\documentclass[10pt,reqno]{amsart}
\usepackage{amsmath,amssymb,latexsym,esint,cite,mathrsfs,graphicx}
\usepackage{verbatim,wasysym}
\usepackage[left=1.9cm,right=1.9cm,top=2.3cm,bottom=2.1cm]{geometry}
\usepackage{tikz,enumitem,graphicx, subfig, microtype, color}
\usepackage{epic,eepic}
\usepackage[colorinlistoftodos]{todonotes}

\usepackage[colorlinks=true,urlcolor=blue, citecolor=red,linkcolor=blue,
linktocpage,pdfpagelabels, bookmarksnumbered,bookmarksopen]{hyperref}
\usepackage[hyperpageref]{backref}
\usepackage[english]{babel}

\numberwithin{equation}{section}

\newtheorem{thm}{Theorem}[section]
\newtheorem{lem}[thm]{Lemma}
\newtheorem{cor}[thm]{Corollary}
\newtheorem{Prop}[thm]{Proposition}
\newtheorem{Def}[thm]{Definition}

\newtheorem{assume}[thm]{Assumptions}

\newcommand{\R}{\mathbb{R}}

\begin{document}
	\baselineskip=14pt
	
	\title[critical Hartree equation]{Blow-up asymptotics for a critical Hartree-type Br\'{e}zis--Nirenberg problem in dimension three}
	
     \author[Y. Wu]{Yue Wu}
	\author[M. Yang]{Minbo Yang$^*$}

	\address{Yue Wu
		\newline\indent Department of Mathematics, Zhejiang Normal University, 
		\newline\indent
		Jinhua 321004, People's Republic of China}
	    \email{\href{mailto:yuewu1@zjnu.edu.cn}{yuewu1@zjnu.edu.cn}}
	
	\address{Minbo Yang 
		\newline\indent Department of Mathematics, Zhejiang Normal University, 
		\newline\indent
		Jinhua 321004, People's Republic of China}
	    \email{\href{mailto:mbyang@zjnu.edu.cn}{mbyang@zjnu.edu.cn}}
	

	\thanks{$^*$Minbo Yang is the corresponding author who was partially supported by National Natural Science Foundation of China (12471114) and Natural Science Foundation of Zhejiang Province (LZ26A010002).}

\begin{abstract}
In this paper,  we study the  following critical Hartree problem
\begin{equation}\label{equationabstract}
	\begin{cases} \displaystyle-\Delta u+(Q+\varepsilon V)u= A_{3,\mu} \left(\int_{\Omega}\frac{u^{6-\mu}(y)}{|x-y|^{\mu}}dy\right) u^{5-\mu}
		&\mathrm{~in~}\Omega,\\ \displaystyle u>0&\mathrm{~in~}\Omega,\\u=0&\mathrm{~on~}\partial\Omega,\end{cases}
\end{equation}
where $\Omega\subset\mathbb{R}^3$ is a bounded open set, $0<\mu<2$ and $A_{3,\mu}>0$ is a normalization constant.
Assume that function $Q$ is critical in the sense of Hebey and Vaugon and the solutions $u_{\varepsilon}$ of \eqref{equationabstract} is a minimizing sequence for a nonlocal Sobolev inequality, Under suitable assumptions on $Q$ and $V$, we derive a precise asymptotic expansion of $u_{\varepsilon}$, determine the exact blow-up rate, and identify the concentration point. We also obtain the pointwise blow-up behavior both near and away from the concentration point. Our theorems generalize the main results in [Frank et al., Anal. PDE, 2024] where the authors characterized the blow-up rate of the solutions to the classical Br\'{e}zis--Nirenberg problem and gave a positve answer to Br\'{e}zis and Peletier's third conjecture in dimension three.
	\end{abstract}

\maketitle

\medskip

\section{Introduction and main results}

In this paper, we are going to consider the    Br\'{e}zis--Nirenberg problem for the critical Hartree equation with the Hardy--Littlewood--Sobolev upper critical exponent in dimension three. We are interested in the asymptotic behavior of the solutions for the following critical Hartree equation 
\begin{equation}\label{equation}
	\begin{cases} \displaystyle-\Delta u+(Q+\varepsilon V)u= A_{3,\mu} \left(\int_{\Omega}\frac{u^{6-\mu}(y)}{|x-y|^{\mu}}dy\right) u^{5-\mu}
		&\mathrm{~in~}\Omega,\\ \displaystyle u>0&\mathrm{~in~}\Omega,\\u=0&\mathrm{~on~}\partial\Omega,\end{cases}
\end{equation}
where $\Omega\subset\mathbb{R}^3$ is a bounded open set, $0<\mu<2$, $
	A_{3,\mu}=
	\frac{3\Gamma\left(\frac{6-\mu}{2}\right)}
	{\pi^{3/2}\Gamma\left(\frac{3-\mu}{2}\right)}
$, the
exponent $6-\mu$ is the upper critical exponent in the sense of the
Hardy--Littlewood--Sobolev inequality. In fact,  recall the famous Hardy--Littlewood--Sobolev (HLS for short) inequality:
\begin{Prop}\label{HLS}
 (Hardy--Littlewood--Sobolev inequality). (See \cite{LL}.) Let $t,r>1$ and $0<\mu<N$ with $1/t+\mu/N+1/r=2$, $f\in L^{t}(\mathbb{R}^N)$ and $h\in L^{r}(\mathbb{R}^N)$. There exists a sharp constant $C(t,N,\mu,r)$, independent of $f,h$, such that
\begin{equation}\label{HLS}
\int_{\mathbb{R}^{N}}\int_{\mathbb{R}^{N}}\frac{f(x)h(y)}{|x-y|^{\mu}}dxdy\leq C(t,N,\mu,r) \left\|f\right\|_{t}\left\|h\right\|_{r}.
\end{equation}
If $t=r=2N/(2N-\mu)$, then
$$
 C(t,N,\mu,r)=C(N,\mu)=\pi^{\frac{\mu}{2}}\frac{\Gamma(\frac{N}{2}-\frac{\mu}{2})}{\Gamma(N-\frac{\mu}{2})}\left\{\frac{\Gamma(\frac{N}{2})}{\Gamma(N)}\right\}^{-1+\frac{\mu}{N}}.
$$
\end{Prop}
\noindent 
By \eqref{HLS} and the Sobolev inequality, we are lead to the following nonlocal Sobolev inequality
\begin{equation}\label{Prm}
\int_{\mathbb{R}^N}|\nabla u|^2 dx
\geq S_{HL}\left(\int_{\mathbb{R}^N}(|x|^{-\mu} \ast|u|^{2^*_\mu})|u|^{2^*_\mu} dx\right)^{\frac{1}{2^*_\mu}}, \quad u\in \mathcal{D}^{1,2}(\mathbb{R}^N),
\end{equation}
for some positive constant $S_{HL}$ depending only on $N$ and $\mu$, where $N\geq 3$, $0<\mu<N$. In particular, for $N=3$, $2^*_\mu=6-\mu$.
The optimal constant in \eqref{Prm} is given by
   \begin{equation}\label{a9}
	S_{H,L}
	:=
	\inf_{u\in D^{1,2}(\mathbb{R}^{N})\setminus\{0\}}
	\frac{\displaystyle\int_{\mathbb{R}^{N}}|\nabla u|^{2}dx}
	{\left(
	\displaystyle\int_{\mathbb{R}^{N}}\int_{\mathbb{R}^{N}}
	\frac{|u(x)|^{2_{\mu}^{\ast}}|u(y)|^{2_{\mu}^{\ast}}}
	{|x-y|^{\mu}}dxdy
	\right)^{\frac{1}{2_{\mu}^{\ast}}}}.
\end{equation}
The motivation for the present work comes from the observation that nonlocal inequality \eqref{Prm} has the same invariance as the best constant Sobolev $S$ for the classical Sobolev inequality. As observed in \cite{GY, DY19}, we know
\begin{equation*}
	S_{H,L}
	=
	\frac{S}{C(N,\mu)^{\frac{1}{2_{\mu}^{\ast}}}}.
\end{equation*}
Denote the Aubin--Talenti functions given by
\begin{equation}\label{atb}
	U_{z,\lambda}(x)
	:=
	\left(
	\frac{\lambda}{1+\lambda^{2}|x-z|^{2}}
	\right)^{\frac{N-2}{2}},
	\qquad
	z\in\mathbb{R}^{N},\quad \lambda>0,
\end{equation} and set
\begin{equation*}
	A_{N,\mu}
	:=
	\frac{
	\left[N(N-2)\right]^{\frac{N-\mu+2}{2}}
	S^{\frac{\mu-N}{2}}
}{
	C(N,\mu)
},
\end{equation*}
we know that $	U_{z,\lambda}$ is the unique family of solutions of the following critical Hartree equation 
\begin{equation}\label{nonlocal critical equation}
	-\Delta U_{z,\lambda}
	=
	A_{N,\mu}
	\left(
	\int_{\mathbb{R}^{N}}
	\frac{U_{z,\lambda}^{2_{\mu}^{\ast}}(y)}
	{|x-y|^{\mu}}dy
	\right)
	U_{z,\lambda}^{2_{\mu}^{\ast}-1}
	\qquad\text{in }\mathbb{R}^{N},
\end{equation}
which coulded be regarded as a nonlcoal counterpart of the so called Yamabe equation
\begin{equation}\label{bec}
-\Delta u=|u|^{2^*-2}u\quad \mbox{in}\ \ \mathbb{R}^N,
\end{equation}
	where $2^{\ast}=2N/(N-2)$. In fact, up to constant mulitiple, problem \eqref{bec} is formally obtained from problem \eqref{nonlocal critical equation} by letting $\mu\to 0^{+}$ or $\mu\to N^{-}$, which means the critical Hartree equation is a nonlocal generalization of the Yamabe equation. 
	
	It is well known that equation \eqref{bec} has a rich history due in part to the fact that its solutions serve as building blocks in geometrical variational problems and critical elliptic equations. In the celebrated paper \cite{BN} Br\'{e}zis and Nirenberg studied the following problem:
\begin{equation}\label{local.S1}
\left\{\begin{array}{l}
\displaystyle-\Delta u=|u|^{2^{\ast}-2}u+\lambda u\ \ \mbox{in}\ \ \Omega,\\
\displaystyle u\in H_{0}^{1}(\Omega),
\end{array}
\right.
\end{equation}
where $\Omega \subset\mathbb{R}^N$ is a bounded domain, $2^{\ast}=\frac{2N}{N-2}$, $ \lambda\in (0,\lambda_{1})$, $\lambda_{1}$ is the first Dirichlet eigenvalue of $-\Delta$ on $\Omega$. They proved that: if $N\geq4$ and $\lambda\in(0,\lambda_{1})$, then problem \eqref{local.S1} has a nontrivial solution and if $N=3$ then there exists a constant $\lambda_{\ast}\in(0, \lambda_{1})$ such that for any $\lambda\in(\lambda_{\ast},\lambda_{1})$ problem \eqref{local.S1} has a positive solution. As a special case, if $\Omega$ is the unit ball in $\R^3$, then $\lambda_{\ast}=\frac{\pi^2}{4}$.
For $N\geq4$, Capozzi, Fortunato and Palmieri \cite{CFP} established the existence of a nontrivial solution to \eqref{local.S1} for every $\lambda>0$. Bahri and Coron \cite{BC} further proved the existence of a positive solution when the domain $\Omega$ has nontrivial topology.
Since the famous result of Br\'{e}zis and Nirenberg, blow-up analysis for solutions of critical equations has been widely studied for more than three decades. This type of problem was raised in a pioneering paper by Br\'{e}zis and Peletier\cite{Brezis-Peletier} about the critical problem
\begin{equation}\label{local.equation}
	\left\{
	\begin{array}{ll}
		\displaystyle
		-\Delta u+(Q+\varepsilon V)u
		=N(N-2)u^{\frac{N+2}{N-2}}
		& \mbox{in } \Omega,\\
		u>0
		& \mbox{in } \Omega,\\
		u=0
		& \mbox{on } \partial\Omega,
	\end{array}
	\right.
\end{equation}
where $\Omega\subset\mathbb{R}^3$ is a bounded open set. They presented a detailed study of the
case where  the domain is a ball and raised three conjectures about the blow-up analysis of the sequence of solutions. The first two of which concern the blow-up behavior of solutions to \eqref{local.equation} with $Q=0$. Concerning the case of general open domain, 
Rey \cite{Re} (independently and using different arguments, by Han \cite{Han}) proved that, as $\lambda\rightarrow0$, the solution $u_{\lambda}$ of \eqref{local.S1} which concentrate around a point of the Robin function that cannot be on the boundary of $\Omega$. Musso and Pistoia \cite{MP1, MP2} constructed families of positive solutions for problem \eqref{local.S1}, which blow-up and concentrate at multiple points. Later, Cao, Luo and Peng \cite{CLP} proved that if $\lambda$ is small, problem \eqref{local.S1} has a unique solution provided the domain $\Omega$ is convex and $N\geq 6$ which depends also on Green's function by using various local Poho\v{z}aev identities and blow-up analysis.
For $N\geq4$ and $N=3$, Frank, K\"{o}nig and Kova\v{r}\'{\i}k \cite{FKK2, FKK1} worked directly
with the variational quotients of the Br\'{e}zis-Nirenberg problem rather than with the Euler-Lagrange equation.  They obtained the asymptotics of the energy functional and gave a precise description of the blow-up profile of minimizing sequences. Moreover, they characterized the concentration points as being extrema of a quotient involving the Robin function. For Br\'{e}zis and Peletier's third conjecture, it concerns the blow-up behavior of solutions of \eqref{local.equation} for
certain nonzero $Q$ in the three-dimensional case.
Frank, K\"{o}nig and Kova\v{r}\'{\i}k \cite{FKK3} considered the equation with nonzero $Q$, they gave a positive answer to the third
Br\'{e}zis--Peletier conjecture and obtained precise blow-up asymptotics
for solutions to \eqref{local.equation} under a natural nondegeneracy assumption.
There is a huge literature on blow-up results for problem \eqref{local.S1} and relevant issues, for instance, \cite{BLR, DDM, DFM1, DFM2,DHR, Esp,FKK3, Gro1, MiPi, PS, Re2, Re3, Wei} and references therein.

Due to the above observation, the following  Hartree-type Br\'{e}zis--Nirenberg problem has attracted considerable attention in recent years, 
\begin{equation}\label{CCE}
	\left\{
	\begin{array}{l}
		\displaystyle-\Delta u
		=\left(\int_{\Omega}
		\frac{|u(y)|^{2_{\mu}^{\ast}}}{|x-y|^{\mu}}dy\right)
		|u|^{2_{\mu}^{\ast}-2}u+\varepsilon u
		\quad\mbox{in }\Omega,\\
		\displaystyle u\in H_{0}^{1}(\Omega),
	\end{array}
	\right.
\end{equation}
where $\Omega$ is a bounded domain in $\mathbb{R}^{N}$, $N\geq3$, $0<\mu<N$.
 Gao and Yang \cite{GY} proved that: when $N\geq4$, problem \eqref{CCE} admits a nontrivial solution for $\varepsilon>0$, provided $\varepsilon$ is not a Dirichlet eigenvalue of $-\Delta$; when $N=3$, there exists $\lambda_{\ast}>0$ such that a nontrivial solution exists for every $\varepsilon>\lambda_{\ast}$, where $\varepsilon$ is not a Dirichlet eigenvalue of $-\Delta$. They also proved that \eqref{CCE} has no nontrivial solution when $N\geq3$, $\varepsilon\leq0$, and $\Omega$ is star-shaped. Involving the existence of blowing-up solutions as $\varepsilon\rightarrow0$. For $N\geq 4$,  $Q\equiv 0$ and $V\equiv -1$, Yang and Zhao \cite{YZ} proved that the solution $u_{\varepsilon}$ of  equation \eqref{CCE} blows up exactly at a critical point of the Robin function that cannot be on the boundary of $\Omega$. While $N\geq5$, under suitable assumptions on $\mu$, Yang, Ye and Zhao \cite{YYZ} constructed a family of solutions for \eqref{CCE} that blow up and concentrate at a critical point of the Robin function. Chen and Wang \cite{CW4D} obtained an analogous result for $0<\mu\leq2$ and $N=4$,. 
Squassina, Yang and Zhao \cite{SYZ} investigated the existence and local uniqueness of the blow-up solutions concentrating at the nondegenerate critical point of the Robin function. Consider the perturbation of the domian, let $\Omega$ be a bounded smooth domain in $\R^N$ that contains the origin and define $\Omega_{\varepsilon}:= \Omega\backslash B(0,\varepsilon)$, Ghimenti, Huang and Pistoia \cite{GHP} constructed bubble solution blowing up at the origin to the following  equation
$$
\left\{\begin{array}{l}
	\displaystyle-\Delta u =\Big(\int_{\Omega_{\varepsilon}}\frac{u(y)^{2_{\mu}^{\ast}}}{|x-y|^{\mu}}dy\Big)u^{2_{\mu}^{\ast}-1}\hspace{4.14mm}\mbox{in}\hspace{1.14mm} \Omega_{\varepsilon},\\
	\displaystyle u\in H_{0}^{1}(\Omega_{\varepsilon}).
\end{array}
\right.
$$
where $N\geq 5$, $0 <\mu < 4$ with $\mu$ sufficiently close to 0. (See also \cite{GHP-1}). Recently, Cingolani, Yang and Zhao \cite{CYZ} studied the asymptotic behavior of least energy solutions for the nearly critical nonlocal problem
$$
\left\{\begin{array}{l}
	\displaystyle-\Delta u=\Big(\int_{\Omega}\frac{u^{p-\varepsilon}}{|x-y|^{N-2}}dy\Big)u^{p-1-\varepsilon} \quad \mbox{in}\quad \Omega,\\
	\displaystyle u\in H_{0}^{1}(\Omega),
\end{array}
\right.
$$
where $\Omega$ is a smooth bounded domain in $\mathbb{R}^N$ for $N=3,4,5$, $\varepsilon>0$ is a small parameter and $p=\frac{N+2}{N-2}$ is energy-critical exponent. They studied the exact rates of blow-up and the location of blow-up points of the solution $u_{\epsilon}$ for the above problem. Especially, the location of blowing-up points are a maximum point of the Robin's function of $\Omega$.

However, the study of blow-up solutions in dimension three remains much less developed. Recently, Chen and Wang \cite{CW3D,CW3DM} constructed both single and multiple blowing-up solutions for \eqref{CCE} as $\varepsilon$ approaches a special value $\varepsilon_{0}>0$, which is characterized by the Green function of $-\Delta-\varepsilon$. Gao, Ghimenti, Yang and Zhao \cite{GGYZ} derived precise energy asymptotics for a three-dimensional nonlocal Br\'{e}zis--Nirenberg problem involving the critical Hartree equation. However, the blow-up tate of the solutions of the solutions remains open. it is important to point out that a feature of the three--dimensional setting is the concept of criticality for the function $Q$. We recall $S$ is the best Sobolev constant and define
\begin{equation*}
	S(Q):=
	\inf_{0\neq u\in H_{0}^{1}(\Omega)}
	\frac{\displaystyle\int_{\Omega}\left(|\nabla u|^{2}+Q u^{2}\right)dx}
	{\left(\displaystyle\int_{\Omega}|u(x)|^{6}dx\right)^{\frac{1}{3}}}. 
\end{equation*}
One consequence of the result of Br\'{e}zis and Nirenberg \cite{BN} is that,
in the three-dimensional case, $S(Q)=S$ if $Q$ is sufficiently small, for
instance in $L^{\infty}(\Omega)$, even when $Q\not\equiv0$. The situation is
quite different in dimensions $N\geq4$, where the corresponding analogue
of $S(Q)$ (with the exponent $6$ replaced by $\frac{2N}{N-2}$)  is strictly
below $S$ whenever $Q$ is negative somewhere in $\Omega$. This phenomenon
leads naturally to the following definition due to Hebey and Vaugon
\cite{HV}. 
\begin{Def}\label{sstar}
Let $Q \in C(\overline{\Omega})$. We say that $Q$ is critical in $\Omega$ if $S(Q)=S$ and if for any continuous function $\widetilde{Q}$ on $\overline{\Omega}$ with $\widetilde{Q}\leq Q$ and $\widetilde{Q}\not\equiv Q$ one has $S(\widetilde{Q})<S(Q)$.
\end{Def}
Throughout the paper, we always assume that $Q$ is critical in $\Omega$.
We next introduce the Green function of $-\Delta+Q$ and the associated Robin function, which will play a central role in our analysis. 
 We use the sign and
normalization convention of Rey \cite{Re}. Since $Q$ is critical, the operator
$-\Delta+Q$ with Dirichlet boundary conditions is coercive. Hence, for each
$x\in\Omega$, it admits a
Green function $G_Q(z,x)$ satisfying
\begin{equation}\label{a41}
	\begin{cases}
		-\Delta_zG_Q(z,x)+Q(z)G_Q(z,x)=4\pi\delta_x,
		& \text{in }\Omega,\\
		G_Q(z,x)=0,
		& \text{on }\partial\Omega.
	\end{cases}
\end{equation}
We define the regular part of $G_Q$ by
\begin{equation}\label{a51}
	H_Q(z,x)=\frac{1}{|z-x|}-G_Q(z,x).
\end{equation}
For each $z\in\Omega$, the function $H_Q(z,\cdot)$, initially defined on
$\Omega\setminus\{z\}$, extends continuously to $\Omega$, and we abbreviate
\begin{equation*}
	\phi_Q(z):=H_Q(z,z).
\end{equation*}
Br\'{e}zis  \cite{B} proved that $\inf_{z\in\Omega}\phi_Q(z)<0$  implies $S(Q) < S$. The reverse implication, $S(Q) < S$ implies $\inf_{z\in\Omega}\phi_Q(z)<0$, was proved by Druet  \cite{Druet1}. Therefore, as a consequence of
criticality we have
\begin{equation}
	\inf_{z\in\Omega}\phi_Q(z)=0,  
\end{equation}
which  implies  that each point $z$  with $\phi_Q(z)=0 $ is a critical point of $\phi_Q$. 

Let us summarize the setting in this paper. In the sequel we set
\begin{equation*}
	\mathcal{N}_Q:=\{x\in\Omega:\phi_Q(x)=0\}.
\end{equation*}
The following assumptions are from \cite{FKK3}.
\begin{assume}
(a) $\Omega\subset\mathbb{R}^3$  is a bounded, open set with $C^2$ boundary. 
\\ (b) $V\in C^{0,1}(\overline{\Omega})$ and  $Q\in C^{0,1}(\overline{\Omega})\cap C_{\mathrm{loc}}^{2,\sigma}(\Omega)$  for some $\sigma >0$.   
\\ (c)  $Q$ is critical in $\Omega$. 
\\ (d) Any point in $\mathcal{N}_Q$  is a nondegenerate critical point of $\phi_Q$, that is,  for any $z_0 \in \mathcal{N}_Q$,  the Hessian    $D^2\phi_Q(z_0)$ does not have a zero eigenvalue. 
\\ (e) $Q<0$ on $\mathcal{N}_Q$.
\end{assume}
%\begin{Rem}
%	The regularity conditions in (a) and (b) are quite modest, which can probably be weakened with further work.  Since Frank, K\"{o}nig and  Kova\v{r}\'{\i}k \cite{FKK3} proved  that $\phi_Q \in C^{2}(\Omega)$,     assumption (d) implies that every  point in $\mathcal{N}_Q$ is a critical point of $\phi_Q$.  Assumption (e) is not severe, since  Corollary 2.2 in  \cite{FKK1}   tells us that    any critical $Q$ satisfies $Q \leq 0$ in $\mathcal{N}_Q$.  In particular, it is fulfilled if $Q$ is a negative constant. 
%\end{Rem}

In this paper, we are interested in the behavior of the  solutions $u=u_\varepsilon$, parametrized
by $\varepsilon>0$, of the critical Hartree equation \eqref{equation} under the assumption that the family
$\{u_\varepsilon\}$ forms a minimizing sequence for the 
Hardy--Littlewood--Sobolev inequality. More precisely, we assume
\begin{equation}\label{assume}
	\frac{\displaystyle\int_{\Omega}|\nabla u_\varepsilon|^2dx}
	{\left(
	\displaystyle\int_{\Omega}\int_{\Omega}
	\dfrac{
	u_\varepsilon^{6-\mu}(x)u_\varepsilon^{6-\mu}(y)}
		{|x-y|^\mu}dxdy
	\right)^{\frac{1}{6-\mu}}}
	\longrightarrow S_{H,L}
	\qquad\text{as }\varepsilon\to0.
\end{equation}
Under suitable assumptions on $Q$ and $V$, the existence of solutions to
\eqref{equation} satisfying \eqref{assume} can be established by a
minimization argument; see \cite{GGYZ}. 
Moreover, the characterization of the optimizers for the sharp
Hardy--Littlewood--Sobolev inequality implies that
$u_{\varepsilon}\rightharpoonup0$ weakly in $H_0^1(\Omega)$ and that
$u_{\varepsilon}^6$ converges weakly in the sense of measures to a
multiple of the Dirac measure concentrated at the blow-up point; see
Proposition \ref{prop22}. In this sense, the family
$\{u_{\varepsilon}\}$ blows up.

The leading-order blow-up profile of solutions to \eqref{equation} is
described by the projection onto $H_0^1(\Omega)$ of a solution to the
whole-space equation \eqref{nonlocal critical equation}. For
$z\in  \Omega$ and $\lambda>0$, let $PU_{z,\lambda}\in H_0^1(\Omega)$
be the unique function satisfying
\begin{equation}\label{a12}
	\Delta PU_{z,\lambda}=\Delta U_{z,\lambda}
	\quad\text{in }\Omega,
	\qquad
	PU_{z,\lambda}=0
	\quad\text{on }\partial\Omega.
\end{equation}
Set
\begin{equation*}
	T_{z,\lambda}
	:=\operatorname{span}\left\{
	PU_{z,\lambda},
	\partial_{\lambda}PU_{z,\lambda},
	\partial_{z_1}PU_{z,\lambda},
	\partial_{z_2}PU_{z,\lambda},
	\partial_{z_3}PU_{z,\lambda}
	\right\}, 
\end{equation*}
and let  $T_{z,\lambda}^{\perp}$ be the orthogonal complement of
$T_{z,\lambda}$ in $H_0^1(\Omega)$ with respect to the inner product
\begin{equation*}
	\langle u,v\rangle
		:=\int_{\Omega}\nabla u\cdot\nabla vdx.
\end{equation*}
The orthogonal projections of $H_0^1(\Omega)$ onto $T_{z,\lambda}$ and
$T_{z,\lambda}^{\perp}$ are denoted by $\Pi_{z,\lambda}$ and
$\Pi_{z,\lambda}^{\perp}$, respectively. To state the asymptotic expansion
below, we introduce
\begin{equation}\label{o117}
	\Theta_{V}(z):=\int_{\Omega}V(x)G_{Q}(z,x)^{2}dx,
	\qquad z\in\Omega.
\end{equation}

We are now ready to state our main results. The first result gives an
asymptotic expansion of $u_{\varepsilon}$ with a remainder in
$H_0^1(\Omega)$.
\begin{thm}[Asymptotic expansion of $u_{\varepsilon}$]\label{main}
	Let $\{u_{\varepsilon}\}$ be a family of solutions to \eqref{equation}
	satisfying \eqref{assume}. Then there exist sequences
	$\{z_{\varepsilon}\}\subset\Omega$,
		$\{\lambda_{\varepsilon}\}\subset(0,\infty)$ and
		$\{\alpha_{\varepsilon}\}\subset\mathbb R$, together with functions
		$r_{\varepsilon}\in T_{z_{\varepsilon},\lambda_{\varepsilon}}^{\perp}$,
		such that
	\begin{equation}
		\label{o118}
		u_{\varepsilon}=\alpha_{\varepsilon}\left(
		PU_{z_{\varepsilon},\lambda_{\varepsilon}}
		-\lambda_{\varepsilon}^{-\frac12}
		\Pi_{z_{\varepsilon},\lambda_{\varepsilon}}^{\perp}
		\left(H_Q(z_{\varepsilon},\cdot)-H_0(z_{\varepsilon},\cdot)\right)
		+r_{\varepsilon}\right).
	\end{equation}
	Moreover, there exists a point $z_0\in\mathcal N_Q$ with
	$\Theta_V(z_0)\leq0$ such that, along a subsequence,
	\begin{equation}
		\label{o119}
		|z_{\varepsilon}-z_0|=o\left(\varepsilon^{\frac12}\right),
	\end{equation}
	\begin{equation}
		\label{o120}
		\phi_Q(z_{\varepsilon})=o(\varepsilon),
	\end{equation}
	\begin{equation}
		\label{o121}
		\lim_{\varepsilon\to0}\varepsilon\lambda_{\varepsilon}
		=\frac{4\pi^2|Q(z_0)|}{|\Theta_V(z_0)|},
	\end{equation}
	\begin{equation}
		\label{o122}
		\alpha_{\varepsilon}
		=1+\frac{4\phi_0(z_0)|\Theta_V(z_0)|}
		{3\pi^3|Q(z_0)|}\varepsilon+o(\varepsilon),
	\end{equation}
	and
	\begin{equation}
		\label{o123}
		\|\nabla r_{\varepsilon}\|_2
		=O\left(\varepsilon^{\frac32}\right).
	\end{equation}
	If $\Theta_V(z_0)=0$, the right-hand side of \eqref{o121} is
	understood as $\infty$.
\end{thm}

The second result describes the pointwise blow-up behavior both at the
blow-up point and away from it.
\begin{thm}[Pointwise blow-up behavior]
	\label{main2}
	Let $\{u_{\varepsilon}\}$ be a family of solutions to \eqref{equation}
	satisfying \eqref{assume}, and let $z_0$ be the concentration point given
	by Theorem \ref{main}.
	\begin{itemize}
		\item[(a)] The asymptotic behavior near $z_0$ is given by
		\begin{equation*}
			\lim_{\varepsilon\to0}\varepsilon
			\|u_{\varepsilon}\|_{\infty}^{2}
			=\lim_{\varepsilon\to0}\varepsilon
			|u_{\varepsilon}(z_{\varepsilon})|^{2}
			=\frac{4\pi^{2}|Q(z_0)|}{|\Theta_V(z_0)|}.
		\end{equation*}
		If $\Theta_V(z_0)=0$, the right-hand side is to be understood
		as $\infty$.
		\item[(b)] The asymptotic behavior away from $z_0$ is given by
		\begin{equation*}
			u_{\varepsilon}(z)
			=\lambda_{\varepsilon}^{-\frac12}G_Q(z,z_0)
			+o\left(\lambda_{\varepsilon}^{-\frac12}\right)
		\end{equation*}
		for every fixed $z\in\Omega\setminus\{z_0\}$. The convergence is uniform
		for $z$ away from $z_0$.
	\end{itemize}
\end{thm}

Theorem \ref{main} shows that the leading-order profile of
$u_{\varepsilon}$ is given by the projected bubble
$PU_{z_{\varepsilon},\lambda_{\varepsilon}}$. Such a single-bubble
decomposition is a standard phenomenon in the blow-up analysis of critical
elliptic problems and critical
 Hartree  problems (see \cite{Han,Re,Re2,YZ}). 
 The three-dimensional case is substantially more delicate than the higher-dimensional setting studied by Yang and Zhao \cite{YZ}. Indeed, the criticality of $Q$ implies that
$\inf_{z\in\Omega}\phi_Q(z)=0$, which causes a cancellation of the leading-order term determining the blow-up behavior in higher dimensions. 
To obtain a more precise description of the blow-up parameters, we need to
extract the leading-order correction to the bubble, which takes the form
$\lambda_{\varepsilon}^{-\frac12}
\Pi_{z_{\varepsilon},\lambda_{\varepsilon}}^{\perp}
\left(H_Q(z_{\varepsilon},\cdot)-H_0(z_{\varepsilon},\cdot)\right)$. 
An asymptotic expansion closely related to that in Theorem \ref{main}  is proved in \cite{GGYZ} for energy-minimizing solutions of \eqref{equation}.  Their analysis does not
require the nondegeneracy of $D^2\phi_Q(z_0)$; instead, it assumes that
$\{z\in\mathcal N_Q:\Theta_V(z)<0\}\neq\varnothing$. In addition, the minimizing property yields
the following characterization of the concentration point
\begin{equation*}
	\frac{\Theta_V(z_0)^2}{|Q(z_0)|}
	=
	\sup_{\substack{z\in\mathcal N_Q\\ \Theta_V(z)<0}}
	\frac{\Theta_V(z)^2}{|Q(z)|}.
\end{equation*}
This property is specific to the energy-minimizing setting and is not available in the general setting considered here.

Our proof is based on an iterative improvement of the approximation of
$u_{\varepsilon}$. The main tool is the coercivity estimate from
\cite{GGYZ}, which is stated in Lemma \ref{E40}. 
To arrive at the level of precision stated in Theorem \ref{main}, two iterations are necessary.
The first
iteration, carried out in Section \ref{A first expansion}, shows that the concentration point stays
away from the boundary and gives the sharp $H_0^1(\Omega)$ estimate for
$\alpha_{\varepsilon}^{-1}u_{\varepsilon}
-PU_{z_{\varepsilon},\lambda_{\varepsilon}}$. 
In Section \ref{sec:refining-expansion}, after extracting the subleading
correction
$\lambda_{\varepsilon}^{-\frac12}
\Pi_{z_{\varepsilon},\lambda_{\varepsilon}}^{\perp}
\left(H_Q(z_{\varepsilon},\cdot)-H_0(z_{\varepsilon},\cdot)\right)$,
we carry out the second iteration and obtain a sharp
$H_0^1(\Omega)$ estimate for the remainder $r_{\varepsilon}$. 
Moreover, suitable Pohozaev identities yield precise asymptotic expansions
for $\alpha_{\varepsilon}$ and $\phi_Q(z_{\varepsilon})$, together with an
estimate for $\nabla\phi_Q(z_{\varepsilon})$. These estimates imply that
$z_0\in\mathcal N_Q$ and $\nabla\phi_Q(z_0)=0$.
In Section \ref{sec:proof-main}, we determine the asymptotic behavior of
$\varepsilon \lambda_{\varepsilon}$ and complete the proof of Theorem \ref{main}. In Section \ref{sec:proof-main2}, we use a Moser iteration argument to show that
$\alpha_{\varepsilon}^{-1}u_{\varepsilon}
-PU_{z_{\varepsilon},\lambda_{\varepsilon}}$ is negligible in the
$L^{\infty}$ norm, which leads to Theorem \ref{main2}. Some technical estimates and auxiliary propositions  are
collected in the two appendices.

\section{A first expansion} \label{A first expansion}
The main result of this section is the following preliminary asymptotic expansion of the family of solutions $\{u_\varepsilon\}$, which is important for the proof of Theorem \ref{main}. 
\begin{Prop}\label{prop21}
	Let $\{u_\varepsilon\}$  be a family of solutions to \eqref{equation} satisfying \eqref{assume}. Then we have, for $\varepsilon$ small enough,
	\begin{equation} 
u_{\varepsilon}=\alpha_{\varepsilon}(PU_{z_{\varepsilon},\lambda_{\varepsilon}}+v_{\varepsilon}),
\end{equation}
where $z_{\varepsilon} \in \Omega, \lambda_{\varepsilon} \in \mathbb{R}^+, \alpha_{\varepsilon}\in \mathbb{R}$ and $v_{\varepsilon} \in  T_{z_{\varepsilon},\lambda_{\varepsilon}}^{\bot}$   satisfying 
\begin{equation}
	z_{\varepsilon} \to z_0 \in \Omega,  \quad \alpha_{\varepsilon} \to 1,  \quad\lambda_{\varepsilon} \to \infty,  \quad \|\nabla v_{\varepsilon}\|_2 =O\left(\lambda^{-\frac{1}{2}}_{\varepsilon}\right)  \quad \text{as }\varepsilon \to 0 . 
\end{equation}
\end{Prop}
We first establish a preliminary version of Proposition \ref{prop21}, which is stated in Proposition \ref{prop22}. This qualitative initial expansion is mainly based on the concentration analysis of Yang and Zhao \cite{YZ} for critical Hartree equations on bounded domains, together with the works of Gao and Yang \cite{GY} and Bahri and Coron \cite{BC}. It has the same form as the decomposition in Proposition \ref{prop21}, but the estimate of the remainder is not yet quantitative and the limiting concentration point may initially lie on $\partial\Omega$.

To pass from Proposition \ref{prop22} to Proposition \ref{prop21}, it remains to establish the bound on $\|\nabla v_{\varepsilon}\|_2$ and the fact that the limiting concentration point $z_0$ lies in the interior of $\Omega$. The estimate for $\|\nabla v_{\varepsilon}\|_2$ follows from the coercivity inequality in \cite[Lemma~5.6]{GGYZ} and estimates of the nonlocal terms based on the HLS inequality. To exclude concentration at the boundary, we use a Pohozaev-type identity and the boundary estimates from \cite{Re}.

We divide the proof of Proposition \ref{prop21} into several subsections. The main conclusion of each subsection is stated as a proposition at its beginning.
\subsection{A qualitative initial expansion.}
\begin{Prop}\label{prop22}
	Let $\{u_\varepsilon\}$  be a family of solutions to \eqref{equation} satisfying \eqref{assume}. Then we have, for $\varepsilon$ small enough,
	\begin{equation} \label{2.1}
u_{\varepsilon}=\alpha_{\varepsilon}(PU_{z_{\varepsilon},\lambda_{\varepsilon}}+v_{\varepsilon}),
\end{equation}
where $z_{\varepsilon} \in \Omega, \lambda_{\varepsilon} \in \mathbb{R}^+, \alpha_{\varepsilon}\in \mathbb{R}$ and $v_{\varepsilon} \in  T_{z_{\varepsilon},\lambda_{\varepsilon}}^{\bot}$   satisfying 
\begin{equation}
	z_{\varepsilon} \to z_0 \in \bar{\Omega},  \quad \alpha_{\varepsilon} \to 1,  \quad\lambda_{\varepsilon}d_{\varepsilon} \to \infty,  \quad \|\nabla v_{\varepsilon}\|_2 \to 0  \quad \text{as } \varepsilon \to 0 ,   
\end{equation}
where  $d_{\varepsilon}=\mbox{dist}(z_{\varepsilon},\partial\Omega)$ is the distance between $z_{\varepsilon}$ and the boundary of  $\Omega$.
\end{Prop}
\begin{proof}
	{\bf Step 1.} We first show that $\{u_\varepsilon\}$ is  bounded in	$H_{0}^{1}(\Omega)$ and that $\left\|u_\varepsilon\right\|_6\gtrsim 1$.  Integrating the equation \eqref{equation} against $u_\varepsilon$, we obtain
\begin{equation}\label{2.3}
	\int_\Omega \left(|\nabla u_\varepsilon|^2+(Q+\varepsilon V)u_\varepsilon^2\right)=A_{3,\mu} \int_\Omega \int_\Omega \frac{u_\varepsilon^{6-\mu}(x)u_\varepsilon^{6-\mu}(y)}{|x-y|^{\mu}}dxdy,
\end{equation}
and therefore
\begin{equation*}
	\frac{\int_\Omega|\nabla u_\varepsilon|^2}{\left(  \int_\Omega \int_\Omega \frac{u_\varepsilon^{6-\mu}(x)u_\varepsilon^{6-\mu}(y)}{|x-y|^{\mu}}dxdy \right)^{\frac{1}{6-\mu}}}+
	\frac{\int_\Omega(Q+\varepsilon V)u_\varepsilon^2}{\left(  \int_\Omega \int_\Omega \frac{u_\varepsilon^{6-\mu}(x)u_\varepsilon^{6-\mu}(y)}{|x-y|^{\mu}}dxdy \right)^{\frac{1}{6-\mu}}}
	=A_{3,\mu} \left(\int_\Omega \int_\Omega \frac{u_\varepsilon^{6-\mu}(x)u_\varepsilon^{6-\mu}(y)}{|x-y|^{\mu}}dxdy\right)^{\frac{5-\mu}{6-\mu}}. 
\end{equation*}
On the left side, the first quotient converges by \eqref{assume}  and the second quotient is bounded by H\"older's inequality. Thus $\int_\Omega \int_\Omega \frac{u_\varepsilon^{6-\mu}(x)u_\varepsilon^{6-\mu}(y)}{|x-y|^{\mu}}dxdy$ is bounded. Using \eqref{assume} again, we obtain that $\{u_\varepsilon\}$ is  bounded in	$H_{0}^{1}(\Omega)$. 

Since $Q$ is critical in $\Omega$, it follows that
 \begin{equation*}
	\int_\Omega(|\nabla u_\varepsilon|^2+ Q u_\varepsilon^2) \geq S \left(\int_\Omega u_\varepsilon^6     \right)^{\frac{1}{3}} .
\end{equation*}
Thus  by Sobolev's inequality,  the left side in \eqref{2.3}  is
bounded from below by a constant times $\|u_{\varepsilon}\|_{6}^{2}$. 
On the other hand, the right  side  is bounded from above by $C(\mu) \|u_{\varepsilon}\|_{6}^{12-2\mu}  $ via the HLS inequality. This yields the lower bound on $\|u_{\varepsilon}\|_{6} \gtrsim 1 $. 
\\ {\bf Step 2.} According to Step 1, $\{u_\varepsilon\}$ is  bounded in	$H_{0}^{1}(\Omega)$.  
Then, up to a subsequence, we may assume that $u_\varepsilon\rightharpoonup u_0$ in $H_{0}^{1}(\Omega)$  with $u_0 \in  H_{0}^{1}(\Omega)$ be a solution of 
\begin{equation}\label{equationu0}
	\begin{cases} \displaystyle-\Delta u+Q u= A_{3,\mu} \left(\int_{\Omega}\frac{u^{6-\mu}(y)}{|x-y|^{\mu}}dy\right) u^{5-\mu}
		&\mathrm{~in~}\Omega,\\u=0&\mathrm{~on~}\partial\Omega. \end{cases}
\end{equation}
Our goal is to show that $u_0\equiv0$. By Rellich's lemma, after passing to a subsequence, we can also suppose 
that  $u_\varepsilon\to u_0$ holds almost everywhere.  Moreover,  up to a subsequence,  we may also assume that 
 $\left\|\nabla u_{\varepsilon}\right\|$  has a limit.  Then, by \eqref{assume}, $\int_{\Omega}\int_{\Omega}\frac{|u_\varepsilon(x)|^{6-\mu}|u_\varepsilon(y)|^{6-\mu}}{|x-y|^\mu}dxdy$  has a limit as well and, by Step 1, none of these limits is zero.

By weak convergence,  it follows that
$$ \mathcal{T}:=\lim_{\varepsilon\to0}\int_\Omega|\nabla(u_\varepsilon-u_0)|^2 \quad \text{exists and satisfies}  \quad \lim_{\varepsilon\to0}\int_\Omega|\nabla u_\varepsilon|^2=\int_\Omega|\nabla u_0|^2+\mathcal{T}.$$
Moreover, by Lemma 2.2 in \cite{GY},  we know 
$$ \mathcal{M}:=\lim_{\varepsilon\to0} \int_{\Omega}\int_{\Omega}\frac{|u_\varepsilon(x)-u_0(x)|^{6-\mu}|u_\varepsilon(y)-u_0(y)|^{6-\mu}}{|x-y|^\mu}dxdy$$
exists and satisfies
$$  \lim_{\varepsilon\to0}  \int_{\Omega}\int_{\Omega}\frac{u_\varepsilon^{6-\mu}(x) u_\varepsilon^{6-\mu}(y)}{|x-y|^\mu} dxdy =\int_{\Omega}\int_{\Omega}\frac{|u_0(x)|^{6-\mu}|u_0(y)|^{6-\mu}}{|x-y|^\mu} dxdy +\mathcal{M} .$$
Thus, \eqref{assume} becomes
\begin{equation*}
	S_{H,L} \left(\int_{\Omega}\int_{\Omega}\frac{|u_0(x)|^{6-\mu}|u_0(y)|^{6-\mu}}{|x-y|^\mu} dxdy +\mathcal{M}\right) ^{\frac{1}{6-\mu}} = \int_\Omega|\nabla u_0|^2+\mathcal{T}.
\end{equation*}
We estimate the left  side from above using the elementary inequality
\begin{equation*}
	\left(\int_{\Omega}\int_{\Omega}\frac{|u_0(x)|^{6-\mu}|u_0(y)|^{6-\mu}}{|x-y|^\mu} dxdy +\mathcal{M}\right) ^{\frac{1}{6-\mu}} \leq \left(\int_{\Omega}\int_{\Omega}\frac{|u_0(x)|^{6-\mu}|u_0(y)|^{6-\mu}}{|x-y|^\mu} dxdy \right) ^{\frac{1}{6-\mu}}+\mathcal{M}^{\frac{1}{6-\mu}}.
\end{equation*}
By the definition of $S_{H,L}$, it is easy to see 
\begin{equation*}
	\mathcal{T} \geq S_{H,L} \mathcal{M}^{\frac{1}{6-\mu}}. 
\end{equation*}
Thus, 
\begin{equation*}
	S_{H,L} \left(\int_{\Omega}\int_{\Omega}\frac{|u_0(x)|^{6-\mu}|u_0(y)|^{6-\mu}}{|x-y|^\mu} dxdy \right) ^{\frac{1}{6-\mu}}  \geq \int_\Omega|\nabla u_0|^2. 
\end{equation*}
 This implies $u_{0} \equiv0$. In fact, if $u_{0}\not\equiv0$, then $u_{0}$ would be a minimizer for  $S_{H,L}$. However, by \cite{GY}, we know that the $S_{H,L}$ problem does not have a minimizer in $\Omega\subsetneq\mathbb{R}^3$.
\\ {\bf Step 3.} We consider the energy functional  defined by
\begin{equation*}
	\mathcal{J}_0(u)=\frac{1}{2}\int_{\mathbb{R}^3}|\nabla u|^2 dx-\frac{A_{3,\mu}}{2\cdot(6-\mu)}\int_{\mathbb{R}^3}\int_{\mathbb{R}^3}\frac{u^{6-\mu}(x)u^{6-\mu}(y)}{|x-y|^\mu}dxdy. 
\end{equation*}
According to Step 2 and  Sobolev embedding, we have
$$u_\varepsilon\rightharpoonup 0  \quad \text{in} \quad  H_{0}^{1}(\Omega)  \quad  \text{and}  \quad  u_\varepsilon\to 0  \quad \text{in} \quad  L^2(\Omega),$$
which implies $\int_{\Omega} (Q+\varepsilon V)u_\varepsilon^2 \to 0 $.
Then a direct calculation shows that
\begin{equation*}
	\int_{\Omega}|\nabla u_{\varepsilon}|^{2}dx\to  \frac{S_{H,L}^{\frac{6-\mu}{5-\mu}}}{A_{3,\mu}^{\frac{1}{5-\mu}}}   \quad \text{and} \quad  \left(\int_{\Omega}\int_{\Omega}\frac{u_\varepsilon^{6-\mu}(x) u_\varepsilon^{6-\mu}(y)}{|x-y|^\mu} dxdy \right)^{\frac{1}{6-\mu}}\to \frac{S_{H,L}^{\frac{1}{5-\mu}}}{A_{3,\mu}^{\frac{1}{5-\mu}}}. 
\end{equation*}
Moreover, 
 $\{u_\varepsilon\}$ is a Palais-Smale sequence for $\mathcal{J}_0(u)$, i.e., 
$$\mathcal{J}_0(u_\varepsilon) \to \frac{\left(5-\mu\right)S_{H,L}^{\frac{6-\mu}{5-\mu}}}{\left(12-2\mu\right)A_{3,\mu}^{\frac{1}{5-\mu}}} \quad \text{and} \quad  \mathcal{J}^{\prime}_0(u_\varepsilon)\to 0   \quad \text{as}  \quad \varepsilon \to 0.$$
Then  \cite[Lemma 2.1]{YZ}   asserts that there  exist 
sequences $\sigma_{\varepsilon}$ of $\mathbb{R}^{+}$ tending to zero and $a_\varepsilon$ of points in $\mathbb{R}^3$  such that 
\begin{equation*}
	w_\varepsilon(x):=\sigma_\varepsilon^{\frac{1}{2}}u_\varepsilon(\sigma_\varepsilon(x-a_\varepsilon))
\end{equation*}
converges weakly in $\mathcal{D}^{1,2}(\mathbb{R}^3)$ to 
$w_{0}$, a nontrivial solution of the equation
\begin{equation*}
-\Delta u = A_{3,\mu}\left( \int_{\mathbb{R}^3} \frac{|u(y)|^{6-\mu}}{|x-y|^\mu} dy \right) |u|^{4-\mu} u \quad \text{in } \mathbb{R}^3,
\end{equation*}
where  $w_\varepsilon(x)=0$  if $\sigma_\varepsilon(x-a_\varepsilon)\in\Omega^c$. Moreover, there exist $\hat{a} \in \mathbb{R}^3$ and $ \hat{\lambda} \in \mathbb{R}^+$  such that
\begin{equation*}
	w_0=U_{\hat{a},\hat{\lambda}}. 
\end{equation*}
Since $U_{\hat{a},\hat{\lambda}}$ satisfies \eqref{nonlocal critical equation} and is
exactly an optimizer of $S_{H,L}$, we have
\begin{equation*}
\int_{\mathbb{R}^3}|\nabla U_{\hat{a},\hat{\lambda}}|^{2}dx=A_{3,\mu}\int_{\mathbb{R}^3}\int_{\mathbb{R}^3}
\frac{|U_{\hat{a},\hat{\lambda}}(x)|^{6-\mu}|U_{\hat{a},\hat{\lambda}}(y)|^{6-\mu}}{|x-y|^{\mu}}dxdy
=   \frac{S_{H,L}^{\frac{6-\mu}{5-\mu}}}{A_{3,\mu}^{\frac{1}{5-\mu}}}      .
\end{equation*}
Consequently, we obtain
\begin{equation*}
	\begin{split}
		\int_{\mathbb{R}^{3}}|\nabla w_{\varepsilon}|^{2}dx
		&=\int_{\mathbb{R}^{3}}\left|
		\nabla\left(\sigma_\varepsilon^{\frac{1}{2}}
		u_\varepsilon(\sigma_\varepsilon(x-a_\varepsilon))\right)
		\right|^{2}dx \\
		&=\int_{\Omega}|\nabla u_{\varepsilon}|^{2}dx
		\to  \frac{S_{H,L}^{\frac{6-\mu}{5-\mu}}}
		{A_{3,\mu}^{\frac{1}{5-\mu}}} \\
		&=\int_{\mathbb{R}^{3}}|\nabla U_{\hat{a},\hat{\lambda}}|^{2}dx
		=\int_{\mathbb{R}^{3}}|\nabla w_{0}|^{2}dx.
	\end{split}
\end{equation*}
Thus, we obtain $w_\varepsilon  \to w_0$  in $\mathcal{D}^{1,2}(\mathbb{R}^3)$, i.e.,
\begin{equation*}
	\sigma_\varepsilon^{\frac{1}{2}}u_\varepsilon(\sigma_\varepsilon(x-a_\varepsilon)) \to w_0= U_{\hat{a},\hat{\lambda}} \quad \text{in} \quad \mathcal{D}^{1,2}(\mathbb{R}^3)
\end{equation*}
as $\varepsilon \to 0$. Taking $\hbar_{\varepsilon} :=\sigma_\varepsilon (\hat{a}-a_\varepsilon)$ and $\tau_{\varepsilon}:=\frac{\hat{\lambda}}{\sigma_{\varepsilon}}$, then  there exists  $\delta_\varepsilon$ such that 
\begin{equation}\label{27}
	u_\varepsilon=  U_{\hbar_{\varepsilon},\tau_{\varepsilon}}+\delta_\varepsilon, 
\end{equation}
	with  $\delta_\varepsilon \to 0$ in  $\mathcal{D}^{1,2}(\mathbb{R}^3)$  as $ \varepsilon \to 0$. Furthermore, we have 
	\begin{equation}\label{23}
			\int_{\Omega}|\nabla u_{\varepsilon}|^{2}dx\to  \frac{S_{H,L}^{\frac{6-\mu}{5-\mu}}}{A_{3,\mu}^{\frac{1}{5-\mu}}}  =\int_{\mathbb{R}^{3}}|\nabla U_{\hbar_{\varepsilon},\tau_{\varepsilon}} |^2dx. 
	\end{equation}
So we also have $\hbar_{\varepsilon} \to z_0 \in \bar{\Omega}$, $\tau_{\varepsilon} \to +\infty$ as $ \varepsilon \to 0$. Using \eqref{23}, we can deduce
\begin{equation}\label{24}
	\hbar_{\varepsilon} \in \Omega, \quad \tau_{\varepsilon}\mbox{dist}(\hbar_{\varepsilon},\partial\Omega) \to +\infty  \quad \text{as} \quad  \varepsilon \to 0. 
\end{equation}
Combining Lemma \ref{A2} and \eqref{27}, we  know 
\begin{equation*}
	u_\varepsilon -PU_{\hbar_{\varepsilon},\tau_{\varepsilon}}  \to 0 \quad  \text{in} \quad H_0^1(\Omega)\quad \text{as} \quad  \varepsilon \to 0. 
\end{equation*}

From \cite[Proposition 7]{BC}, we know that, if $u_\varepsilon \in H_0^1(\Omega) $ with $ \|u_\varepsilon\|_{H_0^1(\Omega)} = 1 $ such that $ \tilde{\eta} $ is small enough and
\begin{equation*}
	\inf_{\begin{array}{c}(z,\lambda)\in\Omega\times\mathbb{R}^+,\\\lambda d>1/\tilde{\eta}\end{array}}\left\|u_\varepsilon-\frac{PU_{z,\lambda}}{\kappa}\right\|_{H_0^1(\Omega)}<\tilde{\eta},
\end{equation*}
where $d=\operatorname{dist}(z,\partial\Omega)$ and $\kappa=\|U_{z,\lambda}\|_{H_0^1(\mathbb{R}^3)}$ is a positive constant, then the problem 
\begin{equation*}
	Minimize\begin{Vmatrix}u_\varepsilon- \alpha  \frac{PU_{z,\lambda}}{\kappa}\end{Vmatrix}_{H_0^1(\Omega)},
\end{equation*}
with $(\alpha,z,\lambda )$ varies in the open set defined by
\begin{equation*}
	\alpha \in(\frac{1}{2},2) \quad \text{and} \quad \lambda d >\frac{1}{4\tilde{\eta}}, 
\end{equation*}
admits a unique solution $(\alpha_\varepsilon,z_\varepsilon,\lambda_\varepsilon )$ for $\varepsilon$ small enough, i.e.,
\begin{equation}
	\label{28}
	u_{\varepsilon}=\alpha_{\varepsilon}(PU_{z_{\varepsilon},\lambda_{\varepsilon}}+v_{\varepsilon})
\end{equation}
with $v_{\varepsilon} \in  T_{z_{\varepsilon},\lambda_{\varepsilon}}^{\bot}$ and $\left\| v_{\varepsilon}   \right\|_{H_0^1(\Omega)}  \to 0$ as $\varepsilon \to 0$. 

Using \eqref{27} and \eqref{28}, we have
\begin{equation*}
	\int_{\Omega}|\nabla (  U_{\hbar_{\varepsilon},\tau_{\varepsilon}}-   \alpha_{\varepsilon}   PU_{z_{\varepsilon},\lambda_{\varepsilon}}     )|^{2}dx =o(1).
\end{equation*}
Thus \begin{equation}
	\label{29} \alpha_{\varepsilon}  \int_{\Omega}\nabla U_{\hbar_\varepsilon,\tau_\varepsilon}\cdot \nabla PU_{z_\varepsilon,\lambda_\varepsilon} dx = \frac{S_{H,L}^{\frac{6-\mu}{5-\mu}}}{A_{3,\mu}^{\frac{1}{5-\mu}}}  +o(1). 
\end{equation}
On the other hand,  it follows
\begin{equation}\label{210}
	\begin{aligned}
		\int_{\Omega} \nabla U_{\hbar_{\varepsilon},\tau_{\varepsilon}} \cdot \nabla  PU_{z_{\varepsilon},\lambda_{\varepsilon}}  dx &= 
		A_{3,\mu} \int_\Omega \int_\Omega \frac{U_{\hbar_{\varepsilon},\tau_{\varepsilon}}^{6-\mu}(x)U_{\hbar_{\varepsilon},\tau_{\varepsilon}}^{5-\mu}(y)PU_{z_{\varepsilon},\lambda_{\varepsilon}}(y)}{|x-y|^{\mu}}dxdy
		\\ &\leq A_{3,\mu} \int_{\mathbb{R}^{3}} \int_{\mathbb{R}^{3}} \frac{U_{\hbar_{\varepsilon},\tau_{\varepsilon}}^{6-\mu}(x)U_{\hbar_{\varepsilon},\tau_{\varepsilon}}^{5-\mu}(y)U_{z_{\varepsilon},\lambda_{\varepsilon}}(y)}{|x-y|^{\mu}}dxdy
		\\ &=3 \int_{\mathbb{R}^{3}}U_{\hbar_{\varepsilon},\tau_{\varepsilon}}^{5}U_{z_{\varepsilon},\lambda_{\varepsilon}} dx, 
		\end{aligned}
\end{equation}
where in the last equality we used \eqref{1esay}. Defining 
\begin{equation*}
	R_{\tau_{\varepsilon},\lambda_{\varepsilon}}:=\left(\frac{1}{\frac{\tau_{\varepsilon}}{\lambda_{\varepsilon}}+\frac{\lambda_{\varepsilon}}{\tau_{\varepsilon}}+\lambda_{\varepsilon}\tau_{\varepsilon}|z_{\varepsilon}-\hbar_{\varepsilon}|^{2}}\right)^{\frac{1}{2}}, 
\end{equation*}
then we have the following estimate from \cite{Bahri}, 
\begin{equation}\label{211}
	\int_{\mathbb{R}^{3}}U_{\hbar_{\varepsilon},\tau_{\varepsilon}}^{5}U_{z_{\varepsilon},\lambda_{\varepsilon}} dx = O(R_{\tau_{\varepsilon},\lambda_{\varepsilon}}+R_{\tau_{\varepsilon},\lambda_{\varepsilon}}^3). 
\end{equation}
Combining \eqref{29}-\eqref{211} and the fact that $\tau_{\varepsilon}\to +\infty$ and $\hbar_{\varepsilon} \to z_0 \in \bar{\Omega} $ as $\varepsilon\to 0$, we can get 
\begin{equation*}
	\lambda_{\varepsilon} \to +\infty \quad \text{and} \quad z_{\varepsilon} \to z_0 \in \bar{\Omega} \quad \text{as} \quad  \varepsilon\to 0.
\end{equation*}
And using \eqref{24} we find 
\begin{equation*}
	\lambda_{\varepsilon}\mbox{dist}(z_{\varepsilon},\partial\Omega) \to +\infty  \quad \text{as} \quad  \varepsilon \to 0. 
\end{equation*}
Finally, by \cite[Lemma A.1]{YZ}, we can derive that $\alpha_{\varepsilon} \to 1$ as $\varepsilon \to 0$. 
\end{proof}

\subsection{Coercivity.} 

The following coercivity inequality, taken from \cite[Lemma 5.6]{GGYZ},  serves as an essential tool for our later refinement of the expansion of  $u_\varepsilon$.  

\begin{lem}\label{E40}
There are constants $T^{\ast}<\infty$ and $\rho> 0$ such that for all $z\in\Omega$, all $\lambda> 0$ with $d(z)\lambda\geq T^{\ast}$ and all $v\in T_{z,\lambda}^{\bot}$,
$$
\aligned
\int_{\Omega}(|\nabla v|^{2}+Qv^{2})dx
&-A_{3,\mu}\Big((5-\mu)\int_{\Omega}\int_{\Omega}
\frac{U_{z,\lambda}^{6-\mu}(x)
U_{z,\lambda}^{4-\mu}(y)v(y)^{2}}{|x-y|^{\mu}}dxdy\\
&+(6-\mu)\int_{\Omega}\int_{\Omega}
\frac{U_{z,\lambda}^{5-\mu}(x)v(x)
U_{z,\lambda}^{5-\mu}(y)v(y)}{|x-y|^{\mu}}dxdy\Big)
\geq\rho\int_{\Omega}|\nabla v|^{2}dx.
\endaligned$$
\end{lem}

\subsection{The bound on $\boldsymbol{\|\nabla v_\varepsilon\|_2}$. } 
In this subsection, we are going to show:
\begin{Prop}\label{prop24}
	As $\varepsilon \to 0$, 
	\begin{equation}
		\label{2.12}
		\| \nabla v_{\varepsilon}\|_2=O\left(\lambda^{-\frac{1}{2}}_{\varepsilon}\right)+O\left((\lambda_{\varepsilon}d_{\varepsilon})^{-1}\right). 
	\end{equation}
\end{Prop}
In Subsection \ref{subsec:boundary-concentration}, we prove that
$d_{\varepsilon}^{-1}=O(1)$. Substituting this bound into \eqref{2.12},
we obtain
\begin{equation*}
	\|\nabla v_{\varepsilon}\|_2
	=O\left(\lambda_{\varepsilon}^{-\frac12}\right),
\end{equation*}
which is the estimate stated in Proposition \ref{prop21}.
\begin{proof}
	We first find the equation satisfied by $v_{\varepsilon}$. Combining  \eqref{equation} and \eqref{2.1},  we obtain
\begin{equation}
	\label{equation.v}
	\begin{split}
		(-\Delta+Q)v_{\varepsilon} = \Delta PU_{z_{\varepsilon},\lambda_{\varepsilon}}&+
A_{3,\mu} \alpha_{\varepsilon}^{10-2\mu} \left(\int_{\Omega}\frac{(PU_{z_{\varepsilon},\lambda_{\varepsilon}}+v_{\varepsilon})^{6-\mu}(y)}{|x-y|^{\mu}}dy\right) (PU_{z_{\varepsilon},\lambda_{\varepsilon}}+v_{\varepsilon})^{5-\mu} \\ 
&-(Q+\varepsilon V)PU_{z_{\varepsilon},\lambda_{\varepsilon}} -\varepsilon V v_{\varepsilon}. 
	\end{split}
\end{equation}
Integrating this equation against $v_{\varepsilon}$ and using
\begin{equation*}
	\int_{\Omega} \Delta PU_{z_{\varepsilon},\lambda_{\varepsilon}} v_{\varepsilon}=- \int_{\Omega} \nabla PU_{z_{\varepsilon},\lambda_{\varepsilon}} \nabla v_{\varepsilon}=0, 
\end{equation*}
we can obtain 
\begin{equation}
	\label{o28}
	\begin{split}
			\int_{\Omega}(|\nabla v_{\varepsilon}|^2+Qv_{\varepsilon}^2)=&A_{3,\mu} \alpha_{\varepsilon}^{10-2\mu}\int_{\Omega}\int_{\Omega}\frac{(PU_{z_{\varepsilon},\lambda_{\varepsilon}}+v_{\varepsilon})^{6-\mu}(y) (PU_{z_{\varepsilon},\lambda_{\varepsilon}}+v_{\varepsilon})^{5-\mu}(x)v_{\varepsilon}(x)
	}{|x-y|^{\mu}}dxdy
\\   & -	\int_{\Omega} (Q+\varepsilon V)PU_{z_{\varepsilon},\lambda_{\varepsilon}}v_{\varepsilon}
-\int_{\Omega} \varepsilon V v_{\varepsilon}^2. 
\end{split}
\end{equation}

We estimate the three terms on the right-hand side individually.  The second and third terms are simple. By Lemma \ref{A1}   and Sobolev embedding,  we have
\begin{equation*}
	\left| 
		\int_{\Omega} (Q+\varepsilon V)PU_{z_{\varepsilon},\lambda_{\varepsilon}}v_{\varepsilon}
	\right|\lesssim 	\| v_{\varepsilon}\|_{6}\| U_{z_{\varepsilon},\lambda_{\varepsilon}}\|_{\frac{6}{5}}    \lesssim  \lambda_{\varepsilon}^{-\frac{1}{2}}  \|\nabla v_{\varepsilon}\|_{2},
\end{equation*}
\begin{equation*}
	\left| 
		\int_{\Omega} \varepsilon V v_{\varepsilon}^2
	\right|\lesssim  \varepsilon	\| v_{\varepsilon}\|^{2}_{6}
	= o(\|\nabla v_{\varepsilon}\|^{2}_{2}).
\end{equation*}
Then we begin to estimate the first term on the right side of  \eqref{o28}.  First, we define $\varphi_{z_{\varepsilon},\lambda_{\varepsilon}} := U_{z_{\varepsilon},\lambda_{\varepsilon}} -PU_{z_{\varepsilon},\lambda_{\varepsilon}}$, then
\begin{equation*}
	(PU_{z_{\varepsilon},\lambda_{\varepsilon}}+v_{\varepsilon})^{6-\mu}= 
	(U_{z_{\varepsilon},\lambda_{\varepsilon}}+v_{\varepsilon}-\varphi_{z_{\varepsilon},\lambda_{\varepsilon}})^{6-\mu}= U_{z_{\varepsilon},\lambda_{\varepsilon}}^{6-\mu}+(6-\mu)U_{z_{\varepsilon},\lambda_{\varepsilon}}^{5-\mu}(v_{\varepsilon}-\varphi_{z_{\varepsilon},\lambda_{\varepsilon}})+h_1,
\end{equation*}
and 
\begin{equation*}
	(PU_{z_{\varepsilon},\lambda_{\varepsilon}}+v_{\varepsilon})^{5-\mu}= 
	(U_{z_{\varepsilon},\lambda_{\varepsilon}}+v_{\varepsilon}-\varphi_{z_{\varepsilon},\lambda_{\varepsilon}})^{5-\mu}= U_{z_{\varepsilon},\lambda_{\varepsilon}}^{5-\mu}+(5-\mu)U_{z_{\varepsilon},\lambda_{\varepsilon}}^{4-\mu}(v_{\varepsilon}-\varphi_{z_{\varepsilon},\lambda_{\varepsilon}})+h_2,
\end{equation*}
where 
\begin{equation*}
	h_1=O(  U_{z_{\varepsilon},\lambda_{\varepsilon}}^{4-\mu}  v_{\varepsilon}^2+   U_{z_{\varepsilon},\lambda_{\varepsilon}}^{4-\mu} \varphi_{z_{\varepsilon},\lambda_{\varepsilon}} ^2
	+\left|v_{\varepsilon}\right|^{6-\mu}  +\varphi_{z_{\varepsilon},\lambda_{\varepsilon}} ^{6-\mu}),
\end{equation*}
\begin{equation*}
	h_2=O(  U_{z_{\varepsilon},\lambda_{\varepsilon}}^{3-\mu}  v_{\varepsilon}^2+   U_{z_{\varepsilon},\lambda_{\varepsilon}}^{3-\mu} \varphi_{z_{\varepsilon},\lambda_{\varepsilon}} ^2
	+\left|v_{\varepsilon}\right|^{5-\mu}  +\varphi_{z_{\varepsilon},\lambda_{\varepsilon}} ^{5-\mu}). 
\end{equation*}
Then  we have 
\begin{equation}\label{d11}
	\begin{split}
		\int_{\Omega}\int_{\Omega}&\frac{(PU_{z_{\varepsilon},\lambda_{\varepsilon}}+v_{\varepsilon})^{6-\mu}(y) (PU_{z_{\varepsilon},\lambda_{\varepsilon}}+v_{\varepsilon})^{5-\mu}(x)v_{\varepsilon}(x)
	}{|x-y|^{\mu}}dxdy =  	\int_{\Omega}\int_{\Omega}\frac{U_{z_{\varepsilon},\lambda_{\varepsilon}}^{6-\mu}(y) U_{z_{\varepsilon},\lambda_{\varepsilon}}^{5-\mu}(x)v_{\varepsilon}(x)
	}{|x-y|^{\mu}}dxdy
	\\ &+(5-\mu)\int_{\Omega}\int_{\Omega}\frac{U_{z_{\varepsilon},\lambda_{\varepsilon}}^{6-\mu}(y) U_{z_{\varepsilon},\lambda_{\varepsilon}}^{4-\mu}(x)(v_{\varepsilon}-\varphi_{z_{\varepsilon},\lambda_{\varepsilon}})(x)v_{\varepsilon}(x)
	}{|x-y|^{\mu}}dxdy+\int_{\Omega}\int_{\Omega}\frac{U_{z_{\varepsilon},\lambda_{\varepsilon}}^{6-\mu}(y) h_2(x)v_{\varepsilon}(x)
	}{|x-y|^{\mu}}dxdy
	\\ &+(6-\mu) \left( \int_{\Omega}\int_{\Omega}\frac{U_{z_{\varepsilon},\lambda_{\varepsilon}}^{5-\mu}(y) (v_{\varepsilon}-\varphi_{z_{\varepsilon},\lambda_{\varepsilon}})(y)U_{z_{\varepsilon},\lambda_{\varepsilon}}^{5-\mu}(x)v_{\varepsilon}(x)
	+U_{z_{\varepsilon},\lambda_{\varepsilon}}^{5-\mu}(y) (v_{\varepsilon}-\varphi_{z_{\varepsilon},\lambda_{\varepsilon}})(y)h_2(x)v_{\varepsilon}(x)}{|x-y|^{\mu}}dxdy\right)
\\ &+	O\left( \int_{\Omega}\int_{\Omega}\frac{U_{z_{\varepsilon},\lambda_{\varepsilon}}^{5-\mu}(y) (v_{\varepsilon}-\varphi_{z_{\varepsilon},\lambda_{\varepsilon}})(y)U_{z_{\varepsilon},\lambda_{\varepsilon}}^{4-\mu}(x) (v_{\varepsilon}-\varphi_{z_{\varepsilon},\lambda_{\varepsilon}})(x)v_{\varepsilon}(x)
	}{|x-y|^{\mu}}dxdy    \right) 
	\\ &+\int_{\Omega}\int_{\Omega}\frac{h_1(y) (U_{z_{\varepsilon},\lambda_{\varepsilon}}+v_{\varepsilon}-\varphi_{z_{\varepsilon},\lambda_{\varepsilon}})^{5-\mu}(x)v_{\varepsilon}(x)
	}{|x-y|^{\mu}}dxdy.  
	\end{split}
\end{equation}
Since 
\begin{equation*}
	-\Delta PU_{z_{\varepsilon},\lambda_{\varepsilon}} =-\Delta U_{z_{\varepsilon},\lambda_{\varepsilon}}= A_{3,\mu}\left( \int_{\mathbb{R}^3} \frac{U_{z_{\varepsilon},\lambda_{\varepsilon}}^{6-\mu}(y)}{|x-y|^\mu} dy \right) U_{z_{\varepsilon},\lambda_{\varepsilon}}^{5-\mu} , 
\end{equation*}
by the fact that $v_{\varepsilon}$ vanishes on the boundary and  $v_{\varepsilon} \in  T_{z_{\varepsilon},\lambda_{\varepsilon}}^{\bot}$, we obtain
\begin{equation}
	\label{d12}
	\int_{\mathbb{R}^3} \int_{\Omega} \frac{U_{z_{\varepsilon},\lambda_{\varepsilon}}^{6-\mu}(y)
	U_{z_{\varepsilon},\lambda_{\varepsilon}}^{5-\mu}(x) v_{\varepsilon}(x)
	}{|x-y|^\mu} dxdy = \frac{1}{A_{3,\mu}} \int_{\Omega} -\Delta PU_{z_{\varepsilon},\lambda_{\varepsilon}} v_{\varepsilon} 
	=\frac{1}{A_{3,\mu}} \int_{\Omega}  \nabla  PU_{z_{\varepsilon},\lambda_{\varepsilon}} \nabla v_{\varepsilon} =0. 
\end{equation}
Now we begin to estimate each   term on the right-hand side of \eqref{d11}. Combining \eqref{d12} with Lemmas \ref{A1} and \ref{A1RN}, we have  
\begin{equation*}
	\begin{split}
		&\left|\int_{\Omega}\int_{\Omega}\frac{U_{z_{\varepsilon},\lambda_{\varepsilon}}^{6-\mu}(y) U_{z_{\varepsilon},\lambda_{\varepsilon}}^{5-\mu}(x)v_{\varepsilon}(x)
	}{|x-y|^{\mu}}dxdy\right|	
	\\ \leq& \left| \int_{\mathbb{R}^3} \int_{\Omega} \frac{U_{z_{\varepsilon},\lambda_{\varepsilon}}^{6-\mu}(y)
	U_{z_{\varepsilon},\lambda_{\varepsilon}}^{5-\mu}(x) v_{\varepsilon}(x)
	}{|x-y|^\mu} dxdy   \right|+\left| \int_{\mathbb{R}^3\backslash\Omega} \int_{\Omega} \frac{U_{z_{\varepsilon},\lambda_{\varepsilon}}^{6-\mu}(y)
	U_{z_{\varepsilon},\lambda_{\varepsilon}}^{5-\mu}(x) v_{\varepsilon}(x)
	}{|x-y|^\mu} dxdy   \right|
	\\ \lesssim&  \| U_{z_{\varepsilon},\lambda_{\varepsilon}}^{6-\mu}\|_{L^{\frac{6}{6-\mu}}(\mathbb{R}^3\backslash\Omega)}
	\| U_{z_{\varepsilon},\lambda_{\varepsilon}}^{5-\mu} v_{\varepsilon}\|_{\frac{6}{6-\mu}}
	\\ \lesssim& \| U_{z_{\varepsilon},\lambda_{\varepsilon}}\|^{6-\mu}_{L^{6}(\mathbb{R}^3\backslash\Omega)}
	\| U_{z_{\varepsilon},\lambda_{\varepsilon}}^{5-\mu} \|_{\frac{6}{5-\mu}}
	\| v_{\varepsilon}\|_{6}
	\\ \lesssim&  \| U_{z_{\varepsilon},\lambda_{\varepsilon}}\|^{6-\mu}_{L^{6}(\mathbb{R}^3\backslash B_{d_{\varepsilon}}(z_\varepsilon)  )}   \| U_{z_{\varepsilon},\lambda_{\varepsilon}} \|^{5-\mu}_{6}  \|\nabla v_{\varepsilon}\|_{2}
	\\ \lesssim& (\lambda_{\varepsilon}d_{\varepsilon})^{-\frac{6-\mu}{2}}  \|\nabla v_{\varepsilon}\|_{2}
	\lesssim (\lambda_{\varepsilon}d_{\varepsilon})^{-1}  \|\nabla v_{\varepsilon}\|_{2}. 
	\end{split}
\end{equation*}
By H\"older's inequality, the HLS inequality, and Lemmas \ref{A1} and \ref{A2}, we can find
\begin{equation*}
	\begin{split}
			&(5-\mu)\int_{\Omega}\int_{\Omega}\frac{U_{z_{\varepsilon},\lambda_{\varepsilon}}^{6-\mu}(y) U_{z_{\varepsilon},\lambda_{\varepsilon}}^{4-\mu}(x)(v_{\varepsilon}-\varphi_{z_{\varepsilon},\lambda_{\varepsilon}})(x)v_{\varepsilon}(x)
	}{|x-y|^{\mu}}dxdy
	\\ &=(5-\mu)\int_{\Omega}\int_{\Omega}\frac{U_{z_{\varepsilon},\lambda_{\varepsilon}}^{6-\mu}(y) U_{z_{\varepsilon},\lambda_{\varepsilon}}^{4-\mu}(x)v^2_{\varepsilon}(x)
	}{|x-y|^{\mu}}dxdy+O\left((\lambda_{\varepsilon}d_{\varepsilon})^{-1}  \|\nabla v_{\varepsilon}\|_{2} \right),
	\end{split}
\end{equation*}
since 
\begin{equation*}
	\begin{split}
		\left| 
			(5-\mu)\int_{\Omega}\int_{\Omega}\frac{U_{z_{\varepsilon},\lambda_{\varepsilon}}^{6-\mu}(y) U_{z_{\varepsilon},\lambda_{\varepsilon}}^{4-\mu}(x)\varphi_{z_{\varepsilon},\lambda_{\varepsilon}}(x)v_{\varepsilon}(x)
	}{|x-y|^{\mu}}dxdy
		\right| &\lesssim \| U_{z_{\varepsilon},\lambda_{\varepsilon}}^{6-\mu}\|_{\frac{6}{6-\mu}}
		 \|    U_{z_{\varepsilon},\lambda_{\varepsilon}}^{4-\mu}\varphi_{z_{\varepsilon},\lambda_{\varepsilon}}v_{\varepsilon}\|_{\frac{6}{6-\mu}}
		 \\ &\lesssim \| \varphi_{z_{\varepsilon},\lambda_{\varepsilon}} \| _{\infty} \| v_{\varepsilon}\|_6
		 \|    U_{z_{\varepsilon},\lambda_{\varepsilon}}^{4-\mu} \| _{\frac{6}{5-\mu}}
		 \\ &\lesssim \lambda_{\varepsilon}^{-\frac{1}{2}}  d_{\varepsilon}^{-1}   \|\nabla v_{\varepsilon}\|_{2} 
		 \lambda_{\varepsilon}^{-\frac{1}{2}}. 
	\end{split}
\end{equation*}
Similarly,  we can deduce 
\begin{equation*}
	\begin{split}
		&(6-\mu)  \int_{\Omega}\int_{\Omega}\frac{U_{z_{\varepsilon},\lambda_{\varepsilon}}^{5-\mu}(y) (v_{\varepsilon}-\varphi_{z_{\varepsilon},\lambda_{\varepsilon}})(y)U_{z_{\varepsilon},\lambda_{\varepsilon}}^{5-\mu}(x)v_{\varepsilon}(x)}{|x-y|^{\mu}}dxdy
		\\ &=(6-\mu)  \int_{\Omega}\int_{\Omega}\frac{U_{z_{\varepsilon},\lambda_{\varepsilon}}^{5-\mu}(y) v_{\varepsilon}(y)U_{z_{\varepsilon},\lambda_{\varepsilon}}^{5-\mu}(x)v_{\varepsilon}(x)}{|x-y|^{\mu}}dxdy+O\left((\lambda_{\varepsilon}d_{\varepsilon})^{-1}  \|\nabla v_{\varepsilon}\|_{2} \right), 
	\end{split}
\end{equation*}
\begin{equation*}
	\begin{split}
		\int_{\Omega}\int_{\Omega}\frac{U_{z_{\varepsilon},\lambda_{\varepsilon}}^{6-\mu}(y) h_2(x)v_{\varepsilon}(x)
	}{|x-y|^{\mu}}dxdy \lesssim & \| U_{z_{\varepsilon},\lambda_{\varepsilon}}^{6-\mu}\|_{\frac{6}{6-\mu}}   \left(   \| U_{z_{\varepsilon},\lambda_{\varepsilon}}^{3-\mu}v_{\varepsilon}^3\|_{\frac{6}{6-\mu}}+  \| U_{z_{\varepsilon},\lambda_{\varepsilon}}^{3-\mu}\varphi_{z_{\varepsilon},\lambda_{\varepsilon}}^2 v_{\varepsilon}\|_{\frac{6}{6-\mu}}  + \|  \left|v_{\varepsilon}\right|^{6-\mu}  +  \varphi_{z_{\varepsilon},\lambda_{\varepsilon}}^{5-\mu} v_{\varepsilon}        \|_{\frac{6}{6-\mu}} \right)
	\\ \lesssim & \|  v_{\varepsilon}\|_6^3 +\| \varphi_{z_{\varepsilon},\lambda_{\varepsilon}}\|_6^2 \|  v_{\varepsilon}\|_6+\|  v_{\varepsilon}\|_6^{6-\mu} +\| \varphi_{z_{\varepsilon},\lambda_{\varepsilon}}\|_6^{5-\mu}\|  v_{\varepsilon}\|_6
	\\ \lesssim & \|\nabla v_{\varepsilon}\|_{2}^3 +(\lambda_{\varepsilon}d_{\varepsilon})^{-1}  \|\nabla v_{\varepsilon}\|_{2} + \|\nabla v_{\varepsilon}\|_{2}^{6-\mu}  
	\\ = & O\left((\lambda_{\varepsilon}d_{\varepsilon})^{-1}  \|\nabla v_{\varepsilon}\|_{2}   \right)+o\left(     \|\nabla v_{\varepsilon}\|_{2}^2    \right),  
	\end{split}
\end{equation*}
\begin{equation*}
	\begin{split}
		(6-\mu)  \int_{\Omega}\int_{\Omega}\frac{U_{z_{\varepsilon},\lambda_{\varepsilon}}^{5-\mu}(y) (v_{\varepsilon}-\varphi_{z_{\varepsilon},\lambda_{\varepsilon}})(y)h_2(x)v_{\varepsilon}(x)}{|x-y|^{\mu}}dxdy=  O\left((\lambda_{\varepsilon}d_{\varepsilon})^{-1}  \|\nabla v_{\varepsilon}\|_{2}   \right)+o\left(     \|\nabla v_{\varepsilon}\|_{2}^2    \right),  
	\end{split}
\end{equation*}
\begin{equation*}
	\begin{split}
		\int_{\Omega}\int_{\Omega}\frac{U_{z_{\varepsilon},\lambda_{\varepsilon}}^{5-\mu}(y) (v_{\varepsilon}-\varphi_{z_{\varepsilon},\lambda_{\varepsilon}})(y)U_{z_{\varepsilon},\lambda_{\varepsilon}}^{4-\mu}(x) (v_{\varepsilon}-\varphi_{z_{\varepsilon},\lambda_{\varepsilon}})(x)v_{\varepsilon}(x)
	}{|x-y|^{\mu}}dxdy   =  O\left((\lambda_{\varepsilon}d_{\varepsilon})^{-1}  \|\nabla v_{\varepsilon}\|_{2}   \right)+o\left(     \|\nabla v_{\varepsilon}\|_{2}^2    \right),  
	\end{split}
\end{equation*}
and
\begin{equation*}
	\int_{\Omega}\int_{\Omega}\frac{h_1(y) (U_{z_{\varepsilon},\lambda_{\varepsilon}}+v_{\varepsilon}-\varphi_{z_{\varepsilon},\lambda_{\varepsilon}})^{5-\mu}(x)v_{\varepsilon}(x)
	}{|x-y|^{\mu}}dxdy  =  O\left((\lambda_{\varepsilon}d_{\varepsilon})^{-1}  \|\nabla v_{\varepsilon}\|_{2}   \right)+o\left(     \|\nabla v_{\varepsilon}\|_{2}^2    \right).  
\end{equation*}
Thus, the first term on the right side of  \eqref{o28}  becomes 
\begin{equation*}
	\begin{split}
		&A_{3,\mu} \alpha_{\varepsilon}^{10-2\mu}\int_{\Omega}\int_{\Omega}\frac{(PU_{z_{\varepsilon},\lambda_{\varepsilon}}+v_{\varepsilon})^{6-\mu}(y) (PU_{z_{\varepsilon},\lambda_{\varepsilon}}+v_{\varepsilon})^{5-\mu}(x)v_{\varepsilon}(x)
	}{|x-y|^{\mu}}dxdy
	\\  =& A_{3,\mu} (5-\mu)\alpha_{\varepsilon}^{10-2\mu}\int_{\Omega}\int_{\Omega}\frac{U_{z_{\varepsilon},\lambda_{\varepsilon}}^{6-\mu}(y) U_{z_{\varepsilon},\lambda_{\varepsilon}}^{4-\mu}(x)v^2_{\varepsilon}(x)
	}{|x-y|^{\mu}}dxdy \\ &\quad + A_{3,\mu}(6-\mu) \alpha_{\varepsilon}^{10-2\mu} \int_{\Omega}\int_{\Omega}\frac{U_{z_{\varepsilon},\lambda_{\varepsilon}}^{5-\mu}(y) v_{\varepsilon}(y)U_{z_{\varepsilon},\lambda_{\varepsilon}}^{5-\mu}(x)v_{\varepsilon}(x)}{|x-y|^{\mu}}dxdy
	+ O\left((\lambda_{\varepsilon}d_{\varepsilon})^{-1}  \|\nabla v_{\varepsilon}\|_{2}   \right)+o\left(     \|\nabla v_{\varepsilon}\|_{2}^2    \right). 
	\end{split}
\end{equation*}
Combining all the estimates, it follows from \eqref{o28} that 
$$
\aligned
\int_{\Omega}&(|\nabla v_{\varepsilon}|^{2}+Qv_{\varepsilon}^{2})dx
-A_{3,\mu} \alpha_{\varepsilon}^{10-2\mu}\Big((5-\mu)\int_{\Omega}\int_{\Omega}
\frac{U_{z_{\varepsilon},\lambda_{\varepsilon}}^{6-\mu}(x)
U_{z_{\varepsilon},\lambda_{\varepsilon}}^{4-\mu}(y)v_{\varepsilon}^2(y)}{|x-y|^{\mu}}dxdy\\
&+(6-\mu)\int_{\Omega}\int_{\Omega}
\frac{U_{z_{\varepsilon},\lambda_{\varepsilon}}^{5-\mu}(x)v_{\varepsilon}(x)
U_{z_{\varepsilon},\lambda_{\varepsilon}}^{5-\mu}(y)v_{\varepsilon}(y)}{|x-y|^{\mu}}dxdy\Big)
= O\left( \lambda_{\varepsilon}^{-\frac{1}{2}}  \|\nabla v_{\varepsilon}\|_{2}  + (\lambda_{\varepsilon}d_{\varepsilon})^{-1}  \|\nabla v_{\varepsilon}\|_{2}   \right)+ o\left(     \|\nabla v_{\varepsilon}\|_{2}^2    \right) .
\endaligned$$
The coercivity inequality of Lemma \ref{E40} implies that the left side is at least some positive constant multiplied by  $\|\nabla v_{\varepsilon}\|_{2}^2$. Hence, we obtain \eqref{2.12}.
\end{proof}
\subsection{Excluding boundary concentration.}\label{subsec:boundary-concentration}
In this subsection, we are going to show:
\begin{Prop} \label{prop25}
	\begin{equation*}
			d_{\varepsilon}^{-1}= O(1). 
	\end{equation*}
\end{Prop}

We first integrate the equation for $u_\varepsilon$ against $\nabla u_\varepsilon$, to obtain 
\begin{equation}
	\label{o291}
	\int_{\Omega} \left( -\Delta u_{\varepsilon}+(Q+\varepsilon V)u_{\varepsilon}  \right)\nabla u_{\varepsilon}= A_{3,\mu} \int_{\Omega}\int_{\Omega}\frac{u_{\varepsilon}^{6-\mu}(y) u_{\varepsilon}^{5-\mu}(x)\nabla u_{\varepsilon}(x)}{|x-y|^{\mu}}dxdy.  
\end{equation}
Applying the divergence theorem in $\Omega$, we obtain
\begin{equation*}
	\int_{\Omega}  -\Delta u_{\varepsilon}  \nabla u_{\varepsilon} =- \frac{1}{2} \int_{\partial \Omega} n \left( \frac{\partial u_{\varepsilon}}{\partial n}\right) ^2, 
\end{equation*}
and 
\begin{equation*}
	\int_{\Omega} (Q+\varepsilon V)u_{\varepsilon} \nabla u_{\varepsilon}= -\frac{1}{2} \int_{ \Omega} (\nabla(Q+\varepsilon V) )u_{\varepsilon}^2. 
\end{equation*}
On the other hand, we have
\begin{equation*}
	A_{3,\mu} \int_{\Omega}\int_{\Omega}\frac{u_{\varepsilon}^{6-\mu}(y) u_{\varepsilon}^{5-\mu}(x)\nabla u_{\varepsilon}(x)}{|x-y|^{\mu}}dxdy= \frac{ \mu  A_{3,\mu}}{6-\mu}  \int_{\Omega}\int_{\Omega}\frac{u_{\varepsilon}^{6-\mu}(y) u_{\varepsilon}^{6-\mu}(x)(x-y)}{|x-y|^{\mu+2}}dxdy. 
\end{equation*}
Using the symmetry, we can also prove
\begin{equation*}
	A_{3,\mu} \int_{\Omega}\int_{\Omega}\frac{u_{\varepsilon}^{6-\mu}(x) u_{\varepsilon}^{5-\mu}(y)\nabla u_{\varepsilon}(y)}{|x-y|^{\mu}}dxdy= \frac{ \mu  A_{3,\mu}}{6-\mu}  \int_{\Omega}\int_{\Omega}\frac{u_{\varepsilon}^{6-\mu}(y) u_{\varepsilon}^{6-\mu}(x)(y-x)}{|x-y|^{\mu+2}}dxdy. 
\end{equation*}
Hence we obtain
\begin{equation*}
	A_{3,\mu} \int_{\Omega}\int_{\Omega}\frac{u_{\varepsilon}^{6-\mu}(y) u_{\varepsilon}^{5-\mu}(x)\nabla u_{\varepsilon}(x)}{|x-y|^{\mu}}dxdy=0,
\end{equation*}
and then \eqref{o291} becomes the following Pohozaev-type identity 
\begin{equation}\label{o29}
	- \int_{ \Omega} (\nabla(Q+\varepsilon V) )u_{\varepsilon}^2 =\int_{\partial \Omega} n \left( \frac{\partial u_{\varepsilon}}{\partial n}\right) ^2. 
\end{equation}
By inserting the decomposition $u_{\varepsilon}=\alpha_{\varepsilon}(PU_{z_{\varepsilon},\lambda_{\varepsilon}}+v_{\varepsilon})$,  we find 
\begin{equation}
	\label{o210}
	\int_{\partial \Omega} n \left( \frac{\partial PU_{z_{\varepsilon},\lambda_{\varepsilon}}}{\partial n}\right) ^2=-\int_{\partial \Omega}n \left( 2\frac{\partial PU_{z_{\varepsilon},\lambda_{\varepsilon}}}{\partial n}
	\frac{\partial v_{\varepsilon}}{\partial n}
	+   \left(\frac{\partial v_{\varepsilon}}{\partial n} \right)^2\right)- \int_{ \Omega} (\nabla(Q+\varepsilon V) )(PU_{z_{\varepsilon},\lambda_{\varepsilon}}+v_{\varepsilon})^2 . 
\end{equation}
By Lemma \ref{A1},  \eqref{2.12} and $Q,V \in C^1(\overline{\Omega})$, the volume integral is estimated by
\begin{equation}
	\label{o211}
	\left|  \int_{ \Omega} (\nabla(Q+\varepsilon V) )(PU_{z_{\varepsilon},\lambda_{\varepsilon}}+v_{\varepsilon})^2   \right|\lesssim \left\| PU_{z_{\varepsilon},\lambda_{\varepsilon}}   \right\|_2^2+ \left\|v_{\varepsilon}  \right\|_2^2 
	\lesssim \lambda_{\varepsilon}^{-1}  +(\lambda_{\varepsilon} d_{\varepsilon})^{-2}. 
\end{equation}
The function $ \frac{\partial PU_{z_{\varepsilon},\lambda_{\varepsilon}}}{\partial n}$ on the boundary is discussed in Lemma \ref{A3}. Then we estimate the function $ \frac{\partial v_{\varepsilon} }{\partial n}$ on the boundary. 
\begin{lem} \label{lemma26}
	\begin{equation*}
	\int_{\partial \Omega}    \left(\frac{\partial v_{\varepsilon}}{\partial n} \right)^2 =O\left(\lambda_\varepsilon^{-1}d_\varepsilon^{-1}\right)+o\left(\lambda_\varepsilon^{-1}d_\varepsilon^{-2}\right).  
	\end{equation*}
\end{lem}
\begin{proof}
	The following proof relies on the inequality
	\begin{equation}\label{o212}
		\left\|\frac{\partial z}{\partial n}\right\|_{L^2(\partial\Omega)}^2\lesssim\left\|\Delta z\right\|_{L^{\frac{3}{2}}(\Omega)}^2\quad\text{for all }z\in H^2(\Omega)\cap H_0^1(\Omega).
	\end{equation}
This inequality appears in \cite[Appendix C]{Re}, and we refer to \cite{HWY} for its proof.

Equation \eqref{equation.v} for $v_\varepsilon$ can be written as $-\Delta v_\varepsilon = \mathcal{F}$ with 
\begin{equation}
	\label{o213}
	\begin{split}
		\mathcal{F}:=&
A_{3,\mu} \alpha_{\varepsilon}^{10-2\mu} \left(\int_{\Omega}\frac{(PU_{z_{\varepsilon},\lambda_{\varepsilon}}+v_{\varepsilon})^{6-\mu}(y)}{|x-y|^{\mu}}dy\right) (PU_{z_{\varepsilon},\lambda_{\varepsilon}}+v_{\varepsilon})^{5-\mu} \\& -A_{3,\mu}\left( \int_{\mathbb{R}^3} \frac{U_{z_{\varepsilon},\lambda_{\varepsilon}}^{6-\mu}(y)}{|x-y|^\mu} dy \right) U_{z_{\varepsilon},\lambda_{\varepsilon}}^{5-\mu}
-(Q+\varepsilon V)(PU_{z_{\varepsilon},\lambda_{\varepsilon}}+v_{\varepsilon} ). 
	\end{split}
\end{equation}
Moreover, the function $\mathcal{F}$ satisfies the simple pointwise bound
\begin{equation}\label{o215}
	\begin{split}
		\left| \mathcal{F}\right| \lesssim&\int_{\mathbb{R}^3} \frac{U_{z_{\varepsilon},\lambda_{\varepsilon}}^{6-\mu}(y)}{|x-y|^\mu} dyU_{z_{\varepsilon},\lambda_{\varepsilon}}^{5-\mu}+ \int_{\mathbb{R}^3} \frac{U_{z_{\varepsilon},\lambda_{\varepsilon}}^{6-\mu}(y)}{|x-y|^\mu} dy\left|v_{\varepsilon}\right|^{5-\mu}
	+ \int_{\Omega} \frac{\left|v_{\varepsilon}\right|^{6-\mu}(y)}{|x-y|^\mu} dyU_{z_{\varepsilon},\lambda_{\varepsilon}}^{5-\mu} \\ &+\int_{\Omega} \frac{\left|v_{\varepsilon}\right|^{6-\mu}(y)}{|x-y|^\mu} dy\left|v_{\varepsilon}\right|^{5-\mu}
	+ U_{z_{\varepsilon},\lambda_{\varepsilon}}+ \left|v_{\varepsilon}\right|. 
	\end{split}
\end{equation}
In order to make the boundary estimate sharper, we introduce a cut-off function
\begin{equation*}
	\zeta(y):=\chi\left(\frac{y-z_{\varepsilon}}{d_{\varepsilon}}\right), 
\end{equation*}
where $\chi \in [0,1]$ is a smooth function with $\chi \equiv0 $ on $\left\{|y|\leq\frac{1}{2}\right\}$  and 
$\chi \equiv 1 $ on $\left\{|y|\geq 1 \right\}$. 
Then $\zeta v_{\varepsilon} \in H^2(\Omega)\cap H_0^1(\Omega) $ and satisfies 
\begin{equation}
	\label{oo215}
	-\Delta(\zeta v_{\varepsilon})=\zeta \mathcal{F}-2\nabla\zeta\cdot\nabla v_{\varepsilon}-(\Delta\zeta)v_{\varepsilon}.
\end{equation}

Taking inequalities \eqref{o212} \eqref{o215} and equality \eqref{oo215} into  consideration,  we have 
\begin{equation}
	\begin{split}
		\left\| \frac{\partial v_{\varepsilon}}{\partial n}  \right\|_{L^2(\partial\Omega)}^2 =&
	\left\| \frac{\partial (\zeta v_{\varepsilon})}{\partial n}  \right\|_{L^2(\partial\Omega)}^2 
	\lesssim\left\|\Delta (\zeta v_{\varepsilon})\right\|_{L^{\frac{3}{2}}(\Omega)}^2
	\lesssim\left\|\zeta \mathcal{F}-2\nabla\zeta\cdot\nabla v_{\varepsilon}-(\Delta\zeta)v_{\varepsilon}\right\|_{L^{\frac{3}{2}}(\Omega)}^2
	\\ \lesssim& \left\|    \int_{\mathbb{R}^3} \frac{U_{z_{\varepsilon},\lambda_{\varepsilon}}^{6-\mu}(y)}{|x-y|^\mu} dyU_{z_{\varepsilon},\lambda_{\varepsilon}}^{5-\mu}  \zeta\right\|_{L^{\frac{3}{2}}(\Omega)}^2
	+\left\|     \int_{\mathbb{R}^3} \frac{U_{z_{\varepsilon},\lambda_{\varepsilon}}^{6-\mu}(y)}{|x-y|^\mu} dy\left|v_{\varepsilon}\right|^{5-\mu}     \zeta\right\|_{L^{\frac{3}{2}}(\Omega)}^2
	\\ &+\left\|      \int_{\Omega} \frac{\left|v_{\varepsilon}\right|^{6-\mu}(y)}{|x-y|^\mu} dyU_{z_{\varepsilon},\lambda_{\varepsilon}}^{5-\mu}    \zeta\right\|_{L^{\frac{3}{2}}(\Omega)}^2
	+\left\|    \int_{\Omega} \frac{\left|v_{\varepsilon}\right|^{6-\mu}(y)}{|x-y|^\mu} dy\left|v_{\varepsilon}\right|^{5-\mu}     \zeta\right\|_{L^{\frac{3}{2}}(\Omega)}^2
	\\ & +\left\|  U_{z_{\varepsilon},\lambda_{\varepsilon}}   \zeta\right\|_{L^{\frac{3}{2}}(\Omega)}^2 
	+\left\|  \left|v_{\varepsilon}\right|   \zeta\right\|_{L^{\frac{3}{2}}(\Omega)}^2 +\left\|    |\nabla \zeta|  |\nabla  v_{\varepsilon}|  \right\|_{L^{\frac{3}{2}}(\Omega)}^2 +
	\left\|   ( \Delta \zeta )   v_{\varepsilon}  \right\|_{L^{\frac{3}{2}}(\Omega)}^2. 
	\end{split}
\end{equation}
Then we begin to estimate the norms on the right-hand side. Taking equation \eqref{1esay}, Lemma \ref{A1} and the estimate $\|v_{\varepsilon}\|_{6} =O\left(\lambda^{-\frac{1}{2}}_{\varepsilon}\right)+O\left((\lambda_{\varepsilon}d_{\varepsilon})^{-1}\right)$ into consideration,  we obtain
\begin{equation*}
	\left\|    \int_{\mathbb{R}^3} \frac{U_{z_{\varepsilon},\lambda_{\varepsilon}}^{6-\mu}(y)}{|x-y|^\mu} dyU_{z_{\varepsilon},\lambda_{\varepsilon}}^{5-\mu}  \zeta\right\|_{L^{\frac{3}{2}}(\Omega)}^2
	\lesssim 	\left\|  U_{z_{\varepsilon},\lambda_{\varepsilon}}^{5}  \zeta\right\|_{L^{\frac{3}{2}}(\Omega)}^2
	\lesssim 	\left\|  U_{z_{\varepsilon},\lambda_{\varepsilon}}  \right\|_{L^{\frac{15}{2}}( \Omega\backslash B_{d_{\varepsilon}/2}(z_{\varepsilon}))}^{10}
	\lesssim  \lambda_{\varepsilon}^{-5} d_{\varepsilon}^{-6} =o(\lambda_{\varepsilon}^{-1} d_{\varepsilon}^{-2}), 
\end{equation*}
\begin{equation*}
	\left\|     \int_{\mathbb{R}^3} \frac{U_{z_{\varepsilon},\lambda_{\varepsilon}}^{6-\mu}(y)}{|x-y|^\mu} dy\left|v_{\varepsilon}\right|^{5-\mu}     \zeta\right\|_{L^{\frac{3}{2}}(\Omega)}^2
	\lesssim  \left\|    U_{z_{\varepsilon},\lambda_{\varepsilon}}^{\mu}\left|v_{\varepsilon}\right|^{5-\mu}     \zeta\right\|_{L^{\frac{3}{2}}(\Omega)}^2
	\lesssim  \left\|    U_{z_{\varepsilon},\lambda_{\varepsilon}}^{\mu}\right\|_{L^{\frac{6}{\mu}}(\Omega)}^2
	\left\|    \left|v_{\varepsilon}\right|^{5-\mu}     \zeta\right\|_{L^{\frac{6}{4-\mu}}(\Omega)}^2
	\lesssim \left\|    \left|v_{\varepsilon}\right|^{5-\mu}     \zeta\right\|_{L^{\frac{6}{4-\mu}}(\Omega)}^2, 
\end{equation*}
\begin{equation*}
	\left\|  U_{z_{\varepsilon},\lambda_{\varepsilon}}   \zeta\right\|_{L^{\frac{3}{2}}(\Omega)}^2 \lesssim 
	\left\|  U_{z_{\varepsilon},\lambda_{\varepsilon}}  \right\|_{L^{\frac{3}{2}}( \Omega  \backslash B_{d_{\varepsilon}/2}(z_{\varepsilon}))}^{2} \lesssim \lambda_{\varepsilon}^{-1} =O\left(\lambda_\varepsilon^{-1}d_\varepsilon^{-1}\right), 
\end{equation*}
\begin{equation*}
	\left\|  \left|v_{\varepsilon}\right|   \zeta\right\|_{L^{\frac{3}{2}}(\Omega)}^2
		\lesssim \left\| v_{\varepsilon}   \right\|_{L^{6}(\Omega)}^2= O\left(\lambda^{-1}_{\varepsilon}\right)+O\left((\lambda_{\varepsilon}d_{\varepsilon})^{-2}\right)=O\left(\lambda_\varepsilon^{-1}d_\varepsilon^{-1}\right)+o\left(\lambda_\varepsilon^{-1}d_\varepsilon^{-2}\right), 
\end{equation*}
\begin{equation*}
	\left\|    |\nabla \zeta|  |\nabla  v_{\varepsilon}|  \right\|_{L^{\frac{3}{2}}(\Omega)}^2 
	\lesssim \|\nabla v_\varepsilon\|_{L^2(\Omega)}^2 \|\nabla\zeta\|_{L^6(\Omega)}^2\lesssim
	\left(\lambda_\varepsilon^{-1}+(\lambda_\varepsilon d_\varepsilon)^{-2}\right)d_\varepsilon^{-1}
		=O\left(\lambda_\varepsilon^{-1}d_\varepsilon^{-1}\right)+o\left(\lambda_\varepsilon^{-1}d_\varepsilon^{-2}\right),
\end{equation*}
and
\begin{equation*}
	\left\|   ( \Delta \zeta )   v_{\varepsilon}  \right\|_{L^{\frac{3}{2}}(\Omega)}^2
	\lesssim 	\left\|   ( \Delta \zeta )     \right\|_{L^{2}(\Omega)}^2 \left\|      v_{\varepsilon}  \right\|_{L^{6}(\Omega)}^2
	\lesssim
	\left(\lambda_\varepsilon^{-1}+(\lambda_\varepsilon d_\varepsilon)^{-2}\right)d_\varepsilon^{-1}
		=O\left(\lambda_\varepsilon^{-1}d_\varepsilon^{-1}\right)+o\left(\lambda_\varepsilon^{-1}d_\varepsilon^{-2}\right). 
\end{equation*}
Moreover, taking Proposition \ref{A5} into  consideration,  we have 
\begin{equation*}\begin{split}
		\left\|      \int_{\Omega} \frac{\left|v_{\varepsilon}\right|^{6-\mu}(y)}{|x-y|^\mu} dyU_{z_{\varepsilon},\lambda_{\varepsilon}}^{5-\mu}    \zeta\right\|_{L^{\frac{3}{2}}(\Omega)}^2
	&\lesssim \left\|      \int_{\Omega} \frac{\left|v_{\varepsilon}\right|^{6-\mu}(y)}{|x-y|^\mu} dy\right\|_{L^{\frac{6}{\mu}}(\Omega)}^2 \left\|     U_{z_{\varepsilon},\lambda_{\varepsilon}}^{5-\mu}    \zeta\right\|_{L^{\frac{6}{4-\mu}}(\Omega)}^2
	 \\ &\lesssim \left\|    \left|v_{\varepsilon}\right|^{6-\mu}\right\|_{L^{\frac{6}{6-\mu}}(\Omega)}^2
		\left\|     U_{z_{\varepsilon},\lambda_{\varepsilon}}\right\|_{L^{\frac{6(5-\mu)}{4-\mu}}(\Omega \backslash B_{d_{\varepsilon}/2}(z_{\varepsilon}))}^{2(5-\mu)}=o\left(\lambda_\varepsilon^{-1}d_\varepsilon^{-2}\right), 
\end{split}
\end{equation*}
and 
\begin{equation*}
	\begin{split}
		\left\|    \int_{\Omega} \frac{\left|v_{\varepsilon}\right|^{6-\mu}(y)}{|x-y|^\mu} dy\left|v_{\varepsilon}\right|^{5-\mu}     \zeta\right\|_{L^{\frac{3}{2}}(\Omega)}^2
&\lesssim \left\|      \int_{\Omega} \frac{\left|v_{\varepsilon}\right|^{6-\mu}(y)}{|x-y|^\mu} dy\right\|_{L^{\frac{6}{\mu}}(\Omega)}^2 \left\|     \left|v_{\varepsilon}\right|^{5-\mu}     \zeta\right\|_{L^{\frac{6}{4-\mu}}(\Omega)}^2
 \\ &\lesssim \left\|    \left|v_{\varepsilon}\right|^{6-\mu}\right\|_{L^{\frac{6}{6-\mu}}(\Omega)}^2 \left\|    \left|v_{\varepsilon}\right|^{5-\mu}     \zeta\right\|_{L^{\frac{6}{4-\mu}}(\Omega)}^2 =o\left(  \left\|    \left|v_{\varepsilon}\right|^{5-\mu}     \zeta\right\|_{L^{\frac{6}{4-\mu}}(\Omega)}^2    \right). 
	\end{split}
\end{equation*}
It remains to bound the term  $\left\|    \left|v_{\varepsilon}\right|^{5-\mu}     \zeta\right\|_{L^{\frac{6}{4-\mu}}(\Omega)}^2 $. 
This term is difficult to estimate,  since  $(5-\mu)\cdot \frac{6}{4-\mu} >6$. We now proceed to estimate it as follows. Firstly, we multiply the equation $-\Delta v_\varepsilon = \mathcal{F}$  by $ \zeta ^{\frac{2}{4-\mu}}\left| v_{\varepsilon}\right|^{\frac{2}{4-\mu}} v_{\varepsilon}$ and integrate over $\Omega$ to obtain
\begin{equation}
	\label{o216}
	\int_{\Omega}\nabla (\zeta ^{\frac{2}{4-\mu}}\left| v_{\varepsilon}\right|^{\frac{2}{4-\mu}} v_{\varepsilon}) \cdot \nabla v_{\varepsilon} \leq \int_{\Omega} \left| \mathcal{F}  \right| \zeta ^{\frac{2}{4-\mu}} \left| v_{\varepsilon}\right|^{\frac{6-\mu}{4-\mu}}. 
\end{equation}
Now we claim that there are universal constants $c > 0$ and $C < \infty$ such that, pointwise a.e.,
\begin{equation}
	\label{o217}
	\nabla (\zeta ^{\frac{2}{4-\mu}}\left| v_{\varepsilon}\right|^{\frac{2}{4-\mu}} v_{\varepsilon}) \cdot \nabla v_{\varepsilon}
	\geq c \left|   \nabla \left(\zeta ^{\frac{1}{4-\mu}}  \left| v_{\varepsilon}\right|^{\frac{1}{4-\mu}} v_{\varepsilon}   \right)  \right|^{2} -C \left| v_{\varepsilon}\right|^{\frac{10-2\mu}{4-\mu}} \left|   \nabla \left(\zeta ^{\frac{1}{4-\mu}}    \right)  \right|^{2}. 
\end{equation}
Actually, using the product rule and chain rule, we find 
\begin{equation*}
	\begin{split}
		\nabla (\zeta ^{\frac{2}{4-\mu}}\left| v_{\varepsilon}\right|^{\frac{2}{4-\mu}} v_{\varepsilon}) \cdot \nabla v_{\varepsilon} =&\frac{(6-\mu)(4-\mu)}{(5-\mu)^2} \left|   \nabla \left(\zeta ^{\frac{1}{4-\mu}}  \left| v_{\varepsilon}\right|^{\frac{1}{4-\mu}} v_{\varepsilon}   \right)  \right|^{2}-\frac{(4-\mu)^2}{(5-\mu)^2}\left| v_{\varepsilon}\right|^{\frac{10-2\mu}{4-\mu}} \left|   \nabla \left(\zeta ^{\frac{1}{4-\mu}}    \right)  \right|^{2}
		\\  &-\frac{2(4-\mu)}{(5-\mu)^2}\left| v_{\varepsilon}\right|^{\frac{1}{4-\mu}} v_{\varepsilon}\nabla \left(\zeta ^{\frac{1}{4-\mu}}    \right) \cdot \nabla \left(\zeta ^{\frac{1}{4-\mu}}  \left| v_{\varepsilon}\right|^{\frac{1}{4-\mu}} v_{\varepsilon}   \right) . 
	\end{split}
\end{equation*} 
By applying Schwarz's inequality $v_1\cdot v_2\geq-\eta|v_1|^2-|v_2|^2/(4\eta)$ to the cross term on the right-hand side and choosing  $\eta >0$ sufficiently small, we obtain the desired inequality \eqref{o217}.

Using \eqref{o217}, we can estimate the left side in \eqref{o216} from below by
\begin{equation*}
	\int_{\Omega}\nabla (\zeta ^{\frac{2}{4-\mu}}\left| v_{\varepsilon}\right|^{\frac{2}{4-\mu}} v_{\varepsilon}) \cdot \nabla v_{\varepsilon}
	\geq c \int_{\Omega}\left|   \nabla \left(\zeta ^{\frac{1}{4-\mu}}  \left| v_{\varepsilon}\right|^{\frac{1}{4-\mu}} v_{\varepsilon}   \right)  \right|^{2} -C \int_{\Omega}\left| v_{\varepsilon}\right|^{\frac{10-2\mu}{4-\mu}} \left|   \nabla \left(\zeta ^{\frac{1}{4-\mu}}    \right)  \right|^{2}. 
\end{equation*}
Combining the Sobolev inequality for the function $\zeta ^{\frac{1}{4-\mu}}  \left| v_{\varepsilon}\right|^{\frac{1}{4-\mu}} v_{\varepsilon} $ and \eqref{o216}, we observe
\begin{equation}
	\label{o218}
	\begin{split}
		\left\|    \left|v_{\varepsilon}\right|^{5-\mu}     \zeta\right\|_{L^{\frac{6}{4-\mu}}(\Omega)}^2 
=&\left( \int_{\Omega} \left| \zeta ^{\frac{1}{4-\mu}}  \left| v_{\varepsilon}\right|^{\frac{1}{4-\mu}} v_{\varepsilon} \right|  ^{6} \right)^{\frac{4-\mu}{3}} \lesssim 
\left(  \int_{\Omega}\left|   \nabla \left(\zeta ^{\frac{1}{4-\mu}}  \left| v_{\varepsilon}\right|^{\frac{1}{4-\mu}} v_{\varepsilon}   \right)  \right|^{2}  \right) ^{4-\mu}
\\
\lesssim & \left( \int_{\Omega}\left| v_{\varepsilon}\right|^{\frac{10-2\mu}{4-\mu}} \left|   \nabla \left(\zeta ^{\frac{1}{4-\mu}}    \right)  \right|^{2}\right)^{4-\mu}
+\left( \int_{\Omega} \left| \mathcal{F}  \right| \zeta ^{\frac{2}{4-\mu}} \left| v_{\varepsilon}\right|^{\frac{6-\mu}{4-\mu}} \right)^{4-\mu}. 
	\end{split}
\end{equation}
As for the first term on the right side, we have 
\begin{equation*}
	\begin{split}
		 \left( \int_{\Omega}\left| v_{\varepsilon}\right|^{\frac{10-2\mu}{4-\mu}} \left|   \nabla \left(\zeta ^{\frac{1}{4-\mu}}    \right)  \right|^{2}\right)^{4-\mu}
	&\lesssim \left\| v_{\varepsilon}   \right\|_{6}^{10-2\mu}   \left( \int_{\Omega} \left|   \nabla \left(\zeta ^{\frac{1}{4-\mu}}    \right)  \right|^{\frac{24-6\mu}{7-2\mu}}\right)^{\frac{7-2\mu}{3}}
	\\ &\lesssim \left(\lambda_\varepsilon^{-5+\mu} +\lambda_\varepsilon^{-10+2\mu} d_\varepsilon^{-10+2\mu}\right) d_\varepsilon^{-1}
		=O\left(\lambda_\varepsilon^{-1}d_\varepsilon^{-1}\right)+o\left(\lambda_\varepsilon^{-1}d_\varepsilon^{-2}\right). 
	\end{split}
\end{equation*}
In order to estimate  the second term on the right side of \eqref{o218}, we take  the pointwise estimate  \eqref{o215} into consideration. For the  contribution of the terms $\int_{\mathbb{R}^3} \frac{U_{z_{\varepsilon},\lambda_{\varepsilon}}^{6-\mu}(y)}{|x-y|^\mu} dy\left|v_{\varepsilon}\right|^{5-\mu}$  and $\int_{\Omega} \frac{\left|v_{\varepsilon}\right|^{6-\mu}(y)}{|x-y|^\mu} dy\left|v_{\varepsilon}\right|^{5-\mu}$,  we have the following estimates:
\begin{equation*}
	\begin{split}
		\left( \int_{\mathbb{R}^3}  \int_{\Omega} \frac{U_{z_{\varepsilon},\lambda_{\varepsilon}}^{6-\mu}(y) \left|v_{\varepsilon}(x)\right|^{5-\mu+\frac{6-\mu}{4-\mu}}   \zeta ^{\frac{2}{4-\mu}}(x)}{|x-y|^\mu} dxdy \right)^{4-\mu}= &
		\left(  \frac{3}{A_{3,\mu}}  \int_{\Omega} U_{z_{\varepsilon},\lambda_{\varepsilon}}^{\mu} \left|v_{\varepsilon}\right|^{4-\mu}   \zeta ^{\frac{2}{4-\mu}}  \left|v_{\varepsilon}\right|^{\frac{10-2\mu}{4-\mu}}   \right)^{4-\mu}
		\\ \lesssim & \left( \left\| U_{z_{\varepsilon},\lambda_{\varepsilon}}^{\mu} \left|v_{\varepsilon}\right|^{4-\mu}  \right\|_{\frac{3}{2}}    \left\|  \zeta ^{\frac{2}{4-\mu}}  \left|v_{\varepsilon}\right|^{\frac{10-2\mu}{4-\mu}}  \right\|_{3}      \right)^{4-\mu}
		\\ \lesssim & \left\| \left|v_{\varepsilon}\right|^{4-\mu}  \right\|_{\frac{6}{4-\mu}}^{4-\mu} 
		  \left\| U_{z_{\varepsilon},\lambda_{\varepsilon}}^{\mu} \right\|_{\frac{6}{\mu}}^{4-\mu}
		\left\|    \left|v_{\varepsilon}\right|^{5-\mu}     \zeta\right\|_{\frac{6}{4-\mu}}^2 =o\left(\left\|    \left|v_{\varepsilon}\right|^{5-\mu}     \zeta\right\|_{\frac{6}{4-\mu}}^2\right)	 
	\end{split}
\end{equation*}
and  
\begin{equation*}
	\begin{split}
		\left( \int_{\Omega}  \int_{\Omega} \frac{\left|v_{\varepsilon}(y) \right|^{6-\mu}\left|v_{\varepsilon}(x)\right|^{5-\mu+\frac{6-\mu}{4-\mu}}   \zeta ^{\frac{2}{4-\mu}}(x)}{|x-y|^\mu} dxdy \right)^{4-\mu}
		&\lesssim \left\|  \int_{\Omega}\frac{\left|v_{\varepsilon}(y) \right|^{6-\mu}}{|x-y|^\mu} dy\right\|_{\frac{6}{\mu}} ^{4-\mu}
		\left\|  \left|v_{\varepsilon}\right|^{5-\mu+\frac{6-\mu}{4-\mu}}   \zeta ^{\frac{2}{4-\mu}}\right\|_{\frac{6}{6-\mu}} ^{4-\mu}
		\\ &\lesssim \left\| \left|v_{\varepsilon}\right| ^{6-\mu}\right\|_{\frac{6}{6-\mu}} ^{4-\mu} \left\|  \zeta ^{\frac{2}{4-\mu}}  \left|v_{\varepsilon}\right|^{\frac{10-2\mu}{4-\mu}}  \right\|_{3}   ^{4-\mu}
		\left\| \left|v_{\varepsilon}\right| ^{4-\mu}\right\|_{\frac{6}{4-\mu}} ^{4-\mu}
		\\ &\lesssim \left\| v_{\varepsilon}\right\|_{6} ^{(4-\mu)(6-\mu)+(4-\mu)^2}\left\|    \left|v_{\varepsilon}\right|^{5-\mu}     \zeta\right\|_{\frac{6}{4-\mu}}^2=o\left(\left\|    \left|v_{\varepsilon}\right|^{5-\mu}     \zeta\right\|_{\frac{6}{4-\mu}}^2\right),   
	\end{split}
\end{equation*}
which can be absorbed into the left side of \eqref{o218}. Next, we turn to estimating the remaining terms. Firstly, we have
\begin{equation*}
	\begin{split}
		\left( \int_{\mathbb{R}^3}  \int_{\Omega} \frac{U_{z_{\varepsilon},\lambda_{\varepsilon}}^{6-\mu}(y) U_{z_{\varepsilon},\lambda_{\varepsilon}}^{5-\mu}(x)\left|v_{\varepsilon}(x)\right|^{\frac{6-\mu}{4-\mu}}   \zeta ^{\frac{2}{4-\mu}}(x)}{|x-y|^\mu} dxdy \right)^{4-\mu}&= 
			\left(  \frac{3}{A_{3,\mu}} \int_{\Omega}   U_{z_{\varepsilon},\lambda_{\varepsilon}}^{5}\left|v_{\varepsilon}\right|^{\frac{6-\mu}{4-\mu}}   \zeta ^{\frac{2}{4-\mu}} \right)^{4-\mu}    \\ 
			&\lesssim  \left\| \left|v_{\varepsilon}\right|^{\frac{6-\mu}{4-\mu}}\right\|_{\frac{6(4-\mu)}{6-\mu}}^{4-\mu}
			 \left\|  U_{z_{\varepsilon},\lambda_{\varepsilon}}^{5}\zeta ^{\frac{2}{4-\mu}}\right\|_{\frac{6(4-\mu)}{18-5\mu}}^{4-\mu}
			\\ &\lesssim  \left\| v_{\varepsilon} \right\|_{6}^{6-\mu} \left\|  U_{z_{\varepsilon},\lambda_{\varepsilon}}   \right\|_{L^{\frac{30(4-\mu)}{18-5\mu}}(\Omega \backslash B_{d_{\varepsilon}/2}(z_{\varepsilon}))}^{5(4-\mu)}
				 =o\left(\lambda_\varepsilon^{-1}d_\varepsilon^{-2}\right), 
	\end{split}
\end{equation*}
\begin{equation*}
	\begin{split}
		\left( \int_{\Omega}  \int_{\Omega} \frac{\left|v_{\varepsilon}(y) \right|^{6-\mu} U_{z_{\varepsilon},\lambda_{\varepsilon}}^{5-\mu}(x) \left|v_{\varepsilon}(x)\right|^{\frac{6-\mu}{4-\mu}}   \zeta ^{\frac{2}{4-\mu}}(x)}{|x-y|^\mu} dxdy \right)^{4-\mu} &\lesssim \left\|  \int_{\Omega}\frac{\left|v_{\varepsilon}(y) \right|^{6-\mu}}{|x-y|^\mu} dy\right\|_{\frac{6}{\mu}} ^{4-\mu}
		\left\| U_{z_{\varepsilon},\lambda_{\varepsilon}}^{5-\mu} \left|v_{\varepsilon}\right|^{\frac{6-\mu}{4-\mu}}   \zeta ^{\frac{2}{4-\mu}}\right\|_{\frac{6}{6-\mu}} ^{4-\mu}
		\\ &\lesssim \left\| \left|v_{\varepsilon}\right| ^{6-\mu}\right\|_{\frac{6}{6-\mu}} ^{4-\mu} \left\| U_{z_{\varepsilon},\lambda_{\varepsilon}}^{5-\mu}  \zeta ^{\frac{2}{4-\mu}}\right\|_{\frac{6(4-\mu)}{(6-\mu)(3-\mu)}} ^{4-\mu}  \left\| \left|v_{\varepsilon}\right|^{\frac{6-\mu}{4-\mu}}\right\|_{\frac{6(4-\mu)}{6-\mu}} ^{4-\mu}
		\\ &\lesssim \left\| v_{\varepsilon} \right\|_{6}^{(6-\mu)(5-\mu)} \left\|  U_{z_{\varepsilon},\lambda_{\varepsilon}}   \right\|_{L^{\frac{6(4-\mu)(5-\mu)}{(6-\mu)(3-\mu)}}(\Omega \backslash B_{d_{\varepsilon}/2}(z_{\varepsilon}))}^{(5-\mu)(4-\mu)} =o\left(\lambda_\varepsilon^{-1}d_\varepsilon^{-2}\right), 
	\end{split}
\end{equation*}
\begin{equation*}
	\begin{split}
		\left( \int_{\Omega}  U_{z_{\varepsilon},\lambda_{\varepsilon}}   \left|v_{\varepsilon}\right|^{\frac{6-\mu}{4-\mu}}   \zeta ^{\frac{2}{4-\mu}}\right)^{4-\mu} &\lesssim  \left\| \left|v_{\varepsilon}\right|^{\frac{6-\mu}{4-\mu}}\right\|_{\frac{6(4-\mu)}{6-\mu}}^{4-\mu}
			 \left\|  U_{z_{\varepsilon},\lambda_{\varepsilon}}\zeta ^{\frac{2}{4-\mu}}\right\|_{\frac{6(4-\mu)}{18-5\mu}}^{4-\mu}
			\\ &\lesssim  \left\| v_{\varepsilon} \right\|_{6}^{6-\mu} \left\|  U_{z_{\varepsilon},\lambda_{\varepsilon}}   \right\|_{L^{\frac{6(4-\mu)}{18-5\mu}}(\Omega \backslash B_{d_{\varepsilon}/2}(z_{\varepsilon}))}^{(4-\mu)}=O\left(\lambda_\varepsilon^{-1}d_\varepsilon^{-1}\right)+o\left(\lambda_\varepsilon^{-1}d_\varepsilon^{-2}\right), 
	\end{split}
\end{equation*}
and 
\begin{equation*}
	\begin{split}
		\left( \int_{\Omega}    \left|v_{\varepsilon}\right|^{\frac{10-2\mu}{4-\mu}}   \zeta ^{\frac{2}{4-\mu}}\right)^{4-\mu}
		\lesssim \left\| \left|v_{\varepsilon}\right|^{\frac{10-2\mu}{4-\mu}}    \right\|_{\frac{6(4-\mu)}{10-2\mu}}    ^{4-\mu}
		\lesssim \left\| v_{\varepsilon}   \right\|_{6}    ^{10-2\mu}. 
			\end{split}
\end{equation*}
Then we obtain the following estimate	
\begin{equation*}
		\left\|    \left|v_{\varepsilon}\right|^{5-\mu}     \zeta\right\|_{L^{\frac{6}{4-\mu}}(\Omega)}^2=O\left(\lambda_\varepsilon^{-1}d_\varepsilon^{-1}\right)+o\left(\lambda_\varepsilon^{-1}d_\varepsilon^{-2}\right).  
\end{equation*}
Thus Lemma \ref{lemma26} follows. 
\end{proof}

We can now easily complete the proof of the main result of this section.
\begin{proof}[Proof of Proposition \ref{prop25}]
	Taking Lemma \ref{A3}(a) and \eqref{o211} into consideration, \eqref{o210} becomes
\begin{equation*}
	C\lambda_{\varepsilon}^{-1}\left| \nabla\phi_0(z_{\varepsilon})  \right|=O(\lambda_{\varepsilon}^{-1})+o(\lambda_{\varepsilon}^{-1}d_{\varepsilon}^{-2})+O\left(\left\|\frac{\partial PU_{z_{\varepsilon},\lambda_{\varepsilon}}}{\partial n}\right\|_{L^2(\partial\Omega)}\left\|\frac{\partial v_{\varepsilon}}{\partial n}\right\|_{L^2(\partial\Omega)}+\left\|\frac{\partial v_{\varepsilon}}{\partial n}\right\|_{L^2(\partial\Omega)}^2\right)
\end{equation*}
for some constant $C>0$. By Lemmas \ref{A3}(c) and \ref{lemma26}, we obtain
\begin{equation*}
	\begin{split}
		\left\|\frac{\partial PU_{z_{\varepsilon},\lambda_{\varepsilon}}}{\partial n}\right\|_{L^2(\partial\Omega)} \left\|\frac{\partial v_{\varepsilon}}{\partial n}\right\|_{L^2(\partial\Omega)}  +\left\|\frac{\partial v_{\varepsilon}}{\partial n}\right\|_{L^2(\partial\Omega)}^2&=O\left(\lambda_\varepsilon^{-1}d_\varepsilon^{-\frac{3}{2}}\right) + o\left( \lambda_\varepsilon^{-1}d_\varepsilon^{-2}\right)+O\left(\lambda_\varepsilon^{-1}d_\varepsilon^{-1}\right)
		\\ &=O\left(\lambda_\varepsilon^{-1}d_\varepsilon^{-\frac{3}{2}}\right) + o\left( \lambda_\varepsilon^{-1}d_\varepsilon^{-2}\right). 
	\end{split}
\end{equation*}
Thus we have 
\begin{equation*}
	\left| \nabla\phi_0(z_{\varepsilon})  \right|=O\left( d_\varepsilon^{-\frac{3}{2}} \right) +o\left(  d_\varepsilon^{-2}  \right). 
\end{equation*}
From \cite[Equation (2.9)]{Re}, we know $\left| \nabla\phi_0(z_{\varepsilon})  \right|\gtrsim d_\varepsilon^{-2}$. It follows that 
\begin{equation*}
	d_\varepsilon^{-2}= O\left( d_\varepsilon^{-\frac{3}{2}} \right) +o\left(  d_\varepsilon^{-2}  \right),
\end{equation*}
which yields $d_\varepsilon^{-1}=O(1)$. The proof is now complete.
\end{proof}

\subsection{Proof of Proposition \ref{prop21}. } \begin{proof}
	Proposition \ref{prop22} guarantees the existence of the expansion. From Proposition \ref{prop25} we obtain $d_{\varepsilon}^{-1}= O(1)$, which forces $z_0$ to lie in $\Omega$. Applying the bound $d_{\varepsilon}^{-1}= O(1)$ to Proposition \ref{prop24} then yields $\|\nabla v_{\varepsilon}\|_2 =O\left(\lambda^{-\frac{1}{2}}_{\varepsilon}\right)$, exactly as required by Proposition \ref{prop21}. The proof of the proposition is therefore complete. 
\end{proof}

\section{Refining the expansion}\label{sec:refining-expansion}

In this section, we are going to improve the decomposition given in Proposition \ref{prop21}. More precisely, 	we will 	derive a better approximation to $u_{\varepsilon}$, which is given by 
\begin{equation}
	\label{o31}
	\psi_{z_{\varepsilon},\lambda_{\varepsilon}}:=
	PU_{z_{\varepsilon},\lambda_{\varepsilon}}
	-\lambda_{\varepsilon}^{-1/2}(H_Q(z_{\varepsilon},\cdot)-H_0(z_{\varepsilon},\cdot)).
\end{equation}
We define 
\begin{equation}
	\label{o32}
	q_\varepsilon:=v_{\varepsilon}+ \lambda_{\varepsilon}^{-1/2}(H_Q(z_{\varepsilon},\cdot)-H_0(z_{\varepsilon},\cdot)),
\end{equation}
which leads to 
\begin{equation*}
	u_{\varepsilon}=\alpha_{\varepsilon}\left( \psi_{z_{\varepsilon},\lambda_{\varepsilon}}+ q_\varepsilon \right). 
\end{equation*}
Moreover, we can decompose $q_\varepsilon$  as follows
\begin{equation}
	\label{o33}
	q_\varepsilon=s_\varepsilon +r_\varepsilon, 
\end{equation}
where $s_\varepsilon \in T_{z_{\varepsilon},\lambda_{\varepsilon}}$ and $r_\varepsilon \in T_{z_{\varepsilon},\lambda_{\varepsilon}}^{\bot}$  given by 
\begin{equation}
	\label{o34}
	s_\varepsilon:= \Pi_{z_{\varepsilon},\lambda_{\varepsilon}} q_\varepsilon \quad \text{and}
	\quad r_\varepsilon:=\Pi_{z_{\varepsilon},\lambda_{\varepsilon}}^{\bot}       q_\varepsilon. 
\end{equation}
Since $v_{\varepsilon}$ obtained in Proposition \ref{prop21} belongs to $T_{z_{\varepsilon},\lambda_{\varepsilon}}^{\bot}$, we have 
\begin{equation}
	\label{o35}
	s_\varepsilon=  \lambda_{\varepsilon}^{-1/2}\Pi_{z_{\varepsilon},\lambda_{\varepsilon}}(H_Q(z_{\varepsilon},\cdot)-H_0(z_{\varepsilon},\cdot)). 
\end{equation}

The following proposition summarizes the results of this section.
\begin{Prop}
	\label{prop31}
	Let $\{u_{\varepsilon}\}$ be a family of solutions to \eqref{equation}
	satisfying \eqref{assume}. Then, up to a subsequence, there exist
	$z_{\varepsilon}\in\Omega$, $\lambda_{\varepsilon}>0$,
	$\alpha_{\varepsilon}\in\mathbb R$,
	$s_{\varepsilon}\in T_{z_{\varepsilon},\lambda_{\varepsilon}}$ and
	$r_{\varepsilon}\in T_{z_{\varepsilon},\lambda_{\varepsilon}}^{\bot}$
	such that
	\begin{equation}
		\label{o36}
		u_{\varepsilon}
		=\alpha_{\varepsilon}\left(
		\psi_{z_{\varepsilon},\lambda_{\varepsilon}}
		+s_{\varepsilon}+r_{\varepsilon}\right)
	\end{equation}
	and a point $z_0\in\Omega$ such that, in addition to Proposition
	\ref{prop21},
	\begin{equation}
		\label{o37}
		\|\nabla r_{\varepsilon}\|_2
		=O\left(\varepsilon\lambda_{\varepsilon}^{-\frac12}\right),
	\end{equation}
	\begin{equation}
		\label{o37a}
		\phi_Q(z_{\varepsilon})
		=\pi Q(z_{\varepsilon})\lambda_{\varepsilon}^{-1}
		-\frac{\varepsilon}{4\pi}\Theta_V(z_{\varepsilon})
		+o\left(\lambda_{\varepsilon}^{-1}\right)+o(\varepsilon),
	\end{equation}
	\begin{equation}
		\label{o37b}
		|\nabla\phi_Q(z_{\varepsilon})|
		\lesssim\varepsilon^{\nu}
		\qquad\text{for any }\nu<1,
	\end{equation}
	\begin{equation}
		\label{o37c}
		\lambda_{\varepsilon}^{-1}=O(\varepsilon),
	\end{equation}
	and
	\begin{equation}
		\label{o37d}
		\alpha_{\varepsilon}^{10-2\mu}
		=1+\frac{16(10-2\mu)}{3\pi}
		\phi_0(z_{\varepsilon})\lambda_{\varepsilon}^{-1}
		+O\left(\varepsilon\lambda_{\varepsilon}^{-1}\right).
	\end{equation}
\end{Prop}
The expansion \eqref{o37a} will also be needed in the proof of Theorem
\ref{main}. In fact, the gradient estimate \eqref{o37b} will allow us to
show that
$\phi_Q(z_{\varepsilon})=o(\lambda_{\varepsilon}^{-1})+o(\varepsilon)$,
which will then determine the limit of
$\varepsilon\lambda_{\varepsilon}$. The estimates in Proposition
\ref{prop31} are established in the following subsections, and its proof
is completed in Subsection \ref{subsec:proof-prop31}.

\subsection{Bounds on $\boldsymbol{ s_\varepsilon}$.}
In this section, we establish bounds for the function $s_\varepsilon$, as well as for the coefficients $\beta ,\gamma $ and $\delta _i$ defined by
\begin{equation}
	\label{o38}
	s_\varepsilon= \Pi_{z_{\varepsilon},\lambda_{\varepsilon}} q_\varepsilon =\beta \lambda_{\varepsilon}^{-1}PU_{z_{\varepsilon},\lambda_{\varepsilon}}                  +\gamma\partial_{\lambda}PU_{z_{\varepsilon},\lambda_{\varepsilon}}+ \sum_{i=1}^{3}\delta_{i}\lambda_{\varepsilon}^{-3}\partial_{z_{i}}PU_{z_{\varepsilon},\lambda_{\varepsilon}}.
\end{equation}
Then the following estimates hold.
\begin{Prop}
	\label{prop32}
	As $\varepsilon\rightarrow 0$, we have
\begin{equation}
	\label{o39}
	\beta,\gamma,\delta_{i}=O(1).
\end{equation}
Moreover, we have the following estimates
\begin{equation}
	\label{o310}
	\|s_\varepsilon\|_\infty=O(\lambda_\varepsilon^{-\frac{1}{2}}),\quad\|\nabla s_\varepsilon\|_2=O(\lambda_\varepsilon^{-1})\quad and\quad\|s_\varepsilon\|_2=O(\lambda_\varepsilon^{-\frac{3}{2}}),
\end{equation}
and 
\begin{equation}
	\label{o311}
	\|\nabla s_\varepsilon\|_{L^2(\Omega\backslash B_{d_{\varepsilon}/2}(z_{\varepsilon}))}=O(\lambda_\varepsilon^{-\frac{3}{2}}). 
\end{equation}
\end{Prop}
\begin{proof}
	The estimate \eqref{o39} follows from the argument of \cite[Lemma~5.12]{GGYZ}. Once \eqref{o39} is known, the estimates \eqref{o310} and \eqref{o311} can be derived in the same way as \cite[Proposition 3.2]{FKK3}. We omit the details.
\end{proof}
\begin{Prop} 
	\label{prop33}
	As $\varepsilon\rightarrow 0$, we have
	\begin{equation}
	\label{o313}
	\beta =\frac{16}{3\pi}(\phi_Q(z_{\varepsilon})-\phi_0(z_{\varepsilon}))+O(\lambda_\varepsilon^{-1}),\quad\gamma=-\frac{8}{5}\beta+O(\lambda_\varepsilon^{-1}).
\end{equation}
\end{Prop}
\begin{proof}
	The proof follows from \cite[Proposition~3.3]{FKK3}.
\end{proof}

\subsection{The bound on $\boldsymbol{ \| \nabla r_\varepsilon\|_{2}}$. } 
In this section, we  are going to prove:
\begin{Prop}
	\label{prop34}
	As $\varepsilon\rightarrow 0$, we have
	\begin{equation}
	\label{o318}
	\|\nabla r_{\varepsilon}\|_2=O(\lambda_\varepsilon^{-\frac{3}{2}}  + \lambda_{\varepsilon}^{-1}\phi_{Q}(z_{\varepsilon})    + \varepsilon \lambda_\varepsilon^{-\frac{1}{2}}). 
\end{equation}
\end{Prop}

Firstly, we define 
\begin{equation}\label{o316}
	f_{z_{\varepsilon}, \lambda_{\varepsilon}} :=U_{z_{\varepsilon},\lambda_{\varepsilon}}- P U_{z_{\varepsilon},\lambda_{\varepsilon}} - \lambda_{\varepsilon}^{-1/2} H_0(z_{\varepsilon}, \cdot) ,
\end{equation}
and 
\begin{equation}\label{o317}
	g_{z_{\varepsilon},\lambda_{\varepsilon}}(x) := \frac{\lambda_{\varepsilon}^{-1/2}}{|z_{\varepsilon} - x|} - U_{z_{\varepsilon},\lambda_{\varepsilon}}(x). 
\end{equation}
Observing  $\Delta(H_Q(z_{\varepsilon},\cdot)-H_0(z_{\varepsilon},\cdot))=-QG_Q(z_{\varepsilon},\cdot)$, it follows from \eqref{equation.v} that
\begin{equation}
	\label{o319}
	\begin{split}
			(-\Delta+Q)r_{\varepsilon} = &\Delta PU_{z_{\varepsilon},\lambda_{\varepsilon}} +A_{3,\mu} \alpha_{\varepsilon}^{10-2\mu} \left(\int_{\Omega}\frac{(\psi_{z_{\varepsilon},\lambda_{\varepsilon}} + s_{\varepsilon} + r_{\varepsilon})^{6-\mu}(y)}{|x-y|^{\mu}}dy\right) (\psi_{z_{\varepsilon},\lambda_{\varepsilon}} + s_{\varepsilon} + r_{\varepsilon})^{5-\mu} \\
			& + Q(f_{z_{\varepsilon},\lambda_{\varepsilon}} + g_{z_{\varepsilon},\lambda_{\varepsilon}}) - Q s_{\varepsilon} - \varepsilon V(\psi_{z_{\varepsilon},\lambda_{\varepsilon}} + s_{\varepsilon} + r_{\varepsilon}) + \Delta s_{\varepsilon}. 
	\end{split}
\end{equation}
Multiplying by $r_{\varepsilon}$ and integrating over  $\Omega$ and  combining the orthogonality conditions 
\begin{equation*}
	\int_\Omega(\Delta PU_{z_{\varepsilon},\lambda_{\varepsilon}} )r_{\varepsilon}=-\int_\Omega\nabla PU_{z_{\varepsilon},\lambda_{\varepsilon}} \cdot\nabla r_{\varepsilon}=0 \quad \text{and} \quad 	\int_\Omega(\Delta s_{\varepsilon} )r_{\varepsilon}=-\int_\Omega\nabla s_{\varepsilon} \cdot\nabla r_{\varepsilon}=0,
\end{equation*}
we have
\begin{equation}
	\label{o320}
	\begin{split}
		\int_\Omega(|\nabla r_{\varepsilon}|^2 + Q r_{\varepsilon}^2) = &
		A_{3,\mu} \alpha_{\varepsilon}^{10-2\mu}  \int_{\Omega}\int_{\Omega}\frac{(\psi_{z_{\varepsilon},\lambda_{\varepsilon}} + s_{\varepsilon} + r_{\varepsilon})^{6-\mu}(y)(\psi_{z_{\varepsilon},\lambda_{\varepsilon}} + s_{\varepsilon} + r_{\varepsilon})^{5-\mu} (x)r_{\varepsilon}(x)}{|x-y|^{\mu}}dxdy  
		\\& - \int_\Omega Q (s_{\varepsilon} - f_{z_{\varepsilon},\lambda_{\varepsilon}} - g_{z_{\varepsilon},\lambda_{\varepsilon}}) r_{\varepsilon} - \int_\Omega \varepsilon V (\psi_{z_{\varepsilon},\lambda_{\varepsilon}} + s_{\varepsilon} + r_{\varepsilon}) r_{\varepsilon}.
	\end{split}
\end{equation}
For the first  term on the right side of \eqref{o320}, we have the following estimate:
\begin{lem}
	\label{lemma35}
	As $\varepsilon\rightarrow 0$, it holds that
	\begin{equation} \label{lemma35eq}
		\begin{split}
			\bigg|&A_{3,\mu} \alpha_{\varepsilon}^{10-2\mu}  \int_{\Omega}\int_{\Omega}\frac{(\psi_{z_{\varepsilon},\lambda_{\varepsilon}} + s_{\varepsilon} + r_{\varepsilon})^{6-\mu}(y)(\psi_{z_{\varepsilon},\lambda_{\varepsilon}} + s_{\varepsilon} + r_{\varepsilon})^{5-\mu} (x)r_{\varepsilon}(x)}{|x-y|^{\mu}}dxdy  \\& -A_{3,\mu} \alpha_{\varepsilon}^{10-2\mu}\int_{\Omega}\int_{\Omega}
\frac{(5-\mu)U_{z_{\varepsilon},\lambda_{\varepsilon}}^{6-\mu}(x)
U_{z_{\varepsilon},\lambda_{\varepsilon}}^{4-\mu}(y) r_{\varepsilon}^{2}(y)+(6-\mu)U_{z_{\varepsilon},\lambda_{\varepsilon}}^{5-\mu}(x)r_{\varepsilon}(x)
U_{z_{\varepsilon},\lambda_{\varepsilon}}^{5-\mu}(y)r_{\varepsilon}(y)}{|x-y|^{\mu}}
dxdy   \bigg|  \\ &\lesssim
 \left( \lambda_\varepsilon^{-\frac{3}{2}}  + \lambda_{\varepsilon}^{-1}\phi_{Q}(z_{\varepsilon})+\lambda_\varepsilon^{-1}\| r_{\varepsilon} \|_6   + \| r_{\varepsilon} \|_6^2  \right)\| r_{\varepsilon} \|_6 .
		\end{split}
	\end{equation}
\end{lem}
\begin{proof}
	A direct computation gives 
		$\psi_{z_{\varepsilon},\lambda_{\varepsilon}}=U_{z_{\varepsilon},\lambda_{\varepsilon}} -  \lambda_{\varepsilon}^{-1/2} H_Q(z_{\varepsilon}, \cdot)-f_{z_{\varepsilon}, \lambda_{\varepsilon}}$. Then we have the following bound pointwise
\begin{equation*}
		\begin{split}
			(\psi_{z_{\varepsilon}, \lambda_{\varepsilon}} + s_{\varepsilon} + r_{\varepsilon})^{6-\mu} = &U_{z_{\varepsilon}, \lambda_{\varepsilon}}^{6-\mu} + (6-\mu)U_{z_{\varepsilon}, \lambda_{\varepsilon}}^{5-\mu}(s_{\varepsilon} + r_{\varepsilon}) + O\Big(U_{z_{\varepsilon}, \lambda_{\varepsilon}}^{5-\mu}\big(\lambda_{\varepsilon}^{-1/2}     |H_Q(z_{\varepsilon}, \cdot)| + |f_{z_{\varepsilon}, \lambda_{\varepsilon}}|\big) \Big)\\&+O\Big( U_{z_{\varepsilon}, \lambda_{\varepsilon}}^{4-\mu}(r_{\varepsilon}^2 + s_{\varepsilon}^2)+\lambda_{\varepsilon}^{-(6-\mu)/2}     |H_Q(z_{\varepsilon}, \cdot)|^{6-\mu} + |f_{z_{\varepsilon}, \lambda_{\varepsilon}}|^{6-\mu} + |r_{\varepsilon}|^{6-\mu} + |s_{\varepsilon}|^{6-\mu}\Big),
		\end{split}
	\end{equation*}
and
\begin{equation*}
		\begin{split}
			(\psi_{z_{\varepsilon}, \lambda_{\varepsilon}} + s_{\varepsilon} + r_{\varepsilon})^{5-\mu} = &U_{z_{\varepsilon}, \lambda_{\varepsilon}}^{5-\mu} + (5-\mu)U_{z_{\varepsilon}, \lambda_{\varepsilon}}^{4-\mu}(s_{\varepsilon} + r_{\varepsilon}) + O\Big(U_{z_{\varepsilon}, \lambda_{\varepsilon}}^{4-\mu}\big(\lambda_{\varepsilon}^{-1/2}     |H_Q(z_{\varepsilon}, \cdot)| + |f_{z_{\varepsilon}, \lambda_{\varepsilon}}|\big) \Big)\\&+O\Big( U_{z_{\varepsilon}, \lambda_{\varepsilon}}^{3-\mu}(r_{\varepsilon}^2 + s_{\varepsilon}^2)+\lambda_{\varepsilon}^{-(5-\mu)/2}     |H_Q(z_{\varepsilon}, \cdot)|^{5-\mu} + |f_{z_{\varepsilon}, \lambda_{\varepsilon}}|^{5-\mu} + |r_{\varepsilon}|^{5-\mu} + |s_{\varepsilon}|^{5-\mu}\Big).
		\end{split}
	\end{equation*}
Thus 
\begin{equation}
	\label{o321} \Small
	\begin{split}
		&\int_{\Omega}\int_{\Omega}\frac{(\psi_{z_{\varepsilon},\lambda_{\varepsilon}} + s_{\varepsilon} + r_{\varepsilon})^{6-\mu}(y)(\psi_{z_{\varepsilon},\lambda_{\varepsilon}} + s_{\varepsilon} + r_{\varepsilon})^{5-\mu} (x)r_{\varepsilon}(x)}{|x-y|^{\mu}}dxdy \\ &=
		\underbrace{\int_{\Omega}\int_{\Omega} \frac{    U_{z_{\varepsilon}, \lambda_{\varepsilon}  }^{6-\mu}(y)     U_{z_{\varepsilon}, \lambda_{\varepsilon}}^{5-\mu} (x)r_{\varepsilon}(x)}{|x-y|^{\mu}}dxdy }_{:=H_1} +
		  \underbrace{\int_{\Omega}\int_{\Omega} \frac{      (5-\mu)  U_{z_{\varepsilon}, \lambda_{\varepsilon}  }^{6-\mu}(y)U_{z_{\varepsilon}, \lambda_{\varepsilon}}^{4-\mu}(x)(s_{\varepsilon} + r_{\varepsilon})(x)r_{\varepsilon}(x)}{|x-y|^{\mu}}dxdy }_{:=H_2}\\ &+
		\underbrace{\int_{\Omega}\int_{\Omega} \frac{    (6-\mu)U_{z_{\varepsilon}, \lambda_{\varepsilon}}^{5-\mu}(y)(s_{\varepsilon} + r_{\varepsilon})(y)     U_{z_{\varepsilon}, \lambda_{\varepsilon}}^{5-\mu} (x)r_{\varepsilon}(x)}{|x-y|^{\mu}}dxdy }_{:=H_3}
		 + \underbrace{ O\left(  \int_{\Omega}\int_{\Omega} \frac{  U_{z_{\varepsilon}, \lambda_{\varepsilon}}^{5-\mu}(y)(s_{\varepsilon} + r_{\varepsilon})(y)     U_{z_{\varepsilon}, \lambda_{\varepsilon}}^{4-\mu}(x)(s_{\varepsilon} + r_{\varepsilon}) (x)r_{\varepsilon}(x)}{|x-y|^{\mu}}dxdy                       \right)    }_{:=H_4}
		\\&+ \underbrace{ O\left(  \int_{\Omega}\int_{\Omega} \frac{  U_{z_{\varepsilon}, \lambda_{\varepsilon}}^{5-\mu}(y)\big(\lambda_{\varepsilon}^{-1/2}     |H_Q(z_{\varepsilon}, \cdot)| + |f_{z_{\varepsilon}, \lambda_{\varepsilon}}|\big)(y)    (\psi_{z_{\varepsilon}, \lambda_{\varepsilon}} + s_{\varepsilon} + r_{\varepsilon})^{5-\mu} (x)r_{\varepsilon}(x)}{|x-y|^{\mu}}dxdy                       \right)    }_{:=H_5}
		\\ &+ \underbrace{ O\left(  \int_{\Omega}\int_{\Omega} \frac{  \Big( U_{z_{\varepsilon}, \lambda_{\varepsilon}}^{4-\mu}(r_{\varepsilon}^2 + s_{\varepsilon}^2)+\lambda_{\varepsilon}^{-(6-\mu)/2}     |H_Q(z_{\varepsilon}, \cdot)|^{6-\mu} + |f_{z_{\varepsilon}, \lambda_{\varepsilon}}|^{6-\mu} + |r_{\varepsilon}|^{6-\mu} + |s_{\varepsilon}|^{6-\mu}\Big)(y)    (\psi_{z_{\varepsilon}, \lambda_{\varepsilon}} + s_{\varepsilon} + r_{\varepsilon})^{5-\mu} (x)r_{\varepsilon}(x)}{|x-y|^{\mu}}dxdy                       \right)    }_{:=H_6}
		\\ &+ \underbrace{ O\left(  \int_{\Omega}\int_{\Omega} \frac{  (\psi_{z_{\varepsilon}, \lambda_{\varepsilon}} + s_{\varepsilon} + r_{\varepsilon})^{6-\mu}(y)   \Big(U_{z_{\varepsilon}, \lambda_{\varepsilon}}^{4-\mu}\big(\lambda_{\varepsilon}^{-1/2}     |H_Q(z_{\varepsilon}, \cdot)| + |f_{z_{\varepsilon}, \lambda_{\varepsilon}}|\big) \Big)         (x)r_{\varepsilon}(x)}{|x-y|^{\mu}}dxdy                       \right)    }_{:=H_7}
		\\ &+ \underbrace{ O\left(  \int_{\Omega}\int_{\Omega} \frac{  (\psi_{z_{\varepsilon}, \lambda_{\varepsilon}} + s_{\varepsilon} + r_{\varepsilon})^{6-\mu}(y)  \Big( U_{z_{\varepsilon}, \lambda_{\varepsilon}}^{3-\mu}(r_{\varepsilon}^2 + s_{\varepsilon}^2)+\lambda_{\varepsilon}^{-(5-\mu)/2}     |H_Q(z_{\varepsilon}, \cdot)|^{5-\mu} + |f_{z_{\varepsilon}, \lambda_{\varepsilon}}|^{5-\mu} + |r_{\varepsilon}|^{5-\mu} + |s_{\varepsilon}|^{5-\mu}\Big)     (x)r_{\varepsilon}(x)}{|x-y|^{\mu}}dxdy                       \right)    }_{:=H_8}
	\end{split}
\end{equation}
Next, we will  estimate $H_1, H_2,\cdots,H_8$  respectively. Combining \eqref{d12} with Lemmas \ref{A1} and \ref{A1RN}, we know   
\begin{equation*}
	\begin{split}
		|H_1|= &\left|\int_{\Omega}\int_{\Omega}\frac{U_{z_{\varepsilon},\lambda_{\varepsilon}}^{6-\mu}(y) U_{z_{\varepsilon},\lambda_{\varepsilon}}^{5-\mu}(x)r_{\varepsilon}(x)
	}{|x-y|^{\mu}}dxdy\right|	
	\\ \leq& \left| \int_{\mathbb{R}^3} \int_{\Omega} \frac{U_{z_{\varepsilon},\lambda_{\varepsilon}}^{6-\mu}(y)
	U_{z_{\varepsilon},\lambda_{\varepsilon}}^{5-\mu}(x) r_{\varepsilon}(x)
	}{|x-y|^\mu} dxdy   \right|+\left| \int_{\mathbb{R}^3\backslash\Omega} \int_{\Omega} \frac{U_{z_{\varepsilon},\lambda_{\varepsilon}}^{6-\mu}(y)
	U_{z_{\varepsilon},\lambda_{\varepsilon}}^{5-\mu}(x)r_{\varepsilon}(x)
	}{|x-y|^\mu} dxdy   \right|
	\\ \lesssim&  \| U_{z_{\varepsilon},\lambda_{\varepsilon}}^{6-\mu}\|_{L^{\frac{6}{6-\mu}}(\mathbb{R}^3\backslash\Omega)}
	\| U_{z_{\varepsilon},\lambda_{\varepsilon}}^{5-\mu} r_{\varepsilon}\|_{\frac{6}{6-\mu}}
	\\ \lesssim& \| U_{z_{\varepsilon},\lambda_{\varepsilon}}\|^{6-\mu}_{L^{6}(\mathbb{R}^3\backslash\Omega)}
	\| U_{z_{\varepsilon},\lambda_{\varepsilon}}^{5-\mu} \|_{\frac{6}{5-\mu}}
	\| r_{\varepsilon}\|_{6}
	\\ \lesssim& (\lambda_{\varepsilon}d_{\varepsilon})^{-\frac{6-\mu}{2}} \| r_{\varepsilon}\|_{6}
	\lesssim \lambda_{\varepsilon}^{-\frac{3}{2}}  \| r_{\varepsilon}\|_{6}. 
	\end{split}
\end{equation*}
A direct computation for  $H_2$ and $H_3$ shows 
\begin{equation*}
	\begin{split}
		H_2+H_3= \underbrace{\int_{\Omega}\int_{\Omega} \frac{      (5-\mu)  U_{z_{\varepsilon}, \lambda_{\varepsilon}  }^{6-\mu}(y)U_{z_{\varepsilon}, \lambda_{\varepsilon}}^{4-\mu}(x)r_{\varepsilon}^2(x)}{|x-y|^{\mu}}dxdy
		+\int_{\Omega}\int_{\Omega} \frac{    (6-\mu)U_{z_{\varepsilon}, \lambda_{\varepsilon}}^{5-\mu}(y) r_{\varepsilon}(y)     U_{z_{\varepsilon}, \lambda_{\varepsilon}}^{5-\mu} (x)r_{\varepsilon}(x)}{|x-y|^{\mu}}dxdy}_{:=H_{2,1}}
		\\
		+ \underbrace{\int_{\Omega}\int_{\Omega} \frac{      (5-\mu)  U_{z_{\varepsilon}, \lambda_{\varepsilon}  }^{6-\mu}(y)U_{z_{\varepsilon}, \lambda_{\varepsilon}}^{4-\mu}(x)s_{\varepsilon} (x)r_{\varepsilon}(x)}{|x-y|^{\mu}}dxdy
		+\int_{\Omega}\int_{\Omega} \frac{    (6-\mu)U_{z_{\varepsilon}, \lambda_{\varepsilon}}^{5-\mu}(y)s_{\varepsilon}(y)     U_{z_{\varepsilon}, \lambda_{\varepsilon}}^{5-\mu} (x)r_{\varepsilon}(x)}{|x-y|^{\mu}}dxdy}_{:=H_{2,2}}. 
	\end{split}
\end{equation*}
In \eqref{lemma35eq}, the contribution of $H_{2,1}$ will  cancel with the term
$$-A_{3,\mu} \alpha_{\varepsilon}^{10-2\mu}\int_{\Omega}\int_{\Omega}
\frac{(5-\mu)U_{z_{\varepsilon},\lambda_{\varepsilon}}^{6-\mu}(x)
U_{z_{\varepsilon},\lambda_{\varepsilon}}^{4-\mu}(y) r_{\varepsilon}^{2}(y)+(6-\mu)U_{z_{\varepsilon},\lambda_{\varepsilon}}^{5-\mu}(x)r_{\varepsilon}(x)
U_{z_{\varepsilon},\lambda_{\varepsilon}}^{5-\mu}(y)r_{\varepsilon}(y)}{|x-y|^{\mu}}
dxdy. $$
In order to estimate $H_{2,2}$, we write $s_\varepsilon$ as 
\begin{equation} \label{sss1}
	s_\varepsilon =\beta \lambda_{\varepsilon}^{-1} U_{z_{\varepsilon},\lambda_{\varepsilon}}                  +\gamma\partial_{\lambda}U_{z_{\varepsilon},\lambda_{\varepsilon}}+\tilde{s}_\varepsilon, 
\end{equation}
where 
\begin{equation*}
		\tilde{s}_\varepsilon=-\beta \lambda_{\varepsilon}^{-1}\varphi_{z_{\varepsilon},\lambda_{\varepsilon}}-\gamma\partial_{\lambda}\varphi_{z_{\varepsilon},\lambda_{\varepsilon}}+ \sum_{i=1}^{3}\delta_{i}\lambda_{\varepsilon}^{-3}\partial_{z_{i}}PU_{z_{\varepsilon},\lambda_{\varepsilon}}.
\end{equation*}
Then $H_{2,2}$ becomes 
\begin{equation*}\small
	\begin{split}
		H_{2,2}= &\underbrace{
		\int_{\Omega}\int_{\Omega} \frac{    (11-2\mu) \beta \lambda_{\varepsilon}^{-1}U_{z_{\varepsilon}, \lambda_{\varepsilon}}^{6-\mu}(y)     U_{z_{\varepsilon}, \lambda_{\varepsilon}}^{5-\mu} (x)r_{\varepsilon}(x)}{|x-y|^{\mu}}dxdy}_{:=H_{2,3}}
		\\ &+ \underbrace{\int_{\Omega}\int_{\Omega} \frac{      (5-\mu) \gamma U_{z_{\varepsilon}, \lambda_{\varepsilon}  }^{6-\mu}(y)U_{z_{\varepsilon}, \lambda_{\varepsilon}}^{4-\mu}(x) \partial_{\lambda}U_{z_{\varepsilon},\lambda_{\varepsilon}} (x)r_{\varepsilon}(x)}{|x-y|^{\mu}}dxdy
		+\int_{\Omega}\int_{\Omega} \frac{    (6-\mu)\gamma U_{z_{\varepsilon}, \lambda_{\varepsilon}}^{5-\mu}(y)\partial_{\lambda}U_{z_{\varepsilon},\lambda_{\varepsilon}}(y)     U_{z_{\varepsilon}, \lambda_{\varepsilon}}^{5-\mu} (x)r_{\varepsilon}(x)}{|x-y|^{\mu}}dxdy}_{:=H_{2,4}}
		\\ &+ \underbrace{\int_{\Omega}\int_{\Omega} \frac{      (5-\mu)  U_{z_{\varepsilon}, \lambda_{\varepsilon}  }^{6-\mu}(y)U_{z_{\varepsilon}, \lambda_{\varepsilon}}^{4-\mu}(x)\tilde{s}_\varepsilon (x)r_{\varepsilon}(x)}{|x-y|^{\mu}}dxdy
		+\int_{\Omega}\int_{\Omega} \frac{    (6-\mu)U_{z_{\varepsilon}, \lambda_{\varepsilon}}^{5-\mu}(y)\tilde{s}_\varepsilon(y)     U_{z_{\varepsilon}, \lambda_{\varepsilon}}^{5-\mu} (x)r_{\varepsilon}(x)}{|x-y|^{\mu}}dxdy}_{:=H_{2,5}}. 
	\end{split}
\end{equation*}
Arguing as in the estimate of $H_1$, we obtain
$
	|H_{2,3}|\lesssim \lambda_{\varepsilon}^{-\frac{3}{2}}  \| r_{\varepsilon}\|_{6}. 
$ 
We now turn to  estimate $H_{2,4}$. Since 
\begin{equation*}
	\begin{split}
			-\Delta  \partial_{\lambda} PU_{z_{\varepsilon},\lambda_{\varepsilon}} =-\Delta \partial_{\lambda} U_{z_{\varepsilon},\lambda_{\varepsilon}}=& A_{3,\mu}\left( \int_{\mathbb{R}^3} \frac{(6-\mu)U_{z_{\varepsilon},\lambda_{\varepsilon}}^{5-\mu}(y)  \partial_{\lambda} U_{z_{\varepsilon},\lambda_{\varepsilon}}(y)}{|x-y|^\mu} dy \right) U_{z_{\varepsilon},\lambda_{\varepsilon}}^{5-\mu}  \\ & +A_{3,\mu}(5-\mu)\left( \int_{\mathbb{R}^3} \frac{U_{z_{\varepsilon},\lambda_{\varepsilon}}^{6-\mu}(y)}{|x-y|^\mu} dy \right) U_{z_{\varepsilon},\lambda_{\varepsilon}}^{4-\mu} \partial_{\lambda} U_{z_{\varepsilon},\lambda_{\varepsilon}} , 
	\end{split}
\end{equation*}
by the fact that $  r_{\varepsilon}$ vanishes on the boundary and  $ r_{\varepsilon}\in  T_{z_{\varepsilon},\lambda_{\varepsilon}}^{\bot}$, we obtain
\begin{equation*}
	\begin{split}
		&\int_{\mathbb{R}^3}\int_{\Omega} \frac{      (5-\mu)  U_{z_{\varepsilon}, \lambda_{\varepsilon}  }^{6-\mu}(y)U_{z_{\varepsilon}, \lambda_{\varepsilon}}^{4-\mu}(x) \partial_{\lambda}U_{z_{\varepsilon},\lambda_{\varepsilon}} (x)r_{\varepsilon}(x)}{|x-y|^{\mu}}dxdy
		+\int_{\mathbb{R}^3}\int_{\Omega} \frac{    (6-\mu) U_{z_{\varepsilon}, \lambda_{\varepsilon}}^{5-\mu}(y)\partial_{\lambda}U_{z_{\varepsilon},\lambda_{\varepsilon}}(y)     U_{z_{\varepsilon}, \lambda_{\varepsilon}}^{5-\mu} (x)r_{\varepsilon}(x)}{|x-y|^{\mu}}dxdy
	\\ = &\frac{1}{A_{3,\mu}} \int_{\Omega} -\Delta \partial_{\lambda} PU_{z_{\varepsilon},\lambda_{\varepsilon}}\cdot r_{\varepsilon} 
	=\frac{1}{A_{3,\mu}} \int_{\Omega} \nabla \partial_{\lambda} PU_{z_{\varepsilon},\lambda_{\varepsilon}} \cdot\nabla r_{\varepsilon} =0. 
	\end{split}
\end{equation*}
Thus 
\begin{equation*}\small
	\begin{split}
		|H_{2,4}| \leq & \int_{\mathbb{R}^3 \backslash\Omega}\int_{\Omega} \frac{      (5-\mu)  U_{z_{\varepsilon}, \lambda_{\varepsilon}  }^{6-\mu}(y)U_{z_{\varepsilon},\lambda_{\varepsilon}}^{4-\mu}(x) \partial_{\lambda}U_{z_{\varepsilon},\lambda_{\varepsilon}} (x)r_{\varepsilon}(x)}{|x-y|^{\mu}}dxdy
		+\int_{\mathbb{R}^3 \backslash\Omega}\int_{\Omega} \frac{    (6-\mu) U_{z_{\varepsilon}, \lambda_{\varepsilon}}^{5-\mu}(y)\partial_{\lambda}U_{z_{\varepsilon},\lambda_{\varepsilon}}(y)     U_{z_{\varepsilon}, \lambda_{\varepsilon}}^{5-\mu} (x)r_{\varepsilon}(x)}{|x-y|^{\mu}}dxdy
		\\ \lesssim&  \| U_{z_{\varepsilon},\lambda_{\varepsilon}}^{6-\mu}\|_{L^{\frac{6}{6-\mu}}(\mathbb{R}^3\backslash\Omega)}
	\| U_{z_{\varepsilon},\lambda_{\varepsilon}}^{4-\mu}\partial_{\lambda}U_{z_{\varepsilon},\lambda_{\varepsilon}}  r_{\varepsilon}\|_{\frac{6}{6-\mu}}
	+\| U_{z_{\varepsilon},\lambda_{\varepsilon}}^{5-\mu} \partial_{\lambda}U_{z_{\varepsilon},\lambda_{\varepsilon}}  \|_{L^{\frac{6}{6-\mu}}(\mathbb{R}^3\backslash\Omega)}
	\| U_{z_{\varepsilon},\lambda_{\varepsilon}}^{5-\mu} r_{\varepsilon}\|_{\frac{6}{6-\mu}}
	\\ \lesssim& \| U_{z_{\varepsilon},\lambda_{\varepsilon}}\|^{6-\mu}_{L^{6}(\mathbb{R}^3\backslash\Omega)}
	\| U_{z_{\varepsilon},\lambda_{\varepsilon}}^{4-\mu} \|_{\frac{6}{4-\mu}}
	\| \partial_{\lambda}U_{z_{\varepsilon},\lambda_{\varepsilon}} \|_{6}
	\| r_{\varepsilon}\|_{6} +\| U_{z_{\varepsilon},\lambda_{\varepsilon}}\|^{5-\mu}_{L^{6}(\mathbb{R}^3\backslash\Omega)}
	\| \partial_{\lambda} U_{z_{\varepsilon},\lambda_{\varepsilon}}\|_{L^{6}(\mathbb{R}^3\backslash\Omega)}
	\| U_{z_{\varepsilon},\lambda_{\varepsilon}}^{5-\mu} \|_{\frac{6}{5-\mu}}
	\| r_{\varepsilon}\|_{6} 
	\\ \lesssim&  \lambda_{\varepsilon}^{-\frac{3}{2}}  \| r_{\varepsilon}\|_{6}. 
	\end{split}
\end{equation*}
For the term $H_{2,5}$, we first estimate the norm of $\|\tilde{s}_\varepsilon\|_{6}$. From Lemmas \ref{A1}, \ref{A2} and Proposition \ref{prop32}, we know
\begin{equation} \label{sss2}
	\|\tilde{s}_\varepsilon\|_6\leq(|\beta|+|\gamma|)(\lambda_\varepsilon^{-1}\|\varphi_{z_{\varepsilon},\lambda_{\varepsilon}}\|_6+\|\partial_\lambda\varphi_{z_{\varepsilon},\lambda_{\varepsilon}}\|_6)+\lambda_\varepsilon^{-3}\sum_{i=1}^3|\delta_i|\|\partial_{z_{i}}PU_{z_{\varepsilon},\lambda_{\varepsilon}}\|_6\lesssim\lambda_\varepsilon^{-3/2}.
\end{equation}
By H\"older's inequality and the HLS inequality, we can find
\begin{equation*}
	\begin{split}
		H_{2,5} &=\int_{\Omega}\int_{\Omega} \frac{      (5-\mu)  U_{z_{\varepsilon}, \lambda_{\varepsilon}  }^{6-\mu}(y)U_{z_{\varepsilon}, \lambda_{\varepsilon}}^{4-\mu}(x)\tilde{s}_\varepsilon (x)r_{\varepsilon}(x)}{|x-y|^{\mu}}dxdy
		+\int_{\Omega}\int_{\Omega} \frac{    (6-\mu)U_{z_{\varepsilon}, \lambda_{\varepsilon}}^{5-\mu}(y)\tilde{s}_\varepsilon(y)     U_{z_{\varepsilon}, \lambda_{\varepsilon}}^{5-\mu} (x)r_{\varepsilon}(x)}{|x-y|^{\mu}}dxdy
		\\
		&\lesssim \| U_{z_{\varepsilon},\lambda_{\varepsilon}}^{6-\mu}\|_{\frac{6}{6-\mu}}
	\| U_{z_{\varepsilon},\lambda_{\varepsilon}}^{4-\mu}\tilde{s}_\varepsilon  r_{\varepsilon}\|_{\frac{6}{6-\mu}}
	+\| U_{z_{\varepsilon},\lambda_{\varepsilon}}^{5-\mu} \tilde{s}_\varepsilon \|_{\frac{6}{6-\mu}}
	\| U_{z_{\varepsilon},\lambda_{\varepsilon}}^{5-\mu}  r_{\varepsilon}\|_{\frac{6}{6-\mu}}
	\\
		&\lesssim \| U_{z_{\varepsilon},\lambda_{\varepsilon}}\|_{6}^{6-\mu}
		\| U_{z_{\varepsilon},\lambda_{\varepsilon}}\|_{6}^{4-\mu}
		\| \tilde{s}_\varepsilon\|_{6}
		\| r_{\varepsilon}\|_{6}
	+\| U_{z_{\varepsilon},\lambda_{\varepsilon}}\|_{6}^{5-\mu}
	\| \tilde{s}_\varepsilon\|_{6}
	\| U_{z_{\varepsilon},\lambda_{\varepsilon}}\|_{6}^{5-\mu}
	\| r_{\varepsilon}\|_{6}
	\\  & \lesssim  \lambda_{\varepsilon}^{-\frac{3}{2}}  \| r_{\varepsilon}\|_{6}. 
	\end{split}
\end{equation*}
Then $|H_{2,2}|\lesssim  \lambda_{\varepsilon}^{-\frac{3}{2}}  \| r_{\varepsilon}\|_{6}$. Using H\"older's inequality, the HLS inequality, and Proposition \ref{prop32}, we can obtain 
\begin{equation*}
	\begin{split}
		 H_4 &\lesssim \int_{\Omega}\int_{\Omega} \frac{  U_{z_{\varepsilon}, \lambda_{\varepsilon}}^{5-\mu}(y)(s_{\varepsilon} + r_{\varepsilon})(y)     U_{z_{\varepsilon}, \lambda_{\varepsilon}}^{4-\mu}(x)(s_{\varepsilon} + r_{\varepsilon}) (x)r_{\varepsilon}(x)}{|x-y|^{\mu}}dxdy    \\
		 &\lesssim  \|U_{z_{\varepsilon}, \lambda_{\varepsilon}}^{5-\mu} (s_{\varepsilon} + r_{\varepsilon})  \|_{\frac{6}{6-\mu}} \| U_{z_{\varepsilon}, \lambda_{\varepsilon}}^{4-\mu}  (s_{\varepsilon} + r_{\varepsilon})r_{\varepsilon}\|_{\frac{6}{6-\mu}}
		  \\
		 &\lesssim \left[  \|U_{z_{\varepsilon}, \lambda_{\varepsilon}}^{5-\mu}  \|_{\frac{6}{5-\mu}}  \left(  \| s_{\varepsilon} \|_6 + \| r_{\varepsilon} \|_6 \right)   \right]  \left[  \|U_{z_{\varepsilon}, \lambda_{\varepsilon}}^{4-\mu}  \|_{\frac{6}{4-\mu}}  \left(  \| s_{\varepsilon} \|_6 + \| r_{\varepsilon} \|_6 \right)   \| r_{\varepsilon} \|_6 \right]
		 \\ &\lesssim \left( \| s_{\varepsilon} \|_6 ^2  + \| s_{\varepsilon} \|_6 \| r_{\varepsilon} \|_6 +\| r_{\varepsilon} \|_6^2  \right)\| r_{\varepsilon} \|_6
		 \\ &\lesssim \left( \lambda_\varepsilon^{-2}  + \lambda_\varepsilon^{-1}\| r_{\varepsilon} \|_6 +\| r_{\varepsilon} \|_6^2  \right)\| r_{\varepsilon} \|_6. 
	\end{split}
\end{equation*}
It  remains to estimate $H_5, H_6,H_7,H_8$. Now we give norm estimates for parts of these terms, for example 
\begin{equation*}
	\left\|   \psi_{z_{\varepsilon}, \lambda_{\varepsilon}} + s_{\varepsilon} + r_{\varepsilon}  \right\|_{6} =   \left\|   PU_{z_{\varepsilon},\lambda_{\varepsilon}}+v_{\varepsilon}  \right\|_{6}    =O(1).   
\end{equation*}
Since $H_Q(z_{\varepsilon}, x)= \phi_{Q}(z_{\varepsilon})+O\left( | z_{\varepsilon}- x|  \right)$, we can deduce
\begin{equation*}
	\begin{split}
		\left\| U_{z_{\varepsilon},\lambda_{\varepsilon}}^{5-\mu} |H_Q(z_{\varepsilon}, \cdot)| \right\|_{\frac{6}{6-\mu}}&= \left(\int_\Omega  U_{z_{\varepsilon},\lambda_{\varepsilon}}^{\frac{6(5-\mu)}{6-\mu}}   |H_Q(z_{\varepsilon}, \cdot)| ^{\frac{6}{6-\mu}}          \right)^{\frac{6-\mu}{6}}  
	\\& \lesssim \left(\int_\Omega  U_{z_{\varepsilon},\lambda_{\varepsilon}}^{\frac{6(5-\mu)}{6-\mu}}   \phi_{Q}^{\frac{6}{6-\mu}}(z_{\varepsilon})           \right)^{\frac{6-\mu}{6}}  +\left(\int_\Omega  U_{z_{\varepsilon},\lambda_{\varepsilon}}^{\frac{6(5-\mu)}{6-\mu}}  | z_{\varepsilon}- x|^{\frac{6}{6-\mu}}          \right)^{\frac{6-\mu}{6}}  \\
		&\lesssim  \phi_{Q}(z_{\varepsilon}) \left\| U_{z_{\varepsilon},\lambda_{\varepsilon}}  \right\|_{\frac{6(5-\mu)}{6-\mu}}^{5-\mu} +\lambda_{\varepsilon}^{-\frac{3}{2}}
		\\
		&\lesssim  \lambda_{\varepsilon}^{-\frac{1}{2}} \phi_{Q}(z_{\varepsilon})  +\lambda_{\varepsilon}^{-\frac{3}{2}}, 
	\end{split}
\end{equation*}
where we used  $\mu <2$ to ensure the bound $\left(\int_\Omega  U_{z_{\varepsilon},\lambda_{\varepsilon}}^{\frac{6(5-\mu)}{6-\mu}}  | z_{\varepsilon}- x|^{\frac{6}{6-\mu}}          \right)^{\frac{6-\mu}{6}} \lesssim \lambda_{\varepsilon}^{-\frac{3}{2}}$ holds. 
Now we estimate $\left\| U_{z_{\varepsilon},\lambda_{\varepsilon}}^{4-\mu} |H_Q(z_{\varepsilon}, \cdot)| \right\|_{\frac{6}{5-\mu}}$. 
Let $0<\rho<\operatorname{dist}(z_{\varepsilon},\partial\Omega)$. Then
\begin{equation*}
\begin{split}
	&\left\|
	U_{z_{\varepsilon},\lambda_{\varepsilon}}^{4-\mu}
	H_Q(z_{\varepsilon},\cdot)
	\right\|_{\frac{6}{5-\mu}}\leq
	\left\|
	U_{z_{\varepsilon},\lambda_{\varepsilon}}^{4-\mu}
	H_Q(z_{\varepsilon},\cdot)
	\right\|_{L^{\frac{6}{5-\mu}}(B_{\rho}(z_{\varepsilon}))}+
	\left\|
	U_{z_{\varepsilon},\lambda_{\varepsilon}}^{4-\mu}
	H_Q(z_{\varepsilon},\cdot)
	\right\|_{L^{\frac{6}{5-\mu}}(\Omega\setminus B_{\rho}(z_{\varepsilon}))}.
\end{split}
\end{equation*}
By Lemma \ref{B2},
\begin{equation*}
	H_Q(z_{\varepsilon},x)
	=\phi_Q(z_{\varepsilon})+O(|x-z_{\varepsilon}|)
	\qquad\forall x\in B_{\rho}(z_{\varepsilon}).
\end{equation*}
Hence, by Lemma \ref{A1} and a direct computation, it holds that 
\begin{equation*}
\begin{split}
	&\left\|
	U_{z_{\varepsilon},\lambda_{\varepsilon}}^{4-\mu}
	H_Q(z_{\varepsilon},\cdot)
	\right\|_{L^{\frac{6}{5-\mu}}(B_{\rho}(z_{\varepsilon}))}\\
	&\lesssim
	\phi_Q(z_{\varepsilon})
	\left\|
	U_{z_{\varepsilon},\lambda_{\varepsilon}}^{4-\mu}
	\right\|_{L^{\frac{6}{5-\mu}}(B_{\rho}(z_{\varepsilon}))}
	+
	\left\|
	U_{z_{\varepsilon},\lambda_{\varepsilon}}^{4-\mu}
	|\cdot-z_{\varepsilon}|
	\right\|_{L^{\frac{6}{5-\mu}}(B_{\rho}(z_{\varepsilon}))}\\
	&\lesssim
	\phi_Q(z_{\varepsilon})\lambda_{\varepsilon}^{-\frac12}
	+
	\begin{cases}
		O(\lambda_{\varepsilon}^{-\frac32}), & \text{if} \quad 0<\mu<1,\\
		O(\lambda_{\varepsilon}^{-\frac32}
		(\log\lambda_{\varepsilon})^{\frac23}), & \text{if} \quad \mu=1,\\
		O\left(\lambda_{\varepsilon}^{-\frac{4-\mu}{2}}\right),
		& \text{if} \quad 1<\mu<2.
	\end{cases}
\end{split}
\end{equation*}
On the other hand, by Lemma \ref{B1},
\begin{equation*}
\begin{split}
	&\left\|
	U_{z_{\varepsilon},\lambda_{\varepsilon}}^{4-\mu}
	H_Q(z_{\varepsilon},\cdot)
	\right\|_{L^{\frac{6}{5-\mu}}(\Omega\setminus B_{\rho}(z_{\varepsilon}))}\lesssim
	\left\|
	U_{z_{\varepsilon},\lambda_{\varepsilon}}^{4-\mu}
	\right\|_{L^{\frac{6}{5-\mu}}(\Omega\setminus B_{\rho}(z_{\varepsilon}))}
	=O\left(\lambda_{\varepsilon}^{-\frac{4-\mu}{2}}\right).
\end{split}
\end{equation*}
Thus it follows that
\begin{equation}
	\label{UUUU}
	\left\| U_{z_{\varepsilon},\lambda_{\varepsilon}}^{4-\mu} |H_Q(z_{\varepsilon}, \cdot)| \right\|_{\frac{6}{5-\mu}}   \lesssim \lambda_{\varepsilon}^{-\frac{1}{2}} \phi_{Q}(z_{\varepsilon})  +\begin{cases}
		O(\lambda_{\varepsilon}^{-\frac32}), & \text{if} \quad 0<\mu<1,\\
		O(\lambda_{\varepsilon}^{-\frac32}
		(\log\lambda_{\varepsilon})^{\frac23}), & \text{if} \quad \mu=1,\\
		O\left(\lambda_{\varepsilon}^{-\frac{4-\mu}{2}}\right),
		& \text{if} \quad 1<\mu<2.
	\end{cases}
\end{equation}
Based on the above estimate and Lemma \ref{A2}, we obtain 
\begin{equation*}
	\begin{split}
		 H_5 &\lesssim \int_{\Omega}\int_{\Omega} \frac{  U_{z_{\varepsilon}, \lambda_{\varepsilon}}^{5-\mu}(y)\big(\lambda_{\varepsilon}^{-1/2}     |H_Q(z_{\varepsilon}, \cdot)| + |f_{z_{\varepsilon}, \lambda_{\varepsilon}}|\big)(y)    (\psi_{z_{\varepsilon}, \lambda_{\varepsilon}} + s_{\varepsilon} + r_{\varepsilon})^{5-\mu} (x)r_{\varepsilon}(x)}{|x-y|^{\mu}}dxdy   
		 \\  &\lesssim  \|U_{z_{\varepsilon}, \lambda_{\varepsilon}}^{5-\mu} (\lambda_{\varepsilon}^{-1/2}     |H_Q(z_{\varepsilon}, \cdot)| + |f_{z_{\varepsilon}, \lambda_{\varepsilon}}|)  \|_{\frac{6}{6-\mu}} \|  (\psi_{z_{\varepsilon}, \lambda_{\varepsilon}} + s_{\varepsilon} + r_{\varepsilon})^{5-\mu}   r_{\varepsilon}    \|_{\frac{6}{6-\mu}}
		 \\  &\lesssim \left( \lambda_{\varepsilon}^{-1/2}  \|U_{z_{\varepsilon}, \lambda_{\varepsilon}}^{5-\mu}     |H_Q(z_{\varepsilon}, \cdot)|   \|_{\frac{6}{6-\mu}}  +  \|U_{z_{\varepsilon}, \lambda_{\varepsilon}}^{5-\mu}       \|_{\frac{6}{6-\mu}} \| f_{z_{\varepsilon}, \lambda_{\varepsilon}}\|_{\infty} \right)
		  \|  (\psi_{z_{\varepsilon}, \lambda_{\varepsilon}} + s_{\varepsilon} + r_{\varepsilon})^{5-\mu}    \|_{\frac{6}{5-\mu}}
		    \|    r_{\varepsilon}    \|_{6}
			 \\  &\lesssim \left( \lambda_{\varepsilon}^{-1}\phi_{Q}(z_{\varepsilon})  +\lambda_{\varepsilon}^{-2}  \right)  \|    r_{\varepsilon}    \|_{6},
	\end{split}
\end{equation*}
and 
\begin{equation*}
	\begin{split}
		H_7 &\lesssim \int_{\Omega}\int_{\Omega} \frac{  (\psi_{z_{\varepsilon}, \lambda_{\varepsilon}} + s_{\varepsilon} + r_{\varepsilon})^{6-\mu}(y)   \Big(U_{z_{\varepsilon}, \lambda_{\varepsilon}}^{4-\mu}\big(\lambda_{\varepsilon}^{-1/2}     |H_Q(z_{\varepsilon}, \cdot)| + |f_{z_{\varepsilon}, \lambda_{\varepsilon}}|\big) \Big)         (x)r_{\varepsilon}(x)}{|x-y|^{\mu}}dxdy     
		\\   &\lesssim   \|  (\psi_{z_{\varepsilon}, \lambda_{\varepsilon}} + s_{\varepsilon} + r_{\varepsilon})^{6-\mu}     \|_{\frac{6}{6-\mu}}    \left(  \lambda_{\varepsilon}^{-1/2}  \|U_{z_{\varepsilon}, \lambda_{\varepsilon}}^{4-\mu}     |H_Q(z_{\varepsilon}, \cdot)|   \|_{\frac{6}{5-\mu}}  +  \|U_{z_{\varepsilon}, \lambda_{\varepsilon}}^{4-\mu}       \|_{\frac{6}{5-\mu}} \| f_{z_{\varepsilon}, \lambda_{\varepsilon}}\|_{\infty}  \right)  \|    r_{\varepsilon}    \|_{6}
		 \\  &\lesssim \left( \lambda_{\varepsilon}^{-1}\phi_{Q}(z_{\varepsilon})  +\lambda_{\varepsilon}^{-\frac{3}{2}}  \right)  \|    r_{\varepsilon}    \|_{6}. 
	\end{split}
\end{equation*}
Moreover, taking  \cite[Equation (2.6)]{FKK1} into consideration, we deduce
\begin{equation*}\small
	\begin{split}
		H_6 &\lesssim  \int_{\Omega}\int_{\Omega} \frac{  \Big( U_{z_{\varepsilon}, \lambda_{\varepsilon}}^{4-\mu}(r_{\varepsilon}^2 + s_{\varepsilon}^2)+\lambda_{\varepsilon}^{-(6-\mu)/2}     |H_Q(z_{\varepsilon}, \cdot)|^{6-\mu} + |f_{z_{\varepsilon}, \lambda_{\varepsilon}}|^{6-\mu} + |r_{\varepsilon}|^{6-\mu} + |s_{\varepsilon}|^{6-\mu}\Big)(y)    (\psi_{z_{\varepsilon}, \lambda_{\varepsilon}} + s_{\varepsilon} + r_{\varepsilon})^{5-\mu} (x)r_{\varepsilon}(x)}{|x-y|^{\mu}}dxdy  
		\\ &\lesssim   \left(\left\|U_{z_{\varepsilon}, \lambda_{\varepsilon}}^{4-\mu}(r_{\varepsilon}^2 + s_{\varepsilon}^2)\right\|_{\frac{6}{6-\mu}}   +  \lambda_{\varepsilon}^{-\frac{6-\mu}{2}}   +\left\| |f_{z_{\varepsilon}, \lambda_{\varepsilon}}|^{6-\mu}  \right\|_{\frac{6}{6-\mu}}  +\left\|   |r_{\varepsilon}|^{6-\mu} \right\|_{\frac{6}{6-\mu}} +\left\| |s_{\varepsilon}|^{6-\mu}  \right\|_{\frac{6}{6-\mu}}       \right) \left\|  (\psi_{z_{\varepsilon}, \lambda_{\varepsilon}} + s_{\varepsilon} + r_{\varepsilon})^{5-\mu}  \right\|_{\frac{6}{5-\mu}} \|r_{\varepsilon}    \|_{6} 
		\\ &\lesssim \left( \left\| r_{\varepsilon}\right\|_{6}^{2}+ \left\| s_{\varepsilon}\right\|_{6}^{2} + \lambda_{\varepsilon}^{-\frac{6-\mu}{2}}+ \left\|f_{z_{\varepsilon}, \lambda_{\varepsilon}}\right\|_{\infty}^{6-\mu} + \left\|r_{\varepsilon}\right\|_{6}^{6-\mu} + \left\|s_{\varepsilon}\right\|_{6}^{6-\mu}\right) \left\|  \psi_{z_{\varepsilon}, \lambda_{\varepsilon}} + s_{\varepsilon} + r_{\varepsilon} \right\|_{6} ^{5-\mu}  \|r_{\varepsilon}    \|_{6} 
		\\ &\lesssim \left( \left\| r_{\varepsilon}\right\|_{6}^{2}+\lambda_{\varepsilon}^{-\frac{3}{2}}  \right)   \|r_{\varepsilon}    \|_{6},
	\end{split}
\end{equation*}
and
\begin{equation*}\small
	\begin{split}
		H_8 &\lesssim \int_{\Omega}\int_{\Omega} \frac{  (\psi_{z_{\varepsilon}, \lambda_{\varepsilon}} + s_{\varepsilon} + r_{\varepsilon})^{6-\mu}(y)  \Big( U_{z_{\varepsilon}, \lambda_{\varepsilon}}^{3-\mu}(r_{\varepsilon}^2 + s_{\varepsilon}^2)+\lambda_{\varepsilon}^{-(5-\mu)/2}     |H_Q(z_{\varepsilon}, \cdot)|^{5-\mu} + |f_{z_{\varepsilon}, \lambda_{\varepsilon}}|^{5-\mu} + |r_{\varepsilon}|^{5-\mu} + |s_{\varepsilon}|^{5-\mu}\Big)     (x)r_{\varepsilon}(x)}{|x-y|^{\mu}}dxdy \\
		&\lesssim     \left\|  (\psi_{z_{\varepsilon}, \lambda_{\varepsilon}} + s_{\varepsilon} + r_{\varepsilon})^{6-\mu}  \right\|_{\frac{6}{6-\mu}} 
		\left( \left\|U_{z_{\varepsilon}, \lambda_{\varepsilon}}^{3-\mu}(r_{\varepsilon}^2 + s_{\varepsilon}^2)\right\|_{\frac{6}{5-\mu}} +  \lambda_{\varepsilon}^{-\frac{5-\mu}{2}}   +\left\| |f_{z_{\varepsilon}, \lambda_{\varepsilon}}|^{5-\mu}  \right\|_{\frac{6}{5-\mu}}  +\left\|   |r_{\varepsilon}|^{5-\mu} \right\|_{\frac{6}{5-\mu}} +\left\| |s_{\varepsilon}|^{5-\mu}  \right\|_{\frac{6}{5-\mu}}        \right)  \|r_{\varepsilon}    \|_{6}
		\\
		&\lesssim \left( \left\| r_{\varepsilon}\right\|_{6}^{2}+ \left\| s_{\varepsilon}\right\|_{6}^{2} + \lambda_{\varepsilon}^{-\frac{5-\mu}{2}}+ \left\|f_{z_{\varepsilon}, \lambda_{\varepsilon}}\right\|_{\infty}^{5-\mu} + \left\|r_{\varepsilon}\right\|_{6}^{5-\mu} + \left\|s_{\varepsilon}\right\|_{6}^{5-\mu}\right)\|r_{\varepsilon}    \|_{6}
		\\
		&\lesssim \left( \left\| r_{\varepsilon}\right\|_{6}^{2}+\lambda_{\varepsilon}^{-\frac{3}{2}}  \right)   \|r_{\varepsilon}    \|_{6}. 
	\end{split}
\end{equation*}
Here we used  $\mu <2$ to deduce $\lambda_{\varepsilon}^{-\frac{5-\mu}{2}} \lesssim \lambda_{\varepsilon}^{-\frac{3}{2}} $.  Then Lemma \ref{lemma35} follows. 
\end{proof}

To estimate the second and third terms on the right-hand side of
\eqref{o320}, we use the same argument as in
\cite[Lemma~3.5(b)]{FKK3}.
\begin{lem}
	\label{lemma36}
	As $\varepsilon\rightarrow 0$,   it holds that
	\begin{equation} \label{lemma36eq}
		\begin{split}
			\left| \int_\Omega Q (s_{\varepsilon} - f_{z_{\varepsilon},\lambda_{\varepsilon}} - g_{z_{\varepsilon},\lambda_{\varepsilon}}) r_{\varepsilon} + \int_\Omega \varepsilon V (\psi_{z_{\varepsilon},\lambda_{\varepsilon}} + s_{\varepsilon} + r_{\varepsilon}) r_{\varepsilon} \right|\lesssim
 \left( \lambda_\varepsilon^{-\frac{3}{2}}  +  \varepsilon \lambda_\varepsilon^{-\frac{1}{2}} +\varepsilon \| r_{\varepsilon} \|_6 \right)\| r_{\varepsilon} \|_6 .
		\end{split}
	\end{equation}
\end{lem}
\begin{proof}[Proof of Proposition \ref{prop34}]
It follows from identity  \eqref{o320}, Lemma \ref{lemma35}  and  Lemma \ref{lemma36} that
\begin{equation*}\small
	\begin{split}
		\int_\Omega(|\nabla r_{\varepsilon}|^2 + Q r_{\varepsilon}^2)-A_{3,\mu} \alpha_{\varepsilon}^{10-2\mu}\int_{\Omega}\int_{\Omega}&
\frac{(5-\mu)U_{z_{\varepsilon},\lambda_{\varepsilon}}^{6-\mu}(x)
U_{z_{\varepsilon},\lambda_{\varepsilon}}^{4-\mu}(y) r_{\varepsilon}^{2}(y)+(6-\mu)U_{z_{\varepsilon},\lambda_{\varepsilon}}^{5-\mu}(x)r_{\varepsilon}(x)
U_{z_{\varepsilon},\lambda_{\varepsilon}}^{5-\mu}(y)r_{\varepsilon}(y)}{|x-y|^{\mu}}
dxdy  
\\ \lesssim&  \left( \lambda_\varepsilon^{-\frac{3}{2}}  + \lambda_{\varepsilon}^{-1}\phi_{Q}(z_{\varepsilon})    + \varepsilon \lambda_\varepsilon^{-\frac{1}{2}} +\varepsilon \| r_{\varepsilon} \|_6+\lambda_\varepsilon^{-1}\| r_{\varepsilon} \|_6   + \| r_{\varepsilon} \|_6^2  \right)\| r_{\varepsilon} \|_6 . 
	\end{split}
\end{equation*}
Since $\alpha_{\varepsilon}\to 1$ and $r_{\varepsilon} \in T_{z_{\varepsilon},\lambda_{\varepsilon}}^{\bot}$, the coercivity property stated in Lemma \ref{E40} ensures that, for sufficiently small $\varepsilon > 0$, there exists a universal constant $c>0$ such that
\begin{equation*} \small
	\begin{split}
		\int_\Omega(|\nabla r_{\varepsilon}|^2 + Q r_{\varepsilon}^2)-A_{3,\mu} \alpha_{\varepsilon}^{10-2\mu}\int_{\Omega}\int_{\Omega}
\frac{(5-\mu)U_{z_{\varepsilon},\lambda_{\varepsilon}}^{6-\mu}(x)
U_{z_{\varepsilon},\lambda_{\varepsilon}}^{4-\mu}(y) r_{\varepsilon}^{2}(y)+(6-\mu)U_{z_{\varepsilon},\lambda_{\varepsilon}}^{5-\mu}(x)r_{\varepsilon}(x)
U_{z_{\varepsilon},\lambda_{\varepsilon}}^{5-\mu}(y)r_{\varepsilon}(y)}{|x-y|^{\mu}}
dxdy  \gtrsim c \|\nabla r_{\varepsilon}\|_2^2.
	\end{split}
\end{equation*}
Consequently, we can bound
\begin{equation*}
	\|\nabla r_{\varepsilon}\|_2\lesssim     \lambda_\varepsilon^{-\frac{3}{2}}  + \lambda_{\varepsilon}^{-1}\phi_{Q}(z_{\varepsilon})    + \varepsilon \lambda_\varepsilon^{-\frac{1}{2}} +\varepsilon \|\nabla r_{\varepsilon}\|_2+\lambda_\varepsilon^{-1}\|\nabla r_{\varepsilon}\|_2   + \|\nabla r_{\varepsilon}\|_2 ^2
\end{equation*}
By choosing $\varepsilon$ sufficiently small, the last three terms on the right-hand side can be absorbed into the left-hand side, which immediately establishes Proposition \ref{prop34}.
\end{proof}

\subsection{The estimate of   $\boldsymbol{\alpha_{\varepsilon}^{10-2\mu}}$. } 
The main purpose of this subsection is to prove the following estimate for $\alpha_{\varepsilon}^{10-2\mu}$.
\begin{Prop}
	\label{prop36o}
	As $\varepsilon\rightarrow 0$, we have
	\begin{equation}
	\label{o324}
	\alpha_{\varepsilon}^{10-2\mu}=1-(10-2\mu) \beta \lambda_\varepsilon^{-1}+O\left(\phi_Q(z_{\varepsilon}) \lambda_{\varepsilon}^{-1} + \lambda_{\varepsilon}^{-2}+\varepsilon \lambda_{\varepsilon}^{-1}  \right), 
\end{equation}
where $\beta$ is the zero-mode coefficient from \eqref{o38}. 
\end{Prop}

To establish \eqref{o324},  we start from the energy identity obtained by multiplying the equation for $u_{\varepsilon}$ by  $u_{\varepsilon}$ and integrating over $\Omega$.   Substituting the decomposition  $u_{\varepsilon}=\alpha_{\varepsilon}\left( \psi_{z_{\varepsilon},\lambda_{\varepsilon}}+ q_\varepsilon \right)$  yields
\begin{equation}\label{kkkk1}\small
	\begin{split}
		\int_\Omega\left(|\nabla (\psi_{z_{\varepsilon},\lambda_{\varepsilon}}+ q_\varepsilon)|^2+(Q+\varepsilon V)( \psi_{z_{\varepsilon},\lambda_{\varepsilon}}+ q_\varepsilon)^2\right)=A_{3,\mu}	\alpha_{\varepsilon}^{10-2\mu} \int_\Omega \int_\Omega \frac{( \psi_{z_{\varepsilon},\lambda_{\varepsilon}}+ q_\varepsilon)^{6-\mu}(x)( \psi_{z_{\varepsilon},\lambda_{\varepsilon}}+ q_\varepsilon)^{6-\mu}(y)}{|x-y|^{\mu}}dxdy. 
	\end{split}
\end{equation}
Taking the  bound pointwise 
\begin{equation*}
		\begin{split}
			( \psi_{z_{\varepsilon},\lambda_{\varepsilon}}+ q_\varepsilon)^{6-\mu} 
				= \psi_{z_{\varepsilon},\lambda_{\varepsilon}}^{6-\mu} +(6-\mu)\psi_{z_{\varepsilon},\lambda_{\varepsilon}}^{5-\mu} q_\varepsilon +O\left(\psi_{z_{\varepsilon},\lambda_{\varepsilon}}^{4-\mu} q_\varepsilon ^{2} +|q_\varepsilon|^{6-\mu}   \right)
		\end{split}
	\end{equation*}
into consideration, we can write the equation \eqref{kkkk1}  as 
\begin{equation}
	\label{o325}
	\begin{split}
		&\int_\Omega \left(|\nabla \psi_{z_{\varepsilon},\lambda_{\varepsilon}}|^2+(Q+\varepsilon V)\psi_{z_{\varepsilon},\lambda_{\varepsilon}}^2\right)
		- A_{3,\mu}	\alpha_{\varepsilon}^{10-2\mu} \int_\Omega \int_\Omega \frac{\psi_{z_{\varepsilon},\lambda_{\varepsilon}}^{6-\mu}(x)\psi_{z_{\varepsilon},\lambda_{\varepsilon}}^{6-\mu}(y)}{|x-y|^{\mu}}dxdy
		\\ &+2\left[ \int_\Omega \left( \nabla \psi_{z_{\varepsilon},\lambda_{\varepsilon}} \cdot  \nabla  q_\varepsilon +\left(Q+\varepsilon V\right) \psi_{z_{\varepsilon},\lambda_{\varepsilon}} q_\varepsilon \right) - A_{3,\mu}	\alpha_{\varepsilon}^{10-2\mu} \int_\Omega \int_\Omega \frac{(6-\mu)\psi_{z_{\varepsilon},\lambda_{\varepsilon}}^{6-\mu}(y)\psi_{z_{\varepsilon},\lambda_{\varepsilon}}^{5-\mu}(x)q_\varepsilon(x)}{|x-y|^{\mu}}dxdy\right]=\mathcal{R}_0,
	\end{split}
\end{equation} 
with
\begin{equation*}
	\begin{split}
		\mathcal{R}_0:=- \int_\Omega \left(|\nabla q_\varepsilon|^2+(Q+\varepsilon V)q_\varepsilon^2\right)   & +A_{3,\mu}	\alpha_{\varepsilon}^{10-2\mu}(6-\mu)^2 \int_\Omega \int_\Omega \frac{\psi_{z_{\varepsilon},\lambda_{\varepsilon}}^{5-\mu}(x)q_\varepsilon(x) \psi_{z_{\varepsilon},\lambda_{\varepsilon}}^{5-\mu}(y)q_\varepsilon(y)}{|x-y|^{\mu}}dxdy 
			\\&+O\left( \int_\Omega \int_\Omega \frac{  (\psi_{z_{\varepsilon},\lambda_{\varepsilon}}^{4-\mu} q_\varepsilon ^{2} +|q_\varepsilon|^{6-\mu}   ) (x)( \psi_{z_{\varepsilon},\lambda_{\varepsilon}}+ q_\varepsilon ) ^{6-\mu} (y)}{|x-y|^{\mu}}dxdy    \right). 
	\end{split}
\end{equation*}

In the following lemma, we give the asymptotic expansions of the terms in \eqref{o325}. 
\begin{lem}
	\label{lemma37o}
	As $\varepsilon\rightarrow 0$, the following hold:
	\begin{itemize}
\item[(a)] \quad $ \displaystyle \int_\Omega \left(|\nabla \psi_{z_{\varepsilon},\lambda_{\varepsilon}}|^2+(Q+\varepsilon V)\psi_{z_{\varepsilon},\lambda_{\varepsilon}}^2\right)
		- A_{3,\mu}	\alpha_{\varepsilon}^{10-2\mu} \int_\Omega \int_\Omega \frac{\psi_{z_{\varepsilon},\lambda_{\varepsilon}}^{6-\mu}(x)\psi_{z_{\varepsilon},\lambda_{\varepsilon}}^{6-\mu}(y)}{|x-y|^{\mu}}dxdy$\\
		  $ \hspace*{6cm}
		 \displaystyle  \quad=(1-\alpha_{\varepsilon}^{10-2\mu}) \frac{3\pi ^2}{4}+O\left(\phi_Q(z_{\varepsilon}) \lambda_{\varepsilon}^{-1} + \lambda_{\varepsilon}^{-2}+\varepsilon \lambda_{\varepsilon}^{-1}  \right). $
\item[(b)]  \quad  $ \displaystyle \int_\Omega \left( \nabla \psi_{z_{\varepsilon},\lambda_{\varepsilon}} \cdot  \nabla  q_\varepsilon +\left(Q+\varepsilon V\right) \psi_{z_{\varepsilon},\lambda_{\varepsilon}} q_\varepsilon \right) - A_{3,\mu}	\alpha_{\varepsilon}^{10-2\mu} \int_\Omega \int_\Omega \frac{(6-\mu)\psi_{z_{\varepsilon},\lambda_{\varepsilon}}^{6-\mu}(y)\psi_{z_{\varepsilon},\lambda_{\varepsilon}}^{5-\mu}(x)q_\varepsilon(x)}{|x-y|^{\mu}}dxdy$ \\
 $ \hspace*{6cm}
		 \displaystyle =(1-(6-\mu)\alpha_{\varepsilon}^{10-2\mu}) \frac{3\pi ^2}{4}\beta \lambda_{\varepsilon}^{-1}+O\left( \lambda_{\varepsilon}^{-2}+\varepsilon \lambda_{\varepsilon}^{-\frac{3}{2}}+ \varepsilon^2 \lambda_{\varepsilon}^{-1} \right). $
\item[(c)] \quad	$\displaystyle \mathcal{R}_0=O\left(   \lambda_{\varepsilon}^{-2} +\varepsilon^{2} \lambda_{\varepsilon}^{-1}\right). $
\end{itemize}
\end{lem}
\begin{proof}
	(a) Using  \cite[Theorem 5.1]{GGYZ} and straightforward calculations, we obtain
	\begin{equation*}
		\int_\Omega \left(|\nabla \psi_{z_{\varepsilon},\lambda_{\varepsilon}}|^2+(Q+\varepsilon V)\psi_{z_{\varepsilon},\lambda_{\varepsilon}}^2\right)=\frac{3\pi ^2}{4} +O\left(\phi_Q(z_{\varepsilon}) \lambda_{\varepsilon}^{-1} + \lambda_{\varepsilon}^{-2}+\varepsilon \lambda_{\varepsilon}^{-1}  \right), 
	\end{equation*}
	\begin{equation*}
		A_{3,\mu}	\int_\Omega \int_\Omega \frac{\psi_{z_{\varepsilon},\lambda_{\varepsilon}}^{6-\mu}(x)\psi_{z_{\varepsilon},\lambda_{\varepsilon}}^{6-\mu}(y)}{|x-y|^{\mu}}dxdy=\frac{3\pi ^2}{4} +O\left(\phi_Q(z_{\varepsilon}) \lambda_{\varepsilon}^{-1} + \lambda_{\varepsilon}^{-2}  \right). 
	\end{equation*}
Then the  asymptotic expansion (a) follows.  \\ 
(b) From \eqref{o31}, \eqref{o316} and \eqref{o317}, we have
\begin{equation*}
	\psi_{z_{\varepsilon},\lambda_{\varepsilon}}
	=\lambda_{\varepsilon}^{-1/2}G_Q(z_{\varepsilon},\cdot)
	-f_{z_{\varepsilon},\lambda_{\varepsilon}}
	-g_{z_{\varepsilon},\lambda_{\varepsilon}}.
\end{equation*}
Moreover, using
$\Delta(H_Q(z_{\varepsilon},\cdot)-H_0(z_{\varepsilon},\cdot))
=-QG_Q(z_{\varepsilon},\cdot)$, we obtain
\begin{equation}\label{o326}
	(-\Delta+Q)\psi_{z_{\varepsilon},\lambda_{\varepsilon}}
	=3U_{z_{\varepsilon},\lambda_{\varepsilon}}^5
	-Q(f_{z_{\varepsilon},\lambda_{\varepsilon}}
	+g_{z_{\varepsilon},\lambda_{\varepsilon}}).
\end{equation}
It follows that
\begin{equation*}
	\begin{split}
		&\int_\Omega \left( \nabla \psi_{z_{\varepsilon},\lambda_{\varepsilon}} \cdot  \nabla  q_\varepsilon +\left(Q+\varepsilon V\right) \psi_{z_{\varepsilon},\lambda_{\varepsilon}} q_\varepsilon \right) - A_{3,\mu}	\alpha_{\varepsilon}^{10-2\mu} \int_\Omega \int_\Omega \frac{(6-\mu)\psi_{z_{\varepsilon},\lambda_{\varepsilon}}^{6-\mu}(y)\psi_{z_{\varepsilon},\lambda_{\varepsilon}}^{5-\mu}(x)q_\varepsilon(x)}{|x-y|^{\mu}}dxdy
		\\ &=\int_\Omega \left(  3 U_{z_{\varepsilon}, \lambda_{\varepsilon}} ^{5} q_\varepsilon -Q(f_{z_{\varepsilon}, \lambda_{\varepsilon}}+g_{z_{\varepsilon},\lambda_{\varepsilon}})q_\varepsilon +\varepsilon V   \psi_{z_{\varepsilon},\lambda_{\varepsilon}} q_\varepsilon  \right) - A_{3,\mu}	\alpha_{\varepsilon}^{10-2\mu} \int_\Omega \int_\Omega \frac{(6-\mu)\psi_{z_{\varepsilon},\lambda_{\varepsilon}}^{6-\mu}(y)\psi_{z_{\varepsilon},\lambda_{\varepsilon}}^{5-\mu}(x)q_\varepsilon(x)}{|x-y|^{\mu}}dxdy. 
	\end{split}
\end{equation*}
By \eqref{1esay}, it holds that
\begin{equation*}
	A_{3,\mu}  \int_{\mathbb{R}^3} \int_\Omega  \frac{U_{z_{\varepsilon}, \lambda_{\varepsilon}} ^{6-\mu}(y)U_{z_{\varepsilon}, \lambda_{\varepsilon}} ^{5-\mu}(x)q_\varepsilon(x)}{|x-y|^{\mu}}dxdy= 3\int_\Omega U_{z_{\varepsilon}, \lambda_{\varepsilon}} ^{5} q_\varepsilon 
\end{equation*}
Thus, 
\begin{equation*}
	\begin{split}
		&\int_\Omega \left(  3 U_{z_{\varepsilon}, \lambda_{\varepsilon}} ^{5} q_\varepsilon -Q(f_{z_{\varepsilon}, \lambda_{\varepsilon}}+g_{z_{\varepsilon},\lambda_{\varepsilon}})q_\varepsilon +\varepsilon V   \psi_{z_{\varepsilon},\lambda_{\varepsilon}} q_\varepsilon  \right) - A_{3,\mu}	\alpha_{\varepsilon}^{10-2\mu} \int_\Omega \int_\Omega \frac{(6-\mu)\psi_{z_{\varepsilon},\lambda_{\varepsilon}}^{6-\mu}(y)\psi_{z_{\varepsilon},\lambda_{\varepsilon}}^{5-\mu}(x)q_\varepsilon(x)}{|x-y|^{\mu}}dxdy
		\\= &3\left[ 1-(6-\mu)\alpha_{\varepsilon}^{10-2\mu} \right]\int_\Omega   U_{z_{\varepsilon}, \lambda_{\varepsilon}} ^{5} q_\varepsilon    +\int_\Omega \left(  -Q(f_{z_{\varepsilon}, \lambda_{\varepsilon}}+g_{z_{\varepsilon},\lambda_{\varepsilon}})q_\varepsilon +\varepsilon V   \psi_{z_{\varepsilon},\lambda_{\varepsilon}} q_\varepsilon  \right)
		\\& - A_{3,\mu}(6-\mu)	\alpha_{\varepsilon}^{10-2\mu}\left[\int_\Omega \int_\Omega \frac{\psi_{z_{\varepsilon},\lambda_{\varepsilon}}^{6-\mu}(y)\psi_{z_{\varepsilon},\lambda_{\varepsilon}}^{5-\mu}(x)q_\varepsilon(x)}{|x-y|^{\mu}}dxdy -\int_{\mathbb{R}^3} \int_\Omega  \frac{U_{z_{\varepsilon}, \lambda_{\varepsilon}} ^{6-\mu}(y)U_{z_{\varepsilon}, \lambda_{\varepsilon}} ^{5-\mu}(x)q_\varepsilon(x)}{|x-y|^{\mu}}dxdy \right].  
	\end{split}
\end{equation*}
From the orthogonality condition and the computations carried out in the proof of  \cite[Proposition 3.3]{FKK3},
\begin{equation*}
	3\left[ 1-(6-\mu)\alpha_{\varepsilon}^{10-2\mu} \right]\int_\Omega   U_{z_{\varepsilon}, \lambda_{\varepsilon}} ^{5} q_\varepsilon = \left[ 1-(6-\mu)\alpha_{\varepsilon}^{10-2\mu} \right]\int_\Omega  \nabla s_{\varepsilon} \cdot  \nabla PU_{z_{\varepsilon},\lambda_{\varepsilon}} =   \left[ 1-(6-\mu)\alpha_{\varepsilon}^{10-2\mu} \right]   \frac{3\pi ^2}{4}\beta  \lambda_{\varepsilon}^{-1} +O\left(  \lambda_{\varepsilon}^{-2}  \right). 
\end{equation*}
From Propositions \ref{prop32} and Propositions \ref{prop34}  we have 
\begin{equation}\label{o327}
	 \left\| q_\varepsilon  \right\|_{6}\lesssim \left\| \nabla q_\varepsilon  \right\|_{2} \lesssim \lambda_{\varepsilon}^{-1} +\varepsilon \lambda_{\varepsilon}^{-\frac{1}{2}}.
\end{equation}
Since $\psi_{z_{\varepsilon},\lambda_{\varepsilon}}=U_{z_{\varepsilon},\lambda_{\varepsilon}} -  \lambda_{\varepsilon}^{-1/2} H_Q(z_{\varepsilon}, \cdot)-f_{z_{\varepsilon}, \lambda_{\varepsilon}}$, it follows from Lemmas \ref{A1}, \ref{A2} and \cite[Eq.~(B-1)]{FKK3} that $$\left\| \psi_{z_{\varepsilon},\lambda_{\varepsilon}}  \right\|_{\frac{6}{5}} \lesssim \lambda_{\varepsilon}^{-\frac{1}{2}}. $$
Moreover, Lemma \ref{A2} gives $\left\| f_{z_{\varepsilon}, \lambda_{\varepsilon}}  \right\|_{\infty}\lesssim  \lambda_{\varepsilon}^{-\frac{5}{2}}$
, and Lemma \ref{A4} provides $\left\| g_{z_{\varepsilon},\lambda_{\varepsilon}}  \right\|_{\frac{6}{5}} \lesssim  \lambda_{\varepsilon}^{-2}$, which implies
\begin{equation*}
	\begin{split}
		\left|  \int_\Omega \left(  -Q(f_{z_{\varepsilon}, \lambda_{\varepsilon}}+g_{z_{\varepsilon},\lambda_{\varepsilon}})q_\varepsilon +\varepsilon V   \psi_{z_{\varepsilon},\lambda_{\varepsilon}} q_\varepsilon  \right)  \right| \lesssim \left(   \left\| f_{z_{\varepsilon}, \lambda_{\varepsilon}}  \right\|_{\frac{6}{5}}+ \left\| g_{z_{\varepsilon},\lambda_{\varepsilon}}  \right\|_{\frac{6}{5}}+\varepsilon  \left\| \psi_{z_{\varepsilon},\lambda_{\varepsilon}}  \right\|_{\frac{6}{5}}  \right)\left\| q_\varepsilon  \right\|_{6} \lesssim  \lambda_{\varepsilon}^{-2} +\varepsilon \lambda_{\varepsilon}^{-\frac{3}{2}}  +\varepsilon ^{2}\lambda_{\varepsilon}^{-1}. 
	\end{split}
\end{equation*}

On the other hand, we decompose
\begin{equation*}
	\begin{split}
		&\int_\Omega \int_\Omega \frac{\psi_{z_{\varepsilon},\lambda_{\varepsilon}}^{6-\mu}(y)\psi_{z_{\varepsilon},\lambda_{\varepsilon}}^{5-\mu}(x)q_\varepsilon(x)}{|x-y|^{\mu}}dxdy -\int_{\mathbb{R}^3} \int_\Omega  \frac{U_{z_{\varepsilon}, \lambda_{\varepsilon}} ^{6-\mu}(y)U_{z_{\varepsilon}, \lambda_{\varepsilon}} ^{5-\mu}(x)q_\varepsilon(x)}{|x-y|^{\mu}}dxdy 
		\\ =& \underbrace{\int_\Omega \int_\Omega \frac{\psi_{z_{\varepsilon},\lambda_{\varepsilon}}^{6-\mu}(y)\psi_{z_{\varepsilon},\lambda_{\varepsilon}}^{5-\mu}(x)q_\varepsilon(x)}{|x-y|^{\mu}}dxdy -\int_\Omega \int_\Omega \frac{U_{z_{\varepsilon}, \lambda_{\varepsilon}}^{6-\mu}(y)\psi_{z_{\varepsilon},\lambda_{\varepsilon}}^{5-\mu}(x)q_\varepsilon(x)}{|x-y|^{\mu}}dxdy}_{:=I_1}
		\\ &+ \underbrace{\int_\Omega \int_\Omega \frac{U_{z_{\varepsilon}, \lambda_{\varepsilon}}^{6-\mu}(y)\psi_{z_{\varepsilon},\lambda_{\varepsilon}}^{5-\mu}(x)q_\varepsilon(x)}{|x-y|^{\mu}}dxdy -\int_\Omega  \int_\Omega  \frac{U_{z_{\varepsilon}, \lambda_{\varepsilon}} ^{6-\mu}(y)U_{z_{\varepsilon}, \lambda_{\varepsilon}} ^{5-\mu}(x)q_\varepsilon(x)}{|x-y|^{\mu}}dxdy}_{:=I_2}
		\\ &-\underbrace{ \int_{\mathbb{R}^3 \backslash \Omega } \int_\Omega  \frac{U_{z_{\varepsilon}, \lambda_{\varepsilon}} ^{6-\mu}(y)U_{z_{\varepsilon}, \lambda_{\varepsilon}} ^{5-\mu}(x)q_\varepsilon(x)}{|x-y|^{\mu}}dxdy.  }_{:=I_3}
	\end{split}
\end{equation*}
Taking the bound pointwise 
\begin{equation*}
	\begin{split}
		\left| \psi_{z_{\varepsilon},\lambda_{\varepsilon}}^{6-\mu} - U_{z_{\varepsilon}, \lambda_{\varepsilon}}^{6-\mu} \right| \lesssim&  U_{z_{\varepsilon}, \lambda_{\varepsilon}}^{5-\mu} \left(   \lambda_{\varepsilon}^{-\frac{1}{2}} \left| H_Q(z_{\varepsilon}, \cdot)\right|+\left|f_{z_{\varepsilon}, \lambda_{\varepsilon}} \right|    \right)+\left(   \lambda_{\varepsilon}^{-\frac{1}{2}} \left| H_Q(z_{\varepsilon}, \cdot)\right|+\left|f_{z_{\varepsilon}, \lambda_{\varepsilon}} \right|    \right)^{6-\mu}
\\ \lesssim &U_{z_{\varepsilon}, \lambda_{\varepsilon}}^{5-\mu} \left(   \lambda_{\varepsilon}^{-\frac{1}{2}} \left| H_Q(z_{\varepsilon}, \cdot)\right|+ \left|f_{z_{\varepsilon}, \lambda_{\varepsilon}} \right|    \right)+\lambda_{\varepsilon}^{-\frac{6-\mu}{2}}  \left| H_Q(z_{\varepsilon}, \cdot)\right|^{6-\mu} + \left|f_{z_{\varepsilon}, \lambda_{\varepsilon}} \right|^{6-\mu}
	\end{split}
\end{equation*}
and the estimate $\left\| \psi_{z_{\varepsilon},\lambda_{\varepsilon}}  \right\|_6 =O(1)$ into consideration, we obtain
\begin{equation*}
	\begin{split}
		I_1 & \lesssim  \left\|  \psi_{z_{\varepsilon},\lambda_{\varepsilon}}^{6-\mu} - U_{z_{\varepsilon}, \lambda_{\varepsilon}}^{6-\mu} \right\|_{\frac{6}{6-\mu}} 
		 \left\|  \psi_{z_{\varepsilon},\lambda_{\varepsilon}}^{5-\mu} q_\varepsilon     \right\|_{\frac{6}{6-\mu}} 
		\\& \lesssim \left(  \lambda_{\varepsilon}^{-\frac{1}{2}}\left\|U_{z_{\varepsilon}, \lambda_{\varepsilon}}^{5-\mu}  \right\|_{\frac{6}{6-\mu}} +\left\| f_{z_{\varepsilon}, \lambda_{\varepsilon}}  \right\|_{\infty} +\lambda_{\varepsilon}^{-\frac{6-\mu}{2}}   \right)    \left\|  \psi_{z_{\varepsilon},\lambda_{\varepsilon}}^{5-\mu}  \right\|_{\frac{6}{5-\mu}}  \left\|  q_\varepsilon     \right\|_{6} 
		\\&  \lesssim \lambda_{\varepsilon}^{-1} \left( \lambda_{\varepsilon}^{-1} +\varepsilon \lambda_{\varepsilon}^{-\frac{1}{2}}\right)
		\\&  \lesssim \lambda_{\varepsilon}^{-2} +\varepsilon \lambda_{\varepsilon}^{-\frac{3}{2}}. 
	\end{split}
\end{equation*}
Similarly, we have
\begin{equation*}
	\begin{split}
		I_2 & \lesssim  \left\|   U_{z_{\varepsilon}, \lambda_{\varepsilon}}^{6-\mu} \right\|_{\frac{6}{6-\mu}} \left\|  \psi_{z_{\varepsilon},\lambda_{\varepsilon}}^{5-\mu} - U_{z_{\varepsilon}, \lambda_{\varepsilon}}^{5-\mu} \right\|_{\frac{6}{5-\mu}} \left\|  q_\varepsilon     \right\|_{6} 
		\\& \lesssim \left(  \lambda_{\varepsilon}^{-\frac{1}{2}}\left\|U_{z_{\varepsilon}, \lambda_{\varepsilon}}^{4-\mu}  \right\|_{\frac{6}{5-\mu}} +\left\| f_{z_{\varepsilon}, \lambda_{\varepsilon}}  \right\|_{\infty} +\lambda_{\varepsilon}^{-\frac{5-\mu}{2}}   \right)    \left\|  q_\varepsilon     \right\|_{6} 
		\\&  \lesssim \lambda_{\varepsilon}^{-1} \left( \lambda_{\varepsilon}^{-1} +\varepsilon \lambda_{\varepsilon}^{-\frac{1}{2}}\right)
		\\&  \lesssim \lambda_{\varepsilon}^{-2} +\varepsilon \lambda_{\varepsilon}^{-\frac{3}{2}}, 
	\end{split}
\end{equation*}
and
\begin{equation*}
	\begin{split}
		I_3    \lesssim&  \| U_{z_{\varepsilon},\lambda_{\varepsilon}}^{6-\mu}\|_{L^{\frac{6}{6-\mu}}(\mathbb{R}^3\backslash\Omega)}
	\| U_{z_{\varepsilon},\lambda_{\varepsilon}}^{5-\mu}   q_\varepsilon\|_{\frac{6}{6-\mu}}
	\\ \lesssim& \| U_{z_{\varepsilon},\lambda_{\varepsilon}}\|^{6-\mu}_{L^{6}(\mathbb{R}^3\backslash\Omega)}
	\| U_{z_{\varepsilon},\lambda_{\varepsilon}}^{5-\mu} \|_{\frac{6}{5-\mu}}
	\| q_\varepsilon\|_{6}
	\\ \lesssim& \lambda_{\varepsilon}^{-\frac{6-\mu}{2}} \|   q_\varepsilon\|_{6}
	\\ \lesssim  &\lambda_{\varepsilon}^{-2} +\varepsilon \lambda_{\varepsilon}^{-\frac{3}{2}}. 
	\end{split}
\end{equation*}
Thus the  asymptotic expansion (b) follows. \\ 
(c)   Using H\"older's inequality, the HLS inequality, and \eqref{o327}, we have
\begin{equation*}
	 \int_\Omega \left(|\nabla q_\varepsilon|^2+(Q+\varepsilon V)q_\varepsilon^2\right) \lesssim  \lambda_{\varepsilon}^{-2} +\varepsilon^{2} \lambda_{\varepsilon}^{-1},
\end{equation*}
\begin{equation*}
	\begin{split}
		\int_\Omega \int_\Omega \frac{\psi_{z_{\varepsilon},\lambda_{\varepsilon}}^{5-\mu}(x)q_\varepsilon(x) \psi_{z_{\varepsilon},\lambda_{\varepsilon}}^{5-\mu}(y)q_\varepsilon(y)}{|x-y|^{\mu}}dxdy  \lesssim   \left\|  \psi_{z_{\varepsilon},\lambda_{\varepsilon}}^{5-\mu}  \right\|_{\frac{6}{5-\mu}}  \left\|  q_\varepsilon     \right\|_{6} \left\|  \psi_{z_{\varepsilon},\lambda_{\varepsilon}}^{5-\mu}  \right\|_{\frac{6}{5-\mu}}  \left\|  q_\varepsilon     \right\|_{6}  \lesssim  \lambda_{\varepsilon}^{-2} +\varepsilon^{2} \lambda_{\varepsilon}^{-1},
	\end{split}
\end{equation*}
and 
\begin{equation*}
	\begin{split}
		 \int_\Omega \int_\Omega \frac{  (\psi_{z_{\varepsilon},\lambda_{\varepsilon}}^{4-\mu} q_\varepsilon ^{2} +|q_\varepsilon|^{6-\mu}   ) (x)( \psi_{z_{\varepsilon},\lambda_{\varepsilon}}+ q_\varepsilon ) ^{6-\mu} (y)}{|x-y|^{\mu}}dxdy &\lesssim \left(  \left\|  \psi_{z_{\varepsilon},\lambda_{\varepsilon}}^{4-\mu} q_\varepsilon ^{2}  \right\|_{\frac{6}{6-\mu}} +  \left\|  |q_\varepsilon|^{6-\mu} \right\|_{\frac{6}{6-\mu}}     \right) \left\|  ( \psi_{z_{\varepsilon},\lambda_{\varepsilon}}+ q_\varepsilon ) ^{6-\mu} \right\|_{\frac{6}{6-\mu}}
		 \\& \lesssim \left\| q_\varepsilon\right\|_{6}^{2} +\left\| q_\varepsilon\right\|_{6}^{6-\mu} \lesssim  \lambda_{\varepsilon}^{-2} +\varepsilon^{2} \lambda_{\varepsilon}^{-1}. 
	\end{split}
\end{equation*}
Thus the asymptotic expansion (c) is obtained, and the proof is complete.
\end{proof}

\begin{proof}[Proof of Proposition \ref{prop36o}]
This proposition follows immediately from \eqref{o325} and Lemma \ref{lemma37o}. 
\end{proof}

\subsection{The expansion of   $\boldsymbol{\phi_Q(z_{\varepsilon})}$. } 
The main purpose of this subsection is to prove the following expansion of $\phi_Q(z_{\varepsilon})$. 
\begin{Prop}
	\label{prop38o}
	As $\varepsilon\rightarrow 0$, it holds  that 
	\begin{equation}
		\label{o328}
		\phi_Q(z_{\varepsilon})
		=\pi Q(z_{\varepsilon})\lambda_{\varepsilon}^{-1}
		-\frac{\varepsilon}{4\pi}\Theta_V(z_{\varepsilon})
		+o\left(\lambda_{\varepsilon}^{-1}\right)
		+o(\varepsilon).
	\end{equation}
\end{Prop}

The proof of \eqref{o328} is based on the identity obtained by integrating the equation for $u_{\varepsilon}$ against $\partial_{\lambda}\psi_{z_{\varepsilon},\lambda_{\varepsilon}}$.
Taking the following pointwise expansions
\begin{equation*}
	\begin{split}
		(\psi_{z_{\varepsilon},\lambda_{\varepsilon}}+q_{\varepsilon})^{6-\mu}
		=\psi_{z_{\varepsilon},\lambda_{\varepsilon}}^{6-\mu}
		+(6-\mu)\psi_{z_{\varepsilon},\lambda_{\varepsilon}}^{5-\mu}q_{\varepsilon}
		+\frac{(6-\mu)(5-\mu)}{2}\psi_{z_{\varepsilon},\lambda_{\varepsilon}}^{4-\mu}q_{\varepsilon}^{2}
		+O\left(\psi_{z_{\varepsilon},\lambda_{\varepsilon}}^{3-\mu}|q_{\varepsilon}|^{3}
		+|q_{\varepsilon}|^{6-\mu}\right)
	\end{split}
\end{equation*}
and
\begin{equation*}
	\begin{split}
		(\psi_{z_{\varepsilon},\lambda_{\varepsilon}}+q_{\varepsilon})^{5-\mu}
		=\psi_{z_{\varepsilon},\lambda_{\varepsilon}}^{5-\mu}
		+(5-\mu)\psi_{z_{\varepsilon},\lambda_{\varepsilon}}^{4-\mu}q_{\varepsilon}
		+\frac{(5-\mu)(4-\mu)}{2}\psi_{z_{\varepsilon},\lambda_{\varepsilon}}^{3-\mu}q_{\varepsilon}^{2}
		+O\left(\psi_{z_{\varepsilon},\lambda_{\varepsilon}}^{2-\mu}|q_{\varepsilon}|^{3}
		+|q_{\varepsilon}|^{5-\mu}\right)
	\end{split}
\end{equation*}
into  consideration,  we write the resulting identity in the form 
\begin{equation}\label{o330}\small
	\begin{split}
		&\int_{\Omega}\left(\nabla\psi_{z_{\varepsilon},\lambda_{\varepsilon}}\cdot\nabla\partial_{\lambda}\psi_{z_{\varepsilon},\lambda_{\varepsilon}}
		+(Q+\varepsilon V)\psi_{z_{\varepsilon},\lambda_{\varepsilon}}\partial_{\lambda}\psi_{z_{\varepsilon},\lambda_{\varepsilon}}\right)
		-A_{3,\mu}\alpha_{\varepsilon}^{10-2\mu}\int_{\Omega}\int_{\Omega}
		\frac{\psi_{z_{\varepsilon},\lambda_{\varepsilon}}^{6-\mu}(y)
		\psi_{z_{\varepsilon},\lambda_{\varepsilon}}^{5-\mu}(x)
		\partial_{\lambda}\psi_{z_{\varepsilon},\lambda_{\varepsilon}}(x)}
		{|x-y|^{\mu}}dxdy\\
		=&-\int_{\Omega}\left(\nabla q_{\varepsilon}\cdot\nabla\partial_{\lambda}\psi_{z_{\varepsilon},\lambda_{\varepsilon}}
		+Qq_{\varepsilon}\partial_{\lambda}\psi_{z_{\varepsilon},\lambda_{\varepsilon}}\right)\\
		&+A_{3,\mu}\alpha_{\varepsilon}^{10-2\mu}
		\int_{\Omega}\int_{\Omega}
		\frac{ (6-\mu)  \psi_{z_{\varepsilon},\lambda_{\varepsilon}}^{5-\mu}(y)q_{\varepsilon}(y)
		\psi_{z_{\varepsilon},\lambda_{\varepsilon}}^{5-\mu}(x)
		\partial_{\lambda}\psi_{z_{\varepsilon},\lambda_{\varepsilon}}(x) +(5-\mu)\psi_{z_{\varepsilon},\lambda_{\varepsilon}}^{6-\mu}(y)
		\psi_{z_{\varepsilon},\lambda_{\varepsilon}}^{4-\mu}(x)q_{\varepsilon}(x)
		\partial_{\lambda}\psi_{z_{\varepsilon},\lambda_{\varepsilon}}(x)     }
		{|x-y|^{\mu}}dxdy\\
		&+A_{3,\mu}\alpha_{\varepsilon}^{10-2\mu}\Bigg[
		\frac{(6-\mu)(5-\mu)}{2}
		\int_{\Omega}\int_{\Omega}
		\frac{\psi_{z_{\varepsilon},\lambda_{\varepsilon}}^{4-\mu}(y)q_{\varepsilon}^{2}(y)
		\psi_{z_{\varepsilon},\lambda_{\varepsilon}}^{5-\mu}(x)
		\partial_{\lambda}\psi_{z_{\varepsilon},\lambda_{\varepsilon}}(x)}
		{|x-y|^{\mu}}dxdy\\
		&\hspace*{4cm}+(6-\mu)(5-\mu)
		\int_{\Omega}\int_{\Omega}
		\frac{\psi_{z_{\varepsilon},\lambda_{\varepsilon}}^{5-\mu}(y)q_{\varepsilon}(y)
		\psi_{z_{\varepsilon},\lambda_{\varepsilon}}^{4-\mu}(x)q_{\varepsilon}(x)
		\partial_{\lambda}\psi_{z_{\varepsilon},\lambda_{\varepsilon}}(x)}
		{|x-y|^{\mu}}dxdy\\
		&\hspace*{4cm}+\frac{(5-\mu)(4-\mu)}{2}
		\int_{\Omega}\int_{\Omega}
		\frac{\psi_{z_{\varepsilon},\lambda_{\varepsilon}}^{6-\mu}(y)
		\psi_{z_{\varepsilon},\lambda_{\varepsilon}}^{3-\mu}(x)q_{\varepsilon}^{2}(x)
		\partial_{\lambda}\psi_{z_{\varepsilon},\lambda_{\varepsilon}}(x)}
		{|x-y|^{\mu}}dxdy\Bigg]+\mathcal{R}_{\varepsilon},
	\end{split}
\end{equation}
where
\begin{equation*}\small
	\begin{split}
		\mathcal{R}_{\varepsilon}
		=&-\varepsilon\int_{\Omega}Vq_{\varepsilon}
		\partial_{\lambda}\psi_{z_{\varepsilon},\lambda_{\varepsilon}}
		+O\Bigg(\int_{\Omega}\int_{\Omega}
		\frac{\left(\psi_{z_{\varepsilon},\lambda_{\varepsilon}}^{3-\mu}(y)|q_{\varepsilon}(y)|^{3}
		+|q_{\varepsilon}(y)|^{6-\mu}\right)
		\left(\psi_{z_{\varepsilon},\lambda_{\varepsilon}}(x)+q_{\varepsilon}(x)\right)^{5-\mu}}
		{|x-y|^{\mu}}
		\left|\partial_{\lambda}\psi_{z_{\varepsilon},\lambda_{\varepsilon}}(x)\right|dxdy\Bigg)\\
		&+O\Bigg(\int_{\Omega}\int_{\Omega}
		\frac{\left(\psi_{z_{\varepsilon},\lambda_{\varepsilon}}(y)+q_{\varepsilon}(y)\right)^{6-\mu}
		\left(\psi_{z_{\varepsilon},\lambda_{\varepsilon}}^{2-\mu}(x)|q_{\varepsilon}(x)|^{3}
		+|q_{\varepsilon}(x)|^{5-\mu}\right)}
		{|x-y|^{\mu}}
		\left|\partial_{\lambda}\psi_{z_{\varepsilon},\lambda_{\varepsilon}}(x)\right|dxdy\Bigg)\\
		&+O\Bigg(\int_{\Omega}\int_{\Omega}
		\frac{\psi_{z_{\varepsilon},\lambda_{\varepsilon}}^{4-\mu}(y)|q_{\varepsilon}(y)|^{2}
		\psi_{z_{\varepsilon},\lambda_{\varepsilon}}^{4-\mu}(x)|q_{\varepsilon}(x)| +\psi_{z_{\varepsilon},\lambda_{\varepsilon}}^{5-\mu}(y)|q_{\varepsilon}(y)|
		\psi_{z_{\varepsilon},\lambda_{\varepsilon}}^{3-\mu}(x)|q_{\varepsilon}(x)|^{2}}
		{|x-y|^{\mu}}
		\left|\partial_{\lambda}\psi_{z_{\varepsilon},\lambda_{\varepsilon}}(x)\right|dxdy\Bigg)\\
		&+O\Bigg(\int_{\Omega}\int_{\Omega}
		\frac{\psi_{z_{\varepsilon},\lambda_{\varepsilon}}^{4-\mu}(y)|q_{\varepsilon}(y)|^{2}
		\psi_{z_{\varepsilon},\lambda_{\varepsilon}}^{3-\mu}(x)|q_{\varepsilon}(x)|^{2}}
		{|x-y|^{\mu}}
		\left|\partial_{\lambda}\psi_{z_{\varepsilon},\lambda_{\varepsilon}}(x)\right|dxdy\Bigg). 
	\end{split}
\end{equation*}

The terms appearing in \eqref{o330} satisfy the following expansions.
\begin{lem}
	\label{lemma310}
	As $\varepsilon\rightarrow 0$, the following hold:
	\begin{itemize}
		\item[(a)] \quad $\displaystyle
		\int_{\Omega}\left(
		\nabla\psi_{z_{\varepsilon},\lambda_{\varepsilon}}
		\cdot\nabla\partial_{\lambda}\psi_{z_{\varepsilon},\lambda_{\varepsilon}}
		+(Q+\varepsilon V)\psi_{z_{\varepsilon},\lambda_{\varepsilon}}
		\partial_{\lambda}\psi_{z_{\varepsilon},\lambda_{\varepsilon}}
		\right)$\\
		$\hspace*{3cm}\displaystyle
		-A_{3,\mu}\alpha_{\varepsilon}^{10-2\mu}
		\int_{\Omega}\int_{\Omega}
		\frac{\psi_{z_{\varepsilon},\lambda_{\varepsilon}}^{6-\mu}(y)
		\psi_{z_{\varepsilon},\lambda_{\varepsilon}}^{5-\mu}(x)
		\partial_{\lambda}\psi_{z_{\varepsilon},\lambda_{\varepsilon}}(x)}
		{|x-y|^{\mu}}dxdy
		$\\
		$\displaystyle
		=-2\pi\phi_Q(z_{\varepsilon})\lambda_{\varepsilon}^{-2}
		+4\pi\left(1-\alpha_{\varepsilon}^{10-2\mu}\right)
		\phi_Q(z_{\varepsilon})\lambda_{\varepsilon}^{-2}
		-\frac{\varepsilon}{2}\Theta_{V}(z_{\varepsilon})
		\lambda_{\varepsilon}^{-2}$\\
		$\hspace*{3cm}\displaystyle
		+\left[2\pi^2Q(z_{\varepsilon})
		+\left((6-\mu)d_4+3\pi^2(5-\mu)\right)
		\phi_Q(z_{\varepsilon})^2\right]\lambda_{\varepsilon}^{-3}
		+o\left(\lambda_{\varepsilon}^{-3}\right)
		+o\left(\varepsilon\lambda_{\varepsilon}^{-2}\right),$
		\\ where 
		\begin{equation*}
			d_4=A_{3,\mu}\int_{\mathbb{R}^3}\int_{\mathbb{R}^3}
	\frac{U_{0,1}^{5-\mu}(y)U_{0,1}^{5-\mu}(x)}
	{|x-y|^{\mu}}dxdy.  
		\end{equation*}

		\item[(b)] \quad $\displaystyle
		\int_{\Omega}\left(
		\nabla q_{\varepsilon}\cdot
		\nabla\partial_{\lambda}\psi_{z_{\varepsilon},\lambda_{\varepsilon}}
		+Qq_{\varepsilon}
		\partial_{\lambda}\psi_{z_{\varepsilon},\lambda_{\varepsilon}}
		\right)$\\
		$\hspace*{3cm}\displaystyle
		-A_{3,\mu}\alpha_{\varepsilon}^{10-2\mu}(6-\mu)
		\int_{\Omega}\int_{\Omega}
		\frac{\psi_{z_{\varepsilon},\lambda_{\varepsilon}}^{5-\mu}(y)
		q_{\varepsilon}(y)
		\psi_{z_{\varepsilon},\lambda_{\varepsilon}}^{5-\mu}(x)
		\partial_{\lambda}\psi_{z_{\varepsilon},\lambda_{\varepsilon}}(x)}
		{|x-y|^{\mu}}dxdy$\\
		$\hspace*{3cm}\displaystyle
		-A_{3,\mu}\alpha_{\varepsilon}^{10-2\mu}(5-\mu)
		\int_{\Omega}\int_{\Omega}
		\frac{\psi_{z_{\varepsilon},\lambda_{\varepsilon}}^{6-\mu}(y)
		\psi_{z_{\varepsilon},\lambda_{\varepsilon}}^{4-\mu}(x)
		q_{\varepsilon}(x)
		\partial_{\lambda}\psi_{z_{\varepsilon},\lambda_{\varepsilon}}(x)}
		{|x-y|^{\mu}}dxdy
		$\\
		$\displaystyle
		=-2\pi\left(1-\alpha_{\varepsilon}^{10-2\mu}\right)
		\left(\phi_Q(z_{\varepsilon})-\phi_0(z_{\varepsilon})\right)
		\lambda_{\varepsilon}^{-2}
		+O\left(\phi_Q(z_{\varepsilon})\lambda_{\varepsilon}^{-3}\right)
		+o\left(\varepsilon\lambda_{\varepsilon}^{-2}\right)
		+o\left(\lambda_{\varepsilon}^{-3}\right).$

		\item[(c)] \quad $\displaystyle
		A_{3,\mu}\alpha_{\varepsilon}^{10-2\mu}\Bigg[
		\frac{(6-\mu)(5-\mu)}{2}
		\int_{\Omega}\int_{\Omega}
		\frac{\psi_{z_{\varepsilon},\lambda_{\varepsilon}}^{4-\mu}(y)
		q_{\varepsilon}^{2}(y)
		\psi_{z_{\varepsilon},\lambda_{\varepsilon}}^{5-\mu}(x)
		\partial_{\lambda}\psi_{z_{\varepsilon},\lambda_{\varepsilon}}(x)}
		{|x-y|^{\mu}}dxdy$\\
		$\hspace*{3cm}\displaystyle
		+(6-\mu)(5-\mu)
		\int_{\Omega}\int_{\Omega}
		\frac{\psi_{z_{\varepsilon},\lambda_{\varepsilon}}^{5-\mu}(y)
		q_{\varepsilon}(y)
		\psi_{z_{\varepsilon},\lambda_{\varepsilon}}^{4-\mu}(x)
		q_{\varepsilon}(x)
		\partial_{\lambda}\psi_{z_{\varepsilon},\lambda_{\varepsilon}}(x)}
		{|x-y|^{\mu}}dxdy$\\
		$\hspace*{3cm}\displaystyle
		+\frac{(5-\mu)(4-\mu)}{2}
		\int_{\Omega}\int_{\Omega}
		\frac{\psi_{z_{\varepsilon},\lambda_{\varepsilon}}^{6-\mu}(y)
		\psi_{z_{\varepsilon},\lambda_{\varepsilon}}^{3-\mu}(x)
		q_{\varepsilon}^{2}(x)
		\partial_{\lambda}\psi_{z_{\varepsilon},\lambda_{\varepsilon}}(x)}
		{|x-y|^{\mu}}dxdy\Bigg]$\\
		$\displaystyle=\frac{15\pi^2(5-\mu)}{32}
			\beta\gamma
		\lambda_{\varepsilon}^{-3}
		+O\left(\phi_Q(z_{\varepsilon})\lambda_{\varepsilon}^{-3}\right)
		+o\left(\lambda_{\varepsilon}^{-3}\right)
		+o\left(\varepsilon\lambda_{\varepsilon}^{-2}\right).$

			\item[(d)] \quad $\displaystyle
			\mathcal{R}_{\varepsilon}=O\left(\phi_Q(z_{\varepsilon})\lambda_{\varepsilon}^{-3}\right)
		+o\left(\lambda_{\varepsilon}^{-3}\right)
		+o\left(\varepsilon\lambda_{\varepsilon}^{-2}\right).$
		\end{itemize}
	\end{lem}

\begin{proof}
(a) By \eqref{o326} and integration by parts, we obtain
\begin{equation*}
	\begin{split}
		&\int_{\Omega}\left(
		\nabla\psi_{z_{\varepsilon},\lambda_{\varepsilon}}
		\cdot\nabla\partial_{\lambda}\psi_{z_{\varepsilon},\lambda_{\varepsilon}}
		+(Q+\varepsilon V)\psi_{z_{\varepsilon},\lambda_{\varepsilon}}
		\partial_{\lambda}\psi_{z_{\varepsilon},\lambda_{\varepsilon}}
		\right)
		\\
		={}&3\int_{\Omega}
		U_{z_{\varepsilon},\lambda_{\varepsilon}}^{5}
		\partial_{\lambda}\psi_{z_{\varepsilon},\lambda_{\varepsilon}}
		-\int_{\Omega}
		Q\left(f_{z_{\varepsilon},\lambda_{\varepsilon}}
		+g_{z_{\varepsilon},\lambda_{\varepsilon}}\right)
		\partial_{\lambda}\psi_{z_{\varepsilon},\lambda_{\varepsilon}}
		+\varepsilon\int_{\Omega}
		V\psi_{z_{\varepsilon},\lambda_{\varepsilon}}
		\partial_{\lambda}\psi_{z_{\varepsilon},\lambda_{\varepsilon}}.
	\end{split}
\end{equation*}
Moreover, \eqref{1esay} gives
\begin{equation*}
	A_{3,\mu}\int_{\mathbb{R}^3}\int_{\Omega}
	\frac{U_{z_{\varepsilon},\lambda_{\varepsilon}}^{6-\mu}(y)
	U_{z_{\varepsilon},\lambda_{\varepsilon}}^{5-\mu}(x)
	\partial_{\lambda}\psi_{z_{\varepsilon},\lambda_{\varepsilon}}(x)}
	{|x-y|^{\mu}}dxdy
	=3\int_{\Omega}U_{z_{\varepsilon},\lambda_{\varepsilon}}^{5}
	\partial_{\lambda}\psi_{z_{\varepsilon},\lambda_{\varepsilon}}.
\end{equation*}
Combining the two identities, we can rewrite the expression in part (a) as
\begin{equation}\label{o331}
	\begin{split}
		&3\left(1-\alpha_{\varepsilon}^{10-2\mu}\right)
		\int_{\Omega}U_{z_{\varepsilon},\lambda_{\varepsilon}}^{5}
		\partial_{\lambda}\psi_{z_{\varepsilon},\lambda_{\varepsilon}}-\int_{\Omega}
		Q\left(f_{z_{\varepsilon},\lambda_{\varepsilon}}
		+g_{z_{\varepsilon},\lambda_{\varepsilon}}\right)
		\partial_{\lambda}\psi_{z_{\varepsilon},\lambda_{\varepsilon}}
		+\varepsilon\int_{\Omega}
		V\psi_{z_{\varepsilon},\lambda_{\varepsilon}}
		\partial_{\lambda}\psi_{z_{\varepsilon},\lambda_{\varepsilon}}
		\\ &+A_{3,\mu}\alpha_{\varepsilon}^{10-2\mu}\Bigg[
		\int_{\mathbb{R}^3}\int_{\Omega}
		\frac{U_{z_{\varepsilon},\lambda_{\varepsilon}}^{6-\mu}(y)
		U_{z_{\varepsilon},\lambda_{\varepsilon}}^{5-\mu}(x)
		\partial_{\lambda}\psi_{z_{\varepsilon},\lambda_{\varepsilon}}(x)}
		{|x-y|^{\mu}}dxdy
		-\int_{\Omega}\int_{\Omega}
		\frac{\psi_{z_{\varepsilon},\lambda_{\varepsilon}}^{6-\mu}(y)
		\psi_{z_{\varepsilon},\lambda_{\varepsilon}}^{5-\mu}(x)
		\partial_{\lambda}\psi_{z_{\varepsilon},\lambda_{\varepsilon}}(x)}
		{|x-y|^{\mu}}dxdy\Bigg].
\end{split}
\end{equation}

We first estimate the first integral in \eqref{o331}. Since
\begin{equation*}
	\partial_{\lambda}\psi_{z_{\varepsilon},\lambda_{\varepsilon}}=\partial_{\lambda}U_{z_{\varepsilon},\lambda_{\varepsilon}}+\frac{1}{2}\lambda_{\varepsilon}^{-3/2}H_Q(z_{\varepsilon},\cdot)-\partial_{\lambda}f_{z_{\varepsilon},\lambda_{\varepsilon}},
\end{equation*}
we have
\begin{equation*}
	\int_{\Omega}U_{z_{\varepsilon},\lambda_{\varepsilon}}^{5}\partial_{\lambda}\psi_{z_{\varepsilon},\lambda_{\varepsilon}}=\int_{\Omega}U_{z_{\varepsilon},\lambda_{\varepsilon}}^{5}\partial_{\lambda}U_{z_{\varepsilon},\lambda_{\varepsilon}}+\frac{1}{2}\lambda_{\varepsilon}^{-3/2}\int_{\Omega}U_{z_{\varepsilon},\lambda_{\varepsilon}}^{5}H_Q(z_{\varepsilon},\cdot)-\int_{\Omega}U_{z_{\varepsilon},\lambda_{\varepsilon}}^{5}\partial_{\lambda}f_{z_{\varepsilon},\lambda_{\varepsilon}}.
\end{equation*}
Observe that
\begin{equation*}
	\int_{\mathbb{R}^3}U_{z_{\varepsilon},\lambda_{\varepsilon}}^{5}\partial_{\lambda}U_{z_{\varepsilon},\lambda_{\varepsilon}}=\frac{1}{6}\partial_{\lambda}\int_{\mathbb{R}^3}U_{z_{\varepsilon},\lambda_{\varepsilon}}^{6}=0.
\end{equation*}
Therefore, by Lemma \ref{A1RN},
\begin{equation*}
	\left|\int_{\Omega}U_{z_{\varepsilon},\lambda_{\varepsilon}}^{5}\partial_{\lambda}U_{z_{\varepsilon},\lambda_{\varepsilon}}\right|=\left|\int_{\mathbb{R}^3\setminus\Omega}U_{z_{\varepsilon},\lambda_{\varepsilon}}^{5}\partial_{\lambda}U_{z_{\varepsilon},\lambda_{\varepsilon}}\right|\lesssim\lambda_{\varepsilon}^{-1}\int_{\mathbb{R}^3\setminus\Omega}U_{z_{\varepsilon},\lambda_{\varepsilon}}^{6}=O\left(\lambda_{\varepsilon}^{-4}\right).
\end{equation*}
Moreover, Lemma \ref{A2} yields
\begin{equation*}
	\left|\int_{\Omega}U_{z_{\varepsilon},\lambda_{\varepsilon}}^{5}\partial_{\lambda}f_{z_{\varepsilon},\lambda_{\varepsilon}}\right|\leq\left\|\partial_{\lambda}f_{z_{\varepsilon},\lambda_{\varepsilon}}\right\|_{\infty}\int_{\Omega}U_{z_{\varepsilon},\lambda_{\varepsilon}}^{5}=O\left(\lambda_{\varepsilon}^{-4}\right).
\end{equation*}
On the other hand, by \eqref{oB10},
\begin{equation*}
	\frac{1}{2}\lambda_{\varepsilon}^{-3/2}\int_{\Omega}U_{z_{\varepsilon},\lambda_{\varepsilon}}^{5}H_Q(z_{\varepsilon},\cdot)=\frac{2\pi}{3}\phi_Q(z_{\varepsilon})\lambda_{\varepsilon}^{-2}-\frac{2\pi}{3}Q(z_{\varepsilon})\lambda_{\varepsilon}^{-3}+o\left(\lambda_{\varepsilon}^{-3}\right).
\end{equation*}
It follows from Proposition \ref{prop36o}, which gives $1-\alpha_{\varepsilon}^{10-2\mu}=o(1)$, that
\begin{equation*}
	3\left(1-\alpha_{\varepsilon}^{10-2\mu}\right)\int_{\Omega}U_{z_{\varepsilon},\lambda_{\varepsilon}}^{5}\partial_{\lambda}\psi_{z_{\varepsilon},\lambda_{\varepsilon}}=2\pi\left(1-\alpha_{\varepsilon}^{10-2\mu}\right)\phi_Q(z_{\varepsilon})\lambda_{\varepsilon}^{-2}+o\left(\lambda_{\varepsilon}^{-3}\right).
\end{equation*}

The estimates of the second and third integrals in \eqref{o331} follow from
the computations in \cite[Lemma~3.10(a)]{FKK3}. More precisely,
\begin{equation*}
	-\int_{\Omega}Q\left(f_{z_{\varepsilon},\lambda_{\varepsilon}}+g_{z_{\varepsilon},\lambda_{\varepsilon}}\right)\partial_{\lambda}\psi_{z_{\varepsilon},\lambda_{\varepsilon}}=-2\pi(3-\pi)Q(z_{\varepsilon})\lambda_{\varepsilon}^{-3}+o\left(\lambda_{\varepsilon}^{-3}\right),
\end{equation*}
and
\begin{equation*}
	\varepsilon\int_{\Omega}V\psi_{z_{\varepsilon},\lambda_{\varepsilon}}\partial_{\lambda}\psi_{z_{\varepsilon},\lambda_{\varepsilon}}=-\frac{\varepsilon}{2}\Theta_{V}(z_{\varepsilon})\lambda_{\varepsilon}^{-2}+o\left(\varepsilon\lambda_{\varepsilon}^{-2}\right).
\end{equation*}

We now consider the last term in \eqref{o331}. We first show that all the terms involving $f_{z_{\varepsilon},\lambda_{\varepsilon}}$ are negligible. Set
\begin{equation*}
	F_{z_{\varepsilon},\lambda_{\varepsilon}}:=U_{z_{\varepsilon},\lambda_{\varepsilon}}-\lambda_{\varepsilon}^{-1/2}H_Q(z_{\varepsilon},\cdot).
\end{equation*}
Then
\begin{equation*}
	\psi_{z_{\varepsilon},\lambda_{\varepsilon}}=F_{z_{\varepsilon},\lambda_{\varepsilon}}-f_{z_{\varepsilon},\lambda_{\varepsilon}},\qquad \partial_{\lambda}\psi_{z_{\varepsilon},\lambda_{\varepsilon}}=\partial_{\lambda}F_{z_{\varepsilon},\lambda_{\varepsilon}}-\partial_{\lambda}f_{z_{\varepsilon},\lambda_{\varepsilon}}.
\end{equation*}
We use the pointwise estimates
\begin{equation*}
	\psi_{z_{\varepsilon},\lambda_{\varepsilon}}^{6-\mu}=F_{z_{\varepsilon},\lambda_{\varepsilon}}^{6-\mu}+O\left(F_{z_{\varepsilon},\lambda_{\varepsilon}}^{5-\mu}\left|f_{z_{\varepsilon},\lambda_{\varepsilon}}\right|+\left|f_{z_{\varepsilon},\lambda_{\varepsilon}}\right|^{6-\mu}\right)
\end{equation*}
and
\begin{equation*}
	\psi_{z_{\varepsilon},\lambda_{\varepsilon}}^{5-\mu}=F_{z_{\varepsilon},\lambda_{\varepsilon}}^{5-\mu}+O\left(F_{z_{\varepsilon},\lambda_{\varepsilon}}^{4-\mu}\left|f_{z_{\varepsilon},\lambda_{\varepsilon}}\right|+\left|f_{z_{\varepsilon},\lambda_{\varepsilon}}\right|^{5-\mu}\right), 
\end{equation*}
to obtain
\begin{equation*}
\begin{split}
	&\int_{\Omega}\int_{\Omega}
	\frac{\psi_{z_{\varepsilon},\lambda_{\varepsilon}}^{6-\mu}(y)
	\psi_{z_{\varepsilon},\lambda_{\varepsilon}}^{5-\mu}(x)
	\partial_{\lambda}\psi_{z_{\varepsilon},\lambda_{\varepsilon}}(x)}
	{|x-y|^{\mu}}dxdy\\
	&=\int_{\Omega}\int_{\Omega}
	\frac{F_{z_{\varepsilon},\lambda_{\varepsilon}}^{6-\mu}(y)
	F_{z_{\varepsilon},\lambda_{\varepsilon}}^{5-\mu}(x)
	\partial_{\lambda}F_{z_{\varepsilon},\lambda_{\varepsilon}}(x)}
	{|x-y|^{\mu}}dxdy\\
	&\quad+O\Bigg(
	\int_{\Omega}\int_{\Omega}
	\frac{\left(F_{z_{\varepsilon},\lambda_{\varepsilon}}^{5-\mu}
	|f_{z_{\varepsilon},\lambda_{\varepsilon}}|
	+|f_{z_{\varepsilon},\lambda_{\varepsilon}}|^{6-\mu}\right)(y)
	\psi_{z_{\varepsilon},\lambda_{\varepsilon}}^{5-\mu}(x)
	|\partial_{\lambda}F_{z_{\varepsilon},\lambda_{\varepsilon}}(x)|}
	{|x-y|^{\mu}}dxdy\\
	&\qquad+\int_{\Omega}\int_{\Omega}
	\frac{\psi_{z_{\varepsilon},\lambda_{\varepsilon}}^{6-\mu}(y)
	\left(F_{z_{\varepsilon},\lambda_{\varepsilon}}^{4-\mu}
	|f_{z_{\varepsilon},\lambda_{\varepsilon}}|
	+|f_{z_{\varepsilon},\lambda_{\varepsilon}}|^{5-\mu}\right)(x)
	|\partial_{\lambda}F_{z_{\varepsilon},\lambda_{\varepsilon}}(x)|}
	{|x-y|^{\mu}}dxdy\\
	&\qquad+\int_{\Omega}\int_{\Omega}
	\frac{\psi_{z_{\varepsilon},\lambda_{\varepsilon}}^{6-\mu}(y)
	\psi_{z_{\varepsilon},\lambda_{\varepsilon}}^{5-\mu}(x)
	|\partial_{\lambda}f_{z_{\varepsilon},\lambda_{\varepsilon}}(x)|}
	{|x-y|^{\mu}}dxdy\Bigg).
\end{split}
\end{equation*}
By the HLS inequality and Lemmas \ref{A1} and \ref{A2}, all the remainder terms are of order $o\left(\lambda_{\varepsilon}^{-3}\right)$. Hence,
\begin{equation*}
\begin{split}
	&\int_{\Omega}\int_{\Omega}
	\frac{\psi_{z_{\varepsilon},\lambda_{\varepsilon}}^{6-\mu}(y)
	\psi_{z_{\varepsilon},\lambda_{\varepsilon}}^{5-\mu}(x)
	\partial_{\lambda}\psi_{z_{\varepsilon},\lambda_{\varepsilon}}(x)}
	{|x-y|^{\mu}}dxdy=\int_{\Omega}\int_{\Omega}
	\frac{F_{z_{\varepsilon},\lambda_{\varepsilon}}^{6-\mu}(y)
	F_{z_{\varepsilon},\lambda_{\varepsilon}}^{5-\mu}(x)
	\partial_{\lambda}F_{z_{\varepsilon},\lambda_{\varepsilon}}(x)}
	{|x-y|^{\mu}}dxdy+o\left(\lambda_{\varepsilon}^{-3}\right).
\end{split}
\end{equation*}
Taking the following pointwise expansions
\begin{equation*}
	F_{z_{\varepsilon},\lambda_{\varepsilon}}^{6-\mu}=U_{z_{\varepsilon},\lambda_{\varepsilon}}^{6-\mu}-(6-\mu)\lambda_{\varepsilon}^{-1/2}U_{z_{\varepsilon},\lambda_{\varepsilon}}^{5-\mu}H_Q(z_{\varepsilon},\cdot)+\frac{(6-\mu)(5-\mu)}{2}\lambda_{\varepsilon}^{-1}U_{z_{\varepsilon},\lambda_{\varepsilon}}^{4-\mu}H_Q(z_{\varepsilon},\cdot)^2+\rho_{6-\mu},
\end{equation*}
\begin{equation*}
	F_{z_{\varepsilon},\lambda_{\varepsilon}}^{5-\mu}=U_{z_{\varepsilon},\lambda_{\varepsilon}}^{5-\mu}-(5-\mu)\lambda_{\varepsilon}^{-1/2}U_{z_{\varepsilon},\lambda_{\varepsilon}}^{4-\mu}H_Q(z_{\varepsilon},\cdot)+\frac{(5-\mu)(4-\mu)}{2}\lambda_{\varepsilon}^{-1}U_{z_{\varepsilon},\lambda_{\varepsilon}}^{3-\mu}H_Q(z_{\varepsilon},\cdot)^2+\rho_{5-\mu},
\end{equation*}
with 
\begin{equation*}
	\rho_{6-\mu}
	=O\left(\lambda_{\varepsilon}^{-3/2}
	U_{z_{\varepsilon},\lambda_{\varepsilon}}^{3-\mu}
	\left|H_Q(z_{\varepsilon},\cdot)\right|^3
	+\lambda_{\varepsilon}^{-(6-\mu)/2}
	\left|H_Q(z_{\varepsilon},\cdot)\right|^{6-\mu}\right),
\end{equation*}
\begin{equation*}
	\rho_{5-\mu}
	=O\left(\lambda_{\varepsilon}^{-3/2}
	U_{z_{\varepsilon},\lambda_{\varepsilon}}^{2-\mu}
	\left|H_Q(z_{\varepsilon},\cdot)\right|^3
	+\lambda_{\varepsilon}^{-(5-\mu)/2}
	\left|H_Q(z_{\varepsilon},\cdot)\right|^{5-\mu}\right),
\end{equation*}
and 
\begin{equation*}
	\partial_{\lambda}F_{z_{\varepsilon},\lambda_{\varepsilon}}
	=\partial_{\lambda}U_{z_{\varepsilon},\lambda_{\varepsilon}}
	+\frac{1}{2}\lambda_{\varepsilon}^{-3/2}H_Q(z_{\varepsilon},\cdot)
\end{equation*}
into consideration,  we obtain
\begin{equation*}
\begin{split}
	&\int_{\Omega}\int_{\Omega}
	\frac{F_{z_{\varepsilon},\lambda_{\varepsilon}}^{6-\mu}(y)
	F_{z_{\varepsilon},\lambda_{\varepsilon}}^{5-\mu}(x)
	\partial_{\lambda}F_{z_{\varepsilon},\lambda_{\varepsilon}}(x)}
	{|x-y|^{\mu}}dxdy\\
	&=\int_{\Omega}\int_{\Omega}
	\frac{U_{z_{\varepsilon},\lambda_{\varepsilon}}^{6-\mu}(y)
	U_{z_{\varepsilon},\lambda_{\varepsilon}}^{5-\mu}(x)
	\partial_{\lambda}F_{z_{\varepsilon},\lambda_{\varepsilon}}(x)}
	{|x-y|^{\mu}}dxdy
	+\mathcal{K}_{\varepsilon}+R_0, 
\end{split}
\end{equation*}
where 
\begin{equation*}
\begin{split}
	\mathcal{K}_{\varepsilon}
	:={}&-(6-\mu)\lambda_{\varepsilon}^{-1/2}
	\int_{\Omega}\int_{\Omega}
	\frac{U_{z_{\varepsilon},\lambda_{\varepsilon}}^{5-\mu}(y)
	H_Q(z_{\varepsilon},y)
	U_{z_{\varepsilon},\lambda_{\varepsilon}}^{5-\mu}(x)
	\partial_{\lambda}U_{z_{\varepsilon},\lambda_{\varepsilon}}(x)}
	{|x-y|^{\mu}}dxdy\\
	&\quad-(5-\mu)\lambda_{\varepsilon}^{-1/2}
	\int_{\Omega}\int_{\Omega}
	\frac{U_{z_{\varepsilon},\lambda_{\varepsilon}}^{6-\mu}(y)
	U_{z_{\varepsilon},\lambda_{\varepsilon}}^{4-\mu}(x)
	H_Q(z_{\varepsilon},x)
	\partial_{\lambda}U_{z_{\varepsilon},\lambda_{\varepsilon}}(x)}
	{|x-y|^{\mu}}dxdy\\
	&\quad+\frac{(6-\mu)(5-\mu)}{2}\lambda_{\varepsilon}^{-1}
	\int_{\Omega}\int_{\Omega}
	\frac{U_{z_{\varepsilon},\lambda_{\varepsilon}}^{4-\mu}(y)
	H_Q(z_{\varepsilon},y)^2
	U_{z_{\varepsilon},\lambda_{\varepsilon}}^{5-\mu}(x)
	\partial_{\lambda}U_{z_{\varepsilon},\lambda_{\varepsilon}}(x)}
	{|x-y|^{\mu}}dxdy\\
	&\quad+(6-\mu)(5-\mu)\lambda_{\varepsilon}^{-1}
	\int_{\Omega}\int_{\Omega}
	\frac{U_{z_{\varepsilon},\lambda_{\varepsilon}}^{5-\mu}(y)
	H_Q(z_{\varepsilon},y)
	U_{z_{\varepsilon},\lambda_{\varepsilon}}^{4-\mu}(x)
	H_Q(z_{\varepsilon},x)
	\partial_{\lambda}U_{z_{\varepsilon},\lambda_{\varepsilon}}(x)}
	{|x-y|^{\mu}}dxdy\\
	&\quad+\frac{(5-\mu)(4-\mu)}{2}\lambda_{\varepsilon}^{-1}
	\int_{\Omega}\int_{\Omega}
	\frac{U_{z_{\varepsilon},\lambda_{\varepsilon}}^{6-\mu}(y)
	U_{z_{\varepsilon},\lambda_{\varepsilon}}^{3-\mu}(x)
	H_Q(z_{\varepsilon},x)^2
	\partial_{\lambda}U_{z_{\varepsilon},\lambda_{\varepsilon}}(x)}
	{|x-y|^{\mu}}dxdy\\
	&\quad-\frac{6-\mu}{2}\lambda_{\varepsilon}^{-2}
	\int_{\Omega}\int_{\Omega}
	\frac{U_{z_{\varepsilon},\lambda_{\varepsilon}}^{5-\mu}(y)
	H_Q(z_{\varepsilon},y)
	U_{z_{\varepsilon},\lambda_{\varepsilon}}^{5-\mu}(x)
	H_Q(z_{\varepsilon},x)}
	{|x-y|^{\mu}}dxdy\\
	&\quad-\frac{5-\mu}{2}\lambda_{\varepsilon}^{-2}
	\int_{\Omega}\int_{\Omega}
	\frac{U_{z_{\varepsilon},\lambda_{\varepsilon}}^{6-\mu}(y)
	U_{z_{\varepsilon},\lambda_{\varepsilon}}^{4-\mu}(x)
	H_Q(z_{\varepsilon},x)^2}
	{|x-y|^{\mu}}dxdy,
\end{split}
\end{equation*}
and 
\begin{equation*}
\begin{split}
	R_0=O\Bigg(&\lambda_{\varepsilon}^{-3/2}
	\int_{\Omega}\int_{\Omega}\frac{
	U_{z_{\varepsilon},\lambda_{\varepsilon}}^{4-\mu}(y)
	\left|H_Q(z_{\varepsilon},y)\right|^2
	U_{z_{\varepsilon},\lambda_{\varepsilon}}^{4-\mu}(x)
	\left|H_Q(z_{\varepsilon},x)\right|
	\left|\partial_{\lambda}F_{z_{\varepsilon},\lambda_{\varepsilon}}(x)\right|}
	{|x-y|^{\mu}}dxdy\\
	&+\lambda_{\varepsilon}^{-3/2}
	\int_{\Omega}\int_{\Omega}\frac{
	U_{z_{\varepsilon},\lambda_{\varepsilon}}^{5-\mu}(y)
	\left|H_Q(z_{\varepsilon},y)\right|
	U_{z_{\varepsilon},\lambda_{\varepsilon}}^{3-\mu}(x)
	\left|H_Q(z_{\varepsilon},x)\right|^2
	\left|\partial_{\lambda}F_{z_{\varepsilon},\lambda_{\varepsilon}}(x)\right|}
	{|x-y|^{\mu}}dxdy\\
	&+\lambda_{\varepsilon}^{-2}
	\int_{\Omega}\int_{\Omega}\frac{
	U_{z_{\varepsilon},\lambda_{\varepsilon}}^{4-\mu}(y)
	\left|H_Q(z_{\varepsilon},y)\right|^2
	U_{z_{\varepsilon},\lambda_{\varepsilon}}^{3-\mu}(x)
	\left|H_Q(z_{\varepsilon},x)\right|^2
	\left|\partial_{\lambda}F_{z_{\varepsilon},\lambda_{\varepsilon}}(x)\right|}
	{|x-y|^{\mu}}dxdy\\
	&+\int_{\Omega}\int_{\Omega}\frac{
	\left[\left|\rho_{6-\mu}(y)\right|
	F_{z_{\varepsilon},\lambda_{\varepsilon}}^{5-\mu}(x)
	+F_{z_{\varepsilon},\lambda_{\varepsilon}}^{6-\mu}(y)
	\left|\rho_{5-\mu}(x)\right|\right]
	\left|\partial_{\lambda}F_{z_{\varepsilon},\lambda_{\varepsilon}}(x)\right|}
	{|x-y|^{\mu}}dxdy\\
	&+\lambda_{\varepsilon}^{-5/2}
	\int_{\Omega}\int_{\Omega}\frac{
	U_{z_{\varepsilon},\lambda_{\varepsilon}}^{4-\mu}(y)
	\left|H_Q(z_{\varepsilon},y)\right|^2
	U_{z_{\varepsilon},\lambda_{\varepsilon}}^{5-\mu}(x)
	\left|H_Q(z_{\varepsilon},x)\right|}
	{|x-y|^{\mu}}dxdy\\
	&+\lambda_{\varepsilon}^{-5/2}
	\int_{\Omega}\int_{\Omega}\frac{
	U_{z_{\varepsilon},\lambda_{\varepsilon}}^{5-\mu}(y)
	\left|H_Q(z_{\varepsilon},y)\right|
	U_{z_{\varepsilon},\lambda_{\varepsilon}}^{4-\mu}(x)
	\left|H_Q(z_{\varepsilon},x)\right|^2}
	{|x-y|^{\mu}}dxdy\\
	&+\lambda_{\varepsilon}^{-5/2}
	\int_{\Omega}\int_{\Omega}\frac{
	U_{z_{\varepsilon},\lambda_{\varepsilon}}^{6-\mu}(y)
	U_{z_{\varepsilon},\lambda_{\varepsilon}}^{3-\mu}(x)
	\left|H_Q(z_{\varepsilon},x)\right|^3}
	{|x-y|^{\mu}}dxdy\Bigg).
\end{split}
\end{equation*}
Now we estimate  $R_0$. Since
\begin{equation*}
	\partial_{\lambda}F_{z_{\varepsilon},\lambda_{\varepsilon}}
	=\partial_{\lambda}U_{z_{\varepsilon},\lambda_{\varepsilon}}
	+\frac{1}{2}\lambda_{\varepsilon}^{-3/2}H_Q(z_{\varepsilon},\cdot),
\end{equation*}
Lemmas \ref{A1} and \ref{B1} give
\begin{equation*}
\begin{split}
	&\left\|
	U_{z_{\varepsilon},\lambda_{\varepsilon}}^{4-\mu}
	\left|H_Q(z_{\varepsilon},\cdot)\right|
	\partial_{\lambda}F_{z_{\varepsilon},\lambda_{\varepsilon}}
	\right\|_{\frac{6}{6-\mu}}\\
	&\lesssim\lambda_{\varepsilon}^{-1}
	\left\|U_{z_{\varepsilon},\lambda_{\varepsilon}}^{5-\mu}\right\|_{\frac{6}{6-\mu}}
	+\lambda_{\varepsilon}^{-3/2}
	\left\|U_{z_{\varepsilon},\lambda_{\varepsilon}}^{4-\mu}\right\|_{\frac{6}{6-\mu}}.
\end{split}
\end{equation*}
Hence, by the HLS inequality,
\begin{equation*}
\begin{split}
	&\lambda_{\varepsilon}^{-3/2}
	\int_{\Omega}\int_{\Omega}
	\frac{U_{z_{\varepsilon},\lambda_{\varepsilon}}^{4-\mu}(y)
	\left|H_Q(z_{\varepsilon},y)\right|^2
	U_{z_{\varepsilon},\lambda_{\varepsilon}}^{4-\mu}(x)
	\left|H_Q(z_{\varepsilon},x)\right|
	\left|\partial_{\lambda}F_{z_{\varepsilon},\lambda_{\varepsilon}}(x)\right|}
	{|x-y|^{\mu}}dxdy\\
	&\lesssim\lambda_{\varepsilon}^{-3/2}
	\left\|U_{z_{\varepsilon},\lambda_{\varepsilon}}^{4-\mu}\right\|_{\frac{6}{6-\mu}}
	\left(
	\lambda_{\varepsilon}^{-1}
	\left\|U_{z_{\varepsilon},\lambda_{\varepsilon}}^{5-\mu}\right\|_{\frac{6}{6-\mu}}
	+\lambda_{\varepsilon}^{-3/2}
	\left\|U_{z_{\varepsilon},\lambda_{\varepsilon}}^{4-\mu}\right\|_{\frac{6}{6-\mu}}
	\right)
	\lesssim\lambda_{\varepsilon}^{-4}.
\end{split}
\end{equation*}
Moreover, by the definition of $\rho_{6-\mu}$ and Lemmas \ref{A1} and \ref{B1},
\begin{equation*}
\begin{split}
	\left\|\rho_{6-\mu}\right\|_{\frac{6}{6-\mu}}
	&\lesssim\lambda_{\varepsilon}^{-3/2}
	\left\|U_{z_{\varepsilon},\lambda_{\varepsilon}}^{3-\mu}\right\|_{\frac{6}{6-\mu}}
	+\lambda_{\varepsilon}^{-(6-\mu)/2}\\
	&=\lambda_{\varepsilon}^{-3/2}
	\left\|U_{z_{\varepsilon},\lambda_{\varepsilon}}\right\|_{\frac{6(3-\mu)}{6-\mu}}^{3-\mu}
	+\lambda_{\varepsilon}^{-(6-\mu)/2}
	\lesssim\lambda_{\varepsilon}^{-3+\mu/2},
\end{split}
\end{equation*}
where we used $\frac{6(3-\mu)}{6-\mu}<3$. It follows that
\begin{equation*}
\begin{split}
	&\int_{\Omega}\int_{\Omega}
	\frac{\left|\rho_{6-\mu}(y)\right|
	F_{z_{\varepsilon},\lambda_{\varepsilon}}^{5-\mu}(x)
	\left|\partial_{\lambda}F_{z_{\varepsilon},\lambda_{\varepsilon}}(x)\right|}
	{|x-y|^{\mu}}dxdy
	\lesssim\lambda_{\varepsilon}^{-4+\mu/2}
	=o\left(\lambda_{\varepsilon}^{-3}\right).
\end{split}
\end{equation*}
Here we used $0<\mu<2$, which implies $-4+\mu/2<-3$.
Finally,
\begin{equation*}
\begin{split}
	&\lambda_{\varepsilon}^{-5/2}
	\int_{\Omega}\int_{\Omega}
	\frac{U_{z_{\varepsilon},\lambda_{\varepsilon}}^{4-\mu}(y)
	\left|H_Q(z_{\varepsilon},y)\right|^2
	U_{z_{\varepsilon},\lambda_{\varepsilon}}^{5-\mu}(x)
	\left|H_Q(z_{\varepsilon},x)\right|}
	{|x-y|^{\mu}}dxdy\\
	&\lesssim\lambda_{\varepsilon}^{-5/2}
	\left\|U_{z_{\varepsilon},\lambda_{\varepsilon}}^{4-\mu}\right\|_{\frac{6}{6-\mu}}
	\left\|U_{z_{\varepsilon},\lambda_{\varepsilon}}^{5-\mu}\right\|_{\frac{6}{6-\mu}}
	\lesssim\lambda_{\varepsilon}^{-4}.
\end{split}
\end{equation*}
The remaining terms  in $R_0$ can be estimated in the same way. Therefore,
\begin{equation*}
	R_0=o\left(\lambda_{\varepsilon}^{-3}\right).
\end{equation*}
We now return to the last term in \eqref{o331}. Since
\begin{equation*}
	\partial_{\lambda}\psi_{z_{\varepsilon},\lambda_{\varepsilon}}
	=\partial_{\lambda}F_{z_{\varepsilon},\lambda_{\varepsilon}}
	-\partial_{\lambda}f_{z_{\varepsilon},\lambda_{\varepsilon}},
\end{equation*}
it follows from \eqref{1esay} and Lemma \ref{A2} that
\begin{equation*}
\begin{split}
	&\left|\int_{\mathbb{R}^3}\int_{\Omega}
	\frac{U_{z_{\varepsilon},\lambda_{\varepsilon}}^{6-\mu}(y)
	U_{z_{\varepsilon},\lambda_{\varepsilon}}^{5-\mu}(x)
	\partial_{\lambda}f_{z_{\varepsilon},\lambda_{\varepsilon}}(x)}
	{|x-y|^{\mu}}dxdy\right|\\
	&=\frac{3}{A_{3,\mu}}
	\left|\int_{\Omega}U_{z_{\varepsilon},\lambda_{\varepsilon}}^5
	\partial_{\lambda}f_{z_{\varepsilon},\lambda_{\varepsilon}}\right|
	=O\left(\lambda_{\varepsilon}^{-4}\right).
\end{split}
\end{equation*}
Moreover, by the HLS inequality and Lemmas \ref{A1}, \ref{A1RN}, and \ref{B1}, we have
\begin{equation*}
\begin{split}
	&\left|\int_{\mathbb{R}^3\setminus\Omega}\int_{\Omega}
	\frac{U_{z_{\varepsilon},\lambda_{\varepsilon}}^{6-\mu}(y)
	U_{z_{\varepsilon},\lambda_{\varepsilon}}^{5-\mu}(x)
	\partial_{\lambda}F_{z_{\varepsilon},\lambda_{\varepsilon}}(x)}
	{|x-y|^{\mu}}dxdy\right|\\
	&\lesssim
	\left\|U_{z_{\varepsilon},\lambda_{\varepsilon}}^{6-\mu}
	\right\|_{L^{\frac{6}{6-\mu}}(\mathbb{R}^3\setminus\Omega)}
	\left\|U_{z_{\varepsilon},\lambda_{\varepsilon}}^{5-\mu}
	\partial_{\lambda}F_{z_{\varepsilon},\lambda_{\varepsilon}}
	\right\|_{L^{\frac{6}{6-\mu}}(\Omega)}
	\lesssim\lambda_{\varepsilon}^{-4+\mu/2}
	=o\left(\lambda_{\varepsilon}^{-3}\right),
\end{split}
		\end{equation*}
		where we used $0<\mu<2$.
Consequently, the last term in \eqref{o331} is equal to
\begin{equation*}
	A_{3,\mu}\alpha_{\varepsilon}^{10-2\mu}
	\left[-\mathcal{K}_{\varepsilon}
	+o\left(\lambda_{\varepsilon}^{-3}\right)\right].
\end{equation*}
Taking Lemma \ref{B4} into consideration,  we obtain
\begin{equation*}
\begin{split}
	\mathcal{K}_{\varepsilon}
	={}&\frac{2\pi}{A_{3,\mu}}\phi_Q(z_{\varepsilon})
	\lambda_{\varepsilon}^{-2}
	-\frac{6\pi}{A_{3,\mu}}Q(z_{\varepsilon})
	\lambda_{\varepsilon}^{-3}-\frac{(6-\mu)d_4+3\pi^2(5-\mu)}{A_{3,\mu}}
	\phi_Q(z_{\varepsilon})^2\lambda_{\varepsilon}^{-3}
	+o\left(\lambda_{\varepsilon}^{-3}\right).
\end{split}
\end{equation*}
Therefore, the last term in \eqref{o331} is equal to
\begin{equation*}
\begin{split}
	&-2\pi\alpha_{\varepsilon}^{10-2\mu}
	\phi_Q(z_{\varepsilon})\lambda_{\varepsilon}^{-2}
	+6\pi\alpha_{\varepsilon}^{10-2\mu}
	Q(z_{\varepsilon})\lambda_{\varepsilon}^{-3}+\alpha_{\varepsilon}^{10-2\mu}
	\left[(6-\mu)d_4+3\pi^2(5-\mu)\right]
	\phi_Q(z_{\varepsilon})^2\lambda_{\varepsilon}^{-3}
	+o\left(\lambda_{\varepsilon}^{-3}\right).
\end{split}
\end{equation*}
Combining the  estimates above completes the proof of part (a).

(b) Differentiating \eqref{o326} with respect to $\lambda$, we obtain
\begin{equation*}
	(-\Delta+Q)\partial_{\lambda}\psi_{z_{\varepsilon},\lambda_{\varepsilon}}
	=15U_{z_{\varepsilon},\lambda_{\varepsilon}}^{4}
	\partial_{\lambda}U_{z_{\varepsilon},\lambda_{\varepsilon}}
	-Q\left(\partial_{\lambda}f_{z_{\varepsilon},\lambda_{\varepsilon}}
	+\partial_{\lambda}g_{z_{\varepsilon},\lambda_{\varepsilon}}\right).
\end{equation*}
Hence, by integration by parts, the left-hand side in part (b) can be written as
\begin{equation*}
\begin{split}
	&15\int_{\Omega}q_{\varepsilon}
	U_{z_{\varepsilon},\lambda_{\varepsilon}}^{4}
	\partial_{\lambda}U_{z_{\varepsilon},\lambda_{\varepsilon}}
	-\int_{\Omega}Qq_{\varepsilon}
	\left(\partial_{\lambda}f_{z_{\varepsilon},\lambda_{\varepsilon}}
	+\partial_{\lambda}g_{z_{\varepsilon},\lambda_{\varepsilon}}\right)\\
	&\quad-A_{3,\mu}\alpha_{\varepsilon}^{10-2\mu}(6-\mu)
	\int_{\Omega}\int_{\Omega}
	\frac{\psi_{z_{\varepsilon},\lambda_{\varepsilon}}^{5-\mu}(y)
	q_{\varepsilon}(y)
	\psi_{z_{\varepsilon},\lambda_{\varepsilon}}^{5-\mu}(x)
	\partial_{\lambda}\psi_{z_{\varepsilon},\lambda_{\varepsilon}}(x)}
	{|x-y|^{\mu}}dxdy\\
	&\quad-A_{3,\mu}\alpha_{\varepsilon}^{10-2\mu}(5-\mu)
	\int_{\Omega}\int_{\Omega}
	\frac{\psi_{z_{\varepsilon},\lambda_{\varepsilon}}^{6-\mu}(y)
	\psi_{z_{\varepsilon},\lambda_{\varepsilon}}^{4-\mu}(x)
	q_{\varepsilon}(x)
	\partial_{\lambda}\psi_{z_{\varepsilon},\lambda_{\varepsilon}}(x)}
	{|x-y|^{\mu}}dxdy.
\end{split}
\end{equation*}
On the other hand, differentiating the equation for
$U_{z_{\varepsilon},\lambda_{\varepsilon}}$ with respect to $\lambda$ gives
\begin{equation*}
\begin{split}
	&15U_{z_{\varepsilon},\lambda_{\varepsilon}}^{4}(x)
	\partial_{\lambda}U_{z_{\varepsilon},\lambda_{\varepsilon}}(x)\\
	={}&A_{3,\mu}(6-\mu)
	\int_{\mathbb{R}^{3}}
	\frac{U_{z_{\varepsilon},\lambda_{\varepsilon}}^{5-\mu}(y)
	\partial_{\lambda}U_{z_{\varepsilon},\lambda_{\varepsilon}}(y)}
	{|x-y|^{\mu}}dy\,
	U_{z_{\varepsilon},\lambda_{\varepsilon}}^{5-\mu}(x)+A_{3,\mu}(5-\mu)
	\int_{\mathbb{R}^{3}}
	\frac{U_{z_{\varepsilon},\lambda_{\varepsilon}}^{6-\mu}(y)}
	{|x-y|^{\mu}}dy\,
	U_{z_{\varepsilon},\lambda_{\varepsilon}}^{4-\mu}(x)
	\partial_{\lambda}U_{z_{\varepsilon},\lambda_{\varepsilon}}(x).
\end{split}
\end{equation*}
Multiplying this identity by $q_{\varepsilon}$, integrating over $\Omega$,
 we have
\begin{equation*}
\begin{split}
	&15\int_{\Omega}q_{\varepsilon}
	U_{z_{\varepsilon},\lambda_{\varepsilon}}^{4}
	\partial_{\lambda}U_{z_{\varepsilon},\lambda_{\varepsilon}}\\
	={}&A_{3,\mu}(6-\mu)\int_{\Omega}\int_{\mathbb{R}^{3}}
	\frac{U_{z_{\varepsilon},\lambda_{\varepsilon}}^{5-\mu}(y)
	q_{\varepsilon}(y)
	U_{z_{\varepsilon},\lambda_{\varepsilon}}^{5-\mu}(x)
	\partial_{\lambda}U_{z_{\varepsilon},\lambda_{\varepsilon}}(x)}
	{|x-y|^{\mu}}dxdy\\
	&+A_{3,\mu}(5-\mu)\int_{\mathbb{R}^{3}}\int_{\Omega}
	\frac{U_{z_{\varepsilon},\lambda_{\varepsilon}}^{6-\mu}(y)
	U_{z_{\varepsilon},\lambda_{\varepsilon}}^{4-\mu}(x)
	q_{\varepsilon}(x)
	\partial_{\lambda}U_{z_{\varepsilon},\lambda_{\varepsilon}}(x)}
	{|x-y|^{\mu}}dxdy.
\end{split}
\end{equation*}
Therefore, the left-hand side in part (b) is equal to
\begin{equation}\label{o337}
\begin{split}
	&15\left(1-\alpha_{\varepsilon}^{10-2\mu}\right)
	\int_{\Omega}q_{\varepsilon}
	U_{z_{\varepsilon},\lambda_{\varepsilon}}^{4}
	\partial_{\lambda}U_{z_{\varepsilon},\lambda_{\varepsilon}}\\
	&+A_{3,\mu}\alpha_{\varepsilon}^{10-2\mu}\Bigg[
	(6-\mu)\int_{\Omega}\int_{\mathbb{R}^{3}}
	\frac{U_{z_{\varepsilon},\lambda_{\varepsilon}}^{5-\mu}(y)
	q_{\varepsilon}(y)
	U_{z_{\varepsilon},\lambda_{\varepsilon}}^{5-\mu}(x)
	\partial_{\lambda}U_{z_{\varepsilon},\lambda_{\varepsilon}}(x)}
	{|x-y|^{\mu}}dxdy\\
	&\hspace*{1cm}-(6-\mu)\int_{\Omega}\int_{\Omega}
	\frac{\psi_{z_{\varepsilon},\lambda_{\varepsilon}}^{5-\mu}(y)
	q_{\varepsilon}(y)
	\psi_{z_{\varepsilon},\lambda_{\varepsilon}}^{5-\mu}(x)
	\partial_{\lambda}\psi_{z_{\varepsilon},\lambda_{\varepsilon}}(x)}
	{|x-y|^{\mu}}dxdy\\
	&\hspace*{1cm}+(5-\mu)\int_{\mathbb{R}^{3}}\int_{\Omega}
	\frac{U_{z_{\varepsilon},\lambda_{\varepsilon}}^{6-\mu}(y)
	U_{z_{\varepsilon},\lambda_{\varepsilon}}^{4-\mu}(x)
	q_{\varepsilon}(x)
	\partial_{\lambda}U_{z_{\varepsilon},\lambda_{\varepsilon}}(x)}
	{|x-y|^{\mu}}dxdy\\
	&\hspace*{1cm}-(5-\mu)\int_{\Omega}\int_{\Omega}
	\frac{\psi_{z_{\varepsilon},\lambda_{\varepsilon}}^{6-\mu}(y)
	\psi_{z_{\varepsilon},\lambda_{\varepsilon}}^{4-\mu}(x)
	q_{\varepsilon}(x)
	\partial_{\lambda}\psi_{z_{\varepsilon},\lambda_{\varepsilon}}(x)}
	{|x-y|^{\mu}}dxdy\Bigg]\\
	&-\int_{\Omega}Qq_{\varepsilon}
	\left(\partial_{\lambda}f_{z_{\varepsilon},\lambda_{\varepsilon}}
	+\partial_{\lambda}g_{z_{\varepsilon},\lambda_{\varepsilon}}\right).
\end{split}
\end{equation}
We first estimate the first term in \eqref{o337}. By the identity
$-\Delta\partial_{\lambda}PU_{z_{\varepsilon},\lambda_{\varepsilon}}
=15U_{z_{\varepsilon},\lambda_{\varepsilon}}^{4}
\partial_{\lambda}U_{z_{\varepsilon},\lambda_{\varepsilon}}$ and  the orthogonality
condition, we have
\begin{equation*}
	\int_{\Omega}q_{\varepsilon}
	U_{z_{\varepsilon},\lambda_{\varepsilon}}^{4}
	\partial_{\lambda}U_{z_{\varepsilon},\lambda_{\varepsilon}}
	=\lambda_{\varepsilon}^{-1/2}
	\int_{\Omega}
	\left(H_Q(z_{\varepsilon},\cdot)-H_0(z_{\varepsilon},\cdot)\right)
	U_{z_{\varepsilon},\lambda_{\varepsilon}}^{4}
	\partial_{\lambda}U_{z_{\varepsilon},\lambda_{\varepsilon}}.
\end{equation*}
By Lemma \ref{B3} and Proposition \ref{prop36o}, we obtain
\begin{equation*}
	15\left(1-\alpha_{\varepsilon}^{10-2\mu}\right)\int_{\Omega}q_{\varepsilon}U_{z_{\varepsilon},\lambda_{\varepsilon}}^{4}\partial_{\lambda}U_{z_{\varepsilon},\lambda_{\varepsilon}}
	=-2\pi\left(1-\alpha_{\varepsilon}^{10-2\mu}\right)\left(\phi_Q(z_{\varepsilon})-\phi_0(z_{\varepsilon})\right)\lambda_{\varepsilon}^{-2}
	+o\left(\lambda_{\varepsilon}^{-3}\right).
\end{equation*}
We next consider the second term  in \eqref{o337}. We write
\begin{equation*}
	\psi_{z_{\varepsilon},\lambda_{\varepsilon}}^{5-\mu}
	=U_{z_{\varepsilon},\lambda_{\varepsilon}}^{5-\mu}
	+\varrho,
\end{equation*}
where
\begin{equation*}
	|\varrho|
	\lesssim U_{z_{\varepsilon},\lambda_{\varepsilon}}^{4-\mu}
	\left(\lambda_{\varepsilon}^{-1/2}|H_Q(z_{\varepsilon},\cdot)|
	+|f_{z_{\varepsilon},\lambda_{\varepsilon}}|\right)
	+\lambda_{\varepsilon}^{-(5-\mu)/2}|H_Q(z_{\varepsilon},\cdot)|^{5-\mu}
	+|f_{z_{\varepsilon},\lambda_{\varepsilon}}|^{5-\mu}.
\end{equation*}
Since
\begin{equation*}
	\partial_{\lambda}\psi_{z_{\varepsilon},\lambda_{\varepsilon}}
	=\partial_{\lambda}U_{z_{\varepsilon},\lambda_{\varepsilon}}
	+\frac{1}{2}\lambda_{\varepsilon}^{-3/2}H_Q(z_{\varepsilon},\cdot)
	-\partial_{\lambda}f_{z_{\varepsilon},\lambda_{\varepsilon}},
\end{equation*}
we have
\begin{equation*}
\begin{split}
	&\int_{\Omega}\int_{\mathbb{R}^{3}}
	\frac{U_{z_{\varepsilon},\lambda_{\varepsilon}}^{5-\mu}(y)
	q_{\varepsilon}(y)
	U_{z_{\varepsilon},\lambda_{\varepsilon}}^{5-\mu}(x)
	\partial_{\lambda}U_{z_{\varepsilon},\lambda_{\varepsilon}}(x)}
	{|x-y|^{\mu}}dxdy-\int_{\Omega}\int_{\Omega}
	\frac{\psi_{z_{\varepsilon},\lambda_{\varepsilon}}^{5-\mu}(y)
	q_{\varepsilon}(y)
	\psi_{z_{\varepsilon},\lambda_{\varepsilon}}^{5-\mu}(x)
	\partial_{\lambda}\psi_{z_{\varepsilon},\lambda_{\varepsilon}}(x)}
	{|x-y|^{\mu}}dxdy
	= I_4+I_5,
\end{split}
\end{equation*}
where
\begin{equation*}
	I_4
	=\int_{\Omega}\int_{\mathbb{R}^{3}\setminus\Omega}
	\frac{U_{z_{\varepsilon},\lambda_{\varepsilon}}^{5-\mu}(y)
	q_{\varepsilon}(y)
	U_{z_{\varepsilon},\lambda_{\varepsilon}}^{5-\mu}(x)
	\partial_{\lambda}U_{z_{\varepsilon},\lambda_{\varepsilon}}(x)}
	{|x-y|^{\mu}}dxdy
\end{equation*}
and
\begin{equation*}
\begin{split}
	I_5=O\Bigg(
	&\int_{\Omega}\int_{\Omega}
	\frac{|q_{\varepsilon}(y)||\varrho(y)|
	\psi_{z_{\varepsilon},\lambda_{\varepsilon}}^{5-\mu}(x)
	|\partial_{\lambda}U_{z_{\varepsilon},\lambda_{\varepsilon}}(x)|}
	{|x-y|^{\mu}}dxdy\\
	&+\int_{\Omega}\int_{\Omega}
	\frac{|q_{\varepsilon}(y)|
	\psi_{z_{\varepsilon},\lambda_{\varepsilon}}^{5-\mu}(y)
	|\varrho(x)|
	|\partial_{\lambda}U_{z_{\varepsilon},\lambda_{\varepsilon}}(x)|}
	{|x-y|^{\mu}}dxdy\\
	&+\lambda_{\varepsilon}^{-3/2}\int_{\Omega}\int_{\Omega}
	\frac{|q_{\varepsilon}(y)|
	\psi_{z_{\varepsilon},\lambda_{\varepsilon}}^{5-\mu}(y)
	\psi_{z_{\varepsilon},\lambda_{\varepsilon}}^{5-\mu}(x)
	H_Q(z_{\varepsilon},x)}
	{|x-y|^{\mu}}dxdy\\
	&+\int_{\Omega}\int_{\Omega}
	\frac{|q_{\varepsilon}(y)|
	\psi_{z_{\varepsilon},\lambda_{\varepsilon}}^{5-\mu}(y)
	\psi_{z_{\varepsilon},\lambda_{\varepsilon}}^{5-\mu}(x)
	|\partial_{\lambda}f_{z_{\varepsilon},\lambda_{\varepsilon}}(x)|}
	{|x-y|^{\mu}}dxdy
	\Bigg).
\end{split}
\end{equation*}
Combining the HLS and H\"older inequalities with Lemmas \ref{A1},
\ref{A1RN}, and \ref{A2}, we have
\begin{equation*}
\begin{split}
	|I_4|
	&\lesssim
	\|q_{\varepsilon}\|_{6}
	\left\|U_{z_{\varepsilon},\lambda_{\varepsilon}}^{5-\mu}\right\|_{\frac{6}{5-\mu}}
	\left\|
	U_{z_{\varepsilon},\lambda_{\varepsilon}}^{5-\mu}
	\partial_{\lambda}U_{z_{\varepsilon},\lambda_{\varepsilon}}
	\right\|_{L^{\frac{6}{6-\mu}}(\mathbb{R}^{3}\setminus\Omega)}\\
	&\lesssim
	\lambda_{\varepsilon}^{-1}
	\|q_{\varepsilon}\|_{6}
	\left\|U_{z_{\varepsilon},\lambda_{\varepsilon}}
	\right\|_{L^{6}(\mathbb{R}^{3}\setminus\Omega)}^{6-\mu}\\
	&\lesssim
	\lambda_{\varepsilon}^{-4+\frac{\mu}{2}}
	\|q_{\varepsilon}\|_{6}
\end{split}
\end{equation*}
and
\begin{equation*}
\begin{split}
	|I_5|
	\lesssim{}&\|q_{\varepsilon}\|_{6}
	\left\|\psi_{z_{\varepsilon},\lambda_{\varepsilon}}^{5-\mu}\right\|_{\frac{6}{5-\mu}}
	\|\varrho\|_{\frac{6}{5-\mu}}
	\left\|
	\partial_{\lambda}U_{z_{\varepsilon},\lambda_{\varepsilon}}\right\|_{6}
	     \\
	&+\lambda_{\varepsilon}^{-3/2}
	\|q_{\varepsilon}\|_{6}
	\left\|\psi_{z_{\varepsilon},\lambda_{\varepsilon}}^{5-\mu}\right\|_{\frac{6}{5-\mu}}
	\left\|
	\psi_{z_{\varepsilon},\lambda_{\varepsilon}}^{5-\mu}
	H_Q(z_{\varepsilon},\cdot)
	\right\|_{\frac{6}{6-\mu}}\\
	&+\|q_{\varepsilon}\|_{6}
	\left\|\psi_{z_{\varepsilon},\lambda_{\varepsilon}}^{5-\mu}\right\|_{\frac{6}{5-\mu}}
	\left\|
	\psi_{z_{\varepsilon},\lambda_{\varepsilon}}^{5-\mu}
	\right\|_{\frac{6}{6-\mu}}
	\left\| \partial_{\lambda}f_{z_{\varepsilon},\lambda_{\varepsilon}}  \right\|_{\infty}
	\\
	\lesssim{}&
	\lambda_{\varepsilon}^{-1}
	\|q_{\varepsilon}\|_{6}
	\|\varrho\|_{\frac{6}{5-\mu}}
	+ \lambda_{\varepsilon}^{-3/2}
	\|q_{\varepsilon}\|_{6}
	\left\|
	\psi_{z_{\varepsilon},\lambda_{\varepsilon}}^{5-\mu}
	H_Q(z_{\varepsilon},\cdot)
	\right\|_{\frac{6}{6-\mu}}+\lambda_{\varepsilon}^{-4}
	\|q_{\varepsilon}\|_{6}.
\end{split}
\end{equation*}
By \eqref{o327}, we have
$
	\|q_\varepsilon\|_{6}
	\lesssim \lambda_{\varepsilon}^{-1}
	+\varepsilon\lambda_{\varepsilon}^{-\frac12}.
$
Thus it remains to estimate $\|\varrho\|_{\frac{6}{5-\mu}}$ and $\left\|
	\psi_{z_{\varepsilon},\lambda_{\varepsilon}}^{5-\mu}
	H_Q(z_{\varepsilon},\cdot)
	\right\|_{\frac{6}{6-\mu}}$.
	Now we estimate $\|\varrho\|_{\frac{6}{5-\mu}}$, 
by the definition of $\varrho$, we have
\begin{equation*}
\begin{split}
	\|\varrho\|_{\frac{6}{5-\mu}}
	\lesssim{}&
	\lambda_{\varepsilon}^{-\frac12}
	\left\|
	U_{z_{\varepsilon},\lambda_{\varepsilon}}^{4-\mu}
	H_Q(z_{\varepsilon},\cdot)
	\right\|_{\frac{6}{5-\mu}}
	+\left\|
	U_{z_{\varepsilon},\lambda_{\varepsilon}}^{4-\mu}
	f_{z_{\varepsilon},\lambda_{\varepsilon}}
	\right\|_{\frac{6}{5-\mu}}\\
	&+\lambda_{\varepsilon}^{-\frac{5-\mu}{2}}
	\left\||H_Q(z_{\varepsilon},\cdot)|^{5-\mu}
	\right\|_{\frac{6}{5-\mu}}
	+\left\|
		|f_{z_{\varepsilon},\lambda_{\varepsilon}}|^{5-\mu}
	\right\|_{\frac{6}{5-\mu}}.
\end{split}
\end{equation*}
For the second term, by Lemmas \ref{A1} and \ref{A2}, we obtain
\begin{equation*}
	\left\|
	U_{z_{\varepsilon},\lambda_{\varepsilon}}^{4-\mu}
	f_{z_{\varepsilon},\lambda_{\varepsilon}}
	\right\|_{\frac{6}{5-\mu}}
	\lesssim
	\|f_{z_{\varepsilon},\lambda_{\varepsilon}}\|_{\infty}
	\left\|
	U_{z_{\varepsilon},\lambda_{\varepsilon}}^{4-\mu}
	\right\|_{\frac{6}{5-\mu}}
	=O\left(\lambda_{\varepsilon}^{-3}\right).
\end{equation*}
Similarly, the third and fourth terms are bounded by
$O\left(\lambda_{\varepsilon}^{-\frac{5-\mu}{2}}\right)$.
From \eqref{UUUU}, we have 
\begin{equation*}
	\begin{split}
		\lambda_{\varepsilon}^{-\frac12}
	\left\|
	U_{z_{\varepsilon},\lambda_{\varepsilon}}^{4-\mu}
	H_Q(z_{\varepsilon},\cdot)
	\right\|_{\frac{6}{5-\mu}}\lesssim
	\phi_Q(z_{\varepsilon})\lambda_{\varepsilon}^{-1}
	+
	\begin{cases}
		O(\lambda_{\varepsilon}^{-2}), & \text{if} \quad0<\mu<1,\\
		O(\lambda_{\varepsilon}^{-2}(\log\lambda_{\varepsilon})^{\frac23}), &\text{if} \quad \mu=1,\\
		O\left(\lambda_{\varepsilon}^{-\frac{5-\mu}{2}}\right), &\text{if} \quad 1<\mu<2.\end{cases}
	\end{split}
\end{equation*}
Consequently,
\begin{equation*}
	\|\varrho\|_{\frac{6}{5-\mu}}
	\lesssim
	\phi_Q(z_{\varepsilon})\lambda_{\varepsilon}^{-1}
	+
	\begin{cases}
		O(\lambda_{\varepsilon}^{-2}), &\text{if} \quad 0<\mu<1,\\
		O(\lambda_{\varepsilon}^{-2}(\log\lambda_{\varepsilon})^{\frac23}), & \text{if} \quad\mu=1,\\
		O\left(\lambda_{\varepsilon}^{-\frac{5-\mu}{2}}\right), & \text{if} \quad1<\mu<2.
	\end{cases}
\end{equation*}
Now we estimate
$\left\|\psi_{z_{\varepsilon},\lambda_{\varepsilon}}^{5-\mu}
H_Q(z_{\varepsilon},\cdot)\right\|_{\frac{6}{6-\mu}}$.
Since
$\psi_{z_{\varepsilon},\lambda_{\varepsilon}}^{5-\mu}
=U_{z_{\varepsilon},\lambda_{\varepsilon}}^{5-\mu}+\varrho$, by H\"older's
inequality, Lemma \ref{B1} and the estimate above, we have
\begin{equation*}
\begin{split}
	&\left\|\psi_{z_{\varepsilon},\lambda_{\varepsilon}}^{5-\mu}
	H_Q(z_{\varepsilon},\cdot)\right\|_{\frac{6}{6-\mu}}\\
	&\lesssim
	\left\|U_{z_{\varepsilon},\lambda_{\varepsilon}}\right\|_{6}
	\left\|
	U_{z_{\varepsilon},\lambda_{\varepsilon}}^{4-\mu}
	H_Q(z_{\varepsilon},\cdot)
	\right\|_{\frac{6}{5-\mu}}
	+\|\varrho\|_{\frac{6}{5-\mu}}
	\left\|H_Q(z_{\varepsilon},\cdot)\right\|_{6}\\
	&\lesssim
	\phi_Q(z_{\varepsilon})\lambda_{\varepsilon}^{-\frac12}
	+
	\begin{cases}
		O(\lambda_{\varepsilon}^{-\frac32}), & \text{if} \quad 0<\mu<1,\\
		O(\lambda_{\varepsilon}^{-\frac32}(\log\lambda_{\varepsilon})^{\frac23}), & \text{if} \quad \mu=1,\\
		O\left(\lambda_{\varepsilon}^{-\frac{4-\mu}{2}}\right), & \text{if} \quad 1<\mu<2.
	\end{cases}
\end{split}
\end{equation*}
Using the estimates above, we obtain
\begin{equation*}
	\begin{split}
	&	\left| \int_{\Omega}\int_{\mathbb{R}^{3}}
	\frac{U_{z_{\varepsilon},\lambda_{\varepsilon}}^{5-\mu}(y)
	q_{\varepsilon}(y)
	U_{z_{\varepsilon},\lambda_{\varepsilon}}^{5-\mu}(x)
	\partial_{\lambda}U_{z_{\varepsilon},\lambda_{\varepsilon}}(x)}
	{|x-y|^{\mu}}dxdy-\int_{\Omega}\int_{\Omega}
	\frac{\psi_{z_{\varepsilon},\lambda_{\varepsilon}}^{5-\mu}(y)
	q_{\varepsilon}(y)
	\psi_{z_{\varepsilon},\lambda_{\varepsilon}}^{5-\mu}(x)
	\partial_{\lambda}\psi_{z_{\varepsilon},\lambda_{\varepsilon}}(x)}
	{|x-y|^{\mu}}dxdy   \right|\\
		\lesssim & 
	|I_4|+|I_5|
	=
	O\left(\phi_Q(z_{\varepsilon})\lambda_{\varepsilon}^{-3}\right)
	+o\left(\varepsilon\lambda_{\varepsilon}^{-2}\right)
	+o\left(\lambda_{\varepsilon}^{-3}\right).
	\end{split}
\end{equation*}
Similarly, we can deduce
\begin{equation*}
	\begin{split}
		&\left|\int_{\mathbb{R}^{3}}\int_{\Omega}
	\frac{U_{z_{\varepsilon},\lambda_{\varepsilon}}^{6-\mu}(y)
	U_{z_{\varepsilon},\lambda_{\varepsilon}}^{4-\mu}(x)
	q_{\varepsilon}(x)
	\partial_{\lambda}U_{z_{\varepsilon},\lambda_{\varepsilon}}(x)}
	{|x-y|^{\mu}}dxdy-\int_{\Omega}\int_{\Omega}
	\frac{\psi_{z_{\varepsilon},\lambda_{\varepsilon}}^{6-\mu}(y)
	\psi_{z_{\varepsilon},\lambda_{\varepsilon}}^{4-\mu}(x)
	q_{\varepsilon}(x)
	\partial_{\lambda}\psi_{z_{\varepsilon},\lambda_{\varepsilon}}(x)}
	{|x-y|^{\mu}}dxdy \right|
	\\ =&
	O\left(\phi_Q(z_{\varepsilon})\lambda_{\varepsilon}^{-3}\right)
	+o\left(\varepsilon\lambda_{\varepsilon}^{-2}\right)
	+o\left(\lambda_{\varepsilon}^{-3}\right).
	\end{split}
\end{equation*}
It remains to estimate the last term in \eqref{o337}. By H\"older's inequality,
Lemmas \ref{A2} and \ref{A4}, we have
\begin{equation*}
\begin{split}
	&\left|
	\int_{\Omega}Qq_{\varepsilon}
	\left(\partial_{\lambda}f_{z_{\varepsilon},\lambda_{\varepsilon}}
	+\partial_{\lambda}g_{z_{\varepsilon},\lambda_{\varepsilon}}\right)
	\right|\\
	&\lesssim
	\|q_{\varepsilon}\|_{6}
	\left(
	\left\|\partial_{\lambda}f_{z_{\varepsilon},\lambda_{\varepsilon}}
	\right\|_{\frac65}
	+
	\left\|\partial_{\lambda}g_{z_{\varepsilon},\lambda_{\varepsilon}}
	\right\|_{\frac65}
	\right)\\
	&\lesssim
	\left(\lambda_{\varepsilon}^{-1}
	+\varepsilon\lambda_{\varepsilon}^{-\frac12}\right)
	\left(\lambda_{\varepsilon}^{-\frac72}
	+\lambda_{\varepsilon}^{-3}\right)
	=o\left(\varepsilon\lambda_{\varepsilon}^{-2}\right)
	+o\left(\lambda_{\varepsilon}^{-3}\right).
\end{split}
\end{equation*}
Combining the estimates above with \eqref{o337}, we prove (b).

We next prove (c). By \eqref{o33}, \eqref{o310} and \eqref{o318}, the terms containing $r_{\varepsilon}$ are absorbed into
the remainder. Hence the left-hand side in part (c) is equal to
\begin{equation*}
\begin{split}
	&A_{3,\mu}\alpha_{\varepsilon}^{10-2\mu}\Bigg[
	\frac{(6-\mu)(5-\mu)}{2}
	\int_{\Omega}\int_{\Omega}
	\frac{\psi_{z_{\varepsilon},\lambda_{\varepsilon}}^{4-\mu}(y)
	s_{\varepsilon}^{2}(y)
	\psi_{z_{\varepsilon},\lambda_{\varepsilon}}^{5-\mu}(x)
	\partial_{\lambda}\psi_{z_{\varepsilon},\lambda_{\varepsilon}}(x)}
	{|x-y|^{\mu}}dxdy\\
	&\quad +(6-\mu)(5-\mu)
	\int_{\Omega}\int_{\Omega}
	\frac{\psi_{z_{\varepsilon},\lambda_{\varepsilon}}^{5-\mu}(y)
	s_{\varepsilon}(y)
	\psi_{z_{\varepsilon},\lambda_{\varepsilon}}^{4-\mu}(x)
	s_{\varepsilon}(x)
	\partial_{\lambda}\psi_{z_{\varepsilon},\lambda_{\varepsilon}}(x)}
	{|x-y|^{\mu}}dxdy\\
	&\quad +\frac{(5-\mu)(4-\mu)}{2}
	\int_{\Omega}\int_{\Omega}
	\frac{\psi_{z_{\varepsilon},\lambda_{\varepsilon}}^{6-\mu}(y)
	\psi_{z_{\varepsilon},\lambda_{\varepsilon}}^{3-\mu}(x)
	s_{\varepsilon}^{2}(x)
	\partial_{\lambda}\psi_{z_{\varepsilon},\lambda_{\varepsilon}}(x)}
	{|x-y|^{\mu}}dxdy\Bigg]\\
	&\quad
	+O\left(\phi_Q(z_{\varepsilon})\lambda_{\varepsilon}^{-3}\right)
	+o\left(\lambda_{\varepsilon}^{-3}\right)
	+o\left(\varepsilon\lambda_{\varepsilon}^{-2}\right).
\end{split}
\end{equation*}
Using \eqref{sss1}, \eqref{sss2}, and the same computation as in the proof of
(b), we obtain that the left-hand side in part (c) is equal to
\begin{equation*}
\begin{split}
	&A_{3,\mu}\alpha_{\varepsilon}^{10-2\mu}
	\left[
	\beta^2\lambda_{\varepsilon}^{-2}I_6
	+2\beta\gamma\lambda_{\varepsilon}^{-1}I_7
	+\gamma^2I_8
	\right]
	+O\left(\phi_Q(z_{\varepsilon})\lambda_{\varepsilon}^{-3}\right)
	+o\left(\lambda_{\varepsilon}^{-3}\right)
	+o\left(\varepsilon\lambda_{\varepsilon}^{-2}\right),
\end{split}
\end{equation*}
where
\begin{equation*}
	I_6
	=(5-\mu)(11-2\mu)
	\int_{\Omega}\int_{\Omega}
	\frac{U_{z_{\varepsilon},\lambda_{\varepsilon}}^{6-\mu}(y)
	U_{z_{\varepsilon},\lambda_{\varepsilon}}^{5-\mu}(x)
	\partial_{\lambda}U_{z_{\varepsilon},\lambda_{\varepsilon}}(x)}
	{|x-y|^{\mu}}dxdy,
\end{equation*}
\begin{equation*}
\begin{split}
	I_7={}&
	(6-\mu)(5-\mu)
	\int_{\Omega}\int_{\Omega}
	\frac{U_{z_{\varepsilon},\lambda_{\varepsilon}}^{5-\mu}(y)
	\partial_{\lambda}U_{z_{\varepsilon},\lambda_{\varepsilon}}(y)
	U_{z_{\varepsilon},\lambda_{\varepsilon}}^{5-\mu}(x)
	\partial_{\lambda}U_{z_{\varepsilon},\lambda_{\varepsilon}}(x)}
	{|x-y|^{\mu}}dxdy\\
	&+(5-\mu)^2
	\int_{\Omega}\int_{\Omega}
	\frac{U_{z_{\varepsilon},\lambda_{\varepsilon}}^{6-\mu}(y)
	U_{z_{\varepsilon},\lambda_{\varepsilon}}^{4-\mu}(x)
	(\partial_{\lambda}U_{z_{\varepsilon},\lambda_{\varepsilon}}(x))^2}
	{|x-y|^{\mu}}dxdy,
\end{split}
\end{equation*}
and
\begin{equation*}
\begin{split}
	I_8={}&
	\frac{(6-\mu)(5-\mu)}{2}
	\int_{\Omega}\int_{\Omega}
	\frac{U_{z_{\varepsilon},\lambda_{\varepsilon}}^{4-\mu}(y)
	(\partial_{\lambda}U_{z_{\varepsilon},\lambda_{\varepsilon}}(y))^2
	U_{z_{\varepsilon},\lambda_{\varepsilon}}^{5-\mu}(x)
	\partial_{\lambda}U_{z_{\varepsilon},\lambda_{\varepsilon}}(x)}
	{|x-y|^{\mu}}dxdy\\
	&+(6-\mu)(5-\mu)
	\int_{\Omega}\int_{\Omega}
	\frac{U_{z_{\varepsilon},\lambda_{\varepsilon}}^{5-\mu}(y)
	\partial_{\lambda}U_{z_{\varepsilon},\lambda_{\varepsilon}}(y)
	U_{z_{\varepsilon},\lambda_{\varepsilon}}^{4-\mu}(x)
	(\partial_{\lambda}U_{z_{\varepsilon},\lambda_{\varepsilon}}(x))^2}
	{|x-y|^{\mu}}dxdy\\
	&+\frac{(5-\mu)(4-\mu)}{2}
	\int_{\Omega}\int_{\Omega}
	\frac{U_{z_{\varepsilon},\lambda_{\varepsilon}}^{6-\mu}(y)
	U_{z_{\varepsilon},\lambda_{\varepsilon}}^{3-\mu}(x)
	(\partial_{\lambda}U_{z_{\varepsilon},\lambda_{\varepsilon}}(x))^3}
		{|x-y|^{\mu}}dxdy.
	\end{split}
\end{equation*}
We first compute $I_6$. By its definition,
\begin{equation*}
\begin{split}
	I_6={}&(5-\mu)(11-2\mu)
	\int_{\Omega}\int_{\Omega}
	\frac{U_{z_{\varepsilon},\lambda_{\varepsilon}}^{6-\mu}(y)
	U_{z_{\varepsilon},\lambda_{\varepsilon}}^{5-\mu}(x)
	\partial_{\lambda}U_{z_{\varepsilon},\lambda_{\varepsilon}}(x)}
	{|x-y|^{\mu}}dxdy\\
	={}&(5-\mu)(11-2\mu)\Bigg[
	\int_{\mathbb R^3}\int_{\Omega}
	\frac{U_{z_{\varepsilon},\lambda_{\varepsilon}}^{6-\mu}(y)
	U_{z_{\varepsilon},\lambda_{\varepsilon}}^{5-\mu}(x)
	\partial_{\lambda}U_{z_{\varepsilon},\lambda_{\varepsilon}}(x)}
	{|x-y|^{\mu}}dxdy-\int_{\mathbb R^3\setminus\Omega}\int_{\Omega}
	\frac{U_{z_{\varepsilon},\lambda_{\varepsilon}}^{6-\mu}(y)
	U_{z_{\varepsilon},\lambda_{\varepsilon}}^{5-\mu}(x)
	\partial_{\lambda}U_{z_{\varepsilon},\lambda_{\varepsilon}}(x)}
	{|x-y|^{\mu}}dxdy\Bigg]\\
	={}&\frac{3(5-\mu)(11-2\mu)}{A_{3,\mu}}\int_{\Omega}
	U_{z_{\varepsilon},\lambda_{\varepsilon}}^5
	\partial_{\lambda}U_{z_{\varepsilon},\lambda_{\varepsilon}}
	+o\left(\lambda_{\varepsilon}^{-3}\right)\\
	={}&\frac{3(5-\mu)(11-2\mu)}{A_{3,\mu}} \left(\int_{\mathbb R^3}
	U_{z_{\varepsilon},\lambda_{\varepsilon}}^5
	\partial_{\lambda}U_{z_{\varepsilon},\lambda_{\varepsilon}}
	-\int_{\mathbb R^3\setminus\Omega}
	U_{z_{\varepsilon},\lambda_{\varepsilon}}^5
	\partial_{\lambda}U_{z_{\varepsilon},\lambda_{\varepsilon}}      \right)+o\left(\lambda_{\varepsilon}^{-3}\right)
	\\
	={}&\frac{(5-\mu)(11-2\mu)}{2 A_{3,\mu}}\partial_{\lambda}
	\int_{\mathbb R^3}U_{z_{\varepsilon},\lambda_{\varepsilon}}^6+o\left(\lambda_{\varepsilon}^{-3}\right)\\
	={}&o\left(\lambda_{\varepsilon}^{-3}\right). 
\end{split}
\end{equation*}

Next, differentiating the equation of
$U_{z_{\varepsilon},\lambda_{\varepsilon}}$ with respect to $\lambda$, we have
\begin{equation*}
\begin{split}
	-\Delta\partial_{\lambda}U_{z_{\varepsilon},\lambda_{\varepsilon}}(x)
	={}&A_{3,\mu}(6-\mu)
	\left(\int_{\mathbb R^3}
	\frac{U_{z_{\varepsilon},\lambda_{\varepsilon}}^{5-\mu}(y)
	\partial_{\lambda}U_{z_{\varepsilon},\lambda_{\varepsilon}}(y)}
	{|x-y|^\mu}dy\right)
	U_{z_{\varepsilon},\lambda_{\varepsilon}}^{5-\mu}(x)\\
	&\quad+A_{3,\mu}(5-\mu)
	\left(\int_{\mathbb R^3}
	\frac{U_{z_{\varepsilon},\lambda_{\varepsilon}}^{6-\mu}(y)}
	{|x-y|^\mu}dy\right)
	U_{z_{\varepsilon},\lambda_{\varepsilon}}^{4-\mu}(x)
	\partial_{\lambda}U_{z_{\varepsilon},\lambda_{\varepsilon}}(x).
\end{split}
\end{equation*}
Therefore,
\begin{equation*}
\begin{split}
	I_7
	={}&(6-\mu)(5-\mu)
	\int_{\mathbb R^3}\int_{\Omega}
	\frac{U_{z_{\varepsilon},\lambda_{\varepsilon}}^{5-\mu}(y)
	\partial_{\lambda}U_{z_{\varepsilon},\lambda_{\varepsilon}}(y)
	U_{z_{\varepsilon},\lambda_{\varepsilon}}^{5-\mu}(x)
	\partial_{\lambda}U_{z_{\varepsilon},\lambda_{\varepsilon}}(x)}
	{|x-y|^\mu}dxdy\\
	&+(5-\mu)^2
	\int_{\mathbb R^3}\int_{\Omega}
	\frac{U_{z_{\varepsilon},\lambda_{\varepsilon}}^{6-\mu}(y)
	U_{z_{\varepsilon},\lambda_{\varepsilon}}^{4-\mu}(x)
	(\partial_{\lambda}U_{z_{\varepsilon},\lambda_{\varepsilon}}(x))^2}
	{|x-y|^\mu}dxdy+o\left(\lambda_{\varepsilon}^{-2}\right)\\
	={}&\frac{5-\mu}{A_{3,\mu}}\int_{\Omega}
	\left(-\Delta\partial_{\lambda}
	U_{z_{\varepsilon},\lambda_{\varepsilon}}\right)
	\partial_{\lambda}U_{z_{\varepsilon},\lambda_{\varepsilon}}
	+o\left(\lambda_{\varepsilon}^{-2}\right)\\
	={}&\frac{15(5-\mu)}{A_{3,\mu}}\int_{\Omega}
	U_{z_{\varepsilon},\lambda_{\varepsilon}}^4
	(\partial_{\lambda}U_{z_{\varepsilon},\lambda_{\varepsilon}})^2
	+o\left(\lambda_{\varepsilon}^{-2}\right),  
\end{split}
\end{equation*}
since $-\Delta\partial_\lambda U_{z_\varepsilon,\lambda_\varepsilon}
=
15U_{z_\varepsilon,\lambda_\varepsilon}^{4}
\partial_\lambda U_{z_\varepsilon,\lambda_\varepsilon}$. 
A direct computation gives
\begin{equation*}
	\int_{\Omega}
	U_{z_{\varepsilon},\lambda_{\varepsilon}}^4
	(\partial_{\lambda}U_{z_{\varepsilon},\lambda_{\varepsilon}})^2
	=\frac{\pi^2}{64}\lambda_{\varepsilon}^{-2}
	+o\left(\lambda_{\varepsilon}^{-2}\right).
\end{equation*}
It follows that
\begin{equation*}
	I_7=\frac{15\pi^2(5-\mu)}{64A_{3,\mu}}
	\lambda_{\varepsilon}^{-2}
	+o\left(\lambda_{\varepsilon}^{-2}\right).
\end{equation*}
It remains to estimate $I_8$. Let $I_8^{\mathbb R^3}$ be the corresponding
integral with $\Omega\times\Omega$ replaced by
$\mathbb R^3\times\mathbb R^3$, that is 
\begin{equation*}
\begin{split}
	I_8^{\mathbb R^3}:={}&
	\frac{(6-\mu)(5-\mu)}{2}
	\int_{\mathbb R^3}\int_{\mathbb R^3}
	\frac{U_{z_{\varepsilon},\lambda_{\varepsilon}}^{4-\mu}(y)
	(\partial_{\lambda}U_{z_{\varepsilon},\lambda_{\varepsilon}}(y))^2
	U_{z_{\varepsilon},\lambda_{\varepsilon}}^{5-\mu}(x)
	\partial_{\lambda}U_{z_{\varepsilon},\lambda_{\varepsilon}}(x)}
	{|x-y|^{\mu}}dxdy\\
	&+(6-\mu)(5-\mu)
	\int_{\mathbb R^3}\int_{\mathbb R^3}
	\frac{U_{z_{\varepsilon},\lambda_{\varepsilon}}^{5-\mu}(y)
	\partial_{\lambda}U_{z_{\varepsilon},\lambda_{\varepsilon}}(y)
	U_{z_{\varepsilon},\lambda_{\varepsilon}}^{4-\mu}(x)
	(\partial_{\lambda}U_{z_{\varepsilon},\lambda_{\varepsilon}}(x))^2}
	{|x-y|^{\mu}}dxdy\\
	&+\frac{(5-\mu)(4-\mu)}{2}
	\int_{\mathbb R^3}\int_{\mathbb R^3}
	\frac{U_{z_{\varepsilon},\lambda_{\varepsilon}}^{6-\mu}(y)
	U_{z_{\varepsilon},\lambda_{\varepsilon}}^{3-\mu}(x)
	(\partial_{\lambda}U_{z_{\varepsilon},\lambda_{\varepsilon}}(x))^3}
	{|x-y|^{\mu}}dxdy.
\end{split}
\end{equation*}
  Differentiating the equation of
$U_{z_{\varepsilon},\lambda_{\varepsilon}}$ twice with respect to $\lambda$,
multiplying the resulting identity by
$\partial_{\lambda}U_{z_{\varepsilon},\lambda_{\varepsilon}}$, and integrating
over $\mathbb R^3$, we get
\begin{equation*}
\begin{split}
	&\int_{\mathbb R^3}
	\left(-\Delta\partial_{\lambda\lambda}
	U_{z_{\varepsilon},\lambda_{\varepsilon}}\right)
	\partial_{\lambda}U_{z_{\varepsilon},\lambda_{\varepsilon}}
	=2A_{3,\mu}I_8^{\mathbb R^3}\\
	&\qquad+A_{3,\mu}(6-\mu)
	\int_{\mathbb R^3}\int_{\mathbb R^3}
	\frac{U_{z_{\varepsilon},\lambda_{\varepsilon}}^{5-\mu}(y)
	\partial_{\lambda\lambda}U_{z_{\varepsilon},\lambda_{\varepsilon}}(y)
	U_{z_{\varepsilon},\lambda_{\varepsilon}}^{5-\mu}(x)
	\partial_{\lambda}U_{z_{\varepsilon},\lambda_{\varepsilon}}(x)}
	{|x-y|^\mu}dxdy\\
	&\qquad+A_{3,\mu}(5-\mu)
	\int_{\mathbb R^3}\int_{\mathbb R^3}
	\frac{U_{z_{\varepsilon},\lambda_{\varepsilon}}^{6-\mu}(y)
	U_{z_{\varepsilon},\lambda_{\varepsilon}}^{4-\mu}(x)
	\partial_{\lambda\lambda}U_{z_{\varepsilon},\lambda_{\varepsilon}}(x)
	\partial_{\lambda}U_{z_{\varepsilon},\lambda_{\varepsilon}}(x)}
	{|x-y|^\mu}dxdy.
\end{split}
\end{equation*}
On the other hand, multiplying the equation of
$\partial_{\lambda}U_{z_{\varepsilon},\lambda_{\varepsilon}}$ by
$\partial_{\lambda\lambda}U_{z_{\varepsilon},\lambda_{\varepsilon}}$ and
integrating over $\mathbb R^3$, we obtain
\begin{equation*}
\begin{split}
	&\int_{\mathbb R^3}
	\left(-\Delta\partial_{\lambda}
	U_{z_{\varepsilon},\lambda_{\varepsilon}}\right)
	\partial_{\lambda\lambda}U_{z_{\varepsilon},\lambda_{\varepsilon}}\\
	={}&A_{3,\mu}(6-\mu)
	\int_{\mathbb R^3}\int_{\mathbb R^3}
	\frac{U_{z_{\varepsilon},\lambda_{\varepsilon}}^{5-\mu}(y)
	\partial_{\lambda}U_{z_{\varepsilon},\lambda_{\varepsilon}}(y)
	U_{z_{\varepsilon},\lambda_{\varepsilon}}^{5-\mu}(x)
	\partial_{\lambda\lambda}U_{z_{\varepsilon},\lambda_{\varepsilon}}(x)}
	{|x-y|^\mu}dxdy\\
	&+A_{3,\mu}(5-\mu)
	\int_{\mathbb R^3}\int_{\mathbb R^3}
	\frac{U_{z_{\varepsilon},\lambda_{\varepsilon}}^{6-\mu}(y)
	U_{z_{\varepsilon},\lambda_{\varepsilon}}^{4-\mu}(x)
	\partial_{\lambda}U_{z_{\varepsilon},\lambda_{\varepsilon}}(x)
	\partial_{\lambda\lambda}U_{z_{\varepsilon},\lambda_{\varepsilon}}(x)}
	{|x-y|^\mu}dxdy.
\end{split}
\end{equation*}
Since  $\int_{\mathbb R^3}
	\left(-\Delta\partial_{\lambda\lambda}
	U_{z_{\varepsilon},\lambda_{\varepsilon}}\right)
	\partial_{\lambda}U_{z_{\varepsilon},\lambda_{\varepsilon}}=\int_{\mathbb R^3}
	\left(-\Delta\partial_{\lambda}
	U_{z_{\varepsilon},\lambda_{\varepsilon}}\right)
	\partial_{\lambda\lambda}U_{z_{\varepsilon},\lambda_{\varepsilon}}$, we have 
	\begin{equation*}
	I_8^{\mathbb R^3}=0.
\end{equation*}
	It is easy to see $\left|I_8^{\mathbb R^3}-I_8\right| =  o\left(\lambda_{\varepsilon}^{-3}\right) $, thus 
\begin{equation*}
	I_8=o\left(\lambda_{\varepsilon}^{-3}\right).
\end{equation*}
	Combining these estimates and the fact $1-\alpha_{\varepsilon}^{10-2\mu}=o(1)$, we obtain
	\begin{equation*}
	\begin{split}
		 A_{3,\mu}\alpha_{\varepsilon}^{10-2\mu}
		\left[
		\beta^2\lambda_{\varepsilon}^{-2}I_6
		+2\beta\gamma\lambda_{\varepsilon}^{-1}I_7
		+\gamma^2I_8
		\right]
		=&\frac{15\pi^2(5-\mu)}{32}
		\alpha_{\varepsilon}^{10-2\mu}\beta\gamma
		\lambda_{\varepsilon}^{-3}
		+o\left(\lambda_{\varepsilon}^{-3}\right)
		\\
		=&\frac{15\pi^2(5-\mu)}{32}
			\beta\gamma
		\lambda_{\varepsilon}^{-3}
		+o\left(\lambda_{\varepsilon}^{-3}\right).
	\end{split}
		\end{equation*}
		Thus (c) follows.

		The proof of (d) relies on the same estimates as above, and is omitted.
	\end{proof}

\begin{proof}[Proof of Proposition \ref{prop38o}]
	Inserting the estimates of Lemma \ref{lemma310} into \eqref{o330}, we obtain
	\begin{equation}\label{o343}
	\begin{split}
		&4\pi\phi_Q(z_{\varepsilon})\lambda_{\varepsilon}^{-2}
		+\varepsilon\Theta_{V}(z_{\varepsilon})\lambda_{\varepsilon}^{-2}
		-4\pi^2Q(z_{\varepsilon})\lambda_{\varepsilon}^{-3}
		+\lambda_{\varepsilon}^{-3}\tilde{R}+O\left(\phi_Q(z_{\varepsilon})\lambda_{\varepsilon}^{-3}\right)
		+o\left(\lambda_{\varepsilon}^{-3}\right)
		+o\left(\varepsilon\lambda_{\varepsilon}^{-2}\right)=0,
	\end{split}
	\end{equation}
	where
	\begin{equation*}
	\begin{split}
		\tilde{R}
		=-4\pi\lambda_{\varepsilon}
		\left(1-\alpha_{\varepsilon}^{10-2\mu}\right)
		\left(\phi_Q(z_{\varepsilon})+\phi_0(z_{\varepsilon})\right)-2\left((6-\mu)d_4+3\pi^2(5-\mu)\right)
		\phi_Q(z_{\varepsilon})^2
		+\frac{15\pi^2(5-\mu)}{16}\beta\gamma .
	\end{split}
	\end{equation*}
	Using Proposition \ref{prop36o}, we first get
	\begin{equation*}
		\tilde{R}
		=-8\pi(5-\mu)\beta
		\phi_0(z_{\varepsilon})
		+\frac{15\pi^2(5-\mu)}{16}\beta\gamma
		+O\left(\phi_Q(z_{\varepsilon})+\lambda_{\varepsilon}^{-1}+\varepsilon\right).
	\end{equation*}
	Inserting the expansions of $\beta$ and $\gamma$, we find that
	\begin{equation*}
		\tilde{R}=O\left(\phi_Q(z_{\varepsilon})
		+\lambda_{\varepsilon}^{-1}
		+\varepsilon\right).
	\end{equation*}
	In particular, $\tilde{R}=O(1)$. Inserting this into \eqref{o343}, we obtain
	\begin{equation*}
		\phi_Q(z_{\varepsilon})
		=O\left(\lambda_{\varepsilon}^{-1}+\varepsilon\right).
	\end{equation*}
	Thus
	\begin{equation*}
		O\left(\phi_Q(z_{\varepsilon})\lambda_{\varepsilon}^{-3}\right)
		=o\left(\lambda_{\varepsilon}^{-3}\right)
		+o\left(\varepsilon\lambda_{\varepsilon}^{-2}\right).
	\end{equation*}
	Moreover, using the above estimate of $\phi_Q(z_{\varepsilon})$, we also have
	\begin{equation*}
		\tilde{R}=O\left(\lambda_{\varepsilon}^{-1}+\varepsilon\right).
	\end{equation*}
	Inserting this improved bound into \eqref{o343}, we get
	\begin{equation*}
	\begin{split}
		0={}&4\pi\phi_Q(z_{\varepsilon})\lambda_{\varepsilon}^{-2}
		+\varepsilon\Theta_{V}(z_{\varepsilon})\lambda_{\varepsilon}^{-2}
		-4\pi^2Q(z_{\varepsilon})\lambda_{\varepsilon}^{-3}
		+o\left(\lambda_{\varepsilon}^{-3}\right)
		+o\left(\varepsilon\lambda_{\varepsilon}^{-2}\right).
	\end{split}
	\end{equation*}
	Dividing by $4\pi\lambda_{\varepsilon}^{-2}$, we obtain \eqref{o328}.
		\end{proof}

	\begin{cor}\label{cor39o}
		We have $\phi_Q(z_0)=0$, $\Theta_V(z_0)\leq 0$ and
		\begin{equation}\label{o329}
			\lambda_{\varepsilon}^{-1}=O(\varepsilon),
		\end{equation}
		as $\varepsilon\rightarrow 0$. Moreover,
		$\|\nabla r_{\varepsilon}\|_{2}
		=O\left(\varepsilon\lambda_{\varepsilon}^{-\frac{1}{2}}\right)$ and
		$\displaystyle
			\alpha_{\varepsilon}^{10-2\mu}
			=1+\frac{16(10-2\mu)}{3\pi}
			\phi_0(z_{\varepsilon})\lambda_{\varepsilon}^{-1}
			+O\left(
			\varepsilon\lambda_{\varepsilon}^{-1}\right).
		$
	\end{cor}
	\begin{proof}
		Once Proposition \ref{prop38o} is obtained, the result follows in the same way as the proof of \cite[Corollary~3.9]{FKK3}.
	\end{proof}

\subsection{The bound on   $\boldsymbol{\nabla \phi_Q(z_{\varepsilon})}$. } 
In this subsection, we establish the bound on $\nabla\phi_Q(z_{\varepsilon})$.

\begin{Prop}\label{prop311o}
	For every $0<\nu<1$, as $\varepsilon\rightarrow 0$, we have
	\begin{equation}\label{o346}
		|\nabla\phi_Q(z_{\varepsilon})|
		\lesssim\varepsilon^{\nu}.
	\end{equation}
\end{Prop}

The proof of Proposition \ref{prop311o} is based on the local Pohozaev identity \eqref{o29}. 
For $v,w\in H^1(\Omega)$, we write
\begin{equation}\label{o344}
	\mathcal{I}[v,w]:=\int_{\partial\Omega}
	\frac{\partial v}{\partial n}
	\frac{\partial w}{\partial n}n
	+\int_{\Omega}(\nabla Q)vw.
\end{equation}
When $v=w$, we simply write $\mathcal{I}[v]$.
Taking $u_{\varepsilon}
=\alpha_{\varepsilon}
(\psi_{z_{\varepsilon},\lambda_{\varepsilon}}+q_{\varepsilon})$
in the Pohozaev identity, we obtain
\begin{equation}\label{o345}
	0=\mathcal{I}[\psi_{z_{\varepsilon},\lambda_{\varepsilon}}]
	+2\mathcal{I}[\psi_{z_{\varepsilon},\lambda_{\varepsilon}},q_{\varepsilon}]
	+\mathcal{I}[q_{\varepsilon}]
	+\varepsilon\int_{\Omega}
	(\nabla V)(\psi_{z_{\varepsilon},\lambda_{\varepsilon}}
	+q_{\varepsilon})^2.
\end{equation}

The following lemma gives the leading contribution from the main term
$\mathcal{I}[\psi_{z_{\varepsilon},\lambda_{\varepsilon}}]$.
\begin{lem}\label{lemma312o}
	For every $0<\nu<1$, as $\varepsilon\rightarrow 0$, we have
	\begin{equation}\label{o347}
		\mathcal{I}[\psi_{z_{\varepsilon},\lambda_{\varepsilon}}]
		=4\pi\nabla\phi_Q(z_{\varepsilon})
		\lambda_{\varepsilon}^{-1}
		+O\left(\lambda_{\varepsilon}^{-1-\nu}\right).
	\end{equation}
\end{lem}
\begin{proof}
	The proof is the same as that of \cite[Lemma~3.12]{FKK3}.
\end{proof}

The next lemma controls the boundary contribution of the error term involving $q_{\varepsilon}$.
\begin{lem}\label{lemma313o}
	As $\varepsilon\rightarrow 0$, we have
	\begin{equation}\label{o348}
		\left\|
		\frac{\partial q_{\varepsilon}}{\partial n}
		\right\|_{L^2(\partial\Omega)}
		\lesssim\varepsilon\lambda_{\varepsilon}^{-\frac12}.
	\end{equation}
\end{lem}
\begin{proof}
	We argue as in the proof of Lemma \ref{lemma26}.
	Combining the equation \eqref{equation.v} for $v_{\varepsilon}$ with
	\begin{equation*}
		\Delta\left(H_Q(z_{\varepsilon},\cdot)-H_0(z_{\varepsilon},\cdot)\right)
=-QG_Q(z_{\varepsilon},\cdot), 
	\end{equation*}
	 we obtain
	\begin{equation*}
		-\Delta q_{\varepsilon}=\mathcal{F},
	\end{equation*}
	where
	\begin{equation*}
	\begin{split}
		\mathcal{F}:={}&
		-3U_{z_{\varepsilon},\lambda_{\varepsilon}}^{5}
		+A_{3,\mu}\alpha_{\varepsilon}^{10-2\mu}
		\left(\int_{\Omega}
		\frac{(\psi_{z_{\varepsilon},\lambda_{\varepsilon}}
		+q_{\varepsilon})^{6-\mu}(y)}
		{|x-y|^{\mu}}dy\right)
		(\psi_{z_{\varepsilon},\lambda_{\varepsilon}}
		+q_{\varepsilon})^{5-\mu}\\
		&-Qq_{\varepsilon}
		+Q(f_{z_{\varepsilon},\lambda_{\varepsilon}}
		+g_{z_{\varepsilon},\lambda_{\varepsilon}})
		-\varepsilon V(\psi_{z_{\varepsilon},\lambda_{\varepsilon}}
		+q_{\varepsilon}).
	\end{split}
	\end{equation*}

		Let $\zeta$ be the cut-off function used in Lemma \ref{lemma26}.
		Since $d_{\varepsilon}^{-1}=O(1)$, the dependence on $d_{\varepsilon}$ will not be written explicitly in the estimates below.
		Then $\zeta q_{\varepsilon}\in H^2(\Omega)\cap H_0^1(\Omega)$ and
		\begin{equation*}
			-\Delta(\zeta q_{\varepsilon})
			=\zeta\mathcal{F}
			-2\nabla\zeta\cdot\nabla q_{\varepsilon}
			-(\Delta\zeta)q_{\varepsilon}.
		\end{equation*}
			On the region $\Omega\setminus B_{d_{\varepsilon}/2}(z_{\varepsilon})$, we have 
			\[
				U_{z_{\varepsilon},\lambda_{\varepsilon}}
				=O\left(\lambda_{\varepsilon}^{-\frac12}\right),
				\qquad
				g_{z_{\varepsilon},\lambda_{\varepsilon}}
				=O\left(\lambda_{\varepsilon}^{-\frac52}\right).
			\]
			Since Corollary \ref{cor39o} gives
$\lambda_{\varepsilon}^{-1}=O(\varepsilon)$, we also have
$\lambda_{\varepsilon}^{-\frac52}
=O\left(\varepsilon\lambda_{\varepsilon}^{-\frac12}\right)$.
			Moreover, since
			\[
				\psi_{z_{\varepsilon},\lambda_{\varepsilon}}
				=U_{z_{\varepsilon},\lambda_{\varepsilon}}
				-\lambda_{\varepsilon}^{-\frac12}
				H_Q(z_{\varepsilon},\cdot)
				-f_{z_{\varepsilon},\lambda_{\varepsilon}},
			\]
			and the terms involving
			$f_{z_{\varepsilon},\lambda_{\varepsilon}}$ and
			$H_Q(z_{\varepsilon},\cdot)$ are controlled by Lemmas \ref{A2}
			and \ref{B1}, we have $$\psi_{z_{\varepsilon},\lambda_{\varepsilon}}^{5-\mu}=O\left(\lambda_{\varepsilon}^{-\frac{5-\mu}{2}}\right)=O\left(\varepsilon\lambda_{\varepsilon}^{-\frac12}\right)$$ on the support of $\zeta$.  Therefore, we have
		\begin{equation}
		\label{o351}
		\begin{split}
			\zeta \left|  \mathcal{F}  \right|\lesssim \left(\int_{\Omega}
		\frac{(\psi_{z_{\varepsilon},\lambda_{\varepsilon}}
		+q_{\varepsilon})^{6-\mu}(y)}
		{|x-y|^{\mu}}dy\right)
		\zeta\left|  q_{\varepsilon}\right|^{5-\mu} 
			    + \left|  q_{\varepsilon}\right|+\varepsilon \zeta  U_{z_{\varepsilon},\lambda_{\varepsilon}}  + \varepsilon\lambda_{\varepsilon}^{-\frac12}. 
		\end{split}
	\end{equation}
	Combining \eqref{o351} with inequality \eqref{o212},  we obtain
	\begin{equation*}
	\begin{split}
		\left\|
		\frac{\partial q_{\varepsilon}}{\partial n}
		\right\|_{L^2(\partial\Omega)}
		={}&
		\left\|
		\frac{\partial(\zeta q_{\varepsilon})}{\partial n}
		\right\|_{L^2(\partial\Omega)}
		\lesssim
		\left\|\Delta(\zeta q_{\varepsilon})\right\|_{\frac32}\\
		\lesssim{}&
		\left\|\zeta\mathcal{F}\right\|_{\frac32}
		+\left\||\nabla\zeta||\nabla q_{\varepsilon}|\right\|_{\frac32}
		+\left\|(\Delta\zeta)q_{\varepsilon}\right\|_{\frac32}\\
		\lesssim{}&
		\left\|
		\left(\int_{\Omega}
		\frac{(\psi_{z_{\varepsilon},\lambda_{\varepsilon}}
		+q_{\varepsilon})^{6-\mu}(y)}
		{|x-y|^{\mu}}dy\right)
		\zeta |q_{\varepsilon}|^{5-\mu}
		\right\|_{\frac32}\\
		&+\varepsilon
		\left\|\zeta U_{z_{\varepsilon},\lambda_{\varepsilon}}
		\right\|_{\frac32}
		+\left\|q_{\varepsilon}\right\|_{\frac32}
		+\varepsilon\lambda_{\varepsilon}^{-\frac12}+\left\||\nabla\zeta||\nabla q_{\varepsilon}|\right\|_{\frac32}
		+\left\|(\Delta\zeta)q_{\varepsilon}\right\|_{\frac32}.
	\end{split}
	\end{equation*}
	All terms except the first can be controlled immediately. Indeed, by \eqref{o310}, \eqref{o311}, Corollary \ref{cor39o}, and the Poincare inequality,
	\begin{equation} \label{estimateql2}
		\left\|q_{\varepsilon}\right\|_2
		+\left\|\nabla q_{\varepsilon}\right\|_{L^2(\Omega\setminus B_{d_{\varepsilon}/2}(z_{\varepsilon}))}
		\lesssim\varepsilon\lambda_{\varepsilon}^{-\frac12}.
	\end{equation}
	Hence
	\begin{equation*}
		\left\|q_{\varepsilon}\right\|_{\frac32}
		+\left\|(\Delta\zeta)q_{\varepsilon}\right\|_{\frac32}
		+\left\||\nabla\zeta||\nabla q_{\varepsilon}|\right\|_{\frac32}
		\lesssim\varepsilon\lambda_{\varepsilon}^{-\frac12}.
	\end{equation*}
	Moreover, by Lemma \ref{A1},
	\begin{equation*}
		\varepsilon
		\left\|\zeta U_{z_{\varepsilon},\lambda_{\varepsilon}}
		\right\|_{\frac32}
		\lesssim
		\varepsilon
		\left\|U_{z_{\varepsilon},\lambda_{\varepsilon}}
		\right\|_{L^{\frac32}(\Omega\setminus B_{d_{\varepsilon}/2}(z_{\varepsilon}))}
		\lesssim\varepsilon\lambda_{\varepsilon}^{-\frac12}.
	\end{equation*}
				
It  remains to estimate the first term $\left\|
		\left(\int_{\Omega}
		\frac{(\psi_{z_{\varepsilon},\lambda_{\varepsilon}}
		+q_{\varepsilon})^{6-\mu}(y)}
		{|x-y|^{\mu}}dy\right)
		\zeta |q_{\varepsilon}|^{5-\mu}
		\right\|_{\frac32}$. Proceeding as in Lemma \ref{lemma26} and using Proposition \ref{A5}, we obtain
\begin{equation*}
	\begin{split}
		\left\|
		\int_{\Omega}
		\frac{\left(\psi_{z_{\varepsilon},\lambda_{\varepsilon}}
		+q_{\varepsilon}\right)^{6-\mu}(y)}
			{|x-y|^\mu}dy|q_{\varepsilon}|^{5-\mu}\zeta
		\right\|_{\frac{3}{2}}
		&\lesssim
		\left\|
		\int_{\Omega}
		\frac{\left(\psi_{z_{\varepsilon},\lambda_{\varepsilon}}
		+q_{\varepsilon}\right)^{6-\mu}(y)}
			{|x-y|^\mu}dy
		\right\|_{\frac{6}{\mu}}
		\left\||q_{\varepsilon}|^{5-\mu}\zeta
		\right\|_{\frac{6}{4-\mu}}
		\\
		&\lesssim
		\left\|
		\left(\psi_{z_{\varepsilon},\lambda_{\varepsilon}}
		+q_{\varepsilon}\right)^{6-\mu}
		\right\|_{\frac{6}{6-\mu}}
		\left\||q_{\varepsilon}|^{5-\mu}\zeta
		\right\|_{\frac{6}{4-\mu}}\lesssim \left\||q_{\varepsilon}|^{5-\mu}\zeta
		\right\|_{\frac{6}{4-\mu}},  
	\end{split}
\end{equation*}
and
\begin{equation*}
\begin{split}
	\left\||q_{\varepsilon}|^{5-\mu}\zeta\right\|_{\frac{6}{4-\mu}}
	={}&\left(
	\int_{\Omega}
	\left|
	\zeta^{\frac{1}{4-\mu}}
	|q_{\varepsilon}|^{\frac{1}{4-\mu}}q_{\varepsilon}
	\right|^6
	\right)^{\frac{4-\mu}{6}}
	\lesssim
	\left(
	\int_{\Omega}
	\left|
	\nabla\left(
	\zeta^{\frac{1}{4-\mu}}
	|q_{\varepsilon}|^{\frac{1}{4-\mu}}q_{\varepsilon}
	\right)
	\right|^2
	\right)^{\frac{4-\mu}{2}}\\
	\lesssim{}&
	\left(
	\int_{\Omega}
	|q_{\varepsilon}|^{\frac{10-2\mu}{4-\mu}}
	\left|
	\nabla\left(\zeta^{\frac{1}{4-\mu}}\right)
	\right|^2
	\right)^{\frac{4-\mu}{2}}
	+\left(
	\int_{\Omega}
	|\mathcal{F}|\,
	\zeta^{\frac{2}{4-\mu}}
	|q_{\varepsilon}|^{\frac{6-\mu}{4-\mu}}
	\right)^{\frac{4-\mu}{2}}
	\\
	\lesssim{}& \left\|  q_{\varepsilon}  \right\|_6^{5-\mu}+\left(
	\int_{\Omega}
	|\mathcal{F}|\,
	\zeta^{\frac{2}{4-\mu}}
	|q_{\varepsilon}|^{\frac{6-\mu}{4-\mu}}
	\right)^{\frac{4-\mu}{2}}. 
\end{split}
\end{equation*}
By \eqref{o327} and the fact $\lambda_{\varepsilon}^{-1}=O(\varepsilon)$, it is easy to see 
\begin{equation*}
	 \left\|  q_{\varepsilon}  \right\|_6^{5-\mu}\lesssim\varepsilon\lambda_{\varepsilon}^{-\frac12}. 
\end{equation*}
Now we estimate $\left(
	\int_{\Omega}
	|\mathcal{F}|  \,
	\zeta^{\frac{2}{4-\mu}}
	|q_{\varepsilon}|^{\frac{6-\mu}{4-\mu}}
\right)^{\frac{4-\mu}{2}}$.  The  pointwise bound \eqref{o351} is equally valid for $ \zeta^{\frac{2}{4-\mu}} \,
	|\mathcal{F}|$, that is 
	\begin{equation*}
		\begin{split}
			\zeta^{\frac{2}{4-\mu}}|\mathcal{F}|
			\lesssim{}&
			\left(\int_{\Omega}
			\frac{(\psi_{z_{\varepsilon},\lambda_{\varepsilon}}
			+q_{\varepsilon})^{6-\mu}(y)}{|x-y|^{\mu}}dy\right)
			\zeta^{\frac{2}{4-\mu}}|q_{\varepsilon}|^{5-\mu}
			+|q_{\varepsilon}|+\varepsilon\zeta^{\frac{2}{4-\mu}}
			U_{z_{\varepsilon},\lambda_{\varepsilon}}
			+\varepsilon\lambda_{\varepsilon}^{-\frac12}.
		\end{split}
	\end{equation*}
		We only need to consider the contribution of the first term in the
		pointwise estimate for
		$\zeta^{\frac{2}{4-\mu}}|\mathcal{F}|$, since the remaining terms can be
		estimated easily.
		By the HLS inequality,   H\"older's inequality and Proposition \ref{A5}, we have
		\begin{equation*}
			\begin{split}
			&\left(\int_{\Omega}
			\left(\int_{\Omega}
			\frac{(\psi_{z_{\varepsilon},\lambda_{\varepsilon}}
			+q_{\varepsilon})^{6-\mu}(y)}{|x-y|^{\mu}}dy\right)
			\zeta^{\frac{2}{4-\mu}}
			|q_{\varepsilon}|^{5-\mu+\frac{6-\mu}{4-\mu}}dx
			\right)^{\frac{4-\mu}{2}}
			\\
			&\lesssim\left(
			\left\|\int_{\Omega}
			\frac{(\psi_{z_{\varepsilon},\lambda_{\varepsilon}}
			+q_{\varepsilon})^{6-\mu}(y)}{|x-y|^{\mu}}dy
			\right\|_{\frac{6}{\mu}}
			\left\||q_{\varepsilon}|^{4-\mu}\right\|_{\frac{6}{4-\mu}}
			\left\|\zeta^{\frac{2}{4-\mu}}
			|q_{\varepsilon}|^{\frac{10-2\mu}{4-\mu}}
			\right\|_3\right)^{\frac{4-\mu}{2}}
			\\
			&\lesssim
			\left\|\psi_{z_{\varepsilon},\lambda_{\varepsilon}}
			+q_{\varepsilon}\right\|_6^{\frac{(6-\mu)(4-\mu)}{2}}
			\left\|q_{\varepsilon}\right\|_6^{\frac{(4-\mu)^2}{2}}
			\left\||q_{\varepsilon}|^{5-\mu}\zeta
			\right\|_{\frac{6}{4-\mu}}
			\\
			&=o\left(
			\left\||q_{\varepsilon}|^{5-\mu}\zeta
			\right\|_{\frac{6}{4-\mu}}
			\right).
				\end{split}
			\end{equation*}
			Combining the estimates above and absorbing the $o\left(
			\left\||q_{\varepsilon}|^{5-\mu}\zeta
			\right\|_{\frac{6}{4-\mu}}
			\right)$ term into the
			left-hand side, we obtain
			\begin{equation*}
				\left\||q_{\varepsilon}|^{5-\mu}\zeta
				\right\|_{\frac{6}{4-\mu}}
				\lesssim
				\varepsilon\lambda_{\varepsilon}^{-\frac12}.
			\end{equation*}
Collecting all the estimates above, we complete the proof of Lemma \ref{lemma313o}.
\end{proof}

\begin{proof}[Proof of Proposition \ref{prop311o}]
	By Lemmas \ref{A1}--\ref{A3} and Lemma \ref{B1}, we have
	\begin{equation*}
		\left\|\psi_{z_{\varepsilon},\lambda_{\varepsilon}}\right\|_2
		+\left\|
		\frac{\partial\psi_{z_{\varepsilon},\lambda_{\varepsilon}}}{\partial n}
		\right\|_{L^2(\partial\Omega)}
		\lesssim\lambda_{\varepsilon}^{-\frac12}.
	\end{equation*}
	Moreover, by \eqref{estimateql2},  and Lemma \ref{lemma313o},
	\begin{equation*}
		\left\|q_{\varepsilon}\right\|_2
		+\left\|
		\frac{\partial q_{\varepsilon}}{\partial n}
		\right\|_{L^2(\partial\Omega)}
		\lesssim\varepsilon\lambda_{\varepsilon}^{-\frac12}.
	\end{equation*}
	It follows from the definition of $\mathcal I$ that
	\begin{equation*}
	\begin{split}
		\left|\mathcal I[\psi_{z_{\varepsilon},\lambda_{\varepsilon}},
		q_{\varepsilon}]\right|
		\lesssim
		\left\|
		\frac{\partial\psi_{z_{\varepsilon},\lambda_{\varepsilon}}}{\partial n}
		\right\|_{L^2(\partial\Omega)}
		\left\|
		\frac{\partial q_{\varepsilon}}{\partial n}
		\right\|_{L^2(\partial\Omega)}+\left\|\psi_{z_{\varepsilon},\lambda_{\varepsilon}}\right\|_2
		\left\|q_{\varepsilon}\right\|_2
		\lesssim\varepsilon\lambda_{\varepsilon}^{-1},
	\end{split}
	\end{equation*}
	and
	\begin{equation*}
		\left|\mathcal I[q_{\varepsilon}]\right|
		\lesssim
		\left\|
		\frac{\partial q_{\varepsilon}}{\partial n}
		\right\|_{L^2(\partial\Omega)}^2
		+\left\|q_{\varepsilon}\right\|_2^2
		\lesssim\varepsilon^2\lambda_{\varepsilon}^{-1}.
	\end{equation*}
	Furthermore,
	\begin{equation*}
		\varepsilon\left|
		\int_{\Omega}(\nabla V)
		(\psi_{z_{\varepsilon},\lambda_{\varepsilon}}+q_{\varepsilon})^2
		\right|
		\lesssim
		\varepsilon
		\left\|\psi_{z_{\varepsilon},\lambda_{\varepsilon}}
		+q_{\varepsilon}\right\|_2^2
		\lesssim\varepsilon\lambda_{\varepsilon}^{-1}.
	\end{equation*}
	Combining these estimates with \eqref{o345} and Lemma
	\ref{lemma312o}, we obtain, for every $0<\nu<1$,
	\begin{equation*}
		|\nabla\phi_Q(z_{\varepsilon})|
		\lesssim
		\varepsilon+\varepsilon^2+\lambda_{\varepsilon}^{-\nu}.
	\end{equation*}
	Finally, Corollary \ref{cor39o} gives
	$\lambda_{\varepsilon}^{-1}=O(\varepsilon)$, and hence
	\begin{equation*}
		|\nabla\phi_Q(z_{\varepsilon})|
		\lesssim\varepsilon^{\nu}.
	\end{equation*}
	This proves the proposition.
\end{proof}

\subsection{Proof of Proposition \ref{prop31}.}
\label{subsec:proof-prop31}
\begin{proof}
	Proposition \ref{prop34} is the first step in proving \eqref{o37}.
	Indeed, Proposition \ref{prop38o} and Corollary \ref{cor39o} give
	\begin{equation*}
		\phi_Q(z_{\varepsilon})
		=O\left(\lambda_{\varepsilon}^{-1}+\varepsilon\right)
		\quad\text{and}\quad
		\lambda_{\varepsilon}^{-1}=O(\varepsilon).
	\end{equation*}
	Substituting these estimates into \eqref{o318}, we obtain \eqref{o37}. Furthermore, Proposition \ref{prop38o}
	yields \eqref{o37a}, Proposition \ref{prop311o} gives \eqref{o37b},
	and Corollary \ref{cor39o} gives \eqref{o37c} and \eqref{o37d}.
	This completes the proof.
\end{proof}

\section{Proof of Theorem \ref{main}}\label{sec:proof-main}

\begin{proof}
Equation \eqref{o118} follows from Proposition \ref{prop31}, together
with \eqref{o31}, \eqref{o33} and \eqref{o35}. 
Moreover, Corollary \ref{cor39o} gives
$\phi_Q(z_0)=0$ and $\Theta_V(z_0)\leq0$. Hence
$z_0\in\mathcal N_Q$.

Now we  prove \eqref{o120}.
Since $Q$ is critical, $\phi_Q\geq 0$ in $\Omega$. Hence $z_0$ is a minimum point of
$\phi_Q$ and $\nabla\phi_Q(z_0)=0$. Moreover, by
\cite[Lemma~4.1]{FKK3}, $\phi_Q\in C^2(\Omega)$. It follows from
Assumption (d) that $z_0$ is a nondegenerate minimum point of $\phi_Q$.
Therefore, applying \cite[Lemma~4.2]{FKK3} to the function
$z\mapsto\phi_Q(z+z_0)$, we obtain
\begin{equation*}
	\phi_Q(z_{\varepsilon})
	\lesssim
	|\nabla\phi_Q(z_{\varepsilon})|^2.
\end{equation*}
By Proposition \ref{prop311o}, for every $0<\nu<1$,
\begin{equation*}
	|\nabla\phi_Q(z_{\varepsilon})|
	\lesssim\varepsilon^{\nu}.
\end{equation*}
Choosing $\frac12<\nu<1$, we conclude that
\begin{equation*}
	\phi_Q(z_{\varepsilon})
	\lesssim\varepsilon^{2\nu}
	=o(\varepsilon),
\end{equation*}
which proves \eqref{o120}.

We next prove \eqref{o119}. Since $z_0$ is a nondegenerate minimum
point of $\phi_Q$, Taylor's formula gives
\begin{equation*}
	\phi_Q(z)
	=\frac12(z-z_0)\cdot D^2\phi_Q(z_0)(z-z_0)
	+o\left(|z-z_0|^2\right)
\end{equation*}
as $z\to z_0$. Since $D^2\phi_Q(z_0)$ is positive definite, it follows
that
\begin{equation*}
	|z-z_0|^2\lesssim\phi_Q(z)
\end{equation*}
for $z$ sufficiently close to $z_0$. Taking $z=z_{\varepsilon}$ and
using \eqref{o120}, we obtain
\begin{equation*}
	|z_{\varepsilon}-z_0|^2
	\lesssim\phi_Q(z_{\varepsilon})
	=o(\varepsilon),
\end{equation*}
and hence \eqref{o119} follows.

We next derive \eqref{o121}. Combining \eqref{o120} with the expansion
\eqref{o328}, and using \eqref{o329}, we obtain
\begin{equation*}
	\pi Q(z_{\varepsilon})\lambda_{\varepsilon}^{-1}
	-\frac{\varepsilon}{4\pi}\Theta_V(z_{\varepsilon})
	=o(\varepsilon).
\end{equation*}
Dividing by $\varepsilon$, we find
\begin{equation*}
	\frac{\pi Q(z_{\varepsilon})}
	{\varepsilon\lambda_{\varepsilon}}
	-\frac{1}{4\pi}\Theta_V(z_{\varepsilon})
	=o(1).
\end{equation*}
Since $z_{\varepsilon}\to z_0$, it follows that
\begin{equation*}
	\frac{1}{\varepsilon\lambda_{\varepsilon}}
	=\frac{\Theta_V(z_0)}{4\pi^2Q(z_0)}+o(1).
\end{equation*}
Here $Q(z_0)<0$ by Assumption (e). If $\Theta_V(z_0)<0$, this proves
\eqref{o121}. If $\Theta_V(z_0)=0$, the preceding identity gives
$\varepsilon\lambda_{\varepsilon}\to\infty$. 

It remains to prove \eqref{o122} and \eqref{o123}. By \eqref{o37d}
and \eqref{o329}, we have $\alpha_{\varepsilon}\to1$. Applying the
Taylor expansion
\begin{equation*}
	t^{\frac{1}{10-2\mu}}
	=1+\frac{1}{10-2\mu}(t-1)
	+O\left(|t-1|^2\right)
	\qquad\text{as }t\to1
\end{equation*}
with $t=\alpha_{\varepsilon}^{10-2\mu}$, we obtain
\begin{equation*}
\begin{split}
	\alpha_{\varepsilon}
	={}&1+\frac{1}{10-2\mu}
	\left(\alpha_{\varepsilon}^{10-2\mu}-1\right)
	+O\left(\left|
	\alpha_{\varepsilon}^{10-2\mu}-1\right|^2\right)\\
	={}&1+\frac{16}{3\pi}
	\phi_0(z_{\varepsilon})\lambda_{\varepsilon}^{-1}
	+o(\varepsilon).
\end{split}
\end{equation*}
If $\Theta_V(z_0)<0$, \eqref{o121} yields
\begin{equation*}
	\lambda_{\varepsilon}^{-1}
	=\frac{|\Theta_V(z_0)|}{4\pi^2|Q(z_0)|}\varepsilon
	+o(\varepsilon).
\end{equation*}
Since $\phi_0(z_{\varepsilon})\to\phi_0(z_0)$, we obtain \eqref{o122}.
If $\Theta_V(z_0)=0$, the proof of \eqref{o121} gives
$\lambda_{\varepsilon}^{-1}=o(\varepsilon)$, and \eqref{o122} follows
as well.

Finally, \eqref{o37} and \eqref{o329} imply
\begin{equation*}
	\|\nabla r_{\varepsilon}\|_2
	=O\left(\varepsilon\lambda_{\varepsilon}^{-\frac12}\right)
	=O\left(\varepsilon^{\frac32}\right),
\end{equation*}
which proves \eqref{o123} and completes the proof.
\end{proof}

\section{Proof of Theorem \ref{main2}}\label{sec:proof-main2}

\subsection{A bound on $\boldsymbol{\|v_{\varepsilon}\|_{\infty}}$.}
In this subsection we estimate the first-order remainder
$v_{\varepsilon}$ in the decomposition
$u_{\varepsilon}
=\alpha_{\varepsilon}(PU_{z_{\varepsilon},\lambda_{\varepsilon}}
+v_{\varepsilon})$, both for its $L^{\infty}$ norm and for 
$L^p$ norms with $p>6$. These estimates are beyond the range given directly by
the Sobolev embedding, and they will be used later in the proof of
Theorem \ref{main2}.

\begin{Prop}\label{prop43o}
	As $\varepsilon\to0$, it holds that
	\begin{equation}
		\label{o401}
		\|v_{\varepsilon}\|_p
		\lesssim \lambda_{\varepsilon}^{-\frac{3}{p}}
		\qquad\text{for every }p\in(6,\infty).
	\end{equation}
	Moreover, for every $\nu>0$,
	\begin{equation}
		\label{o402}
		\|v_{\varepsilon}\|_{\infty}
		=o\left(\lambda_{\varepsilon}^{\nu}\right).
	\end{equation}
\end{Prop}

\begin{proof}
	We first prove \eqref{o401}. Let $r\in(1,\infty)$. By \eqref{o213},
	we have $-\Delta v_{\varepsilon}=\mathcal{F}$. Multiplying this
	equation by $|v_{\varepsilon}|^{r-1}v_{\varepsilon}$ and integrating
	by parts, we obtain
	\begin{equation*}
		\frac{4r}{(r+1)^2}
		\int_{\Omega}
		\left|\nabla
		\left(|v_{\varepsilon}|^{\frac{r+1}{2}}\right)
		\right|^2
		=
		\int_{\Omega}
		\mathcal{F}|v_{\varepsilon}|^{r-1}v_{\varepsilon}.
	\end{equation*}
	Thus, by the Sobolev inequality applied to
	$|v_{\varepsilon}|^{\frac{r+1}{2}}$, it follows that
	\begin{equation}
		\label{o403}
		\|v_{\varepsilon}\|_{3(r+1)}^{r+1}
		\lesssim
		\int_{\Omega}
		|\mathcal{F}||v_{\varepsilon}|^r.
	\end{equation}
	We next estimate the right-hand side of \eqref{o403}. By writing
	$PU_{z_{\varepsilon},\lambda_{\varepsilon}}
	=U_{z_{\varepsilon},\lambda_{\varepsilon}}
	-\varphi_{z_{\varepsilon},\lambda_{\varepsilon}}$, we obtain the
	pointwise bound
	\begin{equation}
		\label{o404}
		\begin{split}
			|\mathcal{F}|\lesssim{}&
			\left|\alpha_{\varepsilon}^{10-2\mu}-1\right|
			\left(\int_{\mathbb{R}^{3}}
			\frac{U_{z_{\varepsilon},\lambda_{\varepsilon}}^{6-\mu}(y)}
			{|x-y|^{\mu}}dy\right)
			U_{z_{\varepsilon},\lambda_{\varepsilon}}^{5-\mu}\\
			&+\left(\int_{\Omega}
			\frac{U_{z_{\varepsilon},\lambda_{\varepsilon}}^{6-\mu}(y)}
			{|x-y|^{\mu}}dy\right)
			\left(
			U_{z_{\varepsilon},\lambda_{\varepsilon}}^{4-\mu}
			|v_{\varepsilon}|
			+U_{z_{\varepsilon},\lambda_{\varepsilon}}^{4-\mu}
			\varphi_{z_{\varepsilon},\lambda_{\varepsilon}}
			+|v_{\varepsilon}|^{5-\mu}
			+\varphi_{z_{\varepsilon},\lambda_{\varepsilon}}^{5-\mu}
			\right)\\
			&+\left(\int_{\Omega}
			\frac{
			U_{z_{\varepsilon},\lambda_{\varepsilon}}^{5-\mu}(y)
			|v_{\varepsilon}(y)|
			+U_{z_{\varepsilon},\lambda_{\varepsilon}}^{5-\mu}(y)
			\varphi_{z_{\varepsilon},\lambda_{\varepsilon}}(y)
			+|v_{\varepsilon}(y)|^{6-\mu}
			+\varphi_{z_{\varepsilon},\lambda_{\varepsilon}}^{6-\mu}(y)}
			{|x-y|^{\mu}}dy\right)
			U_{z_{\varepsilon},\lambda_{\varepsilon}}^{5-\mu}\\
			&+\left(\int_{\mathbb{R}^{3}\setminus\Omega}
			\frac{U_{z_{\varepsilon},\lambda_{\varepsilon}}^{6-\mu}(y)}
			{|x-y|^{\mu}}dy\right)
			U_{z_{\varepsilon},\lambda_{\varepsilon}}^{5-\mu}
			+U_{z_{\varepsilon},\lambda_{\varepsilon}}
			+|v_{\varepsilon}|
			+\varphi_{z_{\varepsilon},\lambda_{\varepsilon}}.
		\end{split}
	\end{equation}
	We estimate the resulting terms separately. Let $\eta>0$ be fixed.
	By Corollary \ref{cor39o}, Lemma \ref{A1}, Lemma \ref{A2}, and
	Proposition \ref{prop21}, we have
	\begin{equation*}
		\left|\alpha_{\varepsilon}^{10-2\mu}-1\right|
		\lesssim \lambda_{\varepsilon}^{-1},\qquad
		\|v_{\varepsilon}\|_6
		+\|\varphi_{z_{\varepsilon},\lambda_{\varepsilon}}\|_{\infty}
		\lesssim \lambda_{\varepsilon}^{-\frac12}.
	\end{equation*}
	We recall the following form of Young's inequality.  Let
	$p,q>1$ satisfy $p^{-1}+q^{-1}=1$. Then, for every $\eta>0$, there
	exists a constant $C_{\eta}>0$ such that, for all $a,b>0$,
	\begin{equation*}
		ab\leq \eta a^p+C_{\eta}b^q.
	\end{equation*}
	Using the identity \eqref{1esay}, H\"older's inequality, and Young's
	inequality, we obtain
	\begin{equation*}
	\begin{split}
		&\left|\alpha_{\varepsilon}^{10-2\mu}-1\right|
		\int_{\Omega}
		\left(\int_{\mathbb{R}^{3}}
		\frac{U_{z_{\varepsilon},\lambda_{\varepsilon}}^{6-\mu}(y)}
		{|x-y|^{\mu}}dy\right)
		U_{z_{\varepsilon},\lambda_{\varepsilon}}^{5-\mu}
		|v_{\varepsilon}|^r\\
		\lesssim&
		\lambda_{\varepsilon}^{-1}
		\int_{\Omega}
		U_{z_{\varepsilon},\lambda_{\varepsilon}}^{5}
		|v_{\varepsilon}|^r\lesssim
		\lambda_{\varepsilon}^{-1}
		\left\|U_{z_{\varepsilon},\lambda_{\varepsilon}}
		\right\|_{5\cdot\frac{3(r+1)}{2r+3}}^5
		\|v_{\varepsilon}\|_{3(r+1)}^r\\
		\lesssim&
		\lambda_{\varepsilon}^{-\frac{r+3}{2(r+1)}}
		\|v_{\varepsilon}\|_{3(r+1)}^r
		\leq
		\eta\|v_{\varepsilon}\|_{3(r+1)}^{r+1}
		+C_{\eta}\lambda_{\varepsilon}^{-\frac{r+3}{2}}.
	\end{split}
	\end{equation*}
	We first establish an estimate that will be needed later, 
	\begin{equation*}
	\begin{split}
		\int_{\Omega}|v_{\varepsilon}|^{r+5}
		=\int_{\Omega}|v_{\varepsilon}|^{r+1}
		|v_{\varepsilon}|^4
		\leq
		\|v_{\varepsilon}\|_{3(r+1)}^{r+1}
		\|v_{\varepsilon}\|_6^{4}\lesssim
		\lambda_{\varepsilon}^{-2}
		\|v_{\varepsilon}\|_{3(r+1)}^{r+1}. 
	\end{split}
	\end{equation*}
	Thus we have
	\begin{equation*}
	\begin{split}
		&\int_{\Omega}
		\left(\int_{\Omega}
		\frac{U_{z_{\varepsilon},\lambda_{\varepsilon}}^{6-\mu}(y)}
		{|x-y|^{\mu}}dy\right)
		U_{z_{\varepsilon},\lambda_{\varepsilon}}^{4-\mu}
		|v_{\varepsilon}|^{r+1}\\
		\lesssim&
		\int_{\Omega}
		U_{z_{\varepsilon},\lambda_{\varepsilon}}^{4}
		|v_{\varepsilon}|^{r+1} \lesssim
		\left(\int_{\Omega}
		U_{z_{\varepsilon},\lambda_{\varepsilon}}^{5}
		|v_{\varepsilon}|^r\right)^{\frac45}
		\left(\int_{\Omega}|v_{\varepsilon}|^{r+5}\right)^{\frac15}\\
		\lesssim&
		\left(
		\lambda_{\varepsilon}^{\frac{r-1}{2(r+1)}}
		\|v_{\varepsilon}\|_{3(r+1)}^r
		\right)^{\frac45}
		\left(
		\lambda_{\varepsilon}^{-2}
		\|v_{\varepsilon}\|_{3(r+1)}^{r+1}
		\right)^{\frac15}\\
		=&
		\lambda_{\varepsilon}^{-\frac{4}{5(r+1)}}
		\|v_{\varepsilon}\|_{3(r+1)}^{r+\frac15}
		\leq
		\eta\|v_{\varepsilon}\|_{3(r+1)}^{r+1}
		+C_{\eta}\lambda_{\varepsilon}^{-1}.
	\end{split}
	\end{equation*}
	Moreover, it follows that
	\begin{equation*}
	\begin{split}
		&\int_{\Omega}
		\left(\int_{\Omega}
		\frac{U_{z_{\varepsilon},\lambda_{\varepsilon}}^{6-\mu}(y)}
		{|x-y|^{\mu}}dy\right)
		|v_{\varepsilon}|^{r+5-\mu}\\
		\lesssim&
		\int_{\Omega}
		U_{z_{\varepsilon},\lambda_{\varepsilon}}^{\mu}
		|v_{\varepsilon}|^{r+5-\mu}
		\lesssim
		\left\|U_{z_{\varepsilon},\lambda_{\varepsilon}}^{\mu}
		\right\|_{\frac{6}{\mu}}
		\|v_{\varepsilon}\|_{3(r+1)}^{r+1}
		\|v_{\varepsilon}\|_6^{4-\mu}\\
		\lesssim&
		\lambda_{\varepsilon}^{-\frac{4-\mu}{2}}
		\|v_{\varepsilon}\|_{3(r+1)}^{r+1},
	\end{split}
	\end{equation*}
	and
	\begin{equation*}
	\begin{split}
		&\int_{\Omega}
		\left(\int_{\Omega}
		\frac{U_{z_{\varepsilon},\lambda_{\varepsilon}}^{6-\mu}(y)}
		{|x-y|^{\mu}}dy\right)
		\left(
		U_{z_{\varepsilon},\lambda_{\varepsilon}}^{4-\mu}
		\varphi_{z_{\varepsilon},\lambda_{\varepsilon}}
		+\varphi_{z_{\varepsilon},\lambda_{\varepsilon}}^{5-\mu}
		\right)|v_{\varepsilon}|^r\\
		\lesssim&
		\|\varphi_{z_{\varepsilon},\lambda_{\varepsilon}}\|_{\infty}
		\int_{\Omega}
		U_{z_{\varepsilon},\lambda_{\varepsilon}}^{4}
		|v_{\varepsilon}|^r
		+\|\varphi_{z_{\varepsilon},\lambda_{\varepsilon}}\|_{\infty}^{5-\mu}
		\int_{\Omega}
		U_{z_{\varepsilon},\lambda_{\varepsilon}}^{\mu}
		|v_{\varepsilon}|^r\\
		\lesssim&
		\|\varphi_{z_{\varepsilon},\lambda_{\varepsilon}}\|_{\infty}
		\left\|U_{z_{\varepsilon},\lambda_{\varepsilon}}
		\right\|_{\frac{12(r+1)}{2r+3}}^4
		\|v_{\varepsilon}\|_{3(r+1)}^r
		+\|\varphi_{z_{\varepsilon},\lambda_{\varepsilon}}\|_{\infty}^{5-\mu}
		\left\|U_{z_{\varepsilon},\lambda_{\varepsilon}}
		\right\|_{\frac{3\mu(r+1)}{2r+3}}^{\mu}
		\|v_{\varepsilon}\|_{3(r+1)}^r\\
		\lesssim&
		\left(
		\lambda_{\varepsilon}^{-\frac{r+3}{2(r+1)}}
		+\lambda_{\varepsilon}^{-\frac52}
		\right)
		\|v_{\varepsilon}\|_{3(r+1)}^r
		\leq
		\eta\|v_{\varepsilon}\|_{3(r+1)}^{r+1}
		+C_{\eta}\lambda_{\varepsilon}^{-\frac{r+3}{2}}.
	\end{split}
	\end{equation*}
	For later use, we record the following estimate:
	\begin{equation}\label{o405}
	\begin{split}
		\left\|
		U_{z_{\varepsilon},\lambda_{\varepsilon}}^{5-\mu}
		|v_{\varepsilon}|^r
		\right\|_{\frac{6}{6-\mu}}
		&\leq
		\left\|
		U_{z_{\varepsilon},\lambda_{\varepsilon}}^{5-\mu}
		\right\|_{\frac{6(r+1)}{(4-\mu)r+6-\mu}}
		\|v_{\varepsilon}\|_{3(r+1)}^r\lesssim
		\lambda_{\varepsilon}^{\frac{r-1}{2(r+1)}}
		\|v_{\varepsilon}\|_{3(r+1)}^r,
	\end{split}
	\end{equation}
	where we used  Lemma \ref{A1}, noting that
$\frac{6(5-\mu)(r+1)}{(4-\mu)r+6-\mu}>3$. 

	The HLS inequality and \eqref{o405} give
	\begin{equation*}
	\begin{split}
		&\int_{\Omega}
		\left(\int_{\Omega}
		\frac{
		U_{z_{\varepsilon},\lambda_{\varepsilon}}^{5-\mu}(y)
		\varphi_{z_{\varepsilon},\lambda_{\varepsilon}}(y)
		+\varphi_{z_{\varepsilon},\lambda_{\varepsilon}}^{6-\mu}(y)}
		{|x-y|^{\mu}}dy\right)
		U_{z_{\varepsilon},\lambda_{\varepsilon}}^{5-\mu}
		|v_{\varepsilon}|^r\\
		\lesssim&
		\left(
		\left\|
		U_{z_{\varepsilon},\lambda_{\varepsilon}}^{5-\mu}
		\varphi_{z_{\varepsilon},\lambda_{\varepsilon}}
		\right\|_{\frac{6}{6-\mu}}
		+\left\|
		\varphi_{z_{\varepsilon},\lambda_{\varepsilon}}^{6-\mu}
		\right\|_{\frac{6}{6-\mu}}
		\right)
		\left\|
		U_{z_{\varepsilon},\lambda_{\varepsilon}}^{5-\mu}
		|v_{\varepsilon}|^r
		\right\|_{\frac{6}{6-\mu}}\\
		\lesssim&
		\left(
		\|\varphi_{z_{\varepsilon},\lambda_{\varepsilon}}\|_{\infty}
		\left\|U_{z_{\varepsilon},\lambda_{\varepsilon}}
		\right\|_{\frac{6(5-\mu)}{6-\mu}}^{5-\mu}
		+\|\varphi_{z_{\varepsilon},\lambda_{\varepsilon}}\|_6^{6-\mu}
		\right)
		\lambda_{\varepsilon}^{\frac{r-1}{2(r+1)}}
		\|v_{\varepsilon}\|_{3(r+1)}^r\\
		\lesssim&
		\lambda_{\varepsilon}^{-\frac{r+3}{2(r+1)}}
		\|v_{\varepsilon}\|_{3(r+1)}^r
		\leq
		\eta\|v_{\varepsilon}\|_{3(r+1)}^{r+1}
		+C_{\eta}\lambda_{\varepsilon}^{-\frac{r+3}{2}},
	\end{split}
	\end{equation*}
	\begin{equation*}
	\begin{split}
		&\int_{\Omega}\int_{\Omega}
		\frac{
		U_{z_{\varepsilon},\lambda_{\varepsilon}}^{5-\mu}(y)
		|v_{\varepsilon}(y)|
		U_{z_{\varepsilon},\lambda_{\varepsilon}}^{5-\mu}(x)
		|v_{\varepsilon}(x)|^r}
		{|x-y|^{\mu}}dxdy\\
		\lesssim&
		\left\|
		U_{z_{\varepsilon},\lambda_{\varepsilon}}^{5-\mu}
		v_{\varepsilon}\right\|_{\frac{6}{6-\mu}}
		\left\|
		U_{z_{\varepsilon},\lambda_{\varepsilon}}^{5-\mu}
		|v_{\varepsilon}|^r\right\|_{\frac{6}{6-\mu}}
		\lesssim
		\lambda_{\varepsilon}^{-\frac12}
		\left(
		\lambda_{\varepsilon}^{\frac{r-1}{2(r+1)}}
		\|v_{\varepsilon}\|_{3(r+1)}^r
		\right)\\
		=&
		\lambda_{\varepsilon}^{-\frac{1}{r+1}}
		\|v_{\varepsilon}\|_{3(r+1)}^r
		\leq
		\eta\|v_{\varepsilon}\|_{3(r+1)}^{r+1}
		+C_{\eta}\lambda_{\varepsilon}^{-1},
	\end{split}
	\end{equation*}
	and
	\begin{equation*}
	\begin{split}
		&\int_{\Omega}\int_{\Omega}
		\frac{
		|v_{\varepsilon}(y)|^{6-\mu}
		U_{z_{\varepsilon},\lambda_{\varepsilon}}^{5-\mu}(x)
		|v_{\varepsilon}(x)|^r}
		{|x-y|^{\mu}}dxdy\\
		\lesssim&
		\|v_{\varepsilon}\|_6^{6-\mu}
		\left\|
		U_{z_{\varepsilon},\lambda_{\varepsilon}}^{5-\mu}
		|v_{\varepsilon}|^r\right\|_{\frac{6}{6-\mu}}
		\lesssim
		\lambda_{\varepsilon}^{-\frac{6-\mu}{2}
		+\frac{r-1}{2(r+1)}}
		\|v_{\varepsilon}\|_{3(r+1)}^r
		\\=&
		\lambda_{\varepsilon}^{-\frac{(5-\mu)r+7-\mu}{2(r+1)}}
		\|v_{\varepsilon}\|_{3(r+1)}^r
		\leq
		\eta\|v_{\varepsilon}\|_{3(r+1)}^{r+1}
		+C_{\eta}\lambda_{\varepsilon}^{-\frac{(5-\mu)r+7-\mu}{2}}.
	\end{split}
	\end{equation*}
		For the remaining terms, the HLS inequality, \eqref{o405}, and
		Lemma  \ref{A1RN} give
	\begin{equation*}
	\begin{split}
		&\int_{\Omega}
		\left(\int_{\mathbb{R}^{3}\setminus\Omega}
		\frac{U_{z_{\varepsilon},\lambda_{\varepsilon}}^{6-\mu}(y)}
		{|x-y|^{\mu}}dy\right)
		U_{z_{\varepsilon},\lambda_{\varepsilon}}^{5-\mu}
		|v_{\varepsilon}|^r\\
		\lesssim&
		\left\|U_{z_{\varepsilon},\lambda_{\varepsilon}}^{6-\mu}
		\right\|_{L^{\frac{6}{6-\mu}}(\mathbb{R}^{3}\setminus\Omega)}
		\left\|U_{z_{\varepsilon},\lambda_{\varepsilon}}^{5-\mu}
		|v_{\varepsilon}|^r\right\|_{\frac{6}{6-\mu}}
		\lesssim
		\lambda_{\varepsilon}^{-\frac{6-\mu}{2}}
		\lambda_{\varepsilon}^{\frac{r-1}{2(r+1)}}
		\|v_{\varepsilon}\|_{3(r+1)}^r
	\\	= &
		\lambda_{\varepsilon}^{-\frac{(5-\mu)r+7-\mu}{2(r+1)}}
		\|v_{\varepsilon}\|_{3(r+1)}^r
		\leq
		\eta\|v_{\varepsilon}\|_{3(r+1)}^{r+1}
		+C_{\eta}\lambda_{\varepsilon}^{-\frac{(5-\mu)r+7-\mu}{2}},
	\end{split}
	\end{equation*}
	\begin{equation*}
	\begin{split}
		\int_{\Omega}U_{z_{\varepsilon},\lambda_{\varepsilon}}
		|v_{\varepsilon}|^r
		&\leq
		\left\|U_{z_{\varepsilon},\lambda_{\varepsilon}}
		\right\|_{\frac{3(r+1)}{2r+3}}
		\|v_{\varepsilon}\|_{3(r+1)}^r
		\lesssim
		\lambda_{\varepsilon}^{-\frac12}
		\|v_{\varepsilon}\|_{3(r+1)}^r
		\leq
		\eta\|v_{\varepsilon}\|_{3(r+1)}^{r+1}
		+C_{\eta}\lambda_{\varepsilon}^{-\frac{r+1}{2}},\\
		\int_{\Omega}
		\varphi_{z_{\varepsilon},\lambda_{\varepsilon}}
		|v_{\varepsilon}|^r
		&\lesssim
		\|\varphi_{z_{\varepsilon},\lambda_{\varepsilon}}\|_{\infty}
		\|v_{\varepsilon}\|_{3(r+1)}^r
		\lesssim
		\lambda_{\varepsilon}^{-\frac12}
		\|v_{\varepsilon}\|_{3(r+1)}^r
		\leq
		\eta\|v_{\varepsilon}\|_{3(r+1)}^{r+1}
		+C_{\eta}\lambda_{\varepsilon}^{-\frac{r+1}{2}},\\
		\int_{\Omega}|v_{\varepsilon}|^{r+1}
		&\lesssim
		\left(\int_{\Omega}|v_{\varepsilon}|^{r+5}\right)^{\frac{r+1}{r+5}}
		\lesssim
		\lambda_{\varepsilon}^{-\frac{2(r+1)}{r+5}}
		\|v_{\varepsilon}\|_{3(r+1)}^{\frac{(r+1)^2}{r+5}}
		\leq
		\eta\|v_{\varepsilon}\|_{3(r+1)}^{r+1}
		+C_{\eta}\lambda_{\varepsilon}^{-\frac{r+1}{2}}.
	\end{split}
	\end{equation*}
	
	Choosing $\eta$ small enough and then taking $\varepsilon$ sufficiently
	small, we absorb the terms
	$\eta\|v_{\varepsilon}\|_{3(r+1)}^{r+1}$,
	$\lambda_{\varepsilon}^{-2}
	\|v_{\varepsilon}\|_{3(r+1)}^{r+1}$ and
	$\lambda_{\varepsilon}^{-\frac{4-\mu}{2}}
	\|v_{\varepsilon}\|_{3(r+1)}^{r+1}$ into the left-hand side of
	\eqref{o403}. Hence
	\begin{equation*}
		\|v_{\varepsilon}\|_{3(r+1)}^{r+1}
		\lesssim
		\lambda_{\varepsilon}^{-\frac{r+3}{2}}
		+\lambda_{\varepsilon}^{-1}
		+\lambda_{\varepsilon}^{-\frac{(5-\mu)r+7-\mu}{2}}
		+\lambda_{\varepsilon}^{-\frac{r+1}{2}}
		\lesssim \lambda_{\varepsilon}^{-1}.
	\end{equation*}
		Therefore, putting $p=3(r+1)$, we obtain \eqref{o401}.

		We next prove \eqref{o402}. Since
		$-\Delta v_{\varepsilon}=\mathcal F$ in $\Omega$ and
		$v_{\varepsilon}=0$ on $\partial\Omega$, the Green function
		representation gives
		\begin{equation}\label{o406}
			v_{\varepsilon}(x)
			=\frac{1}{4\pi}\int_{\Omega}G_0(x,y)\mathcal F(y)dy.
		\end{equation}
		Fix $q\in(3/2,3)$ and let $q'=q/(q-1)$. Since $1<q'<3$ and
		$0\leq G_0(x,y)\leq |x-y|^{-1}$, H\"older's inequality yields
		\begin{equation}\label{o407}
			\|v_{\varepsilon}\|_{\infty}
			\lesssim
			\sup_{x\in\Omega}\|G_0(x,\cdot)\|_{q'}
			\|\mathcal F\|_q
			\lesssim \|\mathcal F\|_q.
		\end{equation}
		Hence it suffices to estimate $\|\mathcal F\|_q$ with some $q\in(3/2,3)$. 
		By Corollary \ref{cor39o}, \eqref{1esay}, and Lemma \ref{A1},
		\begin{equation*}
		\begin{split}
			\left|\alpha_{\varepsilon}^{10-2\mu}-1\right|
			\left\|
			\left(\int_{\mathbb R^3}
			\frac{U_{z_{\varepsilon},\lambda_{\varepsilon}}^{6-\mu}(y)}
			{|x-y|^\mu}dy\right)
			U_{z_{\varepsilon},\lambda_{\varepsilon}}^{5-\mu}
			\right\|_q\lesssim
			\lambda_{\varepsilon}^{-1}
			\left\|U_{z_{\varepsilon},\lambda_{\varepsilon}}^5\right\|_q
			\lesssim
			\lambda_{\varepsilon}^{\frac32-\frac3q}.
		\end{split}
		\end{equation*}

		For the second group of terms on the right-hand side of \eqref{o404},
		\eqref{1esay} and Young's inequality give
		\begin{equation*}
		\begin{split}
			&\left(\int_{\Omega}
			\frac{U_{z_{\varepsilon},\lambda_{\varepsilon}}^{6-\mu}(y)}
			{|x-y|^\mu}dy\right)
			\left(
			U_{z_{\varepsilon},\lambda_{\varepsilon}}^{4-\mu}|v_{\varepsilon}|
			+U_{z_{\varepsilon},\lambda_{\varepsilon}}^{4-\mu}
			\varphi_{z_{\varepsilon},\lambda_{\varepsilon}}
			+|v_{\varepsilon}|^{5-\mu}
			+\varphi_{z_{\varepsilon},\lambda_{\varepsilon}}^{5-\mu}
			\right)\\
			\lesssim&
			U_{z_{\varepsilon},\lambda_{\varepsilon}}^4|v_{\varepsilon}|
			+U_{z_{\varepsilon},\lambda_{\varepsilon}}^4
			\varphi_{z_{\varepsilon},\lambda_{\varepsilon}}
			+U_{z_{\varepsilon},\lambda_{\varepsilon}}^\mu
			|v_{\varepsilon}|^{5-\mu}
			+U_{z_{\varepsilon},\lambda_{\varepsilon}}^\mu
			\varphi_{z_{\varepsilon},\lambda_{\varepsilon}}^{5-\mu}\\
			\lesssim&
			U_{z_{\varepsilon},\lambda_{\varepsilon}}^4|v_{\varepsilon}|
			+U_{z_{\varepsilon},\lambda_{\varepsilon}}^4
			\varphi_{z_{\varepsilon},\lambda_{\varepsilon}}
			+|v_{\varepsilon}|^5
			+\varphi_{z_{\varepsilon},\lambda_{\varepsilon}}^5.
		\end{split}
		\end{equation*}
		By \eqref{o401}, Lemma \ref{A1} and Lemma \ref{A2},  it follows that 
		\begin{equation*}
		\begin{split}
			&\left\|U_{z_{\varepsilon},\lambda_{\varepsilon}}^4
			|v_{\varepsilon}|\right\|_q
			+\left\|U_{z_{\varepsilon},\lambda_{\varepsilon}}^4
			\varphi_{z_{\varepsilon},\lambda_{\varepsilon}}\right\|_q
			+\||v_{\varepsilon}|^5\|_q
			+\|\varphi_{z_{\varepsilon},\lambda_{\varepsilon}}^5\|_q\\
		\lesssim&
			\left\|U_{z_{\varepsilon},\lambda_{\varepsilon}}\right\|_{5q}^4
			\|v_{\varepsilon}\|_{5q}
			+\|\varphi_{z_{\varepsilon},\lambda_{\varepsilon}}\|_{\infty}
			\left\|U_{z_{\varepsilon},\lambda_{\varepsilon}}\right\|_{4q}^4
			+\|v_{\varepsilon}\|_{5q}^5
			+\|\varphi_{z_{\varepsilon},\lambda_{\varepsilon}}\|_{\infty}^5\\
		\lesssim &
			\lambda_{\varepsilon}^{2-\frac3q}
			+\lambda_{\varepsilon}^{\frac32-\frac3q}
			+\lambda_{\varepsilon}^{-\frac3q}
			+\lambda_{\varepsilon}^{-\frac52}
			\lesssim \lambda_{\varepsilon}^{2-\frac3q}.
		\end{split}
		\end{equation*}

		For the contributions from the third and fourth groups of terms on the
		right-hand side of \eqref{o404},
		H\"older's inequality, Proposition \ref{A5} and Lemma \ref{A1} give
		\begin{equation*}
		\begin{split}
			&\left\|
			\left(\int_{\Omega}
			\frac{U_{z_{\varepsilon},\lambda_{\varepsilon}}^{5-\mu}(y)
			|v_{\varepsilon}(y)|}{|x-y|^\mu}dy\right)
			U_{z_{\varepsilon},\lambda_{\varepsilon}}^{5-\mu}
			\right\|_q\\
			\lesssim{}&
			\left\|
			\int_{\Omega}
			\frac{U_{z_{\varepsilon},\lambda_{\varepsilon}}^{5-\mu}(y)
			|v_{\varepsilon}(y)|}{|x-y|^\mu}dy
			\right\|_{\frac6\mu}
			\left\|U_{z_{\varepsilon},\lambda_{\varepsilon}}^{5-\mu}
			\right\|_{\frac{6q}{6-\mu q}}\\
			\lesssim{}&
			\left\|U_{z_{\varepsilon},\lambda_{\varepsilon}}^{5-\mu}
			v_{\varepsilon}\right\|_{\frac6{6-\mu}}
			\left\|U_{z_{\varepsilon},\lambda_{\varepsilon}}^{5-\mu}
			\right\|_{\frac{6q}{6-\mu q}}\\
			\lesssim{}&
			\left\|U_{z_{\varepsilon},\lambda_{\varepsilon}}^{5-\mu}
			\right\|_{\frac6{5-\mu}}
			\|v_{\varepsilon}\|_6
			\lambda_{\varepsilon}^{\frac52-\frac3q}
			\lesssim \lambda_{\varepsilon}^{2-\frac3q}.
		\end{split}
		\end{equation*}
		Similarly, we have
		\begin{equation*}
		\begin{split}
			&\left\|
			\left(\int_{\Omega}
			\frac{U_{z_{\varepsilon},\lambda_{\varepsilon}}^{5-\mu}(y)
			\varphi_{z_{\varepsilon},\lambda_{\varepsilon}}(y)}
			{|x-y|^\mu}dy\right)
			U_{z_{\varepsilon},\lambda_{\varepsilon}}^{5-\mu}
			\right\|_q\\
			\lesssim{}&
			\left\|U_{z_{\varepsilon},\lambda_{\varepsilon}}^{5-\mu}
			\varphi_{z_{\varepsilon},\lambda_{\varepsilon}}
			\right\|_{\frac6{6-\mu}}
			\left\|U_{z_{\varepsilon},\lambda_{\varepsilon}}^{5-\mu}
			\right\|_{\frac{6q}{6-\mu q}}
			\lesssim \lambda_{\varepsilon}^{\frac32-\frac3q},\\
			&\left\|
			\left(\int_{\Omega}
			\frac{|v_{\varepsilon}(y)|^{6-\mu}
			+\varphi_{z_{\varepsilon},\lambda_{\varepsilon}}^{6-\mu}(y)}
			{|x-y|^\mu}dy\right)
			U_{z_{\varepsilon},\lambda_{\varepsilon}}^{5-\mu}
			\right\|_q\\
			\lesssim{}&
			\left(
			\|v_{\varepsilon}\|_6^{6-\mu}
			+\|\varphi_{z_{\varepsilon},\lambda_{\varepsilon}}\|_6^{6-\mu}
			\right)
			\left\|U_{z_{\varepsilon},\lambda_{\varepsilon}}^{5-\mu}
			\right\|_{\frac{6q}{6-\mu q}}
			\lesssim
			\lambda_{\varepsilon}^{\frac{\mu-1}{2}-\frac3q},
		\end{split}
		\end{equation*}
and
		\begin{equation*}
		\begin{split}
			&\left\|
			\left(\int_{\mathbb R^3\setminus\Omega}
			\frac{U_{z_{\varepsilon},\lambda_{\varepsilon}}^{6-\mu}(y)}
			{|x-y|^\mu}dy\right)
			U_{z_{\varepsilon},\lambda_{\varepsilon}}^{5-\mu}
			\right\|_q\\
			\lesssim&
			\left\|U_{z_{\varepsilon},\lambda_{\varepsilon}}^{6-\mu}
			\right\|_{L^{\frac6{6-\mu}}(\mathbb R^3\setminus\Omega)}
			\left\|U_{z_{\varepsilon},\lambda_{\varepsilon}}^{5-\mu}
			\right\|_{\frac{6q}{6-\mu q}}
			\lesssim
			\lambda_{\varepsilon}^{\frac{\mu-1}{2}-\frac3q}.
		\end{split}
		\end{equation*}
		Finally, since $q<3$, Lemma \ref{A1}, Lemma \ref{A2}, and
		Proposition \ref{prop21} imply
		\begin{equation*}
			\left\|U_{z_{\varepsilon},\lambda_{\varepsilon}}
			+\varphi_{z_{\varepsilon},\lambda_{\varepsilon}}
			+|v_{\varepsilon}|\right\|_q
			\leq \|U_{z_{\varepsilon},\lambda_{\varepsilon}}\|_q
			+\|\varphi_{z_{\varepsilon},\lambda_{\varepsilon}}\|_{\infty}
			+\|\nabla v_{\varepsilon}\|_2
			\lesssim \lambda_{\varepsilon}^{-\frac12}.
		\end{equation*}
		 Combining the preceding bounds into \eqref{o404},  we obtain
		\begin{equation}\label{o410}
			\|\mathcal F\|_q
			\lesssim \lambda_{\varepsilon}^{2-\frac3q}
			\qquad\text{for every }q\in(3/2,3).
		\end{equation}
		Given $\nu>0$, we choose $q>3/2$ sufficiently close to $3/2$ such that
		$2-3/q<\nu$. It follows from \eqref{o407} and \eqref{o410} that
		\begin{equation*}
			\|v_{\varepsilon}\|_{\infty}
			\lesssim \lambda_{\varepsilon}^{2-\frac3q}
			=o\left(\lambda_{\varepsilon}^{\nu}\right),
		\end{equation*}
		which proves \eqref{o402}.
	\end{proof}

	\subsection{Proof of Theorem \ref{main2}(a)}
	By Proposition \ref{prop21},
	$u_{\varepsilon}=\alpha_{\varepsilon}
	(PU_{z_{\varepsilon},\lambda_{\varepsilon}}+v_{\varepsilon})$ with
	$\alpha_{\varepsilon}=1+o(1)$. Moreover, Proposition \ref{prop43o},
	applied with $\nu=1/2$, gives
	$\|v_{\varepsilon}\|_{\infty}
	=o(\lambda_{\varepsilon}^{1/2})$. Since
	$z_{\varepsilon}\to z_0\in\Omega$, Lemma \ref{A2} yields
	\begin{equation*}
		\|PU_{z_{\varepsilon},\lambda_{\varepsilon}}\|_{\infty}
		=\|U_{z_{\varepsilon},\lambda_{\varepsilon}}\|_{\infty}
		+O\left(\|\varphi_{z_{\varepsilon},\lambda_{\varepsilon}}\|_{\infty}\right)
		=\lambda_{\varepsilon}^{\frac12}
		+O\left(\lambda_{\varepsilon}^{-\frac12}\right).
	\end{equation*}
	At the same time,
	$U_{z_{\varepsilon},\lambda_{\varepsilon}}(z_{\varepsilon})
	=\lambda_{\varepsilon}^{1/2}$, and Lemma \ref{A2} gives
	$PU_{z_{\varepsilon},\lambda_{\varepsilon}}(z_{\varepsilon})
	=\lambda_{\varepsilon}^{1/2}+O(\lambda_{\varepsilon}^{-1/2})$.
	Consequently, by \eqref{o121},
	\begin{equation*}
	\begin{split}
		\varepsilon\|u_{\varepsilon}\|_{\infty}^{2}
		&=\varepsilon\lambda_{\varepsilon}(1+o(1))
		=\frac{4\pi^{2}|Q(z_0)|}{|\Theta_V(z_0)|}(1+o(1)),\\
		\varepsilon|u_{\varepsilon}(z_{\varepsilon})|^{2}
		&=\varepsilon\lambda_{\varepsilon}(1+o(1))
		=\frac{4\pi^{2}|Q(z_0)|}{|\Theta_V(z_0)|}(1+o(1)).
		\end{split}
		\end{equation*}
		Therefore, part (a) follows.

	\subsection{Proof of Theorem \ref{main2}(b)}
	The Green representation formula associated with $-\Delta+Q$ gives
	\begin{equation}
	\label{o414}
	\begin{split}
		u_{\varepsilon}(z)
		={}&\frac{A_{3,\mu}}{4\pi}
		\int_{\Omega}G_Q(z,x)
		\left(
		\int_{\Omega}
		\frac{u_{\varepsilon}^{6-\mu}(y)}{|x-y|^{\mu}}dy
		\right)
		u_{\varepsilon}^{5-\mu}(x)dx-\frac{\varepsilon}{4\pi}
		\int_{\Omega}G_Q(z,x)V(x)u_{\varepsilon}(x)dx.
	\end{split}
	\end{equation}
	Choose a sequence $\delta_{\varepsilon}>0$ such that
	$\delta_{\varepsilon}=o(1)$ and
	$\lambda_{\varepsilon}^{-1}=o(\delta_{\varepsilon})$.
	Using Proposition \ref{prop21}, Lemmas \ref{A1}, \ref{A2} and \eqref{1esay}, we obtain
	\begin{equation*}
		\frac{A_{3,\mu}}{4\pi}
		\int_{B_{\delta_{\varepsilon}}(z_{\varepsilon})}
		\left(
		\int_{\Omega}
		\frac{u_{\varepsilon}^{6-\mu}(y)}{|x-y|^{\mu}}dy
		\right)
		u_{\varepsilon}^{5-\mu}(x)dx
		=\lambda_{\varepsilon}^{-\frac12}
		+o\left(\lambda_{\varepsilon}^{-\frac12}\right).
	\end{equation*}
	Uniformly for  $z$  away from $z_0$, we have
	\begin{equation*}
	\begin{split}
		&\frac{A_{3,\mu}}{4\pi}
		\int_{B_{\delta_{\varepsilon}}(z_{\varepsilon})}
		G_Q(z,x)
		\left(
		\int_{\Omega}
		\frac{u_{\varepsilon}^{6-\mu}(y)}{|x-y|^{\mu}}dy
		\right)
		u_{\varepsilon}^{5-\mu}(x)dx\\
		={}&\frac{A_{3,\mu}}{4\pi}
		\int_{B_{\delta_{\varepsilon}}(z_{\varepsilon})}
		\left(G_Q(z,z_0)+o(1)\right)
		\left(
		\int_{\Omega}
		\frac{u_{\varepsilon}^{6-\mu}(y)}{|x-y|^{\mu}}dy
		\right)
		u_{\varepsilon}^{5-\mu}(x)dx\\
		={}&\lambda_{\varepsilon}^{-\frac12}G_Q(z,z_0)
		+o\left(\lambda_{\varepsilon}^{-\frac12}\right).
	\end{split}
	\end{equation*}
	On $\Omega\setminus B_{\delta_{\varepsilon}}(z_{\varepsilon})$,
	H\"older's inequality, Proposition \ref{A5}, Lemma \ref{A1}, and
	Proposition \ref{prop43o} give
	\begin{equation*}
	\begin{split}
		&\left|
		\frac{A_{3,\mu}}{4\pi}
		\int_{\Omega\setminus B_{\delta_{\varepsilon}}(z_{\varepsilon})}
		G_Q(z,x)
		\left(
		\int_{\Omega}
		\frac{u_{\varepsilon}^{6-\mu}(y)}{|x-y|^{\mu}}dy
		\right)
		u_{\varepsilon}^{5-\mu}(x)dx
		\right|\\
		\lesssim{}&\|G_Q(z,\cdot)\|_2
		\left\|
		\int_{\Omega}
		\frac{u_{\varepsilon}^{6-\mu}(y)}{|\cdot-y|^{\mu}}dy
		\right\|_{\frac6\mu}
		\left\|u_{\varepsilon}^{5-\mu}\right\|_{L^{\frac6{3-\mu}}
		(\Omega\setminus B_{\delta_{\varepsilon}}(z_{\varepsilon}))}\\
		\lesssim{}&\|G_Q(z,\cdot)\|_2
		\|u_{\varepsilon}\|_6^{6-\mu}
		\left(
		\|U_{z_{\varepsilon},\lambda_{\varepsilon}}\|_{L^{\frac{6(5-\mu)}{3-\mu}}
		(\Omega\setminus B_{\delta_{\varepsilon}}(z_{\varepsilon}))}^{5-\mu}
		+\|v_{\varepsilon}\|_{\frac{6(5-\mu)}{3-\mu}}^{5-\mu}
		\right)\\
		\lesssim{}&
		\lambda_{\varepsilon}^{-\frac{5-\mu}{2}}
		\delta_{\varepsilon}^{-\frac{7-\mu}{2}}
		+\lambda_{\varepsilon}^{-\frac{3-\mu}{2}}.
	\end{split}
	\end{equation*}
	Choosing $\delta_{\varepsilon}=\lambda_{\varepsilon}^{-2/7}$, the
	right-hand side is
	\begin{equation*}
		\lambda_{\varepsilon}^{-\frac{21-5\mu}{14}}
		+\lambda_{\varepsilon}^{-\frac{3-\mu}{2}}
		=o\left(\lambda_{\varepsilon}^{-\frac12}\right),
	\end{equation*}
	where we used $0<\mu<2$.
	The second term on the right-hand side of \eqref{o414} is bounded by
	\begin{equation*}
		\varepsilon\left|
		\int_{\Omega}G_Q(z,x)V(x)u_{\varepsilon}(x)dx
		\right|
		\lesssim \varepsilon\|G_Q(z,\cdot)\|_2
		\left(
		\|U_{z_{\varepsilon},\lambda_{\varepsilon}}\|_2
		+\|v_{\varepsilon}\|_2
		\right)
		\lesssim \varepsilon\lambda_{\varepsilon}^{-\frac12}
		=o\left(\lambda_{\varepsilon}^{-\frac12}\right),
	\end{equation*}
	where we used Proposition \ref{prop21} and Lemma \ref{A1}.
	Combining the preceding estimates, part (b) follows.

\begin{appendix}
\section{Basic estimates}

In this appendix, we first recall that the Aubin--Talenti functions introduced in \eqref{atb}
satisfy the critical Lane--Emden--Fowler equation
\begin{equation*}
	-\Delta U_{z,\lambda}=3U_{z,\lambda}^{5}
	\qquad\text{in }\mathbb{R}^{3}.
\end{equation*}
The function $U_{z,\lambda}$ also solves the whole-space critical Hartree
equation \eqref{nonlocal critical equation}. Comparing the two equations gives
the pointwise convolution identity
\begin{equation}\label{1esay}
	\int_{\mathbb{R}^{3}}
	\frac{U_{z,\lambda}^{6-\mu}(y)}{|x-y|^{\mu}}dy
	=\frac{3}{A_{3,\mu}}U_{z,\lambda}^{\mu}(x)
	\qquad\text{in }\mathbb{R}^{3},
\end{equation}
which has been used repeatedly in the preceding computations.

The following equivalent form of the Hardy--Littlewood--Sobolev inequality  for Riesz potentials plays a important role in the
analysis of the nonlocal terms.
\begin{Prop} (See \cite{LL}.) \label{A5}
Let $1 < r<s<\infty$ and $0<\mu<3$ satisfy
\begin{equation*}
	\frac{1}{r}-\frac{1}{s}=\frac{3-\mu}{3}.
\end{equation*}
Then there exists a constant $K(\mu,r)>0$ such that, for every
$f\in L^r(\mathbb{R}^3)$,
\begin{equation*}
	\left\|\frac{1}{|x|^\mu}*f\right\|_{L^s(\mathbb{R}^3)}
	\leq K(\mu,r)\|f\|_{L^r(\mathbb{R}^3)}.
\end{equation*}
\end{Prop}

We next collect several estimates used in the preceding sections.
They are taken from the appendix of \cite{FKK3}.
\begin{lem}  \label{A1}
	Let $z \in \Omega$ and let $1 \leq q < \infty$. As $\lambda \to \infty$, we have

$$
\|U_{z,\lambda}\|_{L^q(\Omega)} \lesssim 
\begin{cases} 
\lambda^{-1/2}, & 1 \leq q < 3, \\[4pt]
\lambda^{-1/2}(\log \lambda)^{1/3}, & q = 3, \\[4pt]
\lambda^{1/2-3/q}, & q > 3.
\end{cases}
$$

Moreover, we have
$$
\partial_{z_i} U_{z,\lambda}(y) = \lambda^{5/2} \frac{y_i - z_i}{(1 + \lambda^2 |z - y|^2)^{3/2}},
$$
with
$$
\| \partial_{z_i} U_{z,\lambda} \|_{L^q(\Omega)} \lesssim 
\begin{cases} 
\lambda^{-1/2}, & 1 \leq q < \frac{3}{2}, \\[4pt]
\lambda^{-1/2}(\log \lambda)^{2/3}, & q = \frac{3}{2}, \\[4pt]
\lambda^{3/2-3/q}, & q > \frac{3}{2}, 
\end{cases}
$$
and
$$
\partial_{\lambda} U_{z,\lambda}(y) = \frac{1}{2} \lambda^{-1/2} \frac{1 - \lambda^2 |z - y|^2}{(1 + \lambda^2 |z - y|^2)^{3/2}},
$$
with
$$
\| \partial_{\lambda} U \|_{L^q} \leq \lambda^{-1} \| U \|_{L^q} \quad \text{for any } 1 \leq q \leq \infty.
$$
Moreover, for any $\rho = \rho_\lambda$ with $\rho\lambda \to \infty$,

$$
\|U_{z,\lambda}\|_{L^q(\Omega \setminus B_\rho(z))} \lesssim 
\begin{cases} 
\lambda^{-1/2}, & 1 \leq q < 3, \\[4pt]
\lambda^{-1/2}(\log \lambda)^{1/3}, & q = 3, \\[4pt]
\lambda^{-1/2}\rho^{(3-q)/q}, & q > 3,
\end{cases}
$$
and
$$
\|\partial_\lambda U_{z,\lambda}\|_{L^q(\Omega \setminus B_\rho(z))} \lesssim 
\begin{cases} 
\lambda^{-3/2}, & 1 \leq q < 3, \\[4pt]
\lambda^{-3/2}(\log \lambda)^{1/3}, & q = 3, \\[4pt]
\lambda^{-3/2}\rho^{(3-q)/q}, & q > 3,
\end{cases}
$$
and
$$
\|\partial_{z_i} U_{z,\lambda}\|_{L^q(\Omega \setminus B_\rho(z))} \lesssim 
\begin{cases} 
\lambda^{-1/2}, & 1 \leq q < \frac{3}{2}, \\[4pt]
\lambda^{-1/2}(\log \lambda)^{2/3}, & q = \frac{3}{2}, \\[4pt]
\lambda^{-1/2}\rho^{(3-2q)/q}, & q > \frac{3}{2}.
\end{cases}
$$
\end{lem}

\begin{lem} \label{A1RN}
	Let $z\in\Omega$. As $\lambda\to\infty$, we have
	\begin{equation*}
		\|U_{z,\lambda}\|_{L^q(\mathbb{R}^3)}
		\lesssim\lambda^{\frac12-\frac3q},
		\qquad 3<q<\infty,
	\end{equation*}
	\begin{equation*}
		\|\partial_{\lambda}U_{z,\lambda}\|_{L^q(\mathbb{R}^3)}
		\lesssim\lambda^{-\frac12-\frac3q},
		\qquad 3<q<\infty,
	\end{equation*}
	and
	\begin{equation*}
		\|\partial_{z_i}U_{z,\lambda}\|_{L^q(\mathbb{R}^3)}
		\lesssim\lambda^{\frac32-\frac3q},
		\qquad \frac32<q<\infty.
	\end{equation*}
	Moreover, for any $\rho=\rho_{\lambda}>0$ satisfying
	$\rho\lambda\to\infty$, we have
	\begin{equation*}
		\|U_{z,\lambda}\|_{L^q(\mathbb{R}^3\setminus B_\rho(z))}
		\lesssim\lambda^{-\frac12}\rho^{\frac{3-q}{q}},
		\qquad 3<q<\infty,
	\end{equation*}
	\begin{equation*}
		\|\partial_{\lambda}U_{z,\lambda}\|_{L^q(\mathbb{R}^3\setminus B_\rho(z))}
		\lesssim\lambda^{-\frac32}\rho^{\frac{3-q}{q}},
		\qquad 3<q<\infty,
	\end{equation*}
	and
	\begin{equation*}
		\|\partial_{z_i}U_{z,\lambda}\|_{L^q(\mathbb{R}^3\setminus B_\rho(z))}
		\lesssim\lambda^{-\frac12}\rho^{\frac{3-2q}{q}},
		\qquad \frac32<q<\infty.
	\end{equation*}
\end{lem}
\begin{proof}
	The proof is similar to that of \cite[Lemma A.1]{FKK3} and is therefore omitted.
\end{proof}

\begin{lem}\label{A2}
	Let $z\in\Omega$, $\lambda>0$ and $d=\operatorname{dist}(z,\partial\Omega)$. For the function $f_{z,\lambda}$ defined as in \eqref{o316}, we have
	\begin{equation*}
		\|f_{z,\lambda}\|_{\infty}\lesssim\lambda^{-5/2}d^{-3},\qquad
		\|\partial_\lambda f_{z,\lambda}\|_{\infty}\lesssim\lambda^{-7/2}d^{-3},\qquad
		\|\partial_{z_i}f_{z,\lambda}\|_{\infty}\lesssim\lambda^{-5/2}d^{-4}.
	\end{equation*}
	The function
	\begin{equation*}
		\varphi_{z,\lambda}:=\lambda^{-1/2}H_0(z,\cdot)+f_{z,\lambda}
		=U_{z,\lambda}-PU_{z,\lambda}
	\end{equation*}
	satisfies $0\leq\varphi_{z,\lambda}\leq U_{z,\lambda}$ and
	\begin{equation*}
		\|\varphi_{z,\lambda}\|_6\lesssim\lambda^{-1/2}d^{-1/2},\qquad
		\|\varphi_{z,\lambda}\|_\infty\lesssim\lambda^{-1/2}d^{-1}.
	\end{equation*}
	Moreover,
	\begin{equation*}
		\|\partial_\lambda\varphi_{z,\lambda}\|_6\lesssim\lambda^{-3/2}d^{-1/2},\qquad
		\|\partial_\lambda\varphi_{z,\lambda}\|_\infty\lesssim\lambda^{-3/2}d^{-1},
	\end{equation*}
	and
	\begin{equation*}
		\|\partial_{z_i}\varphi_{z,\lambda}\|_6\lesssim\lambda^{-1/2}d^{-1/2},\qquad
		\|\partial_{z_i}\varphi_{z,\lambda}\|_\infty\lesssim\lambda^{-1/2}d^{-2}.
	\end{equation*}
\end{lem}

\begin{lem}\label{A3}
	Let $z\in\Omega$, $\lambda>0$ and $d=\operatorname{dist}(z,\partial\Omega)$. As $\lambda d\to\infty$, we have
	\begin{itemize}
		\item[(a)] \quad $\displaystyle
			\int_{\partial\Omega}n\left(\frac{\partial PU_{z,\lambda}}{\partial n}\right)^2
			=C\lambda^{-1}\nabla\phi_0(z)+o(\lambda^{-1}d^{-2})
		$ for some constant $C>0$;
		\item[(b)] \quad $\displaystyle
			\int_{\partial\Omega}(x\cdot n)\left(\frac{\partial PU_{z,\lambda}}{\partial n}\right)^2
			=O(\lambda^{-1}d^{-2});$
		\item[(c)] \quad $\displaystyle
			\int_{\partial\Omega}\left(\frac{\partial PU_{z,\lambda}}{\partial n}\right)^2
			=O(\lambda^{-1}d^{-2}).$
	\end{itemize}
\end{lem}

\begin{lem}\label{A4}
	Let $g_{z,\lambda}$ be defined as in \eqref{o317}. As $\lambda\to\infty$, for every $1\leq p<3$, we have
	\begin{equation*}
		\|g_{z,\lambda}\|_{L^p(\mathbb{R}^3)}
		\lesssim\lambda^{1/2-3/p},\qquad
		\|\partial_\lambda g_{z,\lambda}\|_{L^p(\mathbb{R}^3)}
		\lesssim\lambda^{-1/2-3/p}.
	\end{equation*}
	Moreover, $\nabla g_{z,\lambda}\in L^p(\mathbb{R}^3)$ for every $1\leq p<3/2$.
\end{lem}

\section{Properties of the functions $H_b(z,x)$}
We collect some properties of the regular part of the Green function needed in the preceding arguments. They are stated for an auxiliary potential $b$ satisfying
\begin{equation*}
	b\in C(\overline{\Omega})\cap C_{\mathrm{loc}}^{2,\sigma}(\Omega)
	\qquad\text{for some }0<\sigma<1.
\end{equation*}
Throughout this appendix, the operator $-\Delta+b$ is assumed to be coercive in $\Omega$ under the Dirichlet boundary condition. In particular, the choice $b=0$ is allowed.

The following three lemmas, Lemmas \ref{B1}--\ref{B3}, are recalled from \cite[Appendix B]{FKK3}.
\begin{lem}\label{B1}
For every $z\in\Omega$,
\begin{equation*}
	\|H_b(z,\cdot)\|_{\infty}\lesssim d(z)^{-1}.
\end{equation*}
Moreover, let $z,x\in\Omega$ with $z\neq x$. Then
$\nabla_zH_b(z,x)$ and $\nabla_xH_b(z,x)$ exist and satisfy
\begin{align} 
	\sup_{x\in\Omega\setminus\{z\}}
	|\nabla_zH_b(z,x)|&\leq C, \label{oB2}\\
	\sup_{x\in\Omega\setminus\{z\}}
	|\nabla_xH_b(z,x)|&\leq C, \label{oB3}
\end{align}
where $C$ can be chosen uniformly for $z$ in compact subsets of $\Omega$.
\end{lem}

\begin{lem}\label{B2}
Let $0<\nu<1$. As $x\to z$, uniformly for $z$ in compact subsets of $\Omega$,
\begin{equation}\label{oB5}
	H_b(z,x)=\phi_b(z)+\frac{1}{2}\nabla\phi_b(z)\cdot(x-z)
	-\frac{1}{2}b(z)|x-z|+O\left(|x-z|^{1+\nu}\right). 
\end{equation}
\end{lem}

\begin{lem}\label{B3}
For every $z\in\Omega$, as $\lambda\to\infty$,
\begin{align}
	\int_{\Omega}U_{z,\lambda}^{5}H_b(z,\cdot)
	&=\frac{4\pi}{3}\phi_b(z)\lambda^{-1/2}
	-\frac{4\pi}{3}b(z)\lambda^{-3/2}
	+o\left(\lambda^{-3/2}\right), \label{oB10}\\
	\int_{\Omega}U_{z,\lambda}^{4}\partial_{\lambda}U_{z,\lambda}H_b(z,\cdot)
	&=-\frac{2\pi}{15}\phi_b(z)\lambda^{-3/2}
	+\frac{2\pi}{5}b(z)\lambda^{-5/2}
	+o\left(\lambda^{-5/2}\right), \label{oB11}\\
	\int_{\Omega}U_{z,\lambda}^{4}\partial_{z_i}U_{z,\lambda}H_b(z,\cdot)
	&=\frac{2\pi}{15}\partial_{z_i}\phi_b(z)\lambda^{-1/2}
	+o\left(\lambda^{-1/2}\right), \label{oB12}\\
	\int_{\Omega}U_{z,\lambda}^{4}H_b(z,\cdot)^2
	&=\pi^2\phi_b(z)^2\lambda^{-1}
	+o\left(\lambda^{-1}\right), \label{oB13}\\
	\int_{\Omega}U_{z,\lambda}^{3}\partial_{\lambda}U_{z,\lambda}H_b(z,\cdot)^2
	&=-\frac{\pi^2}{4}\phi_b(z)^2\lambda^{-2}
	+o\left(\lambda^{-2}\right). \label{oB14}
\end{align}
The implied constants can be chosen uniformly for $z$ in compact subsets of $\Omega$.
\end{lem}

\begin{lem}\label{B4}
For every $z\in\Omega$, as $\lambda\to\infty$, the following asymptotic expansions hold:
\begin{equation}\label{B4-1}\small
\begin{split}
	\int_{\Omega}\int_{\Omega}
	\frac{U_{z,\lambda}^{5-\mu}(y)H_b(z,y)
	U_{z,\lambda}^{5-\mu}(x)\partial_{\lambda}U_{z,\lambda}(x)}
	{|x-y|^{\mu}}dxdy
	=-\frac{2\pi\mu}{5A_{3,\mu}(6-\mu)}
	\phi_b(z)\lambda^{-3/2}
	+\frac{6\pi\mu}{5A_{3,\mu}(6-\mu)}
	b(z)\lambda^{-5/2}
	+o\left(\lambda^{-5/2}\right),
\end{split}
\end{equation}
\begin{equation}\label{B4-2}\small
\begin{split}
	\int_{\Omega}\int_{\Omega}
	\frac{U_{z,\lambda}^{6-\mu}(y)
	U_{z,\lambda}^{4-\mu}(x)H_b(z,x)
	\partial_{\lambda}U_{z,\lambda}(x)}
	{|x-y|^{\mu}}dxdy
	=-\frac{2\pi}{5A_{3,\mu}}
	\phi_b(z)\lambda^{-3/2}
	+\frac{6\pi}{5A_{3,\mu}}
	b(z)\lambda^{-5/2}
	+o\left(\lambda^{-5/2}\right),
\end{split}
\end{equation}
\begin{equation}\label{B4-3}
	\int_{\Omega}\int_{\Omega}
	\frac{U_{z,\lambda}^{4-\mu}(y)H_b(z,y)^2
	U_{z,\lambda}^{5-\mu}(x)\partial_{\lambda}U_{z,\lambda}(x)}
	{|x-y|^{\mu}}dxdy
	=-\frac{3\pi^2\mu}{4A_{3,\mu}(6-\mu)}
	\phi_b(z)^2\lambda^{-2}
	+o\left(\lambda^{-2}\right),
\end{equation}
{\small
\begin{equation}\label{B4-4}
	\int_{\Omega}\int_{\Omega}
	\frac{U_{z,\lambda}^{5-\mu}(y)H_b(z,y)
	U_{z,\lambda}^{4-\mu}(x)H_b(z,x)
	\partial_{\lambda}U_{z,\lambda}(x)}
	{|x-y|^{\mu}}dxdy
	=-\frac{d_4}{2A_{3,\mu}(5-\mu)}
	\phi_b(z)^2\lambda^{-2}
	+o\left(\lambda^{-2}\right),
\end{equation}
}
\begin{equation}\label{B4-5}
	\int_{\Omega}\int_{\Omega}
	\frac{U_{z,\lambda}^{6-\mu}(y)
	U_{z,\lambda}^{3-\mu}(x)H_b(z,x)^2
	\partial_{\lambda}U_{z,\lambda}(x)}
	{|x-y|^{\mu}}dxdy
	=-\frac{3\pi^2}{4A_{3,\mu}}
	\phi_b(z)^2\lambda^{-2}
	+o\left(\lambda^{-2}\right),
\end{equation}
\begin{equation}\label{B4-6}
	\int_{\Omega}\int_{\Omega}
	\frac{U_{z,\lambda}^{5-\mu}(y)H_b(z,y)
	U_{z,\lambda}^{5-\mu}(x)H_b(z,x)}
	{|x-y|^{\mu}}dxdy
	=\frac{d_4}{A_{3,\mu}}
	\phi_b(z)^2\lambda^{-1}
	+o\left(\lambda^{-1}\right),
\end{equation}
\begin{equation}\label{B4-7}
	\int_{\Omega}\int_{\Omega}
	\frac{U_{z,\lambda}^{6-\mu}(y)
	U_{z,\lambda}^{4-\mu}(x)H_b(z,x)^2}
	{|x-y|^{\mu}}dxdy
	=\frac{3\pi^2}{A_{3,\mu}}
	\phi_b(z)^2\lambda^{-1}
	+o\left(\lambda^{-1}\right),
\end{equation}
where $$d_4=A_{3,\mu}\int_{\mathbb{R}^3}\int_{\mathbb{R}^3}
	\frac{U_{0,1}^{5-\mu}(y)U_{0,1}^{5-\mu}(x)}
	{|x-y|^{\mu}}dxdy$$  and the implied constants can be chosen uniformly for $z$ in compact subsets of $\Omega$.
\end{lem}

\begin{proof}
We first prove \eqref{B4-1}. By extending the domain of integration with respect to $x$ to $\mathbb{R}^3$, we have
\begin{equation*}
\begin{split}
	&\int_{\Omega}\int_{\Omega}
	\frac{U_{z,\lambda}^{5-\mu}(y)H_b(z,y)
	U_{z,\lambda}^{5-\mu}(x)\partial_{\lambda}U_{z,\lambda}(x)}
	{|x-y|^{\mu}}dxdy\\
	&
	=\int_{\Omega}\int_{\mathbb{R}^3}
	\frac{U_{z,\lambda}^{5-\mu}(y)H_b(z,y)
	U_{z,\lambda}^{5-\mu}(x)\partial_{\lambda}U_{z,\lambda}(x)}
	{|x-y|^{\mu}}dxdy-\int_{\Omega}\int_{\mathbb{R}^3\setminus\Omega}
	\frac{U_{z,\lambda}^{5-\mu}(y)H_b(z,y)
	U_{z,\lambda}^{5-\mu}(x)\partial_{\lambda}U_{z,\lambda}(x)}
	{|x-y|^{\mu}}dxdy.
\end{split}
\end{equation*}
Differentiating \eqref{1esay} with respect to $\lambda$, we obtain
\begin{equation}\label{B-convolution}
	\int_{\mathbb{R}^3}
	\frac{U_{z,\lambda}^{5-\mu}(x)\partial_{\lambda}U_{z,\lambda}(x)}
	{|x-y|^{\mu}}dx
	=\frac{3\mu}{A_{3,\mu}(6-\mu)}
	U_{z,\lambda}^{\mu-1}(y)\partial_{\lambda}U_{z,\lambda}(y).
\end{equation}
Therefore, by Lemma \ref{B3},
\begin{align*}
	&\int_{\Omega}\int_{\mathbb{R}^3}
	\frac{U_{z,\lambda}^{5-\mu}(y)H_b(z,y)
	U_{z,\lambda}^{5-\mu}(x)\partial_{\lambda}U_{z,\lambda}(x)}
	{|x-y|^{\mu}}dxdy\\
	={}&\frac{3\mu}{A_{3,\mu}(6-\mu)}
	\int_{\Omega}U_{z,\lambda}^{4}
	\partial_{\lambda}U_{z,\lambda}H_b(z,\cdot)\\
	={}&-\frac{2\pi\mu}{5A_{3,\mu}(6-\mu)}
	\phi_b(z)\lambda^{-3/2}
	+\frac{6\pi\mu}{5A_{3,\mu}(6-\mu)}
	b(z)\lambda^{-5/2}
	+o\left(\lambda^{-5/2}\right).
\end{align*}
On the other hand, combining the HLS inequality with Lemmas \ref{A1},
\ref{A1RN}, and \ref{B1}, we obtain
\begin{align*}
	&\left|
	\int_{\Omega}\int_{\mathbb{R}^3\setminus\Omega}
	\frac{U_{z,\lambda}^{5-\mu}(y)H_b(z,y)
	U_{z,\lambda}^{5-\mu}(x)\partial_{\lambda}U_{z,\lambda}(x)}
	{|x-y|^{\mu}}dxdy
	\right|\\
	&\lesssim
	\left\|U_{z,\lambda}^{5-\mu}H_b(z,\cdot)\right\|_{\frac{6}{6-\mu}}
	\left\|U_{z,\lambda}^{5-\mu}\partial_{\lambda}U_{z,\lambda}
	\right\|_{L^{\frac{6}{6-\mu}}(\mathbb{R}^3\setminus\Omega)}\\
	&\lesssim
	\lambda^{-1/2}\lambda^{-1-\frac{6-\mu}{2}}
	=\lambda^{-\frac{9-\mu}{2}}
	=o\left(\lambda^{-5/2}\right).
\end{align*}
It then follows that \eqref{B4-1} holds.

Similarly, for \eqref{B4-2}, we have
\begin{equation*}
\begin{split}
	&\int_{\Omega}\int_{\Omega}
	\frac{U_{z,\lambda}^{6-\mu}(y)
	U_{z,\lambda}^{4-\mu}(x)H_b(z,x)
	\partial_{\lambda}U_{z,\lambda}(x)}
	{|x-y|^{\mu}}dxdy\\
	&=\int_{\mathbb{R}^3}\int_{\Omega}
	\frac{U_{z,\lambda}^{6-\mu}(y)
	U_{z,\lambda}^{4-\mu}(x)H_b(z,x)
	\partial_{\lambda}U_{z,\lambda}(x)}
	{|x-y|^{\mu}}dxdy-\int_{\mathbb{R}^3\setminus\Omega}\int_{\Omega}
	\frac{U_{z,\lambda}^{6-\mu}(y)
	U_{z,\lambda}^{4-\mu}(x)H_b(z,x)
	\partial_{\lambda}U_{z,\lambda}(x)}
	{|x-y|^{\mu}}dxdy.
\end{split}
\end{equation*}
By \eqref{1esay} and Lemma \ref{B3}, we obtain
\begin{align*}
	&\int_{\mathbb{R}^3}\int_{\Omega}
	\frac{U_{z,\lambda}^{6-\mu}(y)
	U_{z,\lambda}^{4-\mu}(x)H_b(z,x)
	\partial_{\lambda}U_{z,\lambda}(x)}
	{|x-y|^{\mu}}dxdy\\
	={}&\frac{3}{A_{3,\mu}}
	\int_{\Omega}U_{z,\lambda}^{4}
	\partial_{\lambda}U_{z,\lambda}H_b(z,\cdot)\\
	={}&-\frac{2\pi}{5A_{3,\mu}}
	\phi_b(z)\lambda^{-3/2}
	+\frac{6\pi}{5A_{3,\mu}}
	b(z)\lambda^{-5/2}
	+o\left(\lambda^{-5/2}\right).
\end{align*}
For the remaining term, we have
\begin{align*}
	&\left|
	\int_{\mathbb{R}^3\setminus\Omega}\int_{\Omega}
	\frac{U_{z,\lambda}^{6-\mu}(y)
	U_{z,\lambda}^{4-\mu}(x)H_b(z,x)
	\partial_{\lambda}U_{z,\lambda}(x)}
	{|x-y|^{\mu}}dxdy
	\right|\\
	&\lesssim
	\left\|U_{z,\lambda}^{6-\mu}\right\|_
	{L^{\frac{6}{6-\mu}}(\mathbb{R}^3\setminus\Omega)}
	\left\|U_{z,\lambda}^{4-\mu}H_b(z,\cdot)
	\partial_{\lambda}U_{z,\lambda}\right\|_{\frac{6}{6-\mu}}\\
	&\lesssim
	\lambda^{-\frac{6-\mu}{2}}\lambda^{-3/2}
	=\lambda^{-\frac{9-\mu}{2}}
	=o\left(\lambda^{-5/2}\right).
\end{align*}
It then follows that \eqref{B4-2} holds.

For \eqref{B4-3}, using \eqref{B-convolution} and Lemma \ref{B3}, we have
\begin{align*}
	&\int_{\Omega}\int_{\mathbb{R}^3}
	\frac{U_{z,\lambda}^{4-\mu}(y)H_b(z,y)^2
	U_{z,\lambda}^{5-\mu}(x)\partial_{\lambda}U_{z,\lambda}(x)}
	{|x-y|^{\mu}}dxdy\\
	&=\frac{3\mu}{A_{3,\mu}(6-\mu)}
	\int_{\Omega}U_{z,\lambda}^{3}
	\partial_{\lambda}U_{z,\lambda}H_b(z,\cdot)^2\\
	&=-\frac{3\pi^2\mu}{4A_{3,\mu}(6-\mu)}
	\phi_b(z)^2\lambda^{-2}
	+o\left(\lambda^{-2}\right).
\end{align*}
Moreover,
\begin{align*}
	&\left|
	\int_{\Omega}\int_{\mathbb{R}^3\setminus\Omega}
	\frac{U_{z,\lambda}^{4-\mu}(y)H_b(z,y)^2
	U_{z,\lambda}^{5-\mu}(x)\partial_{\lambda}U_{z,\lambda}(x)}
	{|x-y|^{\mu}}dxdy
	\right|\\
	&\lesssim
	\left\|U_{z,\lambda}^{4-\mu}H_b(z,\cdot)^2\right\|_{\frac{6}{6-\mu}}
	\left\|U_{z,\lambda}^{5-\mu}\partial_{\lambda}U_{z,\lambda}
	\right\|_{L^{\frac{6}{6-\mu}}(\mathbb{R}^3\setminus\Omega)}\\
	&\lesssim
	\lambda^{-1}\lambda^{-1-\frac{6-\mu}{2}}
	=o\left(\lambda^{-2}\right).
\end{align*}
This proves \eqref{B4-3}.

For \eqref{B4-4}, let $0<\rho<d(z)$. Splitting the domain of integration and using the HLS inequality together with Lemmas \ref{A1} and \ref{B1}, we obtain
{\small
\begin{equation*}
\begin{split}
	&\int_{\Omega}\int_{\Omega}
	\frac{U_{z,\lambda}^{5-\mu}(y)H_b(z,y)
	U_{z,\lambda}^{4-\mu}(x)H_b(z,x)
	\partial_{\lambda}U_{z,\lambda}(x)}
	{|x-y|^{\mu}}dxdy\\
	&=\int_{B_\rho(z)}\int_{B_\rho(z)}
	\frac{U_{z,\lambda}^{5-\mu}(y)H_b(z,y)
	U_{z,\lambda}^{4-\mu}(x)H_b(z,x)
	\partial_{\lambda}U_{z,\lambda}(x)}
	{|x-y|^{\mu}}dxdy+o\left(\lambda^{-2}\right).
\end{split}
\end{equation*}
}
Moreover, by the HLS inequality and Lemma \ref{B2}, we have
{\small
\begin{equation*}
\begin{split}
	&\int_{B_\rho(z)}\int_{B_\rho(z)}
	\frac{U_{z,\lambda}^{5-\mu}(y)H_b(z,y)
	U_{z,\lambda}^{4-\mu}(x)H_b(z,x)
	\partial_{\lambda}U_{z,\lambda}(x)}
	{|x-y|^{\mu}}dxdy\\
	={}&\phi_b(z)^2\int_{B_\rho(z)}\int_{B_\rho(z)}
	\frac{U_{z,\lambda}^{5-\mu}(y)U_{z,\lambda}^{4-\mu}(x)
	\partial_{\lambda}U_{z,\lambda}(x)}
	{|x-y|^{\mu}}dxdy\\
	&+O\Bigg(\int_{B_\rho(z)}\int_{B_\rho(z)}
	\frac{U_{z,\lambda}^{5-\mu}(y)U_{z,\lambda}^{4-\mu}(x)
	|\partial_{\lambda}U_{z,\lambda}(x)|
	\left(|x-z|+|y-z|\right)}
	{|x-y|^{\mu}}dxdy\Bigg)\\
	&+O\Bigg(\int_{B_\rho(z)}\int_{B_\rho(z)}
	\frac{U_{z,\lambda}^{5-\mu}(y)U_{z,\lambda}^{4-\mu}(x)
	|\partial_{\lambda}U_{z,\lambda}(x)|
	|x-z||y-z|}
	{|x-y|^{\mu}}dxdy\Bigg).
\end{split}
\end{equation*}
}
By Lemma \ref{A1} and a direct computation, we have
\begin{equation*}
	\left\|U_{z,\lambda}^{5-\mu}\right\|_{\frac{6}{6-\mu}}
	\lesssim\lambda^{-1/2},
	\qquad
	\left\|U_{z,\lambda}^{5-\mu}|\cdot-z|\right\|_{\frac{6}{6-\mu}}
	\lesssim\lambda^{-3/2}.
\end{equation*}
Since
\begin{equation*}
	\left|\partial_{\lambda}U_{z,\lambda}\right|
	\lesssim\lambda^{-1}U_{z,\lambda},
\end{equation*}
the last two terms are bounded by $O\left(\lambda^{-3}\right)
	=o\left(\lambda^{-2}\right)$. 
It follows that
\begin{equation*}
\begin{split}
	&\int_{\Omega}\int_{\Omega}
	\frac{U_{z,\lambda}^{5-\mu}(y)H_b(z,y)
	U_{z,\lambda}^{4-\mu}(x)H_b(z,x)
	\partial_{\lambda}U_{z,\lambda}(x)}
	{|x-y|^{\mu}}dxdy\\
	&=\phi_b(z)^2\int_{B_\rho(z)}\int_{B_\rho(z)}
	\frac{U_{z,\lambda}^{5-\mu}(y)U_{z,\lambda}^{4-\mu}(x)
	\partial_{\lambda}U_{z,\lambda}(x)}
	{|x-y|^{\mu}}dxdy+o\left(\lambda^{-2}\right).
\end{split}
\end{equation*}
Applying the HLS inequality and Lemma \ref{A1RN} gives
\begin{equation*}
\begin{split}
	&\int_{B_\rho(z)}\int_{B_\rho(z)}
	\frac{U_{z,\lambda}^{5-\mu}(y)U_{z,\lambda}^{4-\mu}(x)
	\partial_{\lambda}U_{z,\lambda}(x)}
	{|x-y|^{\mu}}dxdy=\int_{\mathbb{R}^3}\int_{\mathbb{R}^3}
	\frac{U_{z,\lambda}^{5-\mu}(y)U_{z,\lambda}^{4-\mu}(x)
	\partial_{\lambda}U_{z,\lambda}(x)}
	{|x-y|^{\mu}}dxdy+o\left(\lambda^{-2}\right).
\end{split}
\end{equation*}
Set
\begin{equation}\label{B-d4}
	d_4:=A_{3,\mu}\int_{\mathbb{R}^3}\int_{\mathbb{R}^3}
	\frac{U_{0,1}^{5-\mu}(y)U_{0,1}^{5-\mu}(x)}
	{|x-y|^{\mu}}dxdy.
\end{equation}
Then
\begin{equation*}
	\int_{\mathbb{R}^3}\int_{\mathbb{R}^3}
	\frac{U_{z,\lambda}^{5-\mu}(y)U_{z,\lambda}^{5-\mu}(x)}
	{|x-y|^{\mu}}dxdy
	=\frac{d_4}{A_{3,\mu}}\lambda^{-1}.
\end{equation*}
Differentiating the last identity with respect to $\lambda$ and using symmetry, we obtain
\begin{equation*}
\begin{split}
	&\int_{\mathbb{R}^3}\int_{\mathbb{R}^3}
	\frac{U_{z,\lambda}^{5-\mu}(y)U_{z,\lambda}^{4-\mu}(x)
	\partial_{\lambda}U_{z,\lambda}(x)}
	{|x-y|^{\mu}}dxdy=-\frac{d_4}{2A_{3,\mu}(5-\mu)}\lambda^{-2}.
\end{split}
\end{equation*}
Combining the estimates above, we obtain \eqref{B4-4}.

For \eqref{B4-5}, proceeding as in the proof of \eqref{B4-2} and using \eqref{1esay}, we obtain
\begin{equation*}
\begin{split}
	&\int_{\Omega}\int_{\Omega}
	\frac{U_{z,\lambda}^{6-\mu}(y)U_{z,\lambda}^{3-\mu}(x)
	H_b(z,x)^2\partial_{\lambda}U_{z,\lambda}(x)}
	{|x-y|^{\mu}}dxdy
	=\frac{3}{A_{3,\mu}}
	\int_{\Omega}U_{z,\lambda}^{3}
	H_b(z,\cdot)^2\partial_{\lambda}U_{z,\lambda}
	+o\left(\lambda^{-2}\right).
\end{split}
\end{equation*}
Hence, by \eqref{oB14}, we obtain \eqref{B4-5}.

For \eqref{B4-6}, proceeding as in the proof of \eqref{B4-4}, we obtain
\begin{equation*}
\begin{split}
	&\int_{\Omega}\int_{\Omega}
	\frac{U_{z,\lambda}^{5-\mu}(y)H_b(z,y)
	U_{z,\lambda}^{5-\mu}(x)H_b(z,x)}
	{|x-y|^{\mu}}dxdy=\phi_b(z)^2
	\int_{\mathbb{R}^3}\int_{\mathbb{R}^3}
	\frac{U_{z,\lambda}^{5-\mu}(y)U_{z,\lambda}^{5-\mu}(x)}
	{|x-y|^{\mu}}dxdy
	+o\left(\lambda^{-1}\right).
\end{split}
\end{equation*}
By the definition of $d_4$,
\begin{equation*}
	\int_{\mathbb{R}^3}\int_{\mathbb{R}^3}
	\frac{U_{z,\lambda}^{5-\mu}(y)U_{z,\lambda}^{5-\mu}(x)}
	{|x-y|^{\mu}}dxdy
	=\frac{d_4}{A_{3,\mu}}\lambda^{-1}.
\end{equation*}
Consequently,
\begin{equation*}
\begin{split}
	&\int_{\Omega}\int_{\Omega}
	\frac{U_{z,\lambda}^{5-\mu}(y)H_b(z,y)
	U_{z,\lambda}^{5-\mu}(x)H_b(z,x)}
	{|x-y|^{\mu}}dxdy=\frac{d_4}{A_{3,\mu}}
	\phi_b(z)^2\lambda^{-1}
	+o\left(\lambda^{-1}\right),
\end{split}
\end{equation*}
which proves \eqref{B4-6}.

For \eqref{B4-7}, proceeding as in the proof of \eqref{B4-2} and using \eqref{1esay}, we obtain
\begin{equation*}
	\int_{\Omega}\int_{\Omega}
	\frac{U_{z,\lambda}^{6-\mu}(y)U_{z,\lambda}^{4-\mu}(x)H_b(z,x)^2}
	{|x-y|^{\mu}}dxdy
	=\frac{3}{A_{3,\mu}}\int_{\Omega}U_{z,\lambda}^{4}H_b(z,\cdot)^2
	+o\left(\lambda^{-1}\right).
\end{equation*}
Hence, by \eqref{oB13}, we obtain \eqref{B4-7}. This completes the proof.
\end{proof}

\end{appendix}

\vspace{0.5cm}
\noindent \textbf{Data Availability.} Data sharing is not applicable to this article as no datasets were generated or analyzed during the current study.

\vspace{0.5cm}
\noindent \textbf{Conflict of interest.} The authors have no competing interests to declare that are relevant to the content of this article.

\end{document}